\documentclass[10pt]{amsart}
\usepackage{amsmath}
\usepackage{amsxtra}
\usepackage{amscd}
\usepackage{amsthm}
\usepackage{amsfonts}
\usepackage{amssymb}
\usepackage{eucal}
\usepackage[all]{xy}
\usepackage{graphicx}

\usepackage[usenames]{color}

\newtheorem{prop}[subsubsection]{Proposition}

\newtheorem{cor}[subsubsection]{Corollary}
\newtheorem{lem}[subsubsection]{Lemma}
\newtheorem{defn}[subsubsection]{Definition}

\newtheorem{thm}[subsubsection]{Theorem}

\theoremstyle{remark}
\newtheorem{rem}[subsubsection]{Remark}
\newtheorem{example}[subsubsection]{Example}

\numberwithin{equation}{section}

\newcommand{\lemref}[1]{Lemma~\ref{#1}}
\newcommand{\thmref}[1]{Theorem~\ref{#1}}
\newcommand{\secref}[1]{Sect.~\ref{#1}}
\newcommand{\corref}[1]{Corollary~\ref{#1}}
\newcommand{\propref}[1]{Proposition~\ref{#1}}

\newcommand{\nc}{\newcommand}

\nc{\ssec}{\subsection}
\nc{\sssec}{\subsubsection}

\newcommand{\iso}{\buildrel{\sim}\over{\longrightarrow}}

\nc{\renc}{\renewcommand}
\nc{\on}{\operatorname}
\nc\ol{\overline}
\nc\wt{\widetilde}
\nc\wh{\widehat}
\nc{\Loc}{\on{Loc}}
\nc{\Bun}{\on{Bun}}
\nc{\BZ}{{\mathbb{Z}}}
\nc{\BQ}{{\mathbb{Q}}}
\nc{\BA}{{\mathbb{A}}}
\nc{\BC}{{\mathbb{C}}}
\nc{\BH}{{\mathbb{H}}}
\nc{\BG}{{\mathbb{G}}}
\nc{\BK}{{\mathbb{K}}}
\nc{\BN}{{\mathbb{N}}}
\nc{\BD}{{\mathbb{D}}}
\nc{\BV}{{\mathbb{V}}}
\nc{\BL}{{\mathbb{L}}}
\nc{\CA}{{\mathcal{A}}}
\nc{\CC}{{\mathcal{C}}}
\nc{\CE}{{\mathcal{E}}}
\nc{\CG}{{\mathcal{G}}}
\nc{\CK}{{\mathcal{K}}}
\nc{\CO}{{\mathcal{O}}}
\nc{\CP}{{\mathcal{P}}}
\nc{\CR}{{\mathcal{R}}}
\nc{\CT}{{\mathcal{T}}}
\nc{\CW}{{\mathcal{W}}}
\nc{\CM}{{\mathcal{M}}}
\nc{\CN}{{\mathcal{N}}}
\nc{\CL}{{\mathcal{L}}}
\nc{\CF}{{\mathcal{F}}}
\nc{\CX}{{\mathcal{X}}}
\nc{\CY}{{\mathcal{Y}}}
\nc{\CZ}{{\mathcal{Z}}}

\nc{\D}{{\mathcal{D}}}
\nc{\fg}{{\mathfrak{g}}}
\nc{\fD}{{\mathfrak{D}}}
\nc{\fh}{{\mathfrak{h}}}
\nc{\fn}{{\mathfrak{n}}}
\nc{\sM}{{\mathsf M}}
\nc{\ppart}{(\!(t)\!)}
\nc{\hg}{{\widehat\fg}}
\nc{\sF}{{\mathsf F}}
\nc{\sG}{{\mathsf G}}

\nc{\bC}{{\mathbf{C}}}
\nc{\bZ}{{\mathbf{Z}}}
\nc{\bD}{{\mathbf{D}}}
\nc{\bO}{{\mathbf{O}}}
\nc{\bU}{{\mathbf{U}}}
\nc{\bc}{{\mathbf{c}}}
\nc{\be}{{\mathbf{e}}}
\nc{\bM}{{\mathbf{M}}}
\nc{\bA}{{\mathbf{A}}}
\nc{\bK}{{\mathbf{K}}}

\nc{\fW}{{\mathfrak{W}}}
\nc{\reg}{{\text{\rm reg}}}
\nc{\nilp}{{\text{\rm nilp}}}
\nc{\cG}{{\check{G}}}
\nc{\cB}{{\check{B}}}
\nc{\cg}{{\check{\fg}}}
\nc{\cb}{{\check{\fb}}}
\nc{\cn}{{\check{\fn}}}
\nc{\mer}{{\on{mer}}}
\nc{\Const}{\mathsf{Const}}
\nc{\Whit}{\on{Whit}}
\nc{\KL}{\on{KL}}
\nc{\FS}{\on{FS}}
\nc{\LocSys}{\on{LocSys}}
\nc{\QCoh}{\on{QCoh}}
\nc{\Coh}{\on{Coh}}
\nc{\IndCoh}{\on{IndCoh}}
\nc{\Cat}{\on{Cat}}
\nc{\Op}{\on{Op}}
\nc{\Gr}{\on{Gr}}
\nc{\Fl}{\on{Fl}}
\nc{\Rep}{\on{Rep}}
\renc{\mod}{{\on{-mod}}}
\nc{\Conn}{\on{Conn}}
\nc{\unit}{{\mathbf{1}}}

\nc{\Ho}{\on{Ho}}
\nc{\Hom}{\on{Hom}}
\nc{\rank}{\on{rank}}
\nc{\End}{\on{End}}
\nc{\Ext}{\on{Ext}}
\nc{\Vect}{\on{Vect}}
\nc{\Av}{\on{Av}}
\nc{\id}{\on{id}}
\nc{\Ind}{\on{Ind}}
\nc{\Spec}{\on{Spec}}
\nc{\KG}{K\backslash G}
\nc{\comult}{{co\text{-}mult}}
\nc{\counit}{{co\text{-}unit}}
\nc{\uHom}{{\underline{\Hom}}}
\nc{\Sch}{\on{Sch}}
\nc{\dgSch}{\on{DGSch}}
\nc{\affSch}{Sch^{aff}}
\nc{\affdgSch}{DGSch^{aff}}
\nc{\Groupoids}{Grpd}
\nc{\inftygroup}{\infty\on{-Grpd}}
\nc{\inftyCat}{\infty\on{-}\on{Cat}}
\nc{\StinftyCat}{\on{DGCat}}
\nc{\MoninftyCat}{\infty\on{-}\on{Cat}^{Mon}}
\nc{\SymMoninftyCat}{\infty\on{-}\on{Cat}^{SymMon}}
\nc{\SymMonStinftyCat}{\on{DGCat}^{SymMon}}
\nc{\MonStinftyCat}{\on{DGCat}^{Mon}}
\nc{\inftystack}{\infty\on{-}Stk}
\nc{\inftystackalg}{\infty\on{-}Stk^{1\text{-}alg}}
\nc{\inftyprestack}{\infty\on{-}preStk}
\nc{\inftydgstack}{\infty\on{-}DGStk}
\nc{\inftydgstackalg}{\infty\on{-}DGStk^{1\text{-}alg}}
\nc{\inftydgprestack}{\infty\on{-}DGpreStk}
\nc{\mmod}{{\on{-}{\mathbf{mod}}}}

\nc{\oZ}{\overset{\circ}{Z}{}}
\nc{\of}{\overset{\circ}{f}{}}
\nc{\oCF}{\overset{\circ}{\CF}{}}
\nc{\oCY}{\overset{\circ}{\CY}{}}
\nc{\oX}{\overset{\circ}X{}}

\nc{\ind}{{\mathbf{ind}}}
\nc{\oblv}{{\mathbf{oblv}}}
\nc{\Oblv}{{\mathbf{Oblv}}}
\nc{\one}{{\mathbf{one}}}

\nc{\dr}{{\on{dR},*}}
\nc{\rd}{{\on{ren-dR}}}
\nc{\Dmod}{\on{D-mod}}

\nc{\lleft}{\on{left}}
\nc{\rright}{\on{right}}

\nc{\sotimes}{\overset{!}\otimes}

\DeclareMathOperator{\colim}{{colim}}
\DeclareMathOperator{\Cone}{{Cone}}
\DeclareMathOperator{\cd}{{cd}}
\nc{\CMaps}{{\mathcal Maps}}

\begin{document}

\title[Finiteness questions]{On some finiteness questions for algebraic stacks}

\author{Vladimir Drinfeld and Dennis Gaitsgory}

\date{\today}

\begin{abstract}
%We establish certain crucial (in)finiteness properties of the suburban railway system.
We prove that under a certain mild hypothesis, the DG category of D-modules on a quasi-compact
algebraic stack is compactly generated. We also show that under the same hypothesis, the
functor of global sections on the DG category of quasi-coherent sheaves is continuous.
\end{abstract}

\maketitle

\tableofcontents

\section*{Introduction}

%The goal of this paper is to establish certain crucial (in)finiteness properties of the suburban railway system.

%\ssec{Some notation}   \label{ss:some_notation}
%Let us introduce some notation. Let $\CY$ be a algebraic stack of finite type over a field $k$ of characteristic $0$.
%Let $\QCoh(\CY)$ denote the DG category of quasi-coherent sheaves on $\CY$
%(see Definition~\ref{sss:QCohdef}), a.k.a. the enhanced quasi-coherent derived category. The symbol $\on{D-mod}(\CY)$ denotes the DG category of 
%D-modules on $\CY$ 
%(see Sect.~\ref{ss:Dmods on prestacks} for a review of the definition). Unlike $\QCoh(\CY)$, the
%DG category $\on{D-mod}(\CY)$ depends only on the classical stack underlying the
%stack~$\CY$. For the duration of the introduction, we shall assume that our algebraic stacks are locally of finite type.

\ssec{Introduction to the introduction}

This paper arose from an attempt to answer the following question: let $\CY$ be a quasi-compact
algebraic stack over a field $k$ of characteristic $0$; is it true that the DG category 
of D-modules on $\CY$, denoted $\on{D-mod}(\CY)$, is compactly generated? %And if yes, what are its compact objects?

\medskip

We should remark that we did not pursue the above question out of pressing practical reasons:
most (if not all) algebraic stacks that one encounters in practice are \emph{perfect}
%\footnote{A stack is said to be \emph{perfect} if it has affine diagonal and $\QCoh(\CY)$ is
%compactly generated by $\QCoh(\CY)^{\on{perf}}$, i.e., by the full subcategory of perfect complexes. 
in the sense of \cite{BFN}, and in this case the compact generation assertion is easy to
prove and probably well-known. According to \cite[Sect. 3.3]{BFN}, the class of perfect stacks is 
quite large. We decided to analyze the case of a general quasi-compact stack for aesthetic reasons. 

\sssec{}

Before we proceed any further let us explain why one should care about such questions as compact generation
of a given DG category, and a description of its compact objects. 

\medskip

First, we should specify what is the world of DG categories that we work in. The world in question is that of 
DG categories that are cocomplete and continuous functors between them, see \secref{sss:DG categories} 
for a brief review. The choice of this particular paradigm for DG categories appears to be a convenient framework in which to study 
various categorical aspects of algebraic geometry. 

\medskip

Compactness (resp., compact generation) are properties of an object in a given cocomplete DG category
(resp., of a DG category). The relevance and usefulness of these notions in algebraic geometry 
was first brought to light in the paper of Thomason and Trobaugh, \cite{TT}.

\medskip

The reasons for the importance of these notions can be summarized as follows: compact objects 
are those for which we can compute (or say something about) Hom out of them; and compactly generated 
categories are those for which we can compute (or say something about) continuous functors out of them. 

\sssec{}  \label{sss:results}

The new results proved in the present paper fall into three distinct groups.

\medskip

\noindent (i) Results about D-modules, that we originally started from, but which we treat last in the paper.

\medskip

\noindent (ii) Results about the DG category of quasi-coherent sheaves on $\CY$, denoted $\QCoh(\CY)$,
which are the most basic, and which are treated first.

\medskip

\noindent (iii) Results about yet another category, namely, $\IndCoh(\CY)$, which forms
a bridge between $\QCoh(\CY)$ and $\Dmod(\CY)$.

\sssec{}  \label{sss:logical structure}

The logical structure of the paper is as follows: 

\medskip

Whatever we prove about $\QCoh(\CY)$ 
will easily imply the relevant results about $\IndCoh(\CY)$: for algebraic stacks the
latter category differs only slightly from the former one. 

\medskip

The results about $\Dmod(\CY)$ are 
deduced from those about $\IndCoh(\CY)$ using a conservative forgetful functor 
$\oblv_{\Dmod(\CY)}:\Dmod(\CY)\to \IndCoh(\CY)$, which admits a left adjoint. 

\sssec{}

There is essentially only one piece of technology used in the proofs of all the main results: we stratify
a given algebraic stack $\CY$ by locally closed substacks, which are essentially of the form 
$Z/G$, where $Z$ is a quasi-compact scheme and $G$ an algebraic group acting on it. 

\sssec{}

Finally, we should comment on why this paper came out so long (the first draft that contained
all the main theorems had only five pages). 

\medskip

The reader will notice that the parts of the paper that
contain any innovation (Sects. \ref{s:deducing}, \ref{s:2Dmods} and \ref{s:3Dmods}) take less
than one fifth of the volume. 

\medskip

The rest of the paper is either abstract nonsense (e.g., Sects. \ref{s:dualizability} and \ref{s:renormalized}), or
background material. 

\medskip

Some of the latter (e.g., the theory of D-modules
on stacks) is included because we could not find adequate references in the literature. Some other things,
especially various notions related to derived algebraic geometry, have been written down thanks to the
work of Lurie and To\"en-Vezzosi, but we decided to review them due to the novelty of the subject,
in order to facilitate the job of the reader.

\ssec{Results on $\on{D-mod}(\CY)$}   \label{ss:introduction-D}

\sssec{}

We have not been able to treat the question of compact generation of $\Dmod(\CY)$ for arbitrary algebraic stacks. But we have obtained
the following partial result (see Theorems \ref{t:Dmods} and \ref{t:Dmods, generalized}):

\begin{thm}  \label{preview: Dmods}
Let $\CY$ be an algebraic stack of finite type over $k$. Assume that the automorphism groups of geometric points of $\CY$ are affine. 
Then $\on{D-mod}(\CY)$ is compactly generated.
\end{thm} 

\sssec{}

In addition to this theorem, and under the above assumptions on $\CY$ (we call algebraic stacks with this property ``QCA"),
we prove a result characterizing the subcategory $\on{D-mod}(\CY)^c$
of compact objects in $\on{D-mod}(\CY)$ inside the larger category $\on{D-mod}_{\on{coh}}(\CY)$ of coherent
objects. (We were inspired by the following well known result: for any noetherian scheme $Y$, a bounded coherent object of $\QCoh (Y)$ 
is compact if and only if it has finite Tor-dimension.) 

\medskip

We characterize $\on{D-mod}(\CY)^c$ by a condition that we call \emph{safety}, see 
Proposition~\ref{compactness via safety} and Theorem~\ref{criter for safety}. We note
that safety of an object can be checked strata-wise:
if $i:\CX\hookrightarrow\CY$ is a closed substack and $j:(\CY-\CX)\hookrightarrow\CY$ the
complementary open, then an object $\CM\in \on{D-mod}(\CY)$ is safe if and only if $i^!(\CF)$ and $j^!(\CF)$ are
(see Corollary~\ref{stratification}). However, the subcategory of safe objects is not preserved by the
truncation functors with respect to the canonical t-structure on $\on{D-mod}(\CY)$.

\medskip

Furthermore, we prove \corref{c:more precise} that characterizes those stacks $\CY$ of finite type over $k$
for which the functor of global De~Rham cohomlogy $\Gamma_{\on{dR}}(\CY,-)$ is continuous (i.e., commutes with colimits): 
this happens if and only if the neutral connected component of the automorphism group of any geometric point of $\CY$ is unipotent. 
We call such stacks \emph{safe}.
For example, any Deligne-Mumford stack is safe. 

\sssec{}

Let $\pi:\CY_1\to \CY_2$ be a morphism between QCA algebraic stacks. The functor of D-module direct image
$\pi_\dr:\Dmod(Z_1)\to \Dmod(Z_2)$ is in general not continuous, and consequently, it fails to have the base change
property or satisfy the projection formula.  In \secref{ss:ren direct image} we introduce a new functor
$\pi_{\blacktriangle}$ of \emph{renormalized direct image}, which fixes the above drawbacks of $\pi_\dr$. There always
is a natural transformation $\pi_{\blacktriangle}\to \pi_\dr$, which is an isomorphism on safe objects. 

\ssec{Results on $\QCoh(\CY)$}
Let $\Vect$ denote the DG category of complexes of vector spaces over $k$.

\sssec{}

We deduce \thmref{preview: Dmods} from the following more basic
result about $\QCoh(\CY)$ (see Theorem~\ref{main}):

\begin{thm}  \label{preview: qc}
Let $k$ be a field of characteristic $0$ and
let $\CY$ be a QCA algebraic stack of finite type over $k$. Then the (always derived) functor of global sections 
$$\Gamma(\CY,-):\QCoh(\CY)\to \Vect$$
commutes with colimits. In other words, the structure sheaf $\CO_\CY$ is a compact object of
$\QCoh(\CY)$.
\end{thm}

We also obtain a relative version of \thmref{preview: qc} for morphisms of algebraic stacks $\pi:\CY_1\to \CY_2$ (see \corref{c:relative}). 
It gives a sufficient condition for the functor $$\pi_*:\QCoh(\CY_1)\to \QCoh(\CY_2)$$  to commute with colimits (and thus have a 
base change property and satisfy the projection formula). 

\sssec{} \label{sss:subtle}
The question of compact generation of $\QCoh(\CY)$ 
%is in general more subtle, and we do not know under what conditions on $\CY$ 
%it is expected to hold. 
is subtle. It is easy to see that $\QCoh(\CY)^c$ is contained in the category 
$\QCoh(\CY)^{\on{perf}}$ of perfect complexes, and if $\CY$ satisfies the assumptions of
Theorem~\ref{preview: qc} then $\QCoh(\CY)^c=\QCoh(\CY)^{\on{perf}}$ (see 
Corollary~\ref{c:perfects_are_compact}). But we do not know
if under these assumptions $\QCoh(\CY)^{\on{perf}}$ always generates $\QCoh(\CY)$.
Ben-Zvi, Francis, and Nadler showed in  \cite[Section 3]{BFN} that this is true for
%perfect stacks and that the class of perfect 
most of  the stacks that one encounters in practice (e.g., see Lemma~\ref{l:BFN} below).

\medskip

However, we were able to establish a property of $\QCoh(\CY)$, which is weaker than compact generation,
but still implies many of the favorable properties enjoyed by compactly generated categories (see \thmref{t:QCoh dualizable}):

\begin{thm} \label{preview:QCoh dualizable}
Let $\CY$ be QCA algebraic stack. Then the category $\QCoh(\CY)$ is \emph{dualizable}.
\end{thm}

We refer the reader to \secref{sss:dualizabilitydef} for a review of the notion of dualizable DG category.

\sssec{}

In addition, we show that for a QCA algebraic stack $\CY$ and for
any (pre)stack $\CY'$, the natural functor
$$\QCoh(\CY )\otimes\QCoh(\CY' )\to \QCoh(\CY\times \CY')$$
is an equivalence (\corref{c:QCoh on product}). 

\sssec{}

We should mention that in reviewing the above results about $\QCoh(\CY)$ we were tacitly assuming that
we were dealing with \emph{classical algebraic stacks}. However, in the main body of the paper, we work
in the setting of derived algebraic geometry, and henceforth by a ``(pre)stack" we shall understand what one 
might call a ``DG (pre)stack".  

\medskip

In particular, some caution is needed when dealing with the notion of algebraic stack of finite type, and for 
boundedness condition of the structure sheaf. We refer the reader to the main body of the text for the precise
formulations of the above results in the DG context.

\ssec{Ind-coherent sheaves}

In addition to the categories $\QCoh(\CY)$ and $\Dmod(\CY)$, there is a third player in this paper,
namely, the DG category of ind-coherent sheaves, denoted $\IndCoh(\CY)$. We refer the reader to 
\cite{IndCoh} where this category is introduced and its basic properties are discussed. 

\medskip

As was mentioned in {\it loc.cit.}, Sects. 0.1 and 0.2, the assignment $\CY\mapsto \IndCoh(\CY)$ 
is a natural sheaf-theoretic context in its own right. In particular,
the category $\IndCoh(\CY)$ is indispensable to treat the spectral side of the Geometric Langlands correspondence, see \cite{AG}.

\medskip

In this paper the category $\IndCoh(\CY)$ is used to prove Theorem~\ref{preview:QCoh dualizable}. More importantly, this category serves as an 
intermediary between D-modules and $\CO$-modules on $\CY$. Below we explain more details on the latter role of  $\IndCoh(\CY)$.

\sssec{}

For an arbitrary (pre)stack, there is a naturally defined conservative forgetful functor 
$$\oblv_{\Dmod(\CY)}:\Dmod(\CY)\to \IndCoh(\CY),$$
and this functor is compatible with morphisms of (pre)stacks $\pi:\CY_1\to \CY_2$
under !-pullback functors on both sides. 

\medskip

Now, for a large class of prestacks, including algebraic stacks, 
the functor $\oblv_{\Dmod(\CY)}$ admits a left adjoint, 
denoted $\ind_{\Dmod(\CY)}$. This adjoint pair of functors plays an important role in this
paper: we use them to deduce \thmref{preview: Dmods} from \thmref{preview: qc}.

\sssec{}

The category $\IndCoh$ may be viewed as an accounting device that encodes the convergence
of certain spectral sequences (equivalently, the basic properties of $\IndCoh$ established
in \cite{IndCoh}, ensure that certain colimits commute with certain limits). 

\medskip

In light of this, the reader who is unfamiliar or not interested in the category $\IndCoh$,
may bypass it and relate the categories $\QCoh(\CY)$ and $\Dmod(\CY)$ directly by the pairs of
adjoint  functors\footnote{The pair $({}'\oblv_{\Dmod(\CY)},{}'\ind_{\Dmod(\CY)})$ is related to the realization of
$\Dmod (\CY )$ as ``right" D-modules. The other pair is related to the realization as ``left" D-modules.} 
$({}'\oblv_{\Dmod(\CY)},{}'\ind_{\Dmod(\CY)})$ or $(\oblv^{\on{left}}_{\Dmod(\CY)},\ind^{\on{left}}_{\Dmod(\CY)})$ 
introduced in Sects. \ref{sss:Dmod and QCoh} and \ref{sss:left realization} (for DG schemes), and \ref{sss:oblv stacks} 
and \ref{sss:Dmod and QCoh on stacks} (for algebraic stacks). The corresponding variant of the proof
of \thmref{preview: Dmods} is given in \secref{ss:variant}. 

\sssec{}

However, without the category $\IndCoh(\CY)$, the treatment of $\Dmod(\CY)$ suffers from a certain awkwardness. 
Let us list three reasons for this in the ascending order of importance:

\medskip

\noindent(i) Let $Z$ be a scheme. The realization of $\Dmod(Z)$ as ``right" D-modules has the advantage
of being compatible with the t-structure, see \cite[Sect. 4.3]{GR1} for a detailed discussion. So, let us say
we want to work with right D-modules.
However, 
if instead of $\IndCoh(Z)$ and the forgetful functor $\oblv_{\Dmod(Z)}$ we use $\QCoh(Z)$ and the
corresponding naive forgetful functor $'\oblv_{\Dmod(Z)}$, we would not be able to 
formulate the compatibility of this forgetful functor with pullbacks. The reason is that for a 
general morphism of schemes $f:Z_1\to Z_2$, the functor $f^!$ is defined and is continuous as a functor 
$\IndCoh(Z_2)\to\IndCoh(Z_1)$ but not as a functor $\QCoh(Z_2)\to\QCoh(Z_1)$.  

\medskip

\noindent(ii) The ``left" forgetful functor $\oblv^{\on{left}}_{\Dmod(\CY)}$ is defined for any pre-stack $\CY$.
However, it does not admit a left adjoint in many situations in which $\oblv_{\Dmod(\CY)}$ does, e.g., for ind-schemes. On the other hand,
the naive ``right" forgetful functor $'\oblv_{\Dmod(\CY)}$ is not defined unless $\CY$ is an {\it algebraic} stack.

\medskip

\noindent(iii) As is explained in \secref{sss:corr formalism}, the natural formalism\footnote{This formalism incorporates the base change isomorphism relating $!$-pullbacks and $*$-pushforwards.}  for the assignment
$Z\mapsto \Dmod(Z)$ is that of a functor from the category whose objects are schemes, and morphisms
are correspondences between schemes. Moreover, we want this functor to be endowed with a natural
transformation to one involving $\CO$-modules (in either $\QCoh$ or $\IndCoh$ incarnation). However,
the construction of this formalism carried out in \cite{GR2} 
%only seems possible for $\IndCoh$. 
using $\IndCoh$ would run into serious problems if one tries to work with $\QCoh$ instead. 
\footnote{A part of the construction is that the functor of pullback under a closed embedding should 
admit a left adjoint; for this it is essential that we use the !-pullback and $\IndCoh$ as our category
of $\CO$-modules.}

\medskip

So, the upshot is that without $\IndCoh$, we  
cannot really construct a workable formalism of D-modules, that
allows to take both direct and inverse images.

\sssec{}

Our main result concerning the category $\IndCoh$ is the following one (see \thmref{IndCoh}):

\begin{thm} \label{preview:IndCoh}
For a QCA algebraic stack $\CY$, the category $\IndCoh(\CY)$ is compactly generated.
The category of its compact objects identifies with $\Coh(\CY)$.
\end{thm}

In the above theorem, $\Coh(\CY)$ is the full subcategory of $\QCoh(\CY)$ of \emph{coherent sheaves},
i.e., of bounded complexes with coherent cohomology. We deduce \thmref{preview:IndCoh} from
\thmref{preview: qc}.

\medskip

As we mentioned in Sect.~\ref{sss:subtle}, for a general QCA stack $\CY$ the problem of
compact generation of $\QCoh(\CY)$ is still open. 

%\Drin{``Main advantage" is too much. The true ``main advantage" of $\IndCoh$ is that ``it 
%works". }As we mentioned in Sect.~\ref{sss:subtle}, the very fact of compact generation of
%$\QCoh(\CY)$ has not been resolved in general. \Drin{A fact cannot be resolved.}

\ssec{Contents of the paper}   \label{ss:structure} 

\sssec{}

In \secref{s:main result} we formulate the main technical result of this paper, \thmref{main}.

\medskip

We first fix our conventions regarding algebraic stacks. In Sects. \ref{s:main result} through
\ref{s:3Dmods} we adopt a definition of algebraic stacks slightly more restrictive than that of
\cite{LM}. Namely, we require the diagonal morphism to be schematic rather than representable.

\medskip

We introduce the notion of QCA algebraic stack and of QCA morphism between arbitrary
(pre)stacks.

\medskip

We recall the definition of the category of $\QCoh(\CY)$ for prestacks and in particular algebraic 
stacks.

\medskip

We formulate \thmref{main}, which is a sharpened version of \thmref{preview: qc} mentioned above.
In \thmref{main} we assert not only that the functor $\Gamma(\CY,-)$ is continuous, but also that
it is of bounded cohomological amplitude. 

\medskip

We also show how \thmref{main} implies its relative version for a QCA morphism between (pre)stacks.  

\sssec{}

In \secref{s:deducing} we prove \thmref{main}. The idea of the proof is very simple.
First, we show that the boundedness of the cohomological dimension implies the
continuity of the functor $\Gamma(\CY,-)$. 

\medskip

We then establish the required boundedness
by stratifying our algebraic stack by locally closed substacks that are gerbes over schemes.
For algebraic stacks of the latter form, one deduces the theorem directly by reducing
to the case of quotient stacks $Z/G$, where $Z$ is a quasi-compact scheme and $G$ is a reductive group.

\medskip

The char. $0$ assumption is essential since we are using the fact that the category of
representations of a reductive group is semi-simple.  

\sssec{}

In Sect. \ref{s:IndCoh} we study the behavior of the category $\IndCoh(\CY)$
for QCA algebraic stacks. 

\medskip

We first recall the definition and basic properties of $\IndCoh(\CY)$. 

\medskip

We deduce \thmref{preview:IndCoh} from \thmref{preview: qc}. 

\medskip

We also introduce and study the direct image functor
$\pi_*^{\IndCoh}$ for a morphism $\pi$ between QCA algebraic stacks.

\sssec{}

In Sect. \ref{s:dualizability} we prove (and study the implications of) the 
\emph{dualizability} property of the categories $\IndCoh(\CY)$ and
$\QCoh(\CY)$ for a QCA algebraic stack $\CY$. 

\medskip

We first recall the notion of dualizable DG category, and then deduce the dualizability of $\IndCoh(\CY)$ from
the fact that it is compactly generated. 

\medskip

We deduce the dualizability
of $\QCoh(\CY)$ from the fact that it is a retract of $\IndCoh(\CY)$. 

\medskip

We then proceed to discuss Serre duality, which we interpret as a datum of
equivalence of between $\IndCoh(\CY)$ and its dual. 

\sssec{}

In \secref{s:Dmods on schemes} we review the theory of D-modules on (DG) schemes.

\medskip

All of this material is well-known at the level of underlying triangulated categories,
but unfortunately there is still no reference in the literature where all the needed
constructions are carried out at the DG level. This is particularly relevant with regard to
base change isomorphisms, where it is not straightforward to even formulate 
what structure they encode at the level of $\infty$-categories. 

\medskip

We also discuss Verdier duality for D-modules, which we interpret as a datum of equivalence between
the category $\Dmod(Z)$ and its dual, and its relation to Serre duality for $\IndCoh(Z)$.

\sssec{}

In \secref{s:Dmods on stacks} we review the theory of D-modules on prestacks and algebraic stacks.
This theory is also ``well-known modulo homotopy-theoretic issues".

\medskip

Having an appropriate formalism for the assignment $Z\rightsquigarrow \Dmod(Z)$ for schemes, 
one defines the category $\Dmod(\CY)$ for an arbitrary prestack $\CY$, along with the naturally
defined functors. The theory becomes richer once we restrict our attention to algebraic stacks;
for example, in this case the category $\Dmod(\CY)$ has a t-structure.

\medskip

For algebraic stacks we construct and study the induction functor
$$\ind_{\Dmod(\CY)}:\IndCoh(\CY)\to \Dmod(\CY),$$ left adjoint to the forgetful functor
$\oblv_{\Dmod(\CY)}$. Its existence and properties are crucial for the proof 
of compact generation of $\Dmod(\CY)$ on QCA algebraic stacks, as well as
for the relation between the conditions of compactness and safety for objects of 
$\Dmod(\CY)$, and for the construction of the renormalized direct image functor.
In short, the functor $\ind_{\Dmod(\CY)}$ produces a supply of objects of
$\Dmod(\CY)$ whose cohomological behavior we can control.

\sssec{}

In \secref{s:deRham stack} we define the functor of de Rham cohomology $\Gamma_{\on{dR}}(\CY,-):\Dmod(\CY)\to \Vect$,
where $\CY$ is an algebraic stack, and discuss its failure to be continuous. 

\medskip

Furthermore, we generalize this to the case of the D-module
direct image functor $\pi_\dr$ for a morphism $\pi$ between algebraic stacks. 

\medskip

We also discuss the condition of \emph{coherence} on an object of $\Dmod(\CY)$, and we explain 
that for quasi-compact algebraic stacks, unlike quasi-compact schemes, the inclusion
$$\Dmod(\CY)^c\subset \Dmod_{\on{coh}}(\CY)$$ is \emph{not} an equality.

\sssec{}

In \secref{s:2Dmods} we prove \thmref{preview: Dmods}. More precisely, we show that for a QCA
algebraic stack $\CY$, the category $\Dmod(\CY)$ is compactly generated by objects of the form
$\ind_{\Dmod(\CY)}(\CF)$ for $\CF\in \Coh(\CY)$. 

\medskip

We also show that \thmref{preview: Dmods}, combined with a compatibility of Serre and Verdier
dualities, imply that for a QCA algebraic stack $\CY$, the category $\Dmod(\CY)$ is equivalent
to its dual, as was the case for schemes.

\medskip

Finally, we show that for $\CY$ as above and any prestack $\CY'$, the natural functor
$$\Dmod(\CY)\times \Dmod(\CY')\to \Dmod(\CY\times \CY')$$
is an equivalence.

\sssec{}

In \secref{s:renormalized} we introduce the functors of renormalized de Rham cohomology
and, more generally, renormalized D-module direct image for morphisms between QCA algebraic
stacks. 

\medskip  

We show that both these functors can be defined as ind-extensions of restrictions of the original
functors $\Gamma_{\on{dR}}(\CY,-)$ and $\pi_\dr$ to the subcategory of compact 
objects.  

\medskip

We show that the renormalized direct image functor $\pi_{\blacktriangle}$, unlike the original functor
$\pi_\dr$, has the base change property and satisfies the projection formula.

\medskip

We introduce the notion of \emph{safe} object of $\Dmod(\CY)$, and we show that for
safe objects $\pi_{\blacktriangle}(\CM)\simeq \pi_\dr(\CM)$. 

\medskip

We also show that compact objects of $\Dmod(\CY)$ can be characterized as those objects
of $\Dmod_{\on{coh}}(\CY)$ that are also safe.

\medskip

Finally, we show that the functor $\pi_{\blacktriangle}$ exhibits a behavior opposite to that
of $\pi_\dr$ with respect to its cohomological amplitude: the functor 
$\pi_\dr$ is left t-exact, up to a cohomological shift, whereas the functor $\pi_{\blacktriangle}$
is right t-exact, up to a cohomological shift.

\sssec{}

In \secref{s:3Dmods} we give geometric descriptions of safe algebraic stacks
(i.e., those QCA stacks, for which all objects of $\Dmod(\CY)$ are safe), and a
geometric criterion for safety of objects of $\Dmod(\CY)$ in general. The latter
description also provides a more explicit description of compact objects of
$\Dmod(\CY)$ inside $\Dmod_{\on{coh}}(\CY)$.

\medskip

We prove that a quasi-compact algebraic stack $\CY$ is safe if and only if the 
neutral components of stabilizers of its geometric points are unipotent. In
particular, any Deligne-Mumford quasi-compact algebraic stack is safe. 

\medskip

The criterion for safety of an object, roughly, looks as follows: a cohomologically
bounded object $\CM\in \Dmod(\CY)$ is safe if and only if for every point $y\in \CY$ with
$G_y=\on{Aut}(y)$, the restriction $\CM|_{BG_y}$ (here $BG_y$ denotes the 
classifying stack of $G_y$ which maps canonically into $\CY$) has the property that
$$\pi_\dr(\CM|_{BG_y})$$ is still cohomologically bounded, where $\pi$ denotes the
map $BG_y\to B\Gamma_y$, where $\Gamma_y=\pi_0(G_y)$.

\medskip

Conversely, we show that every cohomologically bounded safe object of $\Dmod(\CY)$
can be obtained by a finite iteration of taking cones starting from objects of the form
$\phi_\dr(\CN)$, where $\phi:S\to \CY$ with $S$ being a quasi-compact scheme
and $\CN\in \Dmod(S)^b$.

\sssec{}

Finally, in \secref{s:gen alg stacks} we explain how to generalize the results of Sects.
\ref{s:main result}-\ref{s:3Dmods} to the case of algebraic stacks in the sense of
\cite{LM}; we call the latter LM-algebraic stacks.

\medskip

Namely, we explain that since quasi-compact algebraic spaces are QCA when viewed
as algebraic stacks, they can be used as building blocks for the categories
$\QCoh(-)$, $\IndCoh(-)$ and $\Dmod(-)$ instead of schemes. This will imply that
the proofs of all the results of this paper are valid for QCA LM-algebraic stacks and 
morphisms. 

\ssec{Conventions, notation and terminology}   \label{ss:conventions}

We will be working over a fixed ground field $k$ of characteristic 0. Without
loss of generality one can assume that $k$ is algebraically closed. 

\sssec{$\infty$-categories}

Throughout the paper we shall be working with $(\infty,1)$-categories. Our treatment
is not tied to any specific model, but we shall use \cite{Lu1} as our basic reference. 

\medskip

We let $\inftygroup$ denote the $\infty$-category of $\infty$-groupids, a.k.a. ``spaces".
%By a presheaf on an $\infty$-category $\bC$ we shall mean a functor
%$$\bC^{\on{op}}\to \inftygroup.$$

\medskip

If $\bC$ is an $\infty$-category and $\bc_1,\bc_2\in \bC$ are objects, we shall denote by
$\on{Maps}_\bC(\bc_1,\bc_2)$ the $\infty$-groupoid of maps between these
two objects. We shall use the notation $\Hom_\bC(\bc_1,\bc_2)\in \on{Sets}$ for 
$\pi_0(\on{Maps}_\bC(\bc_1,\bc_2))$, i.e., Hom in the homotopy category.

\medskip

We shall often say ``category" when in fact we mean an $\infty$-category. 

\medskip

If $F:\bC'\to \bC$ is a functor between $\infty$-categories, we shall say that $F$ is
\emph{0-fully faithful} (or just \emph{fully faithful}) if $F$ induces an equivalence on $\on{Maps}(-,-)$.
In this case we call the essential image of $\bC'$ a \emph{full subcategory} of $\bC$.

\medskip

We shall say that $F$ is \emph{1-fully faithful} (or just \emph{faithful}) if $F$ induces a 
\emph{monomorphism} on $\on{Maps}(-,-)$, i.e., if the map
\begin{equation} \label{e:map on maps}
\on{Maps}_{\bC'}(\bc'_1,\bc'_2)\to \on{Maps}_{\bC}(F(\bc'_1),F(\bc'_2))
\end{equation}
is the inclusion of a union of some of the connected components.  
If, moreover, the map
\eqref{e:map on maps} is surjective on those connected components of 
$\on{Maps}_{\bC}(F(\bc'_1),F(\bc'_2))$ that correspond to isomorphisms, 
we shall refer to the essential image of $\bC'$ as a \emph{1-full subcategory}
of $\bC$. 

\sssec{DG categories: elementary aspects}  \label{sss:DG categories}

We will be working with DG categories over $k$. Unless explicitly specified 
otherwise, all DG categories will be assumed cocomplete, i.e., contain
infinite direct sums (equivalently, filtered colimits, and equivalently all colimits).

\medskip

We let $\Vect$ denote the DG category of complexes of $k$-vector spaces.

\medskip

For a DG category $\bC$, and $\bc_1,\bc_2\in \bC$ we can form the object
$\CMaps_\bC(\bc_1,\bc_2)\in \Vect$. We have
$$\on{Maps}_\bC(\bc_1,\bc_2)\simeq \tau^{\leq 0}\left(\CMaps_\bC(\bc_1,\bc_2)\right),$$
where in the right-hand side we regard an object of $\Vect^{\leq 0}$ as an object
of $\inftygroup$ via the Dold-Kan functor.

\medskip

We shall use the notation $\Hom^\bullet_\bC(\bc_1,\bc_2)$ to denote the graded vector space
$$\underset{i}\oplus\, H^i\left(\CMaps_\bC(\bc_1,\bc_2)\right)\simeq \underset{i}\oplus\, \Hom_\bC(\bc_1,\bc_2[i]).$$

\medskip

We shall often use the notion of t-structure on a DG category. For $\bC$ endowed with a t-structure,
we shall denote by $\bC^{\leq 0}$, $\bC^{\geq 0}$, $\bC^-$, $\bC^+$, $\bC^b$ the corresponding
subcategories of connective, coconnective, eventually connective (a.k.a. bounded above), eventually coconnective 
(a.k.a. bounded below) and cohomologically bounded objects. 

\medskip

We let $\bC^\heartsuit$ denote the abelian category
equal to the heart (a.k.a. core) of the t-structure. For example, $\Vect^\heartsuit$ is the usual category of
$k$-vector spaces.

\sssec{Functors}  \label{sss:dg functors}

All functors between DG categories considered in this paper, without exception, will
be exact (i.e., map exact triangles to exact triangles). 
\footnote{As a way to deal with set-theoretic issues, we will assume that all our DG categories 
are presentable, and all functors between them are accessible (see \cite[Definitions 5.4.2.5 and 5.5.0.1]{Lu1}); 
an assumption which is always satisfied in practice.}

\medskip

It is a corollary of the adjoint functor theorem a cocomplete DG category also contains all \emph{limits},
see \cite[Corollary 5.5.2.4]{Lu1}.

\medskip

More generally, we have a version of Brown's representability
theorem that says that any exact contravariant functor $F:\bC\to \Vect$ is ind-representable
(see \cite[Corollary 5.3.5.4]{Lu1}), and it is representable
if and only if $F$ takes colimits in $\bC$ to limits in $\Vect$.

\sssec{Continuous functors}

For two DG categories $\bC_1$, $\bC_2$ we shall denote by $\on{Funct}(\bC_1,\bC_2)$ 
the DG category of all (exact) functors $\bC_1\to \bC_2$, and by $\on{Funct}_{\on{cont}}(\bC_1,\bC_2)$
its full DG subcategory consisting of \emph{continuous} functors, i.e., those functors that commute 
with infinite direct sums (equivalently, filtered colimits, and equivalently all colimits). 

\medskip

By default, whenever
we talk about a functor between DG categories, we will mean a continuous functor. We shall also
encounter non-continuous functors, but we will explicitly emphasize whenever this happens.

\medskip

The importance of continuous functors vs. all functors is, among the rest, in the fact that the 
operation of tensor product of DG categories, reviewed in \secref{sss:dualizabilitydef}, is functorial 
with respect to continuous functors.

\sssec{Compactness}

We recall that an object $\bc$ in a DG category is called \emph{compact} if the functor
$$\Hom_\bC(\bc,-):\bC\to \Vect^\heartsuit$$ commutes with direct sums. This is equivalent to requiring that the
functor $$\CMaps_\bC(\bc,-):\bC\to \Vect$$ be continuous, and still equivalent to requiring that the
functor $\on{Maps}_\bC(\bc,-):\bC\to \inftygroup$ commute with filtered colimits; the latter interpretation
of compactness makes sense for an arbitary $\infty$-category closed under filtered colimits. We let
$\bC^c$ denote the full \emph{but not cocomplete} subcategory of $\bC$ spanned by compact objects.
\footnote{The presentability assumption on $\bC$ implies that $\bC^c$ is small.}

\medskip

A DG category $\bC$ is said to be compactly generated if there exists a set of compact objects $\bc_\alpha\in \bC$
that generate it, i.e., $$\CMaps(\bc_\alpha,\bc)=0\, \Rightarrow \bc=0.$$ Equivalently, if $\bC$ does not contain
proper full cocomplete subcategories that contain all the objects $\bc_\alpha$.

\sssec{DG categories: homotopy-theoretic aspects}  \label{sss:DG categories2}

We shall regard the totality of DG categories as an $(\infty,1)$-category
in two ways, denoted $\StinftyCat$ and $\StinftyCat_{\on{cont}}$. 
In both cases the objects are DG categories. In the former case, we take as $1$-morphisms
all (exact) functors, whereas in the latter case we take those (exact) functors that are
continuous. The latter is a 1-full subcategory
of the former. 

\medskip

The above framework for the
theory of DG categories is not fully documented (see, however, \cite{DG} where
the basic facts are summarized). For a better documented theory, one can replace 
the $\infty$-category of DG categories by that of stable $\infty$-categories tensored
over $k$ (the latter theory is defined as a consequence of \cite[Sects. 4.2 and 6.3]{Lu2}).

\sssec{Ind-completions}  \label{sss:ind-compl}

If $\bC^0$ is a small, and hence, \emph{non-cocomplete}, DG category, one can canonically attach to
it a cocomplete one, referred to as the \emph{ind-completion} of $\bC^0$, denoted
$\on{Ind}(\bC^0)$, and characterized by the property that for $\bC\in \StinftyCat_{\on{cont}}$
$$\on{Funct}_{\on{cont}}(\on{Ind}(\bC^0),\bC)$$ is the category of \emph{all} (exact) functors
$\bC^0\to\bC$. For a functor $F:\bC^0\to \bC$, the resulting continuous functor
$\on{Ind}(\bC^0)\to\bC$ is called the ``ind-extension of $F$".

\medskip

The objects of $\bC^0$ are compact when viewed as objects of $\bC$.
It is not true, however, that the inclusion $\bC^0\subset \bC^c$ is equality. Rather,
$\bC^c$ is the Karoubian completion of $\bC^0$, i.e., every object of the former can
be realized as a direct summand of an object of the latter (see 
\cite[Theorem 2.1]{N} or \cite[Prop. 1.4.2]{BeV} for the proof). 

\medskip

A DG category is compactly generated if and only if it is of the form $\on{Ind}(\bC^0)$
for $\bC^0$ as above.

\sssec{DG Schemes}

Throughout the paper we shall work in the context of derived algebraic geometry over the field $k$.
We denote $\Spec(k)=:\on{pt}$.

\medskip

We shall denote by $\on{DGSch}$, $\on{DGSch}_{\on{qs-qc}}$ and $\on{DGSch}^{\on{aff}}$ the 
categories of DG schemes, quasi-separated and quasi-compact DG schemes, and affine DG schemes,
respectively. The fundamental treatment of these objects can be found in \cite{Lu3}. For a brief
review see also \cite{Stacks}, Sect. 3. The above categories contain the full subcategories
$\on{Sch}$, $\on{Sch}_{\on{qs-qc}}$ and $\on{Sch}^{\on{aff}}$ of classical schemes.

\medskip

For the reader's convenience, let us recall the notions of smoothness and flatness
in the DG setting. 

\medskip

A map $\Spec(B)\to \Spec(A)$ between affine DG schemes is said to be flat if
$H^0(B)$ is flat as a module over $H^0(A)$, plus the following equivalent conditions hold:

\begin{itemize} 

\item The natural map $H^0(B)\underset{H^0(A)}\otimes H^i(A)\to H^i(B)$ is an isomorphism for every $i$.

\item For any $A$-module $M$, the natural map $H^0(B)\underset{H^0(A)}\otimes H^i(M)\to H^i(B\underset{A}\otimes M)$
is an isomorphism for every $i$.

\smallskip

\item If an $A$-module $N$ is concentrated in degree 0 then so is $B\underset{A}\otimes N$.

\end{itemize}

The above notion is easily seen to be local in the Zariski topology in both
$\Spec(A)$ and $\Spec(B)$. The notion of flatness for a morphism between DG 
schemes is defined accordingly. 

\medskip

Let $f:S_1\to S_2$ be a morphism of DG schemes. We shall say that it is smooth/flat almost of finite presentation
if the following conditions hold:

\begin{itemize}

\item $f$ is flat (in particular, the base-changed DG scheme $^{cl}\!S_2\underset{S_2}\times S_1$ is
classical), and 

\item the map of classical schemes $^{cl}\!S_2\underset{S_2}\times S_1\to {}^{cl}\!S_2$
is smooth/flat of finite presentation. 

\end{itemize}

In the above formulas, for a DG scheme $S$, we denote by $^{cl}\!S$ the underlying classical
scheme. I.e., locally, if $S=\Spec(A)$, then $^{cl}\!S=\Spec(H^0(A))$.

\medskip

A morphism $f:S_1\to S_2$ is said to be fppf if it is flat almost of finite presentation
and surjective at the level of the underlying classical schemes. 

\sssec{Stacks and prestacks}

By a prestack we shall mean an arbitrary functor 
$$\CY:(\on{DGSch}^{\on{aff}})^{\on{op}}\to \inftygroup.$$
We denote the category of prestacks by $\on{PreStk}$.

\medskip

We should emphasize that the reader who is reluctant do deal with functors
taking values in $\infty$-groupoids, and who is willing to pay the price of staying 
within the world of classical algebraic geometry, may ignore any mention of 
prestacks, and replace them by functors with values in usual (i.e., $1$-truncated) 
groupoids.

\medskip

A prestack is called a stack if it satisfies fppf descent, see \cite{Stacks}, Sect. 2.2.
We denote the full subcategory of $\on{PreStk}$ formed by stacks by
$\on{Stk}$. The embedding $\on{Stk}\hookrightarrow \on{PreStk}$ admits
a left adjoint, denoted $L$, and called a sheafification functor.

\medskip

That said, the distinction between stacks and prestacks will not play a significant role
in this paper, because for a prestack $\CY$, the canonical map $\CY\to L(\CY)$
induces an equivalence on the category $\QCoh(-)$. The same happens for  
$\IndCoh(-)$ and $\Dmod(-)$ in the context of prestacks locally almost of finite type,
considered starting from Sects. \ref{s:IndCoh}-\ref{s:3Dmods} on.

\medskip

We can also consider the category of classical prestacks, denoted $^{cl}\!\on{PreStk}$,
the latter being the category of all functors
$$(\on{Sch}^{\on{aff}})^{\on{op}}\to \inftygroup.$$
We have a natural restriction functor 
$$\on{Res}_{\on{cl}\to\on{DG}}:\on{PreStk}\to {}^{cl}\!\on{PreStk},$$ which admits
a fully faithful left adjoint, given by the procedure of \emph{left Kan extension}, see \cite{Stacks},
Sect. 1.1.3. Let us denote this functor $\on{LKE}_{\on{cl}\to\on{DG}}$.
Thus, the functor $\on{LKE}_{\on{cl}\to\on{DG}}$ allows us to view 
$^{cl}\!\on{PreStk}$ as a full subcategory of $\on{PreStk}$. 

\medskip

For example, the composition
of the Yoneda embedding $\on{Sch}^{\on{aff}}\to {}^{cl}\!\on{PreStk}$ with 
$\on{LKE}_{\on{cl}\to\on{DG}}$ is the composition of the tautological embedding
$\on{Sch}^{\on{aff}}\to \on{DGSch}^{\on{aff}}$, followed by the Yoneda embedding
$\on{DGSch}^{\on{aff}}\to \on{PreStk}$.

\medskip

We also have the corresponding full subcategory $^{cl}\!\on{Stk}\subset {}^{cl}\!\on{PreStk}$.
The functor $\on{Res}_{\on{cl}\to\on{DG}}$ sends $\on{Stk}\subset \on{PreStk}$ to
$^{cl}\!\on{Stk}\subset {}^{cl}\!\on{PreStk}$. However, the functor $\on{LKE}_{\on{cl}\to\on{DG}}$ 
does \emph{not} necessarily send $^{cl}\!\on{Stk}$ to $\on{Stk}$. 

\medskip

Following \cite{Stacks}, Sect. 2.4.7, we shall call a stack \emph{classical} if it can be obtained 
as a sheafification of a classical prestack. This is equivalent to the condition that the
natural map $$L(\on{LKE}_{\on{cl}\to\on{DG}}\circ \on{Res}_{\on{cl}\to\on{DG}}(\CY))\to \CY$$
be an isomorphism. 

\medskip

In particular, it is not true that a classical non-affine DG scheme is classical as a prestack.
But it is classical as a stack.

\medskip

When in the main body of the text we will talk about algebraic stacks, the condition of
being classical is understood in the above sense.

\medskip

For a \emph{p}=prestack/stack/DG scheme/affine DG scheme $\CY$, the expression ``the classical \emph{p}
underlying $\CY$" means the object $\on{Res}_{\on{cl}\to\on{DG}}(\CY)\in {}^{cl}\!\on{PreStk}$
that belongs to the appropriate full subcategory  
$$\on{Sch}^{\on{aff}}\subset \on{Sch}\subset \on{Stk}\subset \on{PreStk}.$$
We will use a shorthand notation for this operation: $\CY\mapsto {}^{cl}\CY$.

\ssec{Acknowledgments}

We are grateful to Alexander Beilinson, Jacob Lurie, Amnon Neeman and Bertrand To\"en
for helping us with various technical questions. We are grateful to Sam Raskin for checking
the convergence arguments in \secref{s:deRham stack}.

\medskip

The research of V.~D. is partially supported by NSF grant DMS-1001660. The research
of D.~G. is partially supported by NSF grant DMS-1063470.

%For $\bC_1,\bC_2\in \StinftyCat$, we will use interchangeably the terms 
%``continuous functors" or ``functors commuting with 
%colimits/filtered colimits/direct sums" for objects of 
%$\on{Funct}_{\on{cont}}(\bC_1,\bC_2)$. 

%\medskip

%By $\Gamma(\CY,-)$ we shall mean the corresponding DG functor
%$$\QCoh(\CY)\to \Vect,$$
%where $\Vect$ is the DG category of vector spaces. 

\section{Results on $\QCoh (\CY )$}  \label{s:main result}
In Sects. ~\ref{ss:assumptions_stacks}-\ref{dir im} we introduce the basic definitions and recall some well-known facts. 
The new results are formulated in Sect. ~\ref{ss:statements}.

\ssec{Assumptions on stacks}  \label{ss:assumptions_stacks}

\sssec{Algebraic stacks}  \label{sss:algebraic stacks}

%Let $\CY$ be a $1$-Artin stack in the sense of \cite{Stacks},  Sect. 4.2. 

%\medskip

In Sections \ref{s:main result}-\ref{s:3Dmods}
we will use the following definition of algebraicity of a stack, which is slightly more restrictive than that of 
\cite{LM} (in the context of classical stacks) or \cite[Sect. 4.2.8 ]{Stacks} (in the DG context).

\sssec{}

First, recall that a morphism $\pi:\CY_1\to \CY_2$ between prestacks is called schematic 
if for any affine DG scheme $S$ equipped with a morphism $S\to\CY_2$ the prestack $S\underset{\CY_2}\times \CY_1$ 
is a DG scheme. The notions of surjectivity/flatness/smoothness/quasi-compactness/quasi-separatedness make sense for schematic
morphisms: $\pi$ has one of the above properties if for every $S\to \CY_2$ as above, the map of DG schemes
$S\underset{\CY_2}\times \CY_1\to S$ has the corresponding property. 

\sssec{}

Let $\CY$ be a stack. We shall say that $\CY$ is a algebraic if

\begin{itemize}

\item The diagonal morphism $\CY\to \CY\times \CY$ is 
schematic,
%and as such is
quasi-separated and quasi-compact.

\smallskip

\item There exists a DG scheme $Z$ and a map $f:Z\to \CY$ (automatically schematic,
by the previous condition) such that $f$ is smooth and surjective.

\end{itemize}

%\begin{rem} It is an easy exercise to see that for $\CY$ as above, the diagonal map
%$\CY\to \CY\times \CY$ is representable. 
%\end{rem}

A pair $(Z,f)$ as above is called a {\it presentation} or {\it atlas} for $\CY$.

\begin{rem}  \label{r:qs}
In \cite{LM} one imposes a slightly stronger condition on the diagonal map $\CY\to \CY\times \CY$. Namely,
in {\it loc.cit.} it is required to be separated rather than quasi-separated. However, the above weaker condition
seems more natural, and it will suffice for our purposes (the latter being \lemref{LM}, that relies on 
\cite[Corollary 10.8]{LM}, while the latter does not require the separated diagonal assumption).
\end{rem} 

\begin{rem}
To get the more general notion of algebraic stack in the spirit of \cite{LM} (for brevity, \emph{LM-algebraic} stack), 
one replaces the word ``schematic" in the above definition by ``representable", see Sect.~\ref{sss:LM-algebraic}. 
\footnote{A morphism $\pi:\CY_1\to \CY_2$ between prestacks is called representable if for every 
affine DG scheme $S$ equipped with a morphism $S\to\CY_2$ the prestack $S\underset{\CY_2}\times \CY_1$ 
is an algebraic space, see \secref{sss:alg spaces} for a review of the latter notion in the context of derived algebraic geometry.}
In fact, \emph{all the results formulated in this paper are valid for LM-algebraic stacks;} in \secref{s:gen alg stacks} 
we shall explain the necessary modifications. On the other hand, most LM-algebraic stacks one encounters in practice 
satisfy the more restrictive definition as well. The advantage of LM-algebraic stacks vs. algebraic stacks defined
above is that the former, unlike the latter, satisfy fppf descent.

\medskip

To recover the even more general notion of algebraic stack (a.k.a. $1$-Artin stack) from \cite[Sect. 4.2.8]{Stacks},
one should omit the condition on the diagonal map to be quasi-separated and quasi-compact. However, these conditions 
are essential for the validity of the results in this paper.

\end{rem}

\begin{defn}
We shall say that an algebraic stack $\CY$ is quasi-compact if admits an atlas $(Z,f)$, where $Z$
is an affine (equivalently, quasi-compact) DG scheme.
\end{defn}

%\sssec{Conventions}

%From now on, unless explicitly specified otherwise, all algebraic stacks 
%will be assumed of locally almost of finite type. 

%\medskip

%Furthermore, from now on, by a prestack (resp., stack, Artin stack) we shall mean 
%a prestack (resp., stack, Artin stack) locally almost of finite type, see \cite{Stacks}, 
%Sect. 1.3.9 (resp., 2.6.5, 4.9), where these notions are introduced. 

%\medskip

%However, we should emphasize that it is \emph{not necessary} to know what a prestack
%locally almost of finite type is, in order to understand the results of the present paper.
%Neither does one need to know what stacks or Artin stacks are: 

%\medskip

%Whenever one of these notions appears, the reader can replace it by that of algebraic stack
%(locally almost of finite type in the sense of \secref{sss:almostfinite})
%and thus obtain a slightly less general assertion. 

\sssec{QCA stacks} \label{sss:QCA}  \hfill

\medskip

QCA is shorthand for ``quasi-compact and with affine automorphism groups".

\begin{defn}   \label{d:QCA}
We shall say that algebraic stack $\CY$ is $QCA$ if 

\medskip

\begin{enumerate}

\item It is quasi-compact;

\smallskip

\item The automorphism groups of its geometric points are affine;

\smallskip

\item The classical inertia stack, i.e., the classical algebraic stack
$^{cl}(\CY\underset{\CY\times \CY}\times \CY)$, is of finite
presentation over $^{cl}\CY$.

\end{enumerate}
\end{defn}

\medskip

In particular, any algebraic space automatically satisfies this condition (indeed, the classical
inertia stack of an algebraic space $\CX$ is isomorphic to $^{cl}\CX$). In addition, it is clear that if
$$
\CD
\CY'  @>>>  \CY \\
@VVV   @VVV  \\
\CX'  @>>>  \CX,
\endCD
$$
is a Cartesian diagram, where $\CX$ and $\CX'$ are algebraic spaces, and $\CY$ is a QCA algebraic stack, then
so is $\CY'$.

\medskip

The class of QCA algebraic stacks will play a fundamental role in this article. 
We also need the relative version of the QCA condition.

%\begin{defn}  \label{d:rel QCA} \hfill
%\begin{enumerate}

%\item[(i)] A classical algebraic stack $\CY$ over a classical affine scheme $S$ is QCA \emph{relative} to $S$ 
%if $\CY$ can be obtained as a base change $S\underset{S'}\times \CY'$, where $S'$ is a classical affine
%scheme of finite type, and $\CY'$ is a classical QCA algebraic stack.\footnote{In the above formula,
%the Cartesian product is taken in the classical world, i.e., as classical prestacks.}

%\item[(ii)] An algebraic stack $\CY$ over an affine DG scheme $S$ is QCA \emph{relative} to $S$, 
%if $^{cl}\CY$ is QCA relative to $^{cl}\!S$. 
%\end{enumerate}
%\end{defn}

\begin{defn} \label{d:rel QCA}  
We shall say that a morphism $\pi:\CY_1\to \CY_2$ between prestacks is QCA if for every affine DG 
scheme $S$ and a morphism $S\to \CY_2$, the base-changed prestack 
$\CY_1\underset{\CY_2}\times S$ is an algebraic stack and is QCA.
\end{defn} 

For example, it is easy to show that if $\CY_1$ is a QCA algebraic stack and $\CY_2$ is any algebraic
stack, then any morphism $\CY_1\to \CY_2$ is QCA. 

\ssec{Quasi-coherent sheaves}  \label{ss:QCoh}

\sssec{Definition}  \label{sss:QCohdef}

Let $\CY$ be any prestack. Let us recall (see e.g. \cite[Sect. 1.1.3]{QCoh})
that the category $\QCoh(\CY)$ is defined as
\begin{equation} \label{e:QCoh as limit}
\underset{(S,g)\in ((\on{DGSch}^{\on{aff}})_{/\CY})^{\on{op}}}{\underset{\longleftarrow}{lim}}\, \QCoh(S).
\end{equation}
Here $(\on{DGSch}^{\on{aff}})_{/\CY}$ is the category of pairs $(S,g)$, where $S$ is an
affine DG scheme, and $g$ is a map $S\to \CY$. 

\sssec{}

Let us comment on the structure of the above definition:

\medskip

We view the assignment
$(S,g)\rightsquigarrow \QCoh(S)$
as a functor between $\infty$-categories
\begin{equation} \label{e:which functor}
((\on{DGSch}^{\on{aff}})_{/\CY})^{\on{op}}\to \StinftyCat_{\on{cont}},
\end{equation}
and the limit is taken in the $(\infty,1)$-category $\StinftyCat_{\on{cont}}$. 
The functor \eqref{e:which functor} is obtained by restriction under the forgetful map 
$(\on{DGSch}^{\on{aff}})_{/\CY}\to \on{DGSch}^{\on{aff}}$
of the functor 
$$\QCoh^*_{\on{Sch}^{\on{aff}}}:(\on{DGSch}^{\on{aff}})^{\on{op}}\to \StinftyCat_{\on{cont}},$$
where for $f:S'\to S$, the map $\QCoh(S)\to \QCoh(S')$ is $f^*$.
(The latter functor can be constructed in a ``hands-on" way; this has been carried
out in detail in \cite{Lu3}.)

\medskip

In other words, an object $\CF\in \QCoh(\CY)$ is an assignment for any $(S,g:S\to \CY)$
of an object $\CF|_S:=g^*(\CF)\in \QCoh(S)$, and a homotopy-coherent system of isomorphisms
$$f^*(g^*(\CF))\simeq (g\circ f)^*(\CF)\in \QCoh(S'),$$
for maps of DG schemes $f:S'\to S$.

\begin{rem}
For $\CY$ classical and algebraic, the definition of $\QCoh(\CY)$ given above is different from the
one of \cite{LM} (in {\it loc.cit.}, at the level of triangulated categories, $\QCoh(\CY)$ is defined as a 
full subcategory in the derived category of the abelian category of sheaves of $\CO$-modules 
on the smooth site of $\CY$). It is easy to show that the eventually coconnective (=bounded
from below) parts of both
categories, i.e., the two versions of $\QCoh(\CY)^+$, are canonically equivalent. However, we have no reasons to
believe that the entire categories are equivalent in general. The reason that we insist on considering
the entire category $\QCoh(\CY)$ is that this paper is largely devoted to the notion of compactness,
which only makes sense in a cocomplete category.
\end{rem}

\begin{rem}
In the definition of $\QCoh(\CY)$, one can replace the category $\on{DGSch}^{\on{aff}}$ of 
affine DG schemes by either $\on{DGSch}_{\on{qs-qc}}$ or of quasi-separated
and quasi-compact DG schemes or just $\on{DGSch}$ of all DG schemes.
The limit category will not change
due to the Zariski descent property of the assignment $S\rightsquigarrow \QCoh(S)$.
\end{rem}

\sssec{}   \label{sss:change index qc}

In the definition of $\QCoh(\CY)$ it is often convenient to replace the category $\on{DGSch}_{/\CY}$
(resp., $(\on{DGSch}_{\on{qs-qc}})_{/\CY}$, $(\on{DGSch}^{\on{aff}})_{/\CY}$) by a another category $A$, 
equipped with a functor 
$$(a\in A)\mapsto (S_a,g_a)$$
to $\on{DGSch}_{/\CY}$ (resp., $(\on{DGSch}_{\on{qs-qc}})_{/\CY}$, 
$(\on{DGSch}^{\on{aff}})_{/\CY}$), (provided that the limit will be the same). 
Below are several examples that will be used in this paper. 

\bigskip

\noindent{(i)} If $\CY$ is classical (as a stack or a prestack), one can take the category  
$A:=(\on{Sch}^{\on{aff}})_{/\CY}$, equipped with the tautological inclusion to 
$(\on{DGSch}^{\on{aff}})_{/\CY}$. I.e., we replace DG schemes by classical schemes.
Indeed, if $\CY$ is
classical as a prestack, the fact that the limit category will be the same follows the
from fact that the property of $\CY$ to be classical means that the inclusion
$(\on{Sch}^{\on{aff}})_{/\CY}\to (\on{DGSch}^{\on{aff}})_{/\CY}$ is cofinal 
(i.e., for every $S\in \on{DGSch}^{\on{aff}}$ and a point $y:S\to \CY$, there 
exists a factorization $S\to S'\to \CY$, where $S'\in \on{Sch}^{\on{aff}}$, and the category
of such factorizations is contractible.) For stacks, this follows from the fact that the map
$\CY\to L(\CY)$, where we remind that $L(-)$ denotes fppf sheafification,
induces an isomorphism on $\QCoh$.

\bigskip

\noindent{(ii)} Let $\CY\to \CY'$ be a schematic (resp., schematic + quasi-separated and quasi-compact; affine)
map between prestacks. 

Then we can take $A$ to be $\on{DGSch}_{/\CY'}$ 
(resp., $(\on{DGSch}_{\on{qs-qc}})_{/\CY'}$; $(\on{DGSch}^{\on{aff}})_{/\CY'}$) 
via the functor
\begin{equation} \label{e:base change for index}
S'\mapsto S:=S'\underset{\CY'}\times \CY.
\end{equation}
Indeed, it is easy to see that the above functor is cofinal. 

\bigskip

\noindent{(iii)} Suppose that $\CY$ is algebraic and let $f:Z\to \CY$ be an fppf atlas.
Then we can replace $\on{DGSch}_{/\CY}$ by the \v{C}ech nerve of $f$. The fact that
the limit category is the same follows from the fppf descent for $\QCoh$ on DG schemes.

\bigskip

\noindent{(iv)} Assume again that $\CY$ is algebraic. We can take $A$ to be the 1-full
subcategory 
$$\on{DGSch}_{/\CY,\on{smooth}}\subset \on{DGSch}_{/\CY},$$
or, respectively, 
$$(\on{DGSch}_{\on{qs-qc}})_{/\CY,\on{smooth}}\subset (\on{DGSch}_{\on{qs-qc}})_{/\CY},\,\,
(\on{DGSch}^{\on{aff}})_{/\CY,\on{smooth}}\subset (\on{DGSch}^{\on{aff}})_{/\CY},$$
where we restrict objects to those
$(S,g)$, for which $g$ is smooth, and $1$-morphisms to those $f:S_1\to S_2$, for which $f$ is smooth. The fact that limit category is the same
is shown in \cite[Sect. 11.2 and particularly Corollary 11.2.3]{IndCoh}. The word ``smooth" can also be replaced by the word ``flat". 
The same proof applies to establish the following generalization:

\begin{lem}  \label{l:replace index schematic}
Let $\CY\to \CY'$ be a schematic (resp., schematic + quasi-separated and quasi-compact; affine) map between algebraic stacks. 
Then the functor \eqref{e:base change for index} 
$$\on{DGSch}_{/\CY',\on{smooth}}\to \on{DGSch}_{/\CY,\on{smooth}}$$ defines an equivalence
$$\QCoh(\CY)= \underset{(S,g)\in (\on{DGSch}_{/\CY})^{\on{op}}}{\underset{\longleftarrow}{lim}}\, \QCoh(S)\to
\underset{(S',g')\in (\on{DGSch}_{/\CY'})^{\on{op}}}{\underset{\longleftarrow}{lim}}\, \QCoh(S),$$
and similarly for the $(\on{DGSch}_{\on{qs-qc}})_{/\CY}$ and $(\on{DGSch}^{\on{aff}})_{/\CY}$ versions. 
\end{lem}

\sssec{t-structure}
For any prestack $\CY$, the category $\QCoh(\CY)$ has a natural t-structure: an object
$\CF\in \QCoh(\CY)$ is connective (i.e., cohomologically $\leq 0$) if its pullback to
any scheme is. 

\medskip

Two important features of this t-structure are summarized in the following lemma:

\begin{lem} \label{properties of t} 
Suppose that $\CY$ is an algebraic stack. 

\smallskip

\noindent{\em(a)} The t-structure on $\QCoh(\CY)$ is compatible with filtered 
colimits.\footnote{By definition, this means that the subcategory $\QCoh(\CY)^{> 0}$ 
is preserved under filtered colimits. Note that the subcategory $\QCoh(\CY)^{\leq 0}$ 
automatically has this property.}

\smallskip

\noindent{\em(b)} The t-structure on $\QCoh(\CY)$ is \emph{left-complete}, 
i.e., for $\CF\in \QCoh(\CY)$, the natural map
$$\CF\to \underset{n\in \BN}{\underset{\longleftarrow}{lim}}\, \tau^{\geq -n}(\CF)$$
is an isomorphism, where $\tau$ denotes the truncation functor.

\smallskip

\noindent{\em(c)} If $f:Z\to \CY$ is a faithfully flat atlas, the functor $f^*:\QCoh(\CY)\to \QCoh(Z)$ is
t-exact and conservative.

\end{lem}

We refer the reader to \cite[Sect. 1.2.1]{Lu2}, 
for a review of the notion of left-completeness of a t-structure.

\medskip

For the proof of the lemma, see \cite[Cor. 5.2.4]{QCoh}. One first reduces to the case where
$\CY$ is an affine DG scheme. In this case $\QCoh(\CY )$ is
left-complete because it admits a conservative t-exact functor to $\Vect$ that commutes
with limits, namely, $\Gamma(\CY,-)$. 

\begin{rem}
Let $\CY$ be an algebraic stack. It is easy to see that the category $\QCoh(\CY)^\heartsuit$ identifies
with $\QCoh({}^{cl}\CY)^\heartsuit$. 
\end{rem}

\begin{rem}    \label{r:left-completion}
Suppose again that $\CY$ is classical and algebraic. Suppose in addition that the diagonal
morphism $\CY\to \CY\times \CY$ is affine. In this case, it is easy to to show that
$\QCoh(\CY)^+$ is canonically equivalent to $D(\QCoh(\CY)^\heartsuit)^+$, see \footnote{Here by $D(\CA)$
for an abelian category $\CA$ we mean the canonical DG category, whose homotopy category is the derived 
category of $\CA$, see \cite[Sect. 1.3.4]{Lu2}.} \cite[Prop. 5.4.3]{QCoh}. It follows
from \lemref{properties of t} that the entire $\QCoh(\CY)$ can be recovered as 
the left completion of $D(\QCoh(\CY)^\heartsuit)$. 
At least, in characteristic $p>0$ it can happen that $D(\QCoh(\CY)^\heartsuit)$ itself is not left-complete 
(e.g., A.~Neeman \cite{Ne-private} showed this if $\CY$ is the classifying stack of the additive group over a field of characteristic $p>0$). 
However, it is easy to formulate sufficient conditions for $D(\QCoh(\CY)^\heartsuit)$ to  be left-complete: for example, this happens when 
$\QCoh(\CY)^\heartsuit$ is generated by (every object of $\QCoh(\CY)^\heartsuit$ is a
filtered colimit of quotients of) objects having finite cohomological dimension. 
E.g., this tautologically happens when $\CY$ is an affine DG scheme, or more
generally, a quasi-projective scheme. From here one deduces that  
this is also true for any $\CY$ of the form $Z/G$, where $Z$ is a quasi-projective
scheme, and $G$ is an affine algebraic group acting linearly on $Z$, provided we are
in characteristic $0$.
\end{rem}

\begin{rem}
If $\CY$ is an algebraic stack, which is not classical, then for two objects
$$\CF_1,\CF_2\in \QCoh(\CY)^\heartsuit\simeq \QCoh({}^{cl}\CY)^\heartsuit$$
the Exts between these objects computed in $\QCoh(\CY)$ and $\QCoh({}^{cl}\CY)$ 
will, of course, be different. \footnote{Sam Raskin points out that the latter observation may serve as
an entry point to the world of derived algebraic geometry for those not a priori
familiar with it: we start with the abelian category $\QCoh({}^{cl}\CY)^\heartsuit$, and the data
of $\CY$ encodes a way to promote it to a DG category, namely, $\QCoh(\CY)$.}
\end{rem}

\ssec{Direct images for quasi-coherent sheaves}  \label{dir im}

\sssec{}  \label{s:dir im gen}

Let $\pi:\CY_1\to \CY_2$ be a morphism between prestacks.
We have a tautologically defined (continuous) functor
$$\pi^*:\QCoh(\CY_2)\to \QCoh(\CY_1).$$
By the adjoint functor theorem (\cite[Cor. 5.5.2.9]{Lu1}), $\pi^*$ admits a right adjoint, denoted $\pi_*$.
However, in general, $\pi_*$ is \emph{not} continuous, i.e., it does \emph{not} commute with colimits. 

\medskip 

For $\CY\in \on{PreStk}$ and $p_\CY:\CY\to \on{pt}$ we shall also use the notation
$$\Gamma(\CY,-):=(p_\CY)_*.$$

\begin{rem}
In fact, $\pi_*$ defined above, is a pretty ``bad" functor. E.g., it does \emph{not} satisfy base change
(see \secref{sss:base change qc} below for what tis means). Neither does is satisfy the projection formula
(see \secref{sss:proj formula qc} for what this means), even for open embeddings.  
One of the purposes of this paper is to give conditions on $\pi$ that ensure that 
$\pi_*$ is continuous and has other nice properties.
\end{rem}

\sssec{Base change}  \label{sss:base change qc}

Let $\phi_2:\CY'_2\to \CY_2$ be another map of prestacks. Consider the Cartesian diagram
$$
\CD
\CY'_1  @>{\phi_1}>>  \CY_1 \\
@V{\pi'}VV    @VV{\pi}V   \\
\CY'_2  @>{\phi_2}>>  \CY_2.
\endCD
$$

By adjunction, for $\CF_1\in \QCoh(\CY_1)$ we obtain a morphism 
\begin{equation} \label{e:base change morphism qc}
\phi_2^*\circ \pi_*(\CF_1)\to \pi'_*\circ \phi_1^*(\CF_1).
\end{equation}

\begin{defn} \label{defn:base change qc} \hfill

\smallskip

\noindent{\em(a)} The triple $(\phi_2,\CF_1,\pi)$ satisfies
base change if the map \eqref{e:base change morphism qc} is an isomorphism.

\smallskip

\noindent{\em(b)} The pair $(\CF_1,\pi)$ satisfies
base change if \eqref{e:base change morphism qc} is an isomorphism for any $\phi_2$.

\smallskip

\noindent{\em(c)} The morphism $\pi$ satisfies base change if 
\eqref{e:base change morphism qc} is an isomorphism for any $\phi_2$ and $\CF_1$.

\end{defn}

\sssec{}    \label{sss:bootstrap base change}

Let us observe the following:

\begin{prop} \label{p:bootstrap base change}
Given $\pi:\CY_1\to \CY_2$, for $\CF_1\in \QCoh(\CY_1)$ the following conditions are equivalent:

\smallskip

\noindent {\em(i)} $(\CF_1,\pi)$ satisfies base change.

\smallskip

\noindent {\em(ii)} $(\phi_2,\CF_1,\pi)$ satisfies base change whenever $\CY_2=S_2\in \on{DGSch}^{\on{aff}}$.

\smallskip

\noindent {\em(iii)} For any $S'_2\overset{f_2}\to S_2\overset{g_2}\to \CY_2$ with $S_2,S'_2\in \on{DGSch}^{\on{aff}}$,
the triple $(f_2,\CF_{S,1},\pi_S)$ satisfies base change, where
$\CF_{S,1}:=\CF_1|_{S_2\underset{\CY_2}\times \CY_1} \text{ and } \pi_S:S_2\underset{\CY_2}\times \CY_1\to S_2$. 

\end{prop}

\begin{proof}

Clearly (i) $\Rightarrow$ (ii) $\Rightarrow$ (iii). Suppose that $(\CF_1,\pi)$ satisfies (iii). Consider the assignment
$$(S_2\in (\on{DGSch}^{\on{aff}})_{/\CY_2})\rightsquigarrow (\pi_{S})_*(\CF_{S,1}).$$
The assumption implies that this assignment defines an object $\pi_{*,?}(\CF_1)\in \QCoh(\CY_2)$. Moreover, it is easy to
see that this object is equipped with a functorial isomorphism
$$\CMaps_{\QCoh(\CY_2)}(\CF_2,\pi_{*,?}(\CF_1))\simeq \CMaps_{\QCoh(\CY_1)}(\pi^*(\CF_2),\CF_1),\quad
\CF_2\in \QCoh(\CY_2).$$

Hence, $\pi_{*,?}(\CF_1)\simeq \pi_*(\CF_1)$, and thus
\begin{equation} \label{e:base change with scheme}
g_2^*(\pi_*(\CF_1))\simeq (\pi_S)_*(\CF_{S,1}).
\end{equation}

\medskip

By the same logic, for any $\phi_2:\CY'_2\to \CY_2$, and $g'_2:S'_2\to \CY'_2$, we obtain that
$$(g'_2)^*(\pi'_*\circ \phi_1^*(\CF_1))\simeq (\pi_{S'})_*(\CF_{S',1}),$$
where 
$\CF_{S',1}:=\CF_1|_{S'_2\underset{\CY'_2}\times \CY'_1} \text{ and } \pi_{S'}:S'_2\underset{\CY'_2}\times \CY'_1\to S'_2$. 
Hence, applying \eqref{e:base change with scheme} to the map $$g'_2\circ \phi_2:S'_2\to \CY_2,$$ we obtain
$$(g'_2)^*(\pi'_*\circ \phi_1^*(\CF_1))\simeq (\pi_{S'})_*(\CF_{S',1})\simeq (g'_2\circ \phi_2)^*(\pi_*(\CF_1))=(g'_2)^*(\phi_2^*\circ \pi_*(\CF_1)),$$
as required.
 
\end{proof}

\sssec{Projection formula}  \label{sss:proj formula qc}

Let $\pi:\CY_1\to \CY_2$ be as above. For $\CF_i\in \QCoh(\CY_i)$ by adjunction we have a canonically defined map
\begin{equation} \label{e:proj formula morphism qc}
\CF_2\otimes \pi_*(\CF_1)\to \pi_*(\pi^*(\CF_2)\otimes \CF_1).
\end{equation}

\begin{defn} \label{defn:proj formula qc} \hfill

\smallskip

\noindent{\em(a)} The triple $(\CF_1,\CF_2,\pi)$ satisfies the projection formula
if the map \eqref{e:proj formula morphism qc} is an isomorphism.

\smallskip

\noindent{\em(b)} The pair $(\CF_2,\pi)$ satisfies
the projection formula if \eqref{e:base change morphism qc} is an isomorphism for any $\CF_1$.

\smallskip

\noindent{\em(c)} The pair $(\CF_1,\pi)$ satisfies
the projection formula if \eqref{e:base change morphism qc} is an isomorphism for any $\CF_2$.

\smallskip

\noindent{\em(d)} The morphism $\pi$ satisfies
the projection formula if \eqref{e:base change morphism qc} is an isomorphism for any $\CF_1$ and $\CF_2$.

\end{defn}

We also give the following definition:

\begin{defn}  \label{defn:strong proj formula qc} 
The morphism $\pi$ strongly satisfies the projection formula if it satisfies base change
and for every $S_2\in (\on{DGSch}^{\on{aff}})_{/\CY_2}$, the morphism 
$$\pi_S:S_2\underset{\CY_2}\times \CY_1\to S_2$$ 
satisfies the projection formula. 
\end{defn}

It is easy to see as in \propref{p:bootstrap base change} that if 
$\pi$ strongly satisfies the projection formula, then it satisfies the projection formula.

\sssec{}   \label{sss:easy case}

Suppose for a moment that $\pi$ is schematic, quasi-separated and quasi-compact. In this
case, from \propref{p:bootstrap base change}, we obtain that $\pi$ strongly satisfies projection
formula. 

\medskip

In particular, for $\pi$ schematic, quasi-separated and quasi-compact, we obtain the following explicit description of $\pi_*(\CF_1)$
for $\CF_1\in \QCoh(\CY_1)$. Namely, for $(S_2,g_2)\in (\on{DGSch}^{\on{aff}})_{/\CY_2}$,
we have
$$g_2^*(\CF_2)\simeq (\pi_S)_*(g_1^*(\CF_1))$$
for the morphisms as in the following Cartesian diagram
$$
\CD
S_1  @>{g_1}>>  \CY_1 \\
@V{\pi'}VV    @VV{\pi}V   \\
S_2  @>{g_2}>>  \CY_2.
\endCD
$$

\begin{rem} \label{r:strong proj formula qc}
From the above observation for schematic, quasi-separated and quasi-compact morphisms combined 
with the implication (iii) $\Rightarrow$ (i) in \propref{p:bootstrap base change}, we obtain that in the 
Definition \ref{defn:strong proj formula qc}, the condition that $\pi$ 
should satisfy base change is automatic. Indeed, in the notation of the proof of \propref{p:bootstrap base change},
express $(\phi_2)_*(\phi_2^*(-))$ as $(\phi_2)_*(\CO_{S'_2})\otimes -$, and similarly for the morphism
$S'_2\underset{\CY_2}\times \CY_1\to S_2\underset{\CY_2}\times \CY_1$. 

\end{rem}

\sssec{}

Assume now that $\CY_2$ is an algebraic stack. Note that any map from an affine (or, more generally, quasi-separated and quasi-compact)
DG scheme to $\CY_2$ is schematic, quasi-separated and quasi-compact. This observation reduces the calculation of $\pi_*$
to one in \secref{sss:easy case}. Namely, we have:

\begin{lem}  \label{l:taut dir im}
Let $A$ be a category mapping to $\on{DGSch}_{/\CY_1}$ \emph{(}respectively, $(\on{DGSch}_{\on{qs-qc}})_{/\CY_1}$\emph{;} 
$(\on{DGSch}^{\on{aff}})_{/\CY_1}$\emph{)}
as in \secref{sss:change index qc}.  Then for every
$\CF_1\in \QCoh(\CY_1)$ we have
$$\pi_*(\CF_1)\simeq \underset{a\in A^{\on{op}}}{\underset{\longleftarrow}{lim}}\, (\pi\circ g_a)_*(g_a^*(\CF_1)).$$
\end{lem}

\begin{proof}
For any $\CF_2\in \QCoh(\CY_2)$ one has
\begin{multline*}
\CMaps_{\QCoh(\CY_1)}(\pi^*(\CE_2) ,\CF_1)\simeq \underset{a\in A^{\on{op}}}{\underset{\longleftarrow}{lim}}\, 
\CMaps_{\QCoh(\CY_1)}(g_a^*\circ \pi^*(\CF_2) ,g_a^*(\CF_1))\simeq \\
\simeq \underset{a\in A^{\on{op}}}{\underset{\longleftarrow}{lim}}\,  \CMaps_{\QCoh(\CY_2)}\left(\CF_2,(\pi\circ g_a)_*(g_a^*(\CF_1))\right),
\end{multline*}
as required.
\end{proof}

\begin{rem}
Inverse limits in $\QCoh(\CY)$ exist for formal 
(i.e., set-theoretical) reasons, see \secref{sss:dg functors}.  We emphasize
that they are \emph{not} computed naively, i.e., the value of an inverse limit on $S$
mapping to $\CY$ is not in general isomorphic to the inverse limit of values.
\end{rem}

A particularly useful special case of \lemref{l:taut dir im} is the following: 

\begin{cor}  \label{c:direct image via Cech}
Suppose that in the situation of \lemref{l:taut dir im}, $\CY_1$ is an algebraic stack, and let 
$f:Z\to \CY_1$ be an fppf atlas. Let $Z^\bullet/{\CY_1}$ be its 
\v{C}ech nerve. Consider the morphisms $f^i:Z^i/{\CY_1}\to \CY_1$
and set $f^\bullet:=\{f^i\}$. Then
\begin{equation}   \label{e:direct image via Cech}
\pi_*(\CF)\simeq \on{Tot}\left((\pi \circ f^\bullet)_*((f^\bullet)^*(\CF))\right). 
\end{equation}
\end{cor}

%\begin{proof}
%By \lemref{l:basic Cech}, $\QCoh(\CY_1)= \on{Tot}(\QCoh(Z^\bullet/\CY_1))$. In particular, for any
%$\CE\in\QCoh(\CY_1)$ one has
%$$\Hom_{\QCoh(\CY_1)}(\CE ,\CF)\simeq \on{Tot}\left(\Hom ((f^\bullet)^*(\CE),(f^\bullet)^*(\CF))\right),$$
%which by adjunction is isomorphic to 
%$$\Hom _{\QCoh(\CY_1)}\left(\CE , \on{Tot}\left(( f^\bullet)_*(f^\bullet)^*(\CF)\right)\right).$$
%So $\CF\simeq \on{Tot}\left(( f^\bullet)_*(f^\bullet)^*\CF)\right).$
%This is formula \eqref{e:direct image via Cech} in the particular case $\CY_1=\CY_2$, $\pi =\id$.
%The general case follows, as $\pi_*$, being a right adjoint, commutes with inverse limits. 
%\end{proof}

\sssec{The bounded below part}

Let $\QCoh(\CY)^+$ be the bounded below (a.k.a. eventually coconnective) part
of $\QCoh(\CY)$, i.e.,
$$\QCoh(\CY)^+:=\underset{n\in \BN}\cup\, \QCoh(\CY)^{\geq -n}.$$

We claim:

\begin{cor}  \label{c:coconnective part} Let $\pi :\CY_1\to\CY_2$ be a quasi-separated and quasi-compact morphism between algebraic stacks.

\smallskip

\noindent{\em(a)}
The functor $$\pi_*:\QCoh(\CY_1)^{\geq -n}\to \QCoh(\CY_2)^{\geq -n}$$
commutes with colimits. 

\smallskip

\noindent{\em(b)} For any $\CF_1\in \QCoh(\CY_1)^+$, the pair $(\CF_1,\pi)$ satisfies
base change with respect to morphisms $\CY'_2\to \CY_2$ that are locally 
of bounded Tor-dimension.

\smallskip

\noindent{\em(c)} For any $\CF_1\in \QCoh(\CY_1)^+$, and for $\CF_2\in \QCoh(\CY_2)^+$ 
\emph{locally of bounded Tor-dimension}, the triple $(\CF_1,\CF_2,\pi)$ satisfies
the projection formula.

\end{cor}

\begin{proof}

As in \propref{p:bootstrap base change}, it is easy to see that we can assume that 
$\CY_2=S$ is an affine DG scheme. \footnote{For point (a) we are using the fact that
in a limit of DG categories $\underset{i}{\underset{\longleftarrow}{lim}}\, \bC_i$, where
the transition functors are continuous, colimits of objects are calculated component-wise.}

\medskip

To prove point (a), it suffices to show that for each $i\in\BZ$ the functor 
$$H^i(\pi_*):\QCoh(\CY_1)^{\geq -n}\to\QCoh(S)^\heartsuit$$
commutes with filtered colimits. By assumption, $\CY_1$ is quasi-compact;
hence it admits an atlas $f:Z\to \CY_1$ with
$Z$ being an affine DG scheme. (In fact, all we need for the argument
below is that $f$ be quasi-separated and quasi-compact.) Since $\CY_1$
is quasi-separated, we obtain that all the terms of the \v{C}ech nerve 
$Z^\bullet/\CY_1$ are also quasi-separated and quasi-compact. 

\medskip

Let us apply \corref{c:direct image via Cech}.
The functors $(\pi \circ f^i)_*\circ (f^i)^*$ from the RHS of \eqref{e:direct image via Cech} commute with 
colimits because the morphisms $\pi\circ f^i:Z^i\to \CY_2$ are schematic, quasi-separated and quasi-compact. 
So for each $m\in\BN$ the functor 
$$\CF\mapsto \on{Tot}_{\le m}\left((\pi \circ f^\bullet)_*((f^\bullet)^*(\CF))\right)$$ commutes with 
colimits. But if $\CF\in\QCoh(\CY_1)^{\geq -n}$ then for each $m>i+n$ the morphism
$\on{Tot}\to\on{Tot}_{\le m}$ induces an isomorphism
$$H^i(\on{Tot}_{\le m}\left((\pi \circ f^\bullet)_*((f^\bullet)^*(\CF))\right))\iso
H^i(\on{Tot}\left((\pi \circ f^\bullet)_*((f^\bullet)^*(\CF))\right)).$$
So $H^i(\pi_*):\QCoh(\CY_1)^{\geq -n}\to\QCoh(\CY)^\heartsuit$
commutes with filtered colimits.
%Choose an atlas $f:Z\to \CY_1$ and apply \lemref{direct image via Cech}. The functors 
%$(\pi \circ f^i)_*((f^i)^*$ from the RHS of \eqref{e:direct image via Cech} commute with colimits
%because the morphisms $\pi\circ f^i:Z^i\to\CY_2$ are quasi-compact and schematic.
%Now the assertion follows from the fact that we are dealing with a spectral sequence 
%concentrated in the positive quadrant. 

\medskip

Points (b) and (c) of the proposition follow similarly. 
\end{proof}

\ssec{Statements of the results on $\QCoh (\CY )$}   \label{ss:statements}

\sssec{The main result}
%The main result of this paper is the following:
%In Sect.~\ref{sss:QCA} we introduced the notion of QCA stack and QCA morphism.
The following theorem will be proved in Sect.~\ref{s:deducing}:

\begin{thm} \label{main}
Let $\CY$ be a QCA algebraic stack. Then 
\begin{enumerate}
\item[(i)]
The functor $\CF\mapsto \Gamma(\CY,\CF):\QCoh(\CY)\to \Vect$ is continuous (i.e., it
commutes with colimits, equivalently, with filtered colimits, and equivalently, with infinite direct sums);

\item[(ii)]
There exists an integer $n$ (that depends only on $\CY$) such that $H^i\left(\Gamma(\CY,\CF)\right)=0$ for all $i>n$ and all $\CF\in \QCoh(\CY)^{\leq 0}$.
\end{enumerate}
\end{thm}

Note that statement (i) can be rephrased as follows:  if $\CY$ is a QCA algebraic stack, then the object 
$\CO_\CY\in \QCoh(\CY)$ is compact.

\begin{cor}   \label{c:perfects_are_compact}
Let $\CY$ be a QCA algebraic stack. Then an object of $\QCoh(\CY)$ is compact if and only if
it is perfect.
\end{cor}

We recall that an object $\CF\in \QCoh(\CY)$ is called \emph{perfect} if its pullback to any
affine DG scheme is perfect. By \cite[Lemma 4.2.2]{QCoh}, this is equivalent to $\CF$ being dualizable
in $\QCoh(\CY)$, regarded as a monoidal category.

\begin{proof}
If $\CF$ is perfect the functor $\CMaps_{\QCoh}(\CF,-)$ can be rewritten as $\Gamma(\CY,\CF^*\otimes -)$, so it is
continuous by Theorem~\ref{main}(i). 

\medskip

On the other hand, for any algebraic stack $\CY$,
any compact object $\CF\in\QCoh(\CY)$ is perfect. Indeed, let $S$ be an affine DG scheme equipped with a morphism $f:S\to\CY$, then the object 
$f^*(\CF)\in\QCoh (S)$ is compact (because its right adjoint $f_*$ is continuous), so $f^*(\CF)$ is perfect (see, e.g.,  \cite[Lemma~3.4]{BFN}).
\end{proof}

%Statement (i) means that if 
%$\CY$ satisfies $(\ast)$ then the object $\CO_\CY\in \QCoh(\CY)$ is compact.

\sssec{A relative version}

\begin{cor} \label{c:relative}
Let $\pi :\CY_1\to \CY_2$ be a QCA morphism between prestacks.
\begin{enumerate}
\item[(i)]
The functor $\pi_*:\QCoh(\CY_1)\to\QCoh(\CY_2)$ is continuous and
strongly satisfies the projection formula. 
\item[(ii)]
If $\CY_2$ is a quasi-compact algebraic stack, \footnote{Or, more generally, if there exists a map $f:Z\to \CY_2$,
where $Z$ is an affine DG scheme, and $f$ is a surjection in the faithfully flat topology (see, e.g., \cite[Sect. 2.3.1]{Stacks},
where the notion of surjectivity is recalled).} there exists $n$ such that 
$\pi_*$ maps $\QCoh(\CY_1)^{\leq 0}$ to $\QCoh(\CY_2)^{\leq n}$.
\end{enumerate}
\end{cor}

\begin{proof}

To prove point (i), by definition, it suffices to consider the case when
$\CY_2$ is an affine DG scheme.

\medskip

In this case the continuity of $\pi_*$ follows immediately from \thmref{main}(i). 
Indeed, the functor
$$\Gamma(\CY_2,-):\QCoh(\CY_2)\to \Vect$$
is continuous and conservative, and  
$\Gamma(\CY_2,-)\circ \pi_*\simeq \Gamma(\CY_1,-)$,
so the continuity of $\Gamma(\CY_1,-)$ implies that for $\pi_*$. 

\medskip

The fact that $\pi$ satisfies the projection formula follows formally
from the continuity of $\pi_*$ (express $\CF_2$ as a colimit of copies
of the structure sheaf). By Remark \ref{r:strong proj formula qc}, the projection formula
implies base change. 

\medskip

To prove point (ii), it is again sufficient to do so after base changing by means
of an fppf map $S_2\to \CY_2$, where $S_2$ is an affine DG scheme. In this case,
the assertion follows from Theorem \ref{main}(ii).

\end{proof}

\sssec{Generation by the heart}

Let $\CY$ be an algebraic stack. 

\begin{defn} \label{d:n_coconnective}
We say that $\CY$ is \emph{$n$-coconnective} if the object $\CO_\CY\in \QCoh(\CY)$ belongs
to $\QCoh(\CY)^{\geq -n}$. 
\end{defn}

\begin{defn} \label{d:ev_coconnective}
We say that $\CY$ is \emph{eventually coconnective} 
if it is $n$-coconnective for some $n$; equivalently, if $\CO_\CY$ is bounded below, i.e.,
is eventually coconnective as an object of $\QCoh(\CY)$.
\end{defn}

\begin{rem}
The notion of $n$-connectivity makes sense for all prestacks, and not just algebraic stacks,
see \cite[Sect. 2.4.7]{Stacks}. The fact that the two notions coincide for algebraic stacks
is established in \cite[Proposition 4.6.4]{Stacks}. 
\end{rem}

\medskip

The stratification technique used in the proof of Theorem~\ref{main} also allows to prove the following result (see Sect.~\ref{ss:proofQCoh_core}):

\begin{thm}   \label{t:QCoh_core}
Suppose that an algebraic stack $\CY$ is QCA and eventually coconnective. Then $\QCoh(\CY )$ is 
generated by $\QCoh(\CY )^\heartsuit$.
\end{thm}

\begin{cor}    \label{c:Coh_generates_QCoh}
Let $\CY$ be a QCA algebraic stack, which is eventually coconnective, and such that the
underlying classical stack $^{cl}\CY$ is Noetherian. Then the category $\QCoh(\CY )$ is 
generated by $\Coh(\CY)^\heartsuit$.
\end{cor}

This follows from Theorem~\ref{t:QCoh_core} and the following fact \cite[Corollary 15.5]{LM}:
every object of $\QCoh(\CY )^\heartsuit$ is a union of its coherent sub-objects. 

\sssec{Other results}
We will also prove Theorem~\ref{t:QCoh dualizable}, which says, among other things, that
in the situation of  Corollary~\ref{c:Coh_generates_QCoh} one has 
$\QCoh(\CY\times \CY')=\QCoh(\CY)\otimes \QCoh(\CY')$
for any prestack $\CY'$.
%the category $\QCoh(\CY)$ is dualizable in the sense of Sect.~\ref{sss:notion of dual} 
%and as a monoidal category, it is rigid in the sense of \cite[Sect. 6]{DG}.

\section{Proof of Theorems~\ref{main} and \ref{t:QCoh_core}}  \label{s:deducing}

The proof of Theorem~\ref{main} occupies Sects.~\ref{ss:reducing to key}-\ref{ss:end_proof_main}.
Theorem~ \ref{t:QCoh_core} is proved in Sect.~\ref{ss:proofQCoh_core}.

\ssec{Reducing the statement to a key lemma} \label{ss:reducing to key}

\sssec{Reducing statement (i) to statement (ii)}   \label{sss:reducing}

Let $\alpha\mapsto \CF_\alpha$ be a collection of objects of $\QCoh(\CY)$.
We need to show that for any $i\in\BZ$ the natural map
$$\underset{\alpha}\oplus\, H^i(\Gamma(\CY,\CF_\alpha))\to 
H^i(\Gamma(\CY,\underset{\alpha}\oplus\, \CF_\alpha))$$
is an isomorphism.  Suppose we have proved Theorem~\ref{main}(ii), i.e.,
there exists $n$ such that the functor $H^i\left(\Gamma(\CY,-)\right)$ vanishes on $\QCoh(\CY)^{<-i-n}$.
Then $$H^i(\Gamma(\CY,\CF_\alpha))=H^i(\Gamma(\CY,\tau^{\ge -i-n-1}(\CF_\alpha))),$$
$$H^i(\Gamma(\CY,\underset{\alpha}\oplus\CF_\alpha))=
H^i(\Gamma(\CY,\tau^{\ge -i-n-1}(\underset{\alpha}\oplus\CF_\alpha))). $$
Since the t-structure on $\QCoh(\CY)$ is compatible
with filtered colimits (see \lemref{properties of t}(a)), the morphism
$$\underset{\alpha}\oplus\,  \tau^{\ge -i-n-1}(\CF_\alpha)\to 
\tau^{\ge -i-n-1}\left(\underset{\alpha}\oplus\,\CF_\alpha\right).$$
is an isomorphism. So we have to prove that the morphism
$$\underset{\alpha}\oplus H^i(\Gamma(\CY,\tau^{\ge -i-n-1}(\CF_\alpha)))\to
H^i(\Gamma(\CY,\underset{\alpha}\oplus\, \tau^{\ge -i-n-1}(\CF_\alpha)))$$
is an isomorphism. We have $\tau^{\ge -i-n-1}(\CF_\alpha)\in\QCoh(\CY)^{\ge r}$, where $r=-i-n-1$.
Now Theorem~\ref{main}(i) follows from \corref{c:coconnective part}.

\qed

%\begin{lem}   \label{l:positive part}
%For any $r\in\BZ$, the functor
%$\Gamma(\CY,-)$ commutes with filtered colimits \emph{taken within}
%$\QCoh(\CY)^{\geq r}$. 
%\end{lem}

%\begin{proof}
%For any $\CF\in \QCoh(\CY)$,
%$$\Gamma(\CY,\CF)=\on{Tot}\left(\Gamma(Z^\bullet/\CY,\CF|_{Z^\bullet/\CY})\right),$$
%where $Z^\bullet/\CY$ is the \v{C}ech simplicial scheme corresponding to the
%cover $f:Z\to \CY$. So the assertion follows from the fact that we are dealing
%with a spectral sequence concentrated in the positive quadrant. 
%\end{proof}

\sssec{Reducing statement (ii) to a key lemma}   

\begin{lem} \label{l:why_estimate_suffices}
Let $n\in\BZ$. Suppose that for any $\CF\in \QCoh(\CY)^\heartsuit$ we have 
\begin{equation}   \label{e:vanishing}
H^i\left(\Gamma(\CY,\CF)\right)=0 \;\mbox{ for }\; i>n.
\end{equation}
Then \eqref{e:vanishing} holds for any $\CF\in \QCoh(\CY)^{\leq 0}$. 
\end{lem} 

\begin{proof}
The statement is clear if $\CF$ is bounded below. To treat the general case, 
recall that the t-structure on $\QCoh(\CY)$ is left-complete, see \lemref{properties of t}(b). Since
the functor $\Gamma(\CY,-)\simeq \CMaps_{\QCoh(\CY)}(\CO_\CY,-)$
commutes with inverse limits this implies that 
%\begin{equation}   \label{e:lim}
$$\Gamma(\CY,\CF)=\underset{m}{\underset{\longleftarrow}{lim}}\, \Gamma(\CY,\tau^{\geq -m}(\CF)).$$
%\end{equation}  
If $\CF\in \QCoh(\CY)^{< 0}$ then the complexes $\Gamma(\CY,\tau^{\geq -m}(\CF))$ are concentrated
in degrees $<n$. Since  the functor $\underset{m}{\underset{\longleftarrow}{lim}}$ in $\Vect$
has cohomological amplitude $[0,1]$ we see that $\Gamma(\CY,\CF)$
%and therefore their inverse limit 
is concentrated in degrees $\le n$.
%(we had to replace the $<$ sign by the $\leq$ sign because the functor of projective limit in $\Vect$ 
%numbered by $\BN$ has cohomological amplitude $[0,1]$).
So \eqref{e:vanishing} holds 
for any $\CF\in \QCoh(\CY)^{< 0}$. Therefore it holds for any $\CF\in \QCoh(\CY)^{\leq 0}$: use
the exact triangle $\tau^{<0}(\CF)\to\CF\to\tau^{\ge 0}(\CF)$.
\end{proof}

By Lemma \ref{l:why_estimate_suffices}, to prove Theorem~\ref{main}(ii) it suffices to prove the
 following key lemma.

\begin{lem} \label{l:estimate}
Let $\CY$ be a QCA stack. Then there exists an integer $n_{\CY}$ 
%(that depends only on $\CY$) 
such that for any $\CF\in \QCoh(\CY)^\heartsuit$ we have 
%\begin{equation}   \label{e:vanishing}
$$H^i\left(\Gamma(\CY,\CF)\right)=0 \;\mbox{ for }\; i>n_{\CY}\, .$$
%\end{equation}
\end{lem}

The lemma will be proved in Sects.~\ref{ss:easy_reduct}-\ref{ss:end_proof_main}.

\ssec{Easy reduction steps}  \label{ss:easy_reduct}

\sssec{Reduction to the classical case}

Let $^{cl}\CY\overset{^{cl}\!i}\hookrightarrow \CY$ be the embedding of the classical stack 
underlying $\CY$.  Any $\CF$ as in the \lemref{l:estimate} belongs to the essential image of the functor
$^{cl}\!i_*$. Since $^{cl}\!i_*$ is t-exact, we can replace the original $\CY$ by $^{cl}\CY$,
with the same estimate for $n$.

\medskip

So for the rest of this section we will assume that $\CY$ is classical. 

\sssec{Reduction to the case when $\CY$ is reduced}  \label{reduced case}

Let $\CY_{red}overset{i_{red}}\hookrightarrow \CY$
be the corresponding reduced substack. 

\medskip

Any $\CF\in \QCoh(\CY)^\heartsuit$ admits an increasing filtration with subquotients
belonging to the essential image of the functor $(i_{red})_*$. Since the functor
$$H^i\left(\Gamma(\CY,-)\right):\QCoh(\CY)^\heartsuit\to \Vect^\heartsuit$$
commutes with filtered colimits (by \corref{c:coconnective part}(a)), by the same logic
as above, we can replace $\CY$ by $\CY_{red}$ with the same estimate on $n_\CY$. 

\medskip

So we can assume that $\CY$ is reduced. 

\ssec{Devissage}  \label{induction step}

\sssec{}

We begin with the following observation. 

\medskip

Let $\CX\overset{\imath}\hookrightarrow \CY$ be a closed substack and 
$\overset{\circ}\CY\overset{\jmath}\hookrightarrow \CY$ the complementary open substack,
such that the map $\jmath$ is quasi-compact. 
Let $d\in\BZ$ be such that the functor $\jmath_*$ has cohomological 
amplitude $\le d$ (it exists because $\CY$ itself is quasi-compact). 

\begin{lem}   \label{l:devissage}
If Lemma~\ref{l:estimate} holds for $\CX$ and $\overset{\circ}\CY$ then it holds for $\CY$ with
$$n_{\CY}:=\on{max}(n_{\overset{\circ}\CY},n_\CX+d+1).$$
\end{lem}

\begin{proof}
%Set $\overset{\circ}{n}:_{\overset{\circ}\CY}$. {\bf Do we need this notation?}
For $\CF\in \QCoh(\CY)^{\heartsuit}$ consider the exact triangle
$$\CF'\to \CF\to \jmath_*\circ \jmath^*(\CF),$$
where $\CF'$ is \emph{set-theoretically} supported on $\CX$.

\medskip

It is enough to show that
\begin{equation}  \label{on open}
H^r\left(\Gamma(\CY,\jmath_*\circ \jmath^*(\CF))\right)=0\;\text{ for }\;  r>n_{\CY}\, ,
\end{equation}
%and
\begin{equation}  \label{on closed}
H^r\left(\Gamma(\CY,\CF')\right)=0 \;   \text{ for }\;    r>n_{\CY}\, .
\end{equation}
The vanishing in \eqref{on open} is clear because
$\Gamma(\CY,\jmath_*\circ \jmath^*(\CF))\simeq \Gamma(\overset{\circ}\CY,\jmath^*(\CF))$
and $n_{\CY}\ge n_{\overset{\circ}\CY}$.

\medskip

Let us prove \eqref{on closed}.  Note that $\CF'$ has finitely many cohomology sheaves and all of them
are in degrees $\le d+1$. We have $n_{\CY}\ge n_\CX+d+1$. So to prove \eqref{on closed}
it suffices to show that if  a sheaf $\CF''\in\QCoh(\CY)^{\heartsuit}$ is set-theoretically
supported on $\CX$ then 
\begin{equation}  \label{e:formal_neighb}
H^r\left(\Gamma(\CY,\CF'')\right)=0 \;   \text{ for }\;    r>n_{\CX}\, .
\end{equation}
Represent $\CF''$ as a filtered colimit of sheaves $\CF''_\alpha$ so that each $\CF''_\alpha$ admits a 
finite filtration with subquotients belonging to the essential image of
$\imath_*:\QCoh(\CX)^\heartsuit\to \QCoh(\CY)^\heartsuit$. By assumption, for each $\alpha$ and each
$r>n_{\CX}$ one has $H^r\left(\Gamma(\CY,\CF''_\alpha )\right)=0$. So \eqref{e:formal_neighb} follows from
\corref{c:coconnective part}.
\end{proof}

\sssec{}

By the above, we can assume that $\CY$ is reduced. The next proposition 
is valid over any ground field. 

\begin{prop} \label{p:stratification}
There exists a finite decomposition of $\CY$ into a union of locally closed reduced algebraic
substacks $\CY_i$, each of which satisfies:

\begin{itemize}

\item The locally closed embedding $\CY_i\hookrightarrow \CY$ is quasi-compact;

\smallskip

\item There exists a finite surjective flat morphism 
$\pi :\CZ_i\to \CY_i$ with $\CZ_i$ being a quotient of a quasi-separated and quasi-compact scheme $Z_i$ 
by an action of an affine algebraic group (of finite type) over $k$. Moreover:

\medskip

\noindent(i) One can arrange so that $Z_i$ are quasi-projective over an affine scheme, and the group action is linear with
respect to this projective embedding. 

\smallskip

\noindent(ii) If $\mbox{char\,}k=0$, $\pi$ can be chosen to be \'etale.

\end{itemize}

\end{prop}

This proposition will be proved in \secref{ss:end_proof_main}.

\begin{rem}
Point (i) of the proposition will be used in the proof of Theorem~\ref{t:QCoh_core} but not in the proof of 
\lemref{l:estimate}.
\end{rem}

%\begin{rem}
%At least, if one does not insist on the linearity of the action, one can, in fact, take $\CZ =\oCY$ and
%$\pi =\id$. We will not use this fact.
%\end{rem}

We are now going to deduce \lemref{l:estimate} for $\CY$ as above from \propref{p:stratification}. 

\sssec{}

By induction and \lemref{l:devissage}, it is enough to prove \lemref{l:estimate}
for the algebraic stacks $\CY_i$. 

\sssec{}  \label{sss:coverings}

Let $\CZ_i\to \CY_i$ be a finite surjective \'etale morphism as in \propref{p:stratification}.
We claim that if Lemma~\ref{l:estimate} holds for $\CZ_i$ then it holds for $\CY_i$. 

\medskip

To see this, note that any $\CF\in \QCoh(\CY_i)$ is a direct summand of 
$\pi_*\circ \pi^*(\CF)=\CF\otimes\pi_*(\CO_{\CZ_i})$ (use the trace morphism  $\pi_*(\CO_{\CZ_i})\to\CO_{\CY_i}$).

\medskip

(Note that the last manipulation used the $\mbox{char\,}k=0$ assumption. But this is not
the most crucial place where we will use it.) 

\medskip

Thus, it is sufficient to prove \lemref{l:estimate} for a stack $\CZ$ of the form $Z/G$,
where $Z$ is a quasi-separated and quasi-compact scheme, and $G$ is an affine algebraic group
of finite type over $k$.

\ssec{Quotients of schemes by algebraic groups}  \label{ss:quotients}

Let $G$ be a reductive algebraic group over $k$. Consider the stack $BG:=\on{pt}/G$. 

\begin{lem}    \label{l:Gamma_of_BG}
The functor 
$$\Gamma(BG,-):\QCoh(BG)\to\Vect$$ 
is t-exact. More precisely,
\begin{equation}   \label{e:Gamma_of_BG}
H^i(\Gamma (BG, M))=(H^i(M))^G, \quad\quad M\in \QCoh(BG).
\end{equation}
\end{lem}

It is here that we use the characteristic $0$ assumption.

\begin{proof}
By Remark~\ref{r:left-completion} (which relies on \cite[Prop. 5.4.3]{QCoh}), $\QCoh (BG)$ is the 
left completion of $D(\CA)$, where $\CA:= \QCoh(BG)^\heartsuit$ is the abelian category of $G$-modules. 
But $\CA$ is semisimple (because $\mbox{char\,}k=0$), so  $D(\CA)$ is  left-complete and
$\QCoh (BG)=D(\CA)$. The lemma follows.
\end{proof}

\begin{rem}
In the proof of Lemma~\ref{l:Gamma_of_BG} we used \cite[Prop. 5.4.3]{QCoh}. Instead, one can argue as follows. By Lemma~\ref{l:why_estimate_suffices} and Corollary~\ref{c:coconnective part}, it suffices to 
prove \eqref{e:Gamma_of_BG} if $M\in\CA:= \QCoh(BG)^\heartsuit$. In this case
applying ~\corref{c:direct image via Cech} to the atlas $\on{pt}\to BG$ we see that $H^i(\Gamma (BG, M))$
identifies with the usual $H^i(G,M)$, and the latter is isomorphic to $\Ext^i_{\CA}(k,M)$. It remains to use the 
semisimplicity of $\CA$. % (which holds because $\mbox{char\,}k=0$).  
\end{rem}

%\begin{rem} \label{r:quotients}
%The most relevant statement from \cite{BFN} is \cite[Corollary 3.20]{BFN}.
\begin{lem}
Let $Z$ be a quasi-separated and quasi-compact scheme equipped with an action of
an affine algebraic group $G$. Then Lemma~\ref{l:estimate} holds for $\CZ=Z/G$.
\end{lem}

\begin{proof}
The canonical morphism $f:\CZ\to BG$ is schematic, quasi-separated and quasi-compact.
Embed $G$ into a reductive group $G'$ and let $f'$ be the 
composition $\CZ{\buildrel{f}\over{\longrightarrow}} BG\to BG'$. Then $f'$ is still schematic, quasi-separated
and quasi-compact, so the cohomological
amplitude of the functor $f'_*: \QCoh(\CZ)\to\Vect$  is bounded above. On the other hand, 
the functor $$\Gamma : \QCoh(BG')\to \Vect$$ is t-exact by Lemma~\ref{l:Gamma_of_BG}.
\end{proof}

\ssec{Proof of \propref{p:stratification}} \label{ss:end_proof_main}
%\ssec{Finding a suitable open substack} 

%By virtue of Sect.~\ref{induction step}-\ref{ss:coverings}, to finish the proof of  \lemref{l:estimate}
%and Theorem~\ref{main}, it suffices to prove the next proposition.

%the statement of Lemma~\ref{l:estimate}
%holds if $\CY$ is a quotient of a quasi-compact scheme by an action of an affine algebraic group. 
%We emphasize that this step relies on the assumption that we are working over a field
%of characteristic $0$.
%
%\medskip
%
%The next proposition shows that a reduced classical algebraic stack always contains a non-empty open %substack of this form. 
This will conclude the proof of \lemref{l:estimate} in view of \secref{induction step}.

\sssec{}

The proof of the proposition is based on the following lemma.
%We will use the following result of \cite{LM}, Theorem 11.5 in {\it loc.cit.}:

\begin{lem}  \label{LM}
Let $\CY\ne\emptyset$ be a classical algebraic stack, which is quasi-compact and whose
inertia stack is of finite presentation over $\CY$. Then there exists a finite decomposition of
$\CY$ into a union of locally closed reduced algebraic substacks $\CY_i$, each of which satisfies:

\begin{itemize}

\item The locally closed embedding $\CY_i\hookrightarrow \CY$ is quasi-compact;

\smallskip

\item Each $\CY_i$ admits a map $\varphi_i:\CY_i\to X'_i$, where $X'_i$ is
an affine scheme with the following property:

\smallskip

\noindent There exists a finite fppf morphism $f_i:X_i\to X'_i$, and a flat group-scheme of finite presentation 
$\CG_i$ over $X_i$ such that  $X_i\underset{X'_i}\times \CY_i$ is isomorphic to the classifying stack $B\CG_i$.

\noindent Moreover, we can always arrange so that $X_i$ and $X'_i$ are integral. In the characteristic 0 case, 
one can choose $f_i$ to be \'etale.

\end{itemize}

\end{lem}

%{\bf Should I require $X'$ to be reduced? irreducible?}

\begin{proof}
We are going to apply \cite[Theorem 11.5]{LM}. We note that in {\it loc.cit.}, it is stated under the assumption
that $\CY$ is Noetherian. However, the only place where the Noetherian hypothesis is used in the proof is to 
ensure that the inertia stack be of finite presentation over $\CY$, which is what we are imposing by assumption.

\medskip

The above theorem yields a decomposition of $\CY$ as in the lemma, with the only difference that
the morphisms $$f_i:X_i\to X'_i$$ are just fppf. We have to show that each $X'_i$ admits a finite
decomposition into a union of locally closed integral subschemes $X'_{i,j}$, each of which satisfies:

\begin{itemize}

\item The locally closed embeddings $X'_{i,j}\hookrightarrow X'_i$ are quasi-compact;

\item For every $j$, there exists a finite fppf map $g_{i,j}:\wt{X}'_{i,j}\to X'_{i,j}$, such that $f_i$ admits
a section after a base change by $g_{i,j}$. 

\smallskip

Moreover, the schemes $\wt{X}'_{i,j}$ can be chosen integral. In the characteristic 0 case, $g_{i,j}$ can be chosen \'etale.

\end{itemize}

We claim, however, that this is the case for any fppf map $f:X\to X'$ between reduced affine schemes. 
Indeed, recall that whenever 
$f:X\to X'$ is an fppf morphism of schemes with $X'$ affine, we can always realize it as a base change
$$
\CD
X  @>>>  X^0 \\
@V{f}VV  @VV{f^0}V  \\
X'  @>>>  X'{}^0,
\endCD
$$
where $f^0:X^0\to X'{}^0$ is an fppf morphisms of schemes of finite type over $k$. Hence,
our assertion reduces to the case when $X'$ is of finite type.

\medskip

In the latter case, by Noetherian induction
it is enough to show that it contains a non-empty open subset $\oX'$ with
a finite flat (in characteristic $0$, \'etale) cover $g:\wt{X}\to \oX'$, such that $f$ admits a 
section after a base change by $g$. 

\medskip

Let $K'$ denote the field of fractions of $X'$. Clearly,
$X$ has a point over some finite extension $\wt{K}'$ of $K'$. 

\medskip

Taking $\wt{X}'$ to be any integral scheme
finite over $X'$ with field of fractions $\wt{K}'$, we obtain that the map $\wt{X}'\to X$ is well-defined
over some non-empty open subset $\oX'\subset X'$, as required. Moreover in characteristic
$0$, the map $\wt{X}'\to X'$ is generically \'etale over $X'$, since $\wt{K}'/K'$ is separable.
%First, let us replace the initial $\CY$ by $\oCY$ provided by \lemref{LM}. The scheme $X$ is automatically 
%reduced (because $\CY$ is),  and we can assume that $X$ is irreducible. Let $K$ be the field of rational 
%functions on $X$. The generic fiber of the morphism $X'\to X$  has a point over a finite extension of $K$ 
%(which is separable by the characteristic 0 assumption). So after shrinking $X$ one gets
%a morphism $X''\to X'$ such that the composition $X''\to X'\to X$ is finite \'etale (and $X'$ is also reduced 
%and irreducible). So, we can replace the initial $X'$ by $X''$, and thus assume that $f$ is finite (and in fact 
%\'etale, but we will not need the latter fact for now).
\end{proof}

\begin{proof}[Proof of Proposition~\ref{p:stratification}]
Let $\CY_i$, $X'_i$, $X_i$, and $\CG_i$ be as in Lemma~\ref{LM}. 
Note that for each field-valued point of $X_i$, 
the fiber of $\CG_i$ at it identifies with the group of automorphisms of the corresponding point of $\CY_i$.
Therefore, by the QCA condition, all these groups are affine.

\medskip

As the index $i$ will be fixed, for the rest of the proof, we shall suppress it from the
notation.

\medskip

It is sufficient to show that $X'$ admits a finite decomposition into a union of locally closed 
reduced subschemes $X'_l$, each of which satisfies:

\begin{itemize}

\item The locally closed embedding $X'_l\hookrightarrow X'$ is quasi-compact;

\smallskip

\item The stack $$\CZ_l:=B\CG\underset{X'}\times X'_l$$
(which is tautologically the same as $(X\underset{X'}\times \CY)\underset{\CY}\times (\CY\underset{X'}\times X'_l)$,
viewed as equipped with a map to $\CY\underset{X'}\times X'_l$), is isomorphic 
to a stack of the form $Z_l/G_l$, where $Z_l$ is a quasi-separated and quasi-compact scheme,
and $G_l$ is an affine algebraic group of finite type over $k$. Moreover, $Z_l$ can be chosen
to be quasi-projective over an affine scheme, and the action of $G_l$ on it linear with respect
this projective embedding.

\end{itemize} 

\medskip

Since $\CG$ and $X$ are of finite presentation over $X'$, they come by base change
from a map $X'\to X'{}^0$, where $X'{}^0$ is of finite type over $k$. Hence, is is enough to prove
the assertion in the case when $X'$ (and hence 
$X$ and $\CG$) are of finite type.

\medskip

In the latter case, by Noetherian induction, it is sufficient to find a non-empty open subset
$\oX'\subset X'$, such that $B\CG\underset{X'}\times \oX'$ is of the form $Z/G$ specified
above. Moreover, since the morphism $X\to X'$ is finite, it is sufficient to find the corresponding
open $\oX$ in $X$. 

\medskip

Recall that $X$ was assumed integral. Let $K$ be the field of fractions of $X$. Let
$$\CG_{K}:=\CG\underset{X}\times \on{pt}$$
be the corresponding algebraic group over $K$. Since $\CG_K$ is affine, we can embed
it into $GL(n)_{K}:=GL(n)\times \on{pt}$. 

\medskip

By Chevalley's theorem,
$$Z_{K}:=GL(n)_{K}/\CG_{K}$$
is a quasi-projective scheme over $K$ equipped with a linear action of $GL(n)$. 

\medskip

Hence, there exists a non-empty open subscheme $\oX\subset X$, such that 
$\CG|_{\oX}$ admits a map into $GL(n)\times \oX$,
and the stack-theoretic quotient 
$$(GL(n)\times \oX)/(\CG|_{\oX})$$
is isomorphic to a quasi-projective scheme $Z$ over $\oX$, and moreover the natural
action of $GL(n)$ on it is linear.

\medskip

Thus, $B\CG|_{\oX}\simeq Z/GL(n)$, as required.

\end{proof}

\ssec{Proof of Theorem~\ref{t:QCoh_core}} \label{ss:proofQCoh_core}
Below we give a direct proof. In the case when $\CY$ is locally almost of finite type, 
one can deduce Theorem~\ref{t:QCoh_core} from Proposition~\ref{p:coh_generates}, 
as explained in Remark~\ref{r:generation of IndCoh}.

\sssec{Reduction to the reduced classical case}  \label{sss:reduction_proofQCoh_core}
Let $^{cl}\CY\overset{^{cl}\!i}\hookrightarrow \CY$ be the embedding of the  classical stack 
underlying $\CY$.  We claim that $\QCoh(\CY)$ is generated by the essential image of the functor
$^{cl}\!i_*$. To see this, use the filtration of $\CF\in\QCoh(\CY )$ by objects
$\CF\otimes\tau^{\le -n}(\CO_{\CY})$, $n\in\BZ_+$, which is finite by the eventual coconnectivity
assumption.

\medskip

So without loss of generality we can assume that $\CY$ is classical.  A similar argument
allows to assume that $\CY$ is reduced.

\sssec{}

Using \propref{p:stratification}, the statement of the theorem results from the combination
of the following three lemmas:

\begin{lem}   \label{l:BFN}
Let $Z$ be a quasi-projective scheme equipped 
with a linear action of an affine algebraic group $G$. 
Then $\QCoh(Z/G)$ is generated by the heart of its t-structure.
\end{lem}

\begin{lem} \label{l:covering plus}
If $\CZ\to \CY$ is a finite \'etale map, and $\QCoh(\CZ)$ is generated by the heart of its
t-strcuture, then the same is true for $\QCoh(\CY)$.
\end{lem}

\begin{lem} \label{l:devissage plus}
In the situation of \lemref{l:devissage}, if both $\QCoh(\oCY)$ and $\QCoh(\CX)$
are generated by the hearts of their t-structures, then the same is true for 
$\QCoh(\CY)$.
\end{lem}

\sssec{Proof of \lemref{l:BFN}}

%We claim that in the situation of the lemma, $\QCoh(Z/G)$ is generated by perfect objects. Indeed,
It is easy to see that $\QCoh(Z/G)$ is generated by objects of the form $\CO_Z(-i)$, where $\CO_Z(1)$
denotes the corresponding ample line bundle on $Z$.

\qed

\sssec{Proof of \lemref{l:covering plus}}

This follows from the fact that every object $\CF\in\CO (\CY)$ is a direct summand of $\pi_*\circ \pi^*(\CF)$, see
\secref{sss:coverings}.

\qed

\sssec{Proof of \lemref{l:devissage plus}}    \label{sss:end_of}

Let $\QCoh(\CY )^\spadesuit\subset\QCoh(\CY )$ be the subcategory generated by $\QCoh(\CY )^\heartsuit$.
The subcategory $\QCoh(\CY )^\spadesuit$ contains the essential images of the functors 
$$\jmath_*:\QCoh(\oCY )\to\QCoh(\CY ),\quad \imath_*:\QCoh(\CX)\to\QCoh(\CY )$$
because Theorem~\ref{t:QCoh_core} holds for $\oCY$ and $\CX$, and the above functors have bounded cohomological amplitude. 
We have to show that each $\CF\in\QCoh(\CY )$ belongs to 
$\QCoh(\CY )^\spadesuit$. 

\medskip

Consider the exact triangle
\begin{equation} \label{e:Cousin}
(\CO_{\CY})_{\CX}\to \CO_{\CY}\to \jmath_*\circ \jmath^*(\CO_{\CY}),
\end{equation}
where 
$$(\CO_{\CY})_{\CX}:=\on{Cone}\left(\CO_{\CY}\to \jmath_*\circ \jmath^*(\CO_{\CY})\right)[-1].$$

\medskip

The object $(\CO_{\CY})_{\CX}$ is bounded, and each of its cohomologies admits a filtration 
with subquotients that lies in the essential image of $\imath_*$. Hence, for any $\CF\in \QCoh(\CY)$,
the object $\CF\otimes H^i((\CO_{\CY})_{\CX})$ also admits a filtration 
with subquotients (i.e., the cones of the maps of one term of the filtration into the next)
that lie in the essential image of $\imath_*$. In particular, 
$\CF\otimes (\CO_{\CY})_{\CX}\in \QCoh(\CY )^\spadesuit$. 

\medskip

Tensoring \eqref{e:Cousin} by $\CF$, we obtain an exact triangle
$$\CF\otimes (\CO_{\CY})_{\CX}\to \CF\to \jmath_*\circ \jmath^*(\CF),$$
which implies our assertion.

\qed

\begin{rem}    \label{r:end_of}
If $\CY$ is locally Noetherian and $\CF$ is perfect, then the object
$\CF\otimes (\CO_{\CY})_{\CX}$ is isomorphic to
$$\underset{n}{\underset{\longrightarrow}{lim}}\, (\imath_n)_*\circ \imath_n^!(\CF),$$
where $\imath_n$ denotes the embedding of the $n$-th infinitesimal neighborhood
of $\CX$. This is not necessarily true without the perfectness condition.
In general, the $!$-pullback functor is ``bad" (no continuity, no commutation with base change), just
like the $*$-pushforward with respect to a non-quasi-compact morphism 
(see Sect.~\ref{s:dir im gen}). 

\medskip

However, this state of affairs with the $!$-pullback functor can be remedied by replacing
the category $\QCoh(\CY)$ by $\IndCoh(\CY)$, considered in the next section.

\end{rem}

\section{Implications for ind-coherent sheaves}  \label{s:IndCoh}

This and the next section are concerned with the category $\IndCoh$ on algebraic stacks
and, more generally, prestacks. As was mentioned in the introduction, $\IndCoh$ is another
natural paradigm for ``sheaf theory" on stacks.

\medskip

However, the reader, who is only interested in applications to D-modules, may skip these two
sections. Although it is more natural to connect D-modules to the category $\IndCoh$,
it will be indicated in \secref{sss:Dmod and QCoh on stacks} that if our algebraic stack is eventually
coconnective, one can bypass $\IndCoh$, and relate $\Dmod$ to $\QCoh$ directly.
The only awkwardness that will occur is the relation between Verdier duality on coherent 
D-modules and Serre duality on coherent sheaves, the latter being more naturally
interpreted within $\IndCoh$ rather than $\QCoh$.

\medskip

The material in this section is organized as follows. In \secref{ss:laft} we recall the 
condition of being ``locally almost of finite type". In \secref{ss:review of IndCoh}
we recall the basic facts about the category $\IndCoh$. In Sects. \ref{ss:coherent}-\ref{ss:CohgeneratesIndCoh} 
we prove the compact generation and describe the category of compact objects of $\IndCoh$ on a QCA
algebraic stack. In \secref{ss:indcoh dir image} we introduce the functor of direct image
on $\IndCoh$ for maps between QCA algebraic stacks. 

\ssec{The ``locally almost of finite type" condition} \label{ss:laft}

Unlike $\QCoh$, the category $\IndCoh$ (and also $\Dmod$, considered later in the paper)
only makes sense on (pre)stacks that satisfy a certain finite-typeness hypothesis, 
called ``locally almost of finite type".

\medskip

For general prestacks this condition may seem as too technical (we review it below). It does appear
simpler when applied to algebraic stacks. The reader will not lose much by considering only
those prestacks that are algebraic stacks; all the new results in this paper that concern $\IndCoh$ 
and $\Dmod$ are about algebraic stacks. 

\medskip

We shall nevertheless, discuss $\IndCoh$ in the framework of arbitrary prestacks locally almost of 
finite type, because this seems to be the natural level of generality.

\sssec{}

An affine DG scheme $\Spec(A)$ is said to be almost of finite type over $k$ if 

\begin{itemize}

\item $H^0(A)$ is a finitely generated algebra over $k$.

\item Each $H^{-i}(A)$ is finitely generated as a module over $H^0(A)$.

\end{itemize} 

The property of being almost of finite type is local with respect to Zariski topology. 
A DG scheme $Z$ is said to be locally of almost finite type if it can be covered by 
affines, each of which is almost of finite type. Equivalently, $Z$ is locally of almost finite 
type if any of its open affine subschemes is of almost finite type.

\medskip

We shall denote the corresponding full subcategories of 
$$\on{DGSch}^{\on{aff}}\subset \on{DGSch}_{\on{qs-qc}}\subset \on{DGSch}$$
by
$$\on{DGSch}^{\on{aff}}_{\on{aft}}\subset \on{DGSch}_{\on{aft}}\subset \on{DGSch}_{\on{laft}},$$
respectively.

\begin{defn}  \label{d:laft}
An algebraic stack $\CY$ is \emph{locally of almost finite type} if it admits an atlas
$(Z,f:Z\to \CY)$, where the DG scheme $Z$ is locally almost of finite type
(in which case, for any atlas, the DG scheme $Z$ will have this property). 
\end{defn}

\sssec{}

We shall now proceed to the definition of prestacks locally almost of finite type. As we mentioned above,
the reader is welcome to skip the remainder of this subsection and replace every occurence of
the word ``prestack" by ``algebraic stack". The material
here is taken from \cite[Sect. 1.3]{Stacks}.

\medskip

First, we fix an integer $n$, anc consider the full subcategory
$$^{\leq n}\!\on{DGSch}^{\on{aff}}\subset \on{DGSch}^{\on{aff}}$$
of $n$-coconnective affine DG schemes, i.e., those $S=\Spec(A)$, for which
$H^{-i}(A)=0$ for $i>n$. 

\medskip

Let $^{\leq n}\!\on{PreStk}$ denote the category of all functors
$$({}^{\leq n}\!\on{DGSch}^{\on{aff}})^{\on{op}}\to \inftygroup.$$

\medskip

\begin{defn}
An object ${}^{\leq n}\!\on{PreStk}$ is said to be \emph{locally of finite type} it it sends filtered
limits in $^{\leq n}\!\on{DGSch}^{\on{aff}}$ to colimits in $\inftygroup$.
\end{defn}

Denote by $^{\leq n}\!\on{PreStk}_{\on{lft}}$ the full subcategory of $^{\leq n}\!\on{PreStk}$ spanned
by objects locally of finite type.

\medskip

Denote 
$$^{\leq n}\!\on{DGSch}^{\on{aff}}_{\on{ft}}:={}^{\leq n}\!\on{DGSch}^{\on{aff}}\cap \on{DGSch}^{\on{aff}}_{\on{aft}}.$$
We note that $^{\leq n}\!\on{DGSch}^{\on{aff}}_{\on{ft}}$ identifies with the subcategory of cocompact objects in 
$^{\leq n}\!\on{DGSch}^{\on{aff}}$. Therefore, the Yoneda functor
$$^{\leq n}\!\on{DGSch}^{\on{aff}}\to {}^{\leq n}\!\on{PreStk}$$ sends 
$^{\leq n}\!\on{DGSch}^{\on{aff}}_{\on{ft}}$ to ${}^{\leq n}\!\on{PreStk}_{\on{lft}}$.

\medskip

It is not difficult to show that the image of entire category
$$^{\leq n}\!\on{DGSch}_{\on{lft}}:={}^{\leq n}\!\on{DGSch}\cap \on{DGSch}_{\on{aft}}$$
under the natural functor $^{\leq n}\!\on{DGSch}\to {}^{\leq n}\!\on{PreStk}$ is contained
in $^{\leq n}\!\on{PreStk}_{\on{lft}}$.

\sssec{}

We can reformulate the condition on an object $\CY\in {}^{\leq n}\!\on{PreStk}$ to be locally of finite type
in any of the following equivalent ways:

\medskip

\noindent(i) $\CY$ is the left Kan extension along the fully faithful embedding
$^{\leq n}\!\on{DGSch}^{\on{aff}}_{\on{ft}}\hookrightarrow {}^{\leq n}\!\on{DGSch}^{\on{aff}}$.

\medskip

\noindent(ii) The functor
$$({}^{\leq n}\!\on{DGSch}^{\on{aff}}_{\on{ft}})_{/\CY}\to ({}^{\leq n}\!\on{DGSch}^{\on{aff}})_{/\CY}$$
is cofinal.

\medskip

\noindent(iii) For every $S\in {}^{\leq n}\!\on{DGSch}^{\on{aff}}$ and $y:S\to \CY$, the category of its
factorizations as $S\to S'\to \CY$, where $S'\in {}^{\leq n}\!\on{DGSch}^{\on{aff}}_{\on{ft}}$, is contractible
(in particular, non-empty).

\sssec{}

We now recall the following definition from \cite[Sect. 1.2]{Stacks}:

\begin{defn}
An object $\CY\in \on{PreStk}$ is \emph{convergent} if for every $S\in \on{DGSch}$,
the natural map
$$\underset{n}{\underset{\longleftarrow}{lim}}\, \CY({}^{\leq n}\!S)\to \CY(S)$$
is an isomorphism in $\inftygroup$.
\end{defn}

In the above formula, the operation $S\mapsto {}^{\leq n}\!S$ is that of $n$-coconnective
truncation, i.e., if $S=\Spec(A)$, then $^{\leq n}\!S=\Spec(\tau^{\geq -n}(A))$. 

\medskip

For example, all algebraic stacks are convergent, see \cite[Proposition 4.5.2]{Stacks}.

\sssec{}

Finally, we can give the following definition:

\begin{defn} \label{d:laft prestacks}
An object $\CY\in \on{PreStk}$ is locally almost of finite type if:

\begin{itemize}

\item It is convergent;

\item For every $n$, the restriction $\CY|_{^{\leq n}\!\on{DGSch}^{\on{aff}}}\in {}^{\leq n}\!\on{PreStk}$ 
belongs to $^{\leq n}\!\on{PreStk}_{\on{lft}}$.

\end{itemize}

\end{defn}

The full subcategory of $\on{PreStk}$ spanned by prestacks locally almost of finite type
is denoted $\on{PreStk}_{\on{laft}}$.

\medskip 

It is shown in \cite[Proposition 4.9.2]{Stacks} that an algebraic stack is locally almost of finite type
in the sense of Definition \ref{d:laft} if and only if it is locally almost of finite type as a prestack
in the sense of Definition \ref{d:laft prestacks}.

\sssec{}

Here is an alternative way to introduce the category $\on{PreStk}_{\on{laft}}$. Let
$^{<\infty}\!\on{DGSch}^{\on{aff}}_{\on{aft}}$ denote the full subcategory of
$\on{DGSch}^{\on{aff}}_{\on{aft}}$ spanned by eventually coconnective affine DG
schemes.

\medskip

We have the following assertion (see \cite[Sect. 1.3.11]{Stacks}):

\begin{lem}
The restriction functor under $^{<\infty}\!\on{DGSch}^{\on{aff}}_{\on{aft}}\hookrightarrow \on{DGSch}^{\on{aff}}$
defines an equivalence
$$\on{PreStk}_{\on{laft}}\to \on{Funct}\left(({}^{<\infty}\!\on{DGSch}^{\on{aff}}_{\on{aft}})^{\on{op}},\inftygroup\right).$$
The inverse functor is the composition of the left Kan extension along
$$^{<\infty}\!\on{DGSch}^{\on{aff}}_{\on{aft}}\hookrightarrow {}^{<\infty}\!\on{DGSch}^{\on{aff}},$$
followed by the right Kan extension along
$$^{<\infty}\!\on{DGSch}^{\on{aff}}\hookrightarrow \on{DGSch}^{\on{aff}}.$$
\end{lem}

\noindent{\bf Change of conventions:} From now and until \secref{s:gen alg stacks},
all DG schemes, algebraic stacks and prestacks will be assumed locally almost of finite type,
unless explicitly specified otherwise.

\medskip

\ssec{The category $\IndCoh$}  \label{ss:review of IndCoh}

For the reader's convenience we shall now summarize some of the key properties 
of the category $\IndCoh$ that will be used in the paper. The general reference
for this material in \cite{IndCoh}.

\sssec{}  \label{sss:IndCoh schemes}

Given a quasi-compact DG scheme $Z$, one introduces the category $\IndCoh(Z)$ as the
ind-completion of the category $\Coh(Z)$, 
the latter being the full subcategory of $\QCoh(Z)$ that consists of bounded complexes 
with coherent cohomology sheaves; see \cite[Sect. 1.1]{IndCoh}. See \secref{sss:ind-compl}
where the notion of ind-completion of a DG category is recalled.

\medskip

The category $\IndCoh(Z)$ is naturally a module over $\QCoh(Z)$, when the latter is regarded 
as a monoidal category with respect to the usual tensor product operation, see \cite[Sect. 1.4]{IndCoh}.

\medskip

For a morphism $f:Z_1\to Z_2$ of quasi-compact DG schemes, we have a canonically
defined functor
$$f^!:\IndCoh(Z_2)\to \IndCoh(Z_1),$$
see \cite[Corollary 5.2.4]{IndCoh}. 

\medskip

Moreover, this functor has a canonically defined structure of map
between module categoried over $\QCoh(Z_2)$, where $\QCoh(Z_2)$ acts on $\IndCoh(Z_1)$ via
the monoidal functor $f^*:\QCoh(Z_2)\to \QCoh(Z_1)$; see \cite[Theorem 5.5.5]{IndCoh}. 

\medskip

The assignment $Z\mapsto \IndCoh(Z)$ with the above !-pullback operation
is a functor
$$(\on{DGSch}_{\on{aft}})^{\on{op}}\to \StinftyCat_{\on{cont}},$$
denoted $\IndCoh^!_{\on{DGSch}_{\on{aft}}}$, see \cite[Sect. 5.6.1]{IndCoh}.

\medskip

We shall denote by $\omega_Z$ the object of $\IndCoh(Z)$ equal to $p_Z^!(k)$, where 
$$p_Z:Z\to \on{pt}.$$
We refer to $\omega_Z$ as the ``dualizing sheaf'' on $Z$.

\medskip

The functor $\IndCoh^!_{\on{DGSch}_{\on{aft}}}$ satisfies Zariski descent (see
\cite[Proposition 4.2.1]{IndCoh}.

\medskip

In fact, something stronger is true: according to \cite[Theorem 8.3.2]{IndCoh},
the functor $\IndCoh^!_{\on{DGSch}_{\on{aft}}}$ satisfies fppf descent.

\medskip

The following property of the !-pullback functor will be used in the sequel
(see \cite[Proposition 8.1.2]{IndCoh}): 

\begin{lem}  \label{l:! cons IndCoh}
Let a morphism $f:Z_1\to Z_2$ be surjective at the level of geometric points. Then the functor
$f^!:\IndCoh(Z_2)\to \IndCoh(Z_1)$ is conservative.
\end{lem}

\sssec{}  \label{sss:IndCoh schemes ten}

For two quasi-compact DG schemes $Z_1$ and $Z_2$ there is a naturally defined functor
$$\IndCoh(Z_1)\otimes \IndCoh(Z_2)\overset{\boxtimes}\longrightarrow \IndCoh(Z_1\times Z_2),$$
which is an equivalence by \cite[Proposition 4.6.2]{IndCoh}. (The last assertion uses the assumption
that $\on{char}(k)=0$ in an essential way.)

\medskip

In particular, we obtain a functor
$$\IndCoh(Z)\otimes \IndCoh(Z)\overset{\boxtimes}\longrightarrow \IndCoh(Z\times Z)\overset{\Delta_Z^!}\longrightarrow \IndCoh(Z),$$
that we shall denote by $\CF_1,\CF_2\mapsto \CF_1\sotimes \CF_2$. This functor makes $\IndCoh(Z)$ into a symmetric monoidal
category with the unit given by $\omega_Z$.

\sssec{}  \label{sss:Ind and QCoh}

The categories $\IndCoh(Z)$ and $\QCoh(Z)$ are closely related:

\medskip

The category $\IndCoh(Z)$ has a naturally defined t-structure (induced by one on $\Coh(Z)$). We also
have a naturally defined t-exact continuous functor 
$$\Psi_Z:\IndCoh(Z)\to \QCoh(Z),$$
characterized by the property that it is the identity functor from $\Coh(Z)\subset \IndCoh(Z)$
to $\Coh(Z)\subset \QCoh(Z)$, see  \cite[Secst. 1.1.5 and 1.2.1]{IndCoh}.

\medskip

The induced functor on the corresponding eventually
coconnective (a.k.a. bounded below) subcategories
$$\IndCoh(Z)^+\to \QCoh(Z)^+$$
is an equivalence, see \cite[Proposition 1.2.4]{IndCoh}. 

\medskip

We should add that the t-structure on $\IndCoh(Z)$ is compatible with filtered colimits,
but it is \emph{not} left-complete, unless $Z$ is a smooth classical scheme, in which case $\Psi_Z$
is an equivalence. In fact, $\QCoh(Z)$ is always equivalent to the \emph{left completion} of
$\IndCoh(Z)$ with respect to its t-structure, \cite[Proposition 1.3.4]{IndCoh}. 

\medskip

When $Z$ is eventually coconnective, the functor $\Psi_Z$ is a colocalization (see \cite[Proposition 1.5.3]{IndCoh});
in particular, in this case it is essentially surjective.

\sssec{}

Let $f:Z_1\to Z_2$ be again a map between quasi-compact DG schemes. There exists a continuous functor 
$$f_*^{\IndCoh}:\IndCoh(Z_1)\to \IndCoh(Z_2),$$
uniquely defined by the condition that the diagram
\begin{equation} \label{e:*-pushforward for IndCoh}
\CD
\IndCoh(Z_1)  @>{\Psi_{Z_1}}>>  \QCoh(Z_1) \\
@V{f^{\IndCoh}_*}VV    @VV{f_*}V  \\
\IndCoh(Z_2)  @>{\Psi_{Z_1}}>>  \QCoh(Z_2).
\endCD
\end{equation}
commutes, see \cite[Proposition 3.1.1]{IndCoh}. 

\medskip

The functors of !-pullback and $(\IndCoh,*)$-pushforward
are endowed with base change isomorphisms for Cartesian squares of DG schemes. I.e., for a Cartesian
square
\begin{equation} \label{e:cart diag schemes}
\CD 
Z'_1  @>{g_1}>> Z_1  \\
@V{f'}VV    @VV{f}V   \\
Z'_2  @>{g_2}>> Z_2
\endCD
\end{equation}
there is a canonical isomorphism
\begin{equation} \label{e:base change IndCoh}
g_2^!\circ f^{\IndCoh}_*\simeq (f')^{\IndCoh}_*\circ g_1^!;
\end{equation}
see \cite[Theorem 5.2.2]{IndCoh} for a precise formulation. Note that in \eqref{e:base change IndCoh}
there is no adjunction that would produce a morphism in either direction.

\medskip

For $\CF_i\in \IndCoh(Z_i)$, consider the object $\CF_1\boxtimes \CF_2\in \IndCoh(Z_1\times Z_2)$. Applying
\eqref{e:base change IndCoh} to
$$
\CD
Z_1 @>{\on{Graph}_f}>>  Z_1\times Z_2 \\
@V{f}VV   @VV{f\times \on{id}}V   \\
Z_2 @>{\Delta_{Z_2}}>>  Z_2\times Z_2,
\endCD
$$
we deduce that $f$ satisfies the projection formula for $\IndCoh$:
\begin{equation} \label{e:proj formula IndCoh}
\CF_2\sotimes f^{\IndCoh}_*(\CF_1)\simeq f^{\IndCoh}_*(f^!(\CF_2)\sotimes \CF_1).
\end{equation}

\sssec{}

Assume that the map $f$ is eventually coconnective; see \cite[Definition 3.5.2]{IndCoh}, where this notion is introduced.
Note that this is equivalent to $f$ being finite Tor-dimension, see \cite[Lemma 3.6.3]{IndCoh}.

\medskip

In this case there also exists a functor 
$$f^{\IndCoh,*}:\IndCoh(Z_2)\to \IndCoh(Z_1),$$
uniquely defined by the condition that the diagram
\begin{equation} 
\CD
\IndCoh(Z_1)  @>{\Psi_{Z_1}}>>  \QCoh(Z_1) \\
@A{f^{\IndCoh,*}}AA    @AA{f^*}A  \\
\IndCoh(Z_2)  @>{\Psi_{Z_1}}>>  \QCoh(Z_2).
\endCD
\end{equation}
commutes, see \cite[Proposition 3.5.4]{IndCoh}, and which is the
left adjoint to $f_*^{\IndCoh}$.

\medskip

For a Cartesian diagram \eqref{e:cart diag schemes}, in which the vertical arrows are eventually coconnective, 
the natural transformation \begin{equation} \label{e:! and *}
(f')^{\IndCoh,*}\circ g_2^!\to g_1^!\circ g^{\IndCoh,*}
\end{equation}
that arises by adjunction from \eqref{e:base change IndCoh}, is an isomorphism
(see \cite[Proposition 7.1.6]{IndCoh}).

\medskip

If the map $f$ is smooth (or, more generally, Gorenstein), then we have:
\begin{equation} \label{e:! and * pullback}
f^!(-)\simeq \CK_{Z_1/Z_2}\otimes f^{\IndCoh,*}(-),
\end{equation}
where $\CK_{Z_1/Z_2}$ is the relative dualizing graded
line bundle (see \cite[Proposition 7.3.8]{IndCoh}). In the above
formula, tensor product is understood in the sense of the
monoidal action of $\QCoh(Z)$ on $\IndCoh(Z)$.

\medskip

For a Cartesian diagram \eqref{e:cart diag schemes} with the horizontal maps being eventually
coconnective, the natural transformation
\begin{equation} \label{e:usual base change for IndCoh}
g_2^{\IndCoh,*}\circ f^{\IndCoh}_*\to (f')^{\IndCoh,*}\circ g_1^{\IndCoh,*},
\end{equation}
obtained by adjunction from 
$$f^{\IndCoh}_*\circ (g_2)^{\IndCoh}_*\simeq (g_1)^{\IndCoh}_*\circ (f')^{\IndCoh}_*,$$
is an isomorphism, see \cite[Lemma 3.6.9]{IndCoh}.

\sssec{}

Let now $\CY$ be a prestack. We define the category $\IndCoh(\CY)$ as
\begin{equation} \label{e:limit for IndCoh}
\underset{(S,g)\in ((\on{DGSch}_{\on{aft}})_{/\CY})^{\on{op}}}{\underset{\longleftarrow}{lim}}\, \IndCoh(S),
\end{equation}
where we view the assignment $(S,g)\rightsquigarrow \IndCoh(S)$ as a functor between
$\infty$-categories 
$$((\on{DGSch}_{\on{aft}})_{/\CY})^{\on{op}}\to \StinftyCat_{\on{cont}},$$
obtained by restriction under the forgetful map $(\on{DGSch}_{\on{aft}})_{/\CY}\to \on{DGSch}_{\on{aft}}$
of the functor
$$\IndCoh^!_{\on{DGSch}{\on{aft}}}:\on{DGSch}_{\on{aft}}^{\on{op}}\to \StinftyCat_{\on{cont}},$$
mentioned above. As in the case of $\QCoh$, the limit is taken in the $(\infty,1)$-category 
$\StinftyCat_{\on{cont}}$.

\medskip

Concretely, an object $\CF\in \IndCoh(\CY)$ is an assignment for 
$$(g:S\to \CY)\in (\on{DGSch}_{\on{aft}})_{/\CY} \rightsquigarrow g^!(\CF)\in \IndCoh(S),$$ 
and of a homotopy-coherent system of isomorphisms
$$f^!(g^!(\CF))\simeq (g\circ f)^!(\CF)\in \IndCoh(S')$$
for $f:S'\to S$. 

\medskip

In forming the above limit we can replace the category $\on{DGSch}_{\on{aft}}$ of quasi-compact DG schemes
by $\on{DGSch}^{\on{aff}}_{\on{aft}}$ of affine DG schemes; this is due to the Zariski descent property of $\IndCoh$,
see \cite[Corollaries 10.2.2 and 10.5.5]{IndCoh}.
Furthermore, we can replace the category $\on{DGSch}_{\on{aft}}$ (resp., $\on{DGSch}^{\on{aff}}_{\on{aft}}$)
by any of the indexing categories $A$ that appear in \secref{sss:change index qc}.

\medskip

The compatibility of !-pullbacks with the action of $\QCoh$ implies that 
the category $\IndCoh(\CY)$ has a natural structure of module over the monoidal category $\QCoh(\CY)$.

\sssec{}

If $\pi:\CY_1\to \CY_2$ is a map of prestacks, we have a tautologically defined functor
$\pi^!:\CY_1\to \CY_2$. 

\medskip

In particular, for any $\CY$, we obtain a canonical object $\omega_\CY\in \IndCoh(\CY)$
equal to $p_\CY^!(k)$, where $p_\CY:\CY\to \on{pt}$. We refer to $\omega_\CY$ as
``the dualizing sheaf" on $\CY$.

\medskip

For two prestacks $\CY_1$ and $\CY_2$ there exists a naturally defined functor
$$\IndCoh(\CY_1)\otimes \IndCoh(\CY_1)\overset{\boxtimes}\longrightarrow \IndCoh(\CY_1\times \CY_2).$$
In particular, as in the case of schemes, $\IndCoh(\CY)$ acquires a structure of symmetric monoidal
category via the operation $\sotimes$.

\sssec{} \label{sss:functoriality IndCoh prestacks pushforward}

Let $\pi:\CY_1\to \CY_2$ be a schematic and quasi-compact map between prestacks. Then the functor
of direct image on $\IndCoh$ for DG schemes gives rise to a functor
$$\pi_*^{\IndCoh}:\IndCoh(\CY_1)\to \IndCoh(\CY_2).$$

Namely, for $(S_2,g_2)\in (\on{DGSch}_{\on{aft}})_{/\CY}$, we set
$$g_2^!(\pi_*^{\IndCoh}(-)):=(\pi_S)^{\IndCoh}_*\circ g_1^!(-)$$
for the morphisms in the Cartesian diagram
$$
\CD
S_1 @>{g_1}>>  \CY_1 \\
@V{\pi_S}VV    @VV{\pi}V    \\
S_2  @>{g_2}>>  \CY_2.
\endCD
$$

The data of compatibility of the assignment 
$$(S_2,g_2)\rightsquigarrow (\pi_S)^{\IndCoh}_*\circ g_1^!(-)$$
under !-pullbacks for maps in $(\on{DGSch}_{\on{aft}})_{/\CY}$
is given by base change isomorphisms \eqref{e:base change IndCoh};
see \cite[Sect. 10.6]{IndCoh}.

\medskip

The resulting functor $\pi_*^{\IndCoh}$ is itself also endowed with base change isomorphisms
with respect to !-pullbacks for Cartesian diagrams of prestacks
\begin{equation} \label{e:card diag stacks}
\CD
\CY'_1  @>{\phi_1}>> \CY_1 \\
@V{\pi'}VV    @VV{\pi}V  \\
\CY'_2   @>{\phi_2}>>  \CY_2
\endCD
\end{equation}
where the vertical maps are schematic and quasi-compact. 

\medskip

By construction, the projection formula for maps between quasi-compact schemes, i.e., 
\eqref{e:proj formula IndCoh}, implies one for $\pi$. That is, we have a functorial isomorphism
$$\CF_2\sotimes \pi^{\IndCoh}_*(\CF_1)\simeq \pi^{\IndCoh}_*(\pi^!(\CF_2)\sotimes \CF_1),\quad 
\CF_i\in \IndCoh(\CY_i).$$

\sssec{} \label{sss:functoriality IndCoh prestacks * pullback}

Let $\CY_i$ be prestacks, and let  $\pi:\CY_1\to \CY_2$ be a morphism which is $k$-representable for
some $k$. In this paper we will only need the cases of either $\pi$ being schematic, or $1$-representable
(the latter means that the base change of $\pi$ by an affine DG scheme yields a $1$-Artin stack).

\medskip

Assume also that $\pi$ is eventually coconnective, see \cite[Sect. 11.1.2]{IndCoh}. In this case,
by \cite[Sect. 11.6]{IndCoh}, we have a continuous functor 
$$\pi^{\IndCoh,*}:\IndCoh(\CY_2)\to \IndCoh(\CY_1).$$

\medskip

For a Cartesian diagram \eqref{e:card diag stacks}, in which the vertical arrows are $k$-representable and 
eventually coconnective, we have a canonical isomorphism
\begin{equation} \label{e:! and * stacks}
(\pi')^{\IndCoh,*}\circ \phi_2^!\simeq \phi_1^!\circ \pi^{\IndCoh,*},
\end{equation}
see \cite[Proposition 11.6.2]{IndCoh}.
Note that unlike \eqref{e:! and *}, in \eqref{e:! and * stacks} there is no a priori map 
in either direction.

\medskip

If $f$ is smooth (or, more generally, Gorenstein), the functors $\pi^{\IndCoh,*}$ and $\pi^!$ are related 
by the formula
\begin{equation} \label{e:! and * pullback stacks}
\pi^!(-)\simeq \CK_{\CY_1/\CY_2}\otimes \pi^{\IndCoh,*}(-),
\end{equation}
where $\CK_{\CY_1/\CY_2}$ is the relative dualizing line bundle. This is not explicitly stated in \cite{IndCoh},
but can be obtained by combining the functorial isomorphisms \eqref{e:! and * pullback} for 
morphisms between DG schemes, and \eqref{e:! and * stacks}.

\medskip

If $\pi$ is schematic and quasi-compact, the functors $(\pi^{\IndCoh,*},\pi^{\IndCoh}_*)$ form an adjoint pair. The latter fact is not
stated explicitly in \cite{IndCoh} either, but follows from \eqref{e:! and * stacks} via an analog
of \secref{sss:change index qc}(ii) for $\IndCoh$. 

\sssec{}  \label{sss:IndCoh for algebraic stacks}

When $\CY$ is an algebraic stack, the category $\IndCoh(\CY)$ can be described more explicitly.

\medskip

First, as in \secref{sss:change index qc}(iv), in the formation of the limit \eqref{e:limit for IndCoh}, 
we can replace the category $(\on{DGSch}_{\on{aft}})_{/\CY}$ by $\on{DGSch}_{/\CY,\on{smooth}}$, 
see \cite[Corollary 11.2.4]{IndCoh}. 

\medskip

Furthermore, when we use $(\on{DGSch}_{/\CY,\on{smooth}})^{\on{op}}$ as the 
indexing category, $\IndCoh(\CY)$ can be also realized as the limit
\begin{equation} \label{e:*-limit for IndCoh}
\underset{(S,g)\in (\on{DGSch}_{/\CY,\on{smooth}})^{\on{op}}}{\underset{\longleftarrow}{lim}}\, \IndCoh(S),
\end{equation}
where now for a morphism $f:S'\to S'$ in $\on{DGSch}_{/\CY,\on{smooth}}$, the transition 
functor $\IndCoh(S)\to \IndCoh(S')$ is $f^{\IndCoh,*}$, see \cite[Sect. 11.3 and Proposition 11.4.3]{IndCoh}.

\medskip

If $f:Z\to \CY$ is a smooth atlas, the naturally defined functor
\begin{equation} \label{e:Cech for Ind}
\IndCoh(Z)\to \on{Tot}(\IndCoh(Z^\bullet/\CY))
\end{equation}
is an equivalence. In the above formula, the cosimplicial
category $\IndCoh(Z^\bullet/\CY)$ is formed by using either the !-pullback or $(\IndCoh,*)$-pullback
functors along the simplicial DG scheme
$Z^\bullet/\CY$. See \cite[Corollary 11.3.4]{IndCoh} for the proof. 

\medskip

For a Cartesian diagram \eqref{e:card diag stacks} consisting of algebraic stacks, 
in which the vertical arrows are schematic and quasi-compact and the horizontal ones are eventually coconnective, 
we have a canonical isomorphism
\begin{equation} \label{e:usual base change for IndCoh stacks}
\phi_2^{\IndCoh,*}\circ \pi^{\IndCoh}_*\simeq (\pi')^{\IndCoh}_*\circ \phi_1^{\IndCoh,*}.
\end{equation}
It is obtained from the natural transformation \eqref{e:usual base change for IndCoh} using
\eqref{e:*-limit for IndCoh}. Note again that unless the vertical arrows are also
eventually coconnective or the horizontal maps schematic and quasi-compact, there is a priori no morphism 
in either direction in \eqref{e:usual base change for IndCoh stacks}.

\sssec{}  \label{sss:Psi for stacks}

For $\CY$ an algebraic stack, the category $\IndCoh(\CY)$ has a t-structure and the functor
$$\Psi_\CY:\IndCoh(\CY)\to \QCoh(\CY)$$ with the same properties as those for schemes,
reviewed in \secref{sss:Ind and QCoh} above,  see \cite[Sect. 11.7.1 and Proposition 11.7.5]{IndCoh}. 
Namely, the functor $\Psi_\CY$ is determined uniquely by the requirement that for 
$(S,g)\in \on{DGSch}_{/\CY,\on{smooth}}$, the diagram
$$
\CD
\IndCoh(\CY)  @>{g^{\IndCoh,*}}>>  \IndCoh(S)  \\
@V{\Psi_\CY}VV   @VV{\Psi_S}V   \\
\QCoh(\CY)   @>{g^*}>>  \QCoh(S)
\endCD
$$
is supplied with a commutativity isomorphism, functorially in $(S,g)$. The t-structure on $\IndCoh(\CY)$
is determined by the condition that the functors $g^{\IndCoh,*}$ be t-exact.

\medskip

If $\pi:\CY_1\to \CY_2$ is an eventually coconnective morphism between algebraic stacks,
we have a commutative diagram
\begin{equation} \label{e:*-pullback IndCoh stacks}
\CD
\IndCoh(\CY_1)  @>{\Psi_{\CY_1}}>>  \QCoh(\CY_1) \\
@A{\pi^{\IndCoh,*}}AA   @AA{\pi^*}A   \\
\IndCoh(\CY_2)  @>{\Psi_{\CY_2}}>>  \QCoh(\CY_2).
\endCD
\end{equation}

\medskip

For a schematic and quasi-compact map $\pi:\CY_1\to \CY_2$ between algebraic stacks, we have 
a commutative diagram
\begin{equation} \label{e:*-pushforward IndCoh stacks}
\CD
\IndCoh(\CY_1)  @>{\Psi_{\CY_1}}>>  \QCoh(\CY_1) \\
@V{\pi^{\IndCoh}_*}VV   @VV{\pi_*}V   \\
\IndCoh(\CY_2)  @>{\Psi_{\CY_2}}>>  \QCoh(\CY_2).
\endCD
\end{equation}
It arises from the corresponding
commutative diagrams in the case of DG schemes , i.e., \eqref{e:*-pushforward for IndCoh}, using 
the functorial isomorphisms \eqref{e:usual base change for IndCoh stacks}.

\sssec{}   \label{sss:Gamma IndCoh}

For an algebraic stack $\CY$, we shall denote by $\Gamma^{\IndCoh}(\CY,-):\IndCoh(\CY)\to \Vect$ the
\emph{not necessarily} continuous functor equal to
$$\Gamma(\CY,-)\circ \Psi_\CY.$$

From \lemref{l:taut dir im} we obtain that for $\CF\in \IndCoh(\CY)$ there is a canonical isomorphism
\begin{multline} \label{e:Gamma IndCoh}
\Gamma^{\IndCoh}(\CY,\CF)\simeq 
\underset{(S,g)\in (\on{DGSch}_{/\CY,\on{smooth}})^{\on{op}}}
{\underset{\longleftarrow}{lim}}\, \Gamma\left(S,g^*(\Psi_{\CY}(\CF))\right)\simeq \\
\simeq 
\underset{(S,g)\in (\on{DGSch}_{/\CY,\on{smooth}})^{\on{op}}}
{\underset{\longleftarrow}{lim}}\, \Gamma^{\IndCoh}(S,g^{\IndCoh,*}(\CF)).
\end{multline}

\ssec{The coherent subcategory}  \label{ss:coherent}

Let $\CY$ be an algebraic stack. 

\sssec{}

We define $\Coh_{\on{Ind}}(\CY)$ to be the full subcategory of
$\IndCoh(\CY)$ consisting of those objects $\CF$, for which for any affine DG scheme $S$
equipped with a smooth map $g:S\to \CY$, the corresponding object $g^{\IndCoh,*}(\CF)$ belongs to $\Coh(S)\subset \IndCoh(S)$. 
This condition is enough to check for any fixed collection $(S_\alpha,g_\alpha)$ such that the map
$\underset{\alpha}\sqcup\, S_\alpha\to \CY$ is surjective. 

\medskip

Note that in the above definition, we can replace the functors $g^{\IndCoh,*}$ by $g^!$. This follows from either
\eqref{e:! and * stacks} or \eqref{e:! and * pullback stacks}. 

\medskip

We define $\Coh_{\on{Q}}(\CY)$ to be the full subcategory of
$\QCoh(\CY)$ consisting of those objects $\CF$, 
for which for any affine DG scheme $S$
equipped with a smooth map $g:S\to \CY$, the corresponding object $g^*(\CF)$ belongs to $\Coh(S)\subset \QCoh(S)$. 
This condition is enough to check for any fixed collection $(S_\alpha,g_\alpha)$ such that the map
$\underset{\alpha}\sqcup\, S_\alpha\to \CY$ is surjective. 

\medskip

We claim: 

\begin{lem}
The functor $\Psi_\CY$ defines an equivalence $\Coh_{\on{Ind}}(\CY)\to \Coh_{\on{Q}}(\CY)$.
\end{lem}

\begin{proof}
Follows by combining \eqref{e:*-pullback IndCoh stacks} with \eqref{e:*-limit for IndCoh} and 
\secref{sss:change index qc}(iv).
\end{proof}

From now on, we will identify $\Coh_{\on{Ind}}(\CY)$ with $\Coh_{\on{Q}}(\CY)$ and denote the resulting
category simply by $\Coh(\CY)$, unless a confusion is likely to occur.

\sssec{}

Consider the ind-completion $\Ind (\Coh(\CY))$ of the category $\Coh(\CY)$
(see Sect.~\ref{sss:ind-compl} where the notion of ind-completion of a DG category is recalled). 
One has a tautologically defined continuous functor
\begin{equation} \label{e:two versions of Ind}
\Ind (\Coh(\CY))\to \Ind\Coh(\CY).
\end{equation}

However, it is not true that this functor is always an equivalence. For example, it is typically not
an equivalence for non quasi-compact schemes.

\sssec{}

The main result of this section is the following theorem,
which says that $\IndCoh(\CY)=\Ind (\Coh(\CY))$ if $\CY$ is QCA (see Definition~\ref{d:QCA}).

\begin{thm} \label{IndCoh} 
Assume that a stack $\CY$ is QCA. Then the category $\IndCoh(\CY)$ is compactly generated.
Moreover, its subcategory of compact objects equals $\Coh(\CY)$.
\end{thm}

\sssec{}

The proof will be given in Sects.~\ref{ss:compacts_in_IndCoh(Y)}-\ref{ss:CohgeneratesIndCoh} 
(it is based on \thmref{main}). This theorem will imply a number of favorable properties of
the category $\IndCoh$; these will be established in \secref{s:dualizability}, 
see Sects.~\ref{ss:IndCoh is dualizable} and \ref{ss:appl to QCoh}).

\ssec{Description of compact objects of $\IndCoh(\CY)$}   \label{ss:compacts_in_IndCoh(Y)}

\sssec{}   

First, we claim:

\begin{prop}  \label{coh is compact} \hfill

\smallskip

\noindent{\em(a)}
For any algebraic stack, the subcategory $\IndCoh(\CY)^c\subset \IndCoh(\CY)$ is contained in 
$\Coh(\CY)$. 

\smallskip

\noindent{\em(b)}
If $\CY$ is QCA  then $\IndCoh(\CY)^c=\Coh(\CY)$.
\end{prop}

\begin{proof}[Proof of point (a)]

When need to show that for any affine DG scheme $S$ equipped with a smooth map
$g:S\to \CY$, the functor $g^{\IndCoh,*}$ sends $\IndCoh(\CY)^c$ to $\IndCoh(S)^c=\Coh(S)$.

\medskip

Since $\CY$ is an algebraic stack, the morphism $g$ is schematic and quasi-compact. Hence, the functor $g^{\IndCoh,*}$
admits a continuous right adjoint, namely, $g^{\IndCoh}_*$ (see \secref{sss:functoriality IndCoh prestacks * pullback}).
This implies the required assertion.

\end{proof}

\begin{rem}
For point (b), we need to show that, when $\CY$ is QCA and $\CF\in\Coh(\CY)$, the assignment
$$\CF'\mapsto \CMaps_{\IndCoh(\CY)}(\CF,\CF')$$
commutes with colimits in $\CF'$. The idea of the proof is that to $\CF$ and $\CF'$ one can 
assign their \emph{internal} Hom object
$$\underline\Hom_{\QCoh(\CY)}(\CF,\CF')\in \QCoh(\CY),$$
whose formation commutes with colimits in $\CF'$, and such that
$$\CMaps_{\IndCoh(\CY)}(\CF,\CF')\simeq \Gamma\left(\CY,\underline\Hom_{\QCoh(\CY)}(\CF,\CF')\right).$$
Then the assertion of point (b) of the proposition would follow from \thmref{main}.
\end{rem}

\begin{proof}[Proof of point (b)]

Let $\CY$ be QCA and $\CF\in\Coh(\CY)$ and $\CF'\in \IndCoh(\CY)$. We have:
$$\CMaps_{\IndCoh(\CY)}(\CF,\CF')\simeq 
\underset{(S,g)\in (\on{DGSch}_{/\CY,\on{smooth}})^{\on{op}}}{\underset{\longleftarrow}{lim}}\, 
\CMaps_{\IndCoh(S)}(g^{\IndCoh,*}(\CF),g^{\IndCoh,*}(\CF')).$$

\medskip

For every $(S,g)\in \on{DGSch}^{\on{aff}}_{/\CY,\on{smooth}}$ consider the object
$$\underline\Hom_{\QCoh(S)}(g^{\IndCoh,*}(\CF),g^{\IndCoh,*}(\CF'))\in \QCoh(S),$$
(see \cite{DG}, Sect. 5.1.). Namely, for $\CE\in \QCoh(S)$, 
\begin{multline*}
\CMaps_{\QCoh(S)}(\CE,\underline\Hom_{\QCoh(S)}(g^{\IndCoh,*}(\CF),g^{\IndCoh,*}(\CF')))\simeq \\
\simeq \CMaps_{\IndCoh(S)}(\CE\otimes g^{\IndCoh,*}(\CF),g^{\IndCoh,*}(\CF')),
\end{multline*}
where $-\otimes-$ denotes the action of $\QCoh(S)$ on $\IndCoh(S)$.  

\medskip

Since $g^{\IndCoh,*}(\CF)\in \Coh(S)=\IndCoh(S)^c$, and $\QCoh(S)$ is compactly generated, 
by \cite[Lemma 5.1.1]{DG},
the assignment 
$$\CF'\mapsto \underline\Hom_{\QCoh(S)}(g^{\IndCoh,*}(\CF),g^{\IndCoh,*}(\CF'))$$
commutes with colimits.

\medskip

By construction, for every map $f:\wt{S}\to S$ in $\on{DGSch}^{\on{aff}}_{/\CY,\on{smooth}}$,
there is a canonical map
\begin{multline} \label{e:inner transition map}
f^*(\underline\Hom_{\QCoh(S)}(g^{\IndCoh,*}(\CF),g^{\IndCoh,*}(\CF')))\to\\
\to \underline\Hom_{\QCoh(\wt{S})}(\wt{g}^{\IndCoh,*}(\CF),\wt{g}^{\IndCoh,*}(\CF')),
\end{multline} 
where $\wt{g}=g\circ f$.

\begin{lem}  \label{l:inner transition map}
The map \eqref{e:inner transition map} is an isomorphism.
\end{lem}

The proof will be given in \secref{sss:proof of inner transition map}. Thus, we obtain that 
the assignment 
$$(S,g)\mapsto \underline\Hom_{\QCoh(S)}(g^{\IndCoh,*}(\CF),g^{\IndCoh,*}(\CF'))$$
defines an object 
$$\underline\Hom_{\QCoh(\CY)}(\CF,\CF')\in \QCoh(\CY).$$
Moreover, the functor
$$\CF'\mapsto \underline\Hom_{\QCoh(\CY)}(\CF,\CF')$$
commutes with colimits.

\medskip

By construction,
\begin{equation}  \label{e:local_Hom}
\CMaps_{\IndCoh(\CY)}(\CF,\CF')\simeq \Gamma\left(\CY,\underline\Hom_{\QCoh(\CY)}(\CF,\CF')\right).
\end{equation}

Now, the required assertion follows from \thmref{main}.

\end{proof}

\begin{rem}
It is easy to see that the object $\underline\Hom_{\QCoh(\CY)}(\CF,\CF')$ introduced above
is the internal Hom of $\CF$ and $\CF'$ in the sense
of \cite[Sect. 5.1]{DG}, i.e., for $\CE\in \QCoh(\CY)$, we have
$$\CMaps(\CE,\underline\Hom_{\QCoh(\CY)}(\CF,\CF'))\simeq \CMaps_{\IndCoh(\CY)}(\CE\otimes \CF,\CF').$$
\end{rem}

\sssec{Proof of \lemref{l:inner transition map}}  \label{sss:proof of inner transition map}

Let $f:\wt{S}\to S$ be an eventually coconnective map of affine DG schemes, and 
$\CF\in \Coh(S)$, and $\CF'\in \IndCoh(S)$. We claim that the natural map
$$f^*(\underline\Hom_{\QCoh(S)}(\CF,\CF'))\to \underline\Hom_{\QCoh(\wt{S})}(f^{\IndCoh,*}(\CF),f^{\IndCoh,*}(\CF')))$$
is an isomorphism. The latter is equivalent to
$$f_*(\CO_{\wt{S}})\underset{\CO_S}\otimes \underline\Hom_{\QCoh(S)}(\CF,\CF')\to
f_*(\underline\Hom_{\QCoh(\wt{S})}(f^{\IndCoh,*}(\CF),f^{\IndCoh,*}(\CF')))$$
being an isomorphism at the level of global sections. 

\medskip

Now, since $\CF$ is a compact object of $\IndCoh(S)$, by \cite[Lemma 5.1.1]{DG}, we have:
$$f_*(\CO_{\wt{S}})\underset{\CO_S}\otimes \underline\Hom_{\QCoh(S)}(\CF,\CF')
\simeq \underline\Hom_{\QCoh(S)}(\CF, f_*(\CO_{\wt{S}})\underset{\CO_S}\otimes\CF').$$

Thus, we need to show that
\begin{multline*}
\CMaps_{\IndCoh(S)}(\CF,f_*(\CO_{\wt{S}})\underset{\CO_S}\otimes\CF')\to
\CMaps_{\IndCoh(\wt{S})}(f^{\IndCoh,*}(\CF),f^{\IndCoh,*}(\CF'))\simeq \\
\simeq \CMaps_{\IndCoh(S)}(\CF,f^{\IndCoh}_*\circ f^{\IndCoh,*}(\CF')).
\end{multline*}

I.e., it is sufficient to prove that the map
$$f_*(\CO_{\wt{S}})\underset{\CO_S}\otimes\CF'\to f^{\IndCoh}_*\circ f^{\IndCoh,*}(\CF')$$
is an isomorphism. However, the latter is the content of \cite[Proposition 3.6.11]{IndCoh}.

\qed

\ssec{The category $\Coh(\CY)$ generates $\IndCoh(\CY)$}   \label{ss:CohgeneratesIndCoh} \hfill

\medskip

\thmref{IndCoh} follows from Proposition~\ref{coh is compact} and the next one.

\begin{prop}  \label{p:coh_generates}
If $\CY$ is QCA then
the subcategory $\Coh(\CY)^{\heartsuit}$ generates $\IndCoh(\CY)$. 
\end{prop}

The proof of Proposition~\ref{p:coh_generates}, given below, is parallel to the proof of Theorem~\ref{t:QCoh_core}  given in 
Sect.~\ref{ss:proofQCoh_core}. 

\begin{rem} \label{r:generation of IndCoh}
In some sense, the proof of Proposition~\ref{p:coh_generates} is simpler because  for $\IndCoh$ the $!$-pullback is a continuous functor 
(unlike the situation of 
Sect.~\ref{sss:end_of} and Remark \ref{r:end_of}). So one may prefer to
deduce Theorem~\ref{t:QCoh_core} from Proposition~\ref{p:coh_generates} using the functor
$\Psi_{\CY}:\IndCoh (\CY )\to\QCoh (\CY )$, which is essentially surjective if $\CY$ is eventually
coconnective. 
\end{rem}

\sssec{}

First, just as in Sect.~\ref{ss:proofQCoh_core}, one can assume that
$\CY$ is classical and reduced. 

\medskip

Let $\CY_i$ be the locally closed substacks of $\CY$ given by \propref{p:stratification}. With
no restriction of generality, we can assume that all $\CY_i$
are smooth. In this case $\IndCoh(\CY_i)\simeq \QCoh(\CY_i)$, so
Lemmas \ref{l:BFN} and \ref{l:covering plus} imply
that $\IndCoh(\CY_i)$ is generated by $\Coh(\CY_i)^\heartsuit$. 

\medskip

Hence, to prove the theorem, it suffices to prove the following analog
of \lemref{l:devissage plus}:

\begin{lem} \label{l:devissage IndCoh}
Let $\CX$ and $\oCY$ be as in \lemref{l:devissage}. Then if the assertion of
\propref{p:coh_generates} holds for $\CX$ and $\oCY$, then it holds also for
$\CY$.
\end{lem}

\begin{proof}

We have to show that if $\CF\in\IndCoh (\CY )$ and 
$$\CMaps_{\IndCoh(\CY)}(\CE ,\CF)=0 \mbox{ for all }\CE\in\Coh(\CY)^\heartsuit$$ then $\CF =0$.

\medskip

Consider the exact triangle 
\begin{equation} \label{e:indcoh triangle}
(\CF)_{\CX}\to \CF\to \jmath^{\IndCoh}_*\circ \jmath^{\IndCoh,*}(\CF),
\end{equation}
where
$$(\CF)_\CX:=\on{Cone}\left(\jmath^{\IndCoh}_*\circ \jmath^{\IndCoh,*}(\CF)\right)[-1].$$

By \cite[Proposition 4.1.7]{IndCoh} (which is applicable to algebraic stacks),
$$(\CF)_\CX \Leftrightarrow \imath^!(\CF)=0.$$

\medskip

For any $\CF'\in\Coh(\CX)^\heartsuit$ one has 
$$\CMaps_{\IndCoh(\CX)}(\CF',\imath^!(\CF))=\CMaps_{\IndCoh(\CY)}(\imath^{\IndCoh}_*(\CF') ,\CF)=0,$$
and $\imath^{\IndCoh}_*(\CF')\in \Coh(\CY)^\heartsuit$. So, the assumption that \propref{p:coh_generates} holds for $\CX$ implies
that $\imath^!(\CF)=0$. Therefore, $(\CF)_\CX=0$, and, hence,
$$\CF\to \jmath^{\IndCoh}_*\circ \jmath^{\IndCoh,*}(\CF)$$ is
an isomorphism.

\medskip

In particular, for every $\CE\in \IndCoh(\CY)$, we have:
\begin{multline} \label{e:F is ext}
\CMaps_{\IndCoh(\CY)}(\CE,\CF)\simeq \CMaps_{\IndCoh(\CY)}\left(\CE,\jmath^{\IndCoh}_*\circ \jmath^{\IndCoh,*}(\CF)\right)\simeq \\
\simeq \CMaps_{\IndCoh(\oCY)}\left(\jmath^{\IndCoh,*}(\CE),\jmath^{\IndCoh,*}(\CF)\right).
\end{multline}

\medskip

Now we use the following lemma, which immediately follows from  \cite[Corollary 15.5]{LM}.

\begin{lem} %\label{extension}
For every $\overset{\circ}\CE\in \Coh(\oCY)^\heartsuit$, there exists
$\CE\in \Coh(\CY)^\heartsuit$ such that $\jmath^*(\CE)\simeq \overset{\circ}\CE$. 
\end{lem}

By \eqref{e:F is ext}, for every $\overset{\circ}\CE\in \Coh(\oCY)^\heartsuit$ and the corresponding $\CE\in \Coh(\CY)^\heartsuit$, we have:
$$\CMaps_{\IndCoh(\oCY)}\left(\overset{\circ}\CE,\jmath^{\IndCoh,*}(\CF)\right)\simeq \CMaps_{\IndCoh(\CY)}(\CE,\CF)=0.$$
Hence, the $\jmath^{\IndCoh,*}(\CF)=0$, by the assumption that \propref{p:coh_generates} holds for $\oCY$. 

\medskip

Thus, we have $(\CF)_\CX=0$ and $\jmath^{\IndCoh,*}(\CF)=0$, and by \eqref{e:indcoh triangle}, this implies that $\CF=0$.

\end{proof}

\ssec{Direct image functor on $\IndCoh$}  \label{ss:indcoh dir image}

As an application of \thmref{IndCoh}, we shall now construct a functor $\pi_*^{\IndCoh}$
for a morphism $\pi:\CY_1\to \CY_2$ between QCA algebraic stacks. \footnote{For this construction
to make sense we only need $\CY_1$ to be QCA, while $\CY_2$ may be arbitrary.}

\sssec{}

We claim that in this case there exists a unique continuous functor
$$\pi_*^{\IndCoh}:\IndCoh(\CY_1)\to \IndCoh(\CY_2),$$
which is left t-exact and which makes the following diagram commute:
$$
\CD
\IndCoh(\CY_1)  @>{\Psi_{\CY_1}}>>  \QCoh(\CY_1) \\
@V{\pi^{\IndCoh}_*}VV    @VV{\pi_*}V  \\
\IndCoh(\CY_2)  @>{\Psi_{\CY_2}}>>  \QCoh(\CY_2).
\endCD
$$

\medskip

Indeed, the functor $\pi^{\IndCoh}_*$ is obtained as the ind-extension of the functor
$$\Coh(\CY_1)\to \IndCoh(\CY_2)$$ 
equal to the composition
$$\Coh(\CY_1)\hookrightarrow \QCoh(\CY_1)^+\overset{\pi_*}\longrightarrow \QCoh(\CY_2)^+\simeq
\IndCoh(\CY_2)^+\hookrightarrow \IndCoh(\CY_2),$$
where 
$\QCoh(\CY_2)^+\simeq \IndCoh(\CY_2)^+$ is the equivalence inverse to that induced by $\Psi_{\CY_2}$,
see \secref{sss:Psi for stacks}.

\medskip

It is easy to see that when $\pi$ is schematic and quasi-compact, the above functor $\pi^{\IndCoh}_*$ is canonically
isomorphic to the one in \secref{sss:functoriality IndCoh prestacks pushforward}. This follows from the 
defining property of $\pi_*^{\IndCoh}$, using the commutative diagram \eqref{e:*-pushforward IndCoh stacks}.

\sssec{}  \label{sss:Gamma ind}

Consider the particular case when $\CY_1=\CY$ and $\CY_2=\on{pt}$, and $\pi=p_\CY$. Recall the functor
$\Gamma^{\IndCoh}(\CY,-)$, see \secref{sss:Gamma IndCoh}.

\medskip

Since the functor
$\Gamma(\CY,-):\QCoh(\CY)\to \on{pt}$ is continuous, so is the functor $\Gamma^{\IndCoh}(\CY,-)$.

\medskip

We obtain that we have a canonical isomorphism of functors
\begin{equation} \label{e:Gamma ren}
\Gamma^{\IndCoh}(\CY,-)\simeq (p_\CY)^{\IndCoh}.
\end{equation}
(Indeed, the two functors tautologically coincide on $\Coh(\CY)\subset \IndCoh(\CY)$, and 
the isomorphism on all of $\IndCoh(\CY)$ follows by continuity.)

\sssec{}

The defining property of $\pi_*^{\IndCoh}$ implies that it is compatible with compositions. I.e., if 
$$\CY_1\overset{\pi}\longrightarrow \CY_2\overset{\phi}\longrightarrow \CY_3$$
are maps between QCA algebraic stacks, we have
$$\phi_*^{\IndCoh}\circ \pi_*^{\IndCoh}\simeq (\phi\circ \pi)_*^{\IndCoh}.$$
This follows from the diagram
$$
\CD
\IndCoh(\CY_1)  @>{\Psi_{\CY_1}}>>  \QCoh(\CY_1) \\
@V{\pi^{\IndCoh}_*}VV    @VV{\pi_*}V  \\
\IndCoh(\CY_2)  @>{\Psi_{\CY_2}}>>  \QCoh(\CY_2) \\
@V{\phi^{\IndCoh}_*}VV    @VV{\phi_*}V  \\
\IndCoh(\CY_3)  @>{\Psi_{\CY_3}}>>  \QCoh(\CY_3).
\endCD
$$

\ssec{Direct image functor on $\IndCoh$, further constructions}  \label{ss:indcoh dir image further}

The contents of this subsection will not be used elsewhere in the paper. We include it for completeness
as the functor $\pi_{\on{non-ren},*}^{\IndCoh}$ introduced below has features analogous to those 
of the de Rham pushforward functor $\pi_\dr$, considered in \secref{ss:de Rham dir image stacks}.

\sssec{}

Let $\pi:\CY_1\to \CY_2$ is a morphism between arbitrary algebraic stacks. In this case also,
we can introduce a functor $\IndCoh(\CY_1)\to \IndCoh(\CY_2)$, that we denote
$\pi_{\on{non-ren},*}^{\IndCoh}$. This functor is \emph{not necessarily continuous}. 

\medskip

By definition,
\begin{equation} \label{e:IndCog dir image non-renorm}
\pi^{\IndCoh}_{\on{non-ren},*}(\CF):=\underset{(S,g)\in ((\on{DGSch}_{\on{aft}})_{/\CY_1,\on{smooth}})^{\on{op}}}
{\underset{\longleftarrow}{lim}}\, (\pi\circ g)^{\IndCoh}_*(g^{\IndCoh,*}(\CF)),
\end{equation}
where $(\pi\circ g)^{\IndCoh}_*$ is well-defined because the morphism $\pi\circ g$ is
schematic and quasi-compact.

\sssec{}

Let us take for a moment $\CY_1=\CY$ and $\CY_2=\on{pt}$. From \eqref{e:Gamma IndCoh} we obtain that
\begin{equation} \label{e:Gamma non-ren}
(p_\CY)^{\IndCoh}_{\on{non-ren},*}\simeq \Gamma^{\IndCoh}(\CY,-).
\end{equation}

\sssec{}

Note that by construction we have a natural transformation
\begin{equation}  \label{e:non-ren pushforward qc and IndCoh}
\Psi_{\CY_2}\circ \pi^{\IndCoh}_{\on{non-ren},*}\to 
\pi_*\circ \Psi_{\CY_1}.
\end{equation}

We claim:

\begin{lem} \label{l:non-ren pushforward qc and IndCoh}
The natural transformation \eqref{e:non-ren pushforward qc and IndCoh} is an isomorphism when applied 
to objects from $\IndCoh(\CY_1)^+$.
\end{lem}

\begin{proof}

Note that the functors $\Psi_{\CY_i}$ are t-exact, and 
both $\pi_*$ and $\pi^{\IndCoh}_{\on{non-ren},*}$ are left t-exact. 
Hence, it is enough to show that the following diagram of functors
commutes
$$
\CD
\IndCoh(\CY_1)^{\geq n}  @>{\Psi_{\CY_1}}>>  \QCoh(\CY_1)^{\geq n}  \\
@V{\pi^{\IndCoh}_{\on{non-ren},*}}VV    @VV{\pi_*}V  \\
\IndCoh(\CY_2)^{\geq n}   @>{\Psi_{\CY_2}}>>  \QCoh(\CY_2)^{\geq n} 
\endCD
$$
for every given $n$.

\medskip

Note that $\Psi_{\CY_2}$, restricted to 
$\IndCoh(\CY_2)^{\geq n}$, is an equivalence, and hence 
commutes with limits. Hence, for $\CF_1\in \IndCoh(\CY_1)^{\geq n}$ we have:
\begin{multline*}
\Psi_{\CY_2}\circ \pi^{\IndCoh}_{\on{non-ren},*}=
\Psi_{\CY_2}\left(\underset{(S,g)\in ((\on{DGSch}_{\on{aft}})_{/\CY_1,\on{smooth}})^{\on{op}}}
{\underset{\longleftarrow}{lim}}\, (\pi\circ g)^{\IndCoh}_*(g^{\IndCoh,*}(\CF))\right)\simeq \\
\simeq \underset{(S,g)\in ((\on{DGSch}_{\on{aft}})_{/\CY_1,\on{smooth}})^{\on{op}}}
{\underset{\longleftarrow}{lim}}\, \Psi_{\CY_2}\left((\pi\circ g)^{\IndCoh}_*(g^{\IndCoh,*}(\CF))\right).
\end{multline*}

Since the morphisms $\pi\circ g$ are schematic and quasi-compact, and $g$ is eventually 
coconnective, the latter expression can be rewritten as
$$\underset{(S,g)\in ((\on{DGSch}_{\on{aft}})_{/\CY_1,\on{smooth}})^{\on{op}}}
{\underset{\longleftarrow}{lim}}\, \left((\pi\circ g)_*(g^*(\Psi_{\CY_1}(\CF_1)))\right),$$
which is isomorphic to $\pi_*(\Psi_{\CY_1}(\CF_1))$ by \lemref{l:taut dir im},
where we take the indexing category $A$ to be $((\on{DGSch}_{\on{aft}})_{/\CY_1,\on{smooth}}$.

\end{proof}

\sssec{}

Suppose for a moment that $\pi$ is schematic and quasi-compact. It is easy to see that there exists a natural transformation
\begin{equation} \label{e:non-ren for sch}
\pi^{\IndCoh}_*\to \pi^{\IndCoh}_{\on{non-ren},*}.
\end{equation}

The next assertion can be proved by the same method as \propref{p:dr for sch}:

\begin{prop}
The natural transformation \eqref{e:non-ren for sch} is an isomorphism.
\end{prop}

\sssec{}

It is easy to see that when $\pi$ is eventually coconnective, the functor $\pi^{\IndCoh}_{\on{non-ren},*}$ is the right adjoint of
$\pi^{\IndCoh,*}$. 

\begin{rem}
When $\pi$ is not eventually coconnective, we do not know how to 
characterize the functor $\pi^{\IndCoh}_{\on{non-ren},*}$, except by the explicit formula \eqref{e:IndCog dir image non-renorm}.
\end{rem}

\sssec{}

Suppose that the morphism $\pi$ is quasi-compact. Then it is easy to see that, that although
the functor $\pi^{\IndCoh}_{\on{non-ren},*}$ is a priori non-continuous, it has has properties 
parallel to those of $\pi_*$ expressed in \corref{c:coconnective part}(a,b): when restricted to 
$\IndCoh(\CY_1)^{\geq 0}$, it commutes with filtered colimits and is equipped with base change
isomorphisms with respect to !-pullbacks for maps of algebraic stacks $\CY'_2\to \CY_2$.

\medskip

From the base change isomorphism for schematic quasi-compact maps we obtain that for a map $\phi_2:\CY'_2\to \CY_2$
and the corresponding Cartesian square 
\begin{equation}  \label{e:same Cartesian square}
\CD
\CY'_1  @>{\phi_1}>>  \CY_1 \\
@V{\pi'}VV   @VV{\pi}V  \\
\CY'_2  @>{\phi_2}>>  \CY_2
\endCD
\end{equation}
there is a canonical natural transformation
\begin{equation} \label{e:almost base change doe indcoh}
\phi_2^!\circ \pi^{\IndCoh}_{\on{non-ren},*}\to \pi{}'^{\IndCoh}_{\on{non-ren},*}\circ \phi_1^!.
\end{equation}

This natural natural transformation is not necessarily an isomorphism. But as we mentioned above,
if $\pi$ is quasi-compact, it is an isomorphism when applied to objects of $\IndCoh(\CY_1)^+$.

\sssec{}

Suppose now that $\CY_1$ and $\CY_2$ are QCA. It is easy to see from the construction that there exists a canonical 
natural transformation
\begin{equation} \label{e:nat trans for indcoh dir image}
\pi_*^{\IndCoh}\to \pi^{\IndCoh}_{\on{non-ren},*}.
\end{equation}

In \secref{sss:proof of ren vs nonren indcoh} we will show: 

\begin{prop} \label{p:proof of ren vs nonren indcoh}
The natural transformation \eqref{e:nat trans for indcoh dir image} is an isomorphism.
\end{prop}

\begin{rem} \label{r:Gamma ind}
For $\CY_2=\on{pt}$, the assertion of \propref{p:proof of ren vs nonren indcoh} is easy:
indeed, by \eqref{e:Gamma non-ren} and \eqref{e:Gamma ren}, both functors identify
canonically with $\Gamma^{\IndCoh}(\CY,-)$ of \secref{sss:Gamma ind}, 
where $\CY=\CY_1$.
\end{rem}

From \propref{p:proof of ren vs nonren indcoh}, we obtain:

\begin{cor} \label{c:* adj for indcoh}
If $\pi$ is an eventually coconnective morphism between QCA stacks, the functors $(\pi^{\IndCoh,*},\pi_*^{\IndCoh})$ are adjoint.
\end{cor}

In addition, we have:

\begin{cor}  \label{c:base change for indcoh}
For a Cartesian square \eqref{e:same Cartesian square} there is a canonical
isomorphism of functors
$$\phi_2^!\circ \pi^{\IndCoh}_*\to \pi{}'^{\IndCoh}_*\circ \phi_1^!.$$
\end{cor}

\begin{proof}
Both functors are continuous, so it is enough to construct the required
natural transformation when restricted to the subcategory $\Coh(\CY_1)$.
In this case, it follows from \propref{p:proof of ren vs nonren indcoh} and the
isomorphism of \eqref{e:almost base change doe indcoh} on
$\Coh(\CY_1)\subset \IndCoh(\CY_1)^+$.
\end{proof}

\begin{rem} 
One can use \corref{c:base change for indcoh} to define the functor $\pi_*^{\IndCoh}$ 
for QCA morphisms $\pi:\CY_1\to \CY_2$ between prestacks, in a way compatible with base
change.
\end{rem}

\section{Dualizability and behavior with respect to products of stacks}   \label{s:dualizability}

In this section we will show that the category $\IndCoh$ on a QCA algebraic stack locally almost of finite 
type is dualizable, see \corref{c:IndCohdualizable}. This will imply that the category $\QCoh(\CY)$ on such 
a stack is also dualizable, under the additional assumption that $\CY$ be eventually coconnective, see 
\thmref{t:QCoh dualizable}. 

\medskip

These properties of $\IndCoh(\CY)$ and $\QCoh(\CY)$ will imply a ``good"
behavior of $\IndCoh(-)$ and $\QCoh(-)$ when we take a product of $\CY$ with another prestack. 

\medskip

In Sect.~\ref{ss:Serreduality} we shall discuss applications to Serre duality on $\IndCoh(\CY)$.

\ssec{The notion of dualizable DG category}  

\sssec{Definition of dualizability}   \label{sss:dualizabilitydef}
We refer to \cite{Lu2}, Sect. 6.3.1 for the definition of the tensor product functor
$$\otimes :\StinftyCat_{\on{cont}}\times\StinftyCat_{\on{cont}}\to\StinftyCat_{\on{cont}}$$
(see also \cite{DG}, Sect. 1.4 for a brief review). 

\medskip

The above operation makes the $(\infty ,1)$-category 
$\StinftyCat_{\on{cont}}$ into a symmetric monoidal $\infty$-category \footnote{I.e., 
$\StinftyCat_{\on{cont}}$ is a commutative
algebra object in the symmetric monoidal $(1,\infty)$-category of $\infty$-categories with respect
to the Cartesian product, see \cite{Lu2}, Sect. 2.3.1.}, in which the unit object is the category $\Vect$.

\medskip

For an object of any symmetric monoidal category, one can talk about its property of being dualizable
(see \cite{Lu2}, Sect. 4.2.5, or \cite{DG}, Sect. 5.2 for a brief review). When the category is just monoidal,
there are two different notions: left dualizable and right dualizable, see \cite{DG}, Sect. 5.2.

\begin{rem}
Note that dualizability of
an object in not a higher-categorical notion, but only depends on the truncation of the
monoidal $\infty$-category to an ordinary monoidal category. 
\end{rem}

Following Lurie, we say that $\bC\in \StinftyCat_{\on{cont}}$ is \emph{dualizable} if it is dualizable
in the above sense. 

\medskip

For $\bC\in \StinftyCat_{\on{cont}}$ dualizable, we denote by $\bC^\vee$ the corresponding dual category. 
We denote by
$$\bC^\vee\otimes \bC\overset{\epsilon_\bC}\longrightarrow \Vect \;\text{ and }\;
\Vect\overset{\mu_\bC}\longrightarrow \bC\otimes \bC^\vee$$
the corresponding duality data. The functor $\epsilon_\bC$ is called the \emph{co-unit} of the pairing
(or \emph{evaluation}, or \emph{canonical pairing}), and the functor $\mu_\bC$ is called the \emph{unit}
(or, \emph{co-evaluation}).

\sssec{}  \label{sss:properties of duality}

Here are some basic facts related to duality in $\StinftyCat_{\on{cont}}$ (see also \cite{DG}, Sect. 2):

\medskip

\noindent(i) If $\bC$ is dualizable, the category $\bC^\vee$ can be recovered as 
$\on{Funct}_{\on{cont}}(\bC,\Vect)$.

\medskip

\noindent(ii) Any compactly generated DG category is dualizable.

\medskip

\noindent(ii') For $\bC$ compactly generated, $\bC^\vee$ can be explicitly described as the ind-completion of the 
\emph{non-cocomplete} DG category $(\bC^c)^{\on{op}}$. In particular, we have a canonical equivalence:
$$\BD_\bC:(\bC^\vee)^c\simeq (\bC^c)^{\on{op}}.$$
In particular, for $\bC=\on{Ind}(\bC^0)$ (see \secref{sss:ind-compl}), we have $\bC^\vee\simeq \on{Ind}((\bC^0)^{\on{op}})$,
and 
$$\bC^\vee\simeq \on{Funct}(\bC^0,\Vect) \text{ and } \bC\simeq \on{Funct}((\bC^0)^{\on{op}},\Vect),$$
which also gives an explicit construction of $\on{Ind}(\bC^0)$.

\medskip

\noindent(iii) The functor of tensoring by a dualizable category commutes with all limits\footnote{Tensoring by $\bC$
commutes with all colimits in $\StinftyCat_{\on{cont}}$ for any $\bC$.}
taken in $\StinftyCat_{\on{cont}}$. Indeed, if $\bC$ is dualizable then $\bC\otimes -\simeq \on{Funct}_{\on{cont}}(\bC^\vee,-)$.

\sssec{}  \label{sss:dual functors}

Let $\bO$ be an arbitrary symmetric monoidal category, and $\bc_1,\bc_2\in \bO$ two dualizable objects.
Then to any morphism $f:\bc_1\to \bc_2$ one canonically attaches the dual morphism
$$f^\vee:\bc_2^\vee\to \bc_1^\vee,$$
where $\bc_i^\vee$ denotes the dual of $\bc_i$. 

\medskip

This construction has the following interpretation: a datum morphism $f$ as above is equivalent to that of
a point in $\on{Maps}_\bO(1,\bc_1^\vee\otimes \bc_2)$. Then the datum $f^\vee$ corresponds to \emph{the same}
point in
$$\on{Maps}_\bO(1,(\bc^\vee_2)^\vee\otimes \bc^\vee_1)\simeq \on{Maps}_\bO(1,\bc_1^\vee\otimes \bc_2).$$

\medskip

Applying this to $\bO=\StinftyCat_{\on{cont}}$ and two dualizable categories $\bC_1$ and $\bC_2$, we obtain
that to every continuous functor $F:\bC_1\to \bC_2$ there corresponds a dual functor
$$F^\vee:\bC_2^\vee\to \bC_1^\vee.$$

\medskip

In terms of \secref{sss:properties of duality}(i), the functor $F^\vee$ can be described as follows: it sends
an object 
$\Phi\in \on{Funct}_{\on{cont}}(\bC_2,\Vect)$ to $\Phi\circ F\in \on{Funct}_{\on{cont}}(\bC_1,\Vect)$.

\ssec{Dualizability of $\IndCoh$}   \label{ss:IndCoh is dualizable}  \hfill

\sssec{}

From \secref{sss:properties of duality}(ii) and \thmref{IndCoh} we obtain:

\begin{cor}   \label{c:IndCohdualizable}
If $\CY$ is a QCA algebraic stack, then the DG category $\IndCoh(\CY)$ is dualizable.
\end{cor}

As was explained to us by J.~Lurie, \corref{c:IndCohdualizable} implies the following
result (in any sheaf-theoretic context):

\begin{cor}  \label{c:indcoh on product}
Let $\CY_1$ and $\CY_2$ be two prestacks, with $\CY_1$ being a QCA algebraic stack.
Then the natural functor
$$\IndCoh(\CY_1)\otimes \IndCoh(\CY_2)\to \IndCoh(\CY_1\times \CY_2)$$
is an equivalence.
\end{cor}

\begin{proof}

The argument repeats verbatim that of \cite[Proposition 1.4.4]{QCoh}. For completeness, let us reproduce it here:

\medskip

We will show that the equivalence stated in the corollary takes place for any
two prestacks $\CY_1$, $\CY_2$, whenever 
$\IndCoh(\CY_1)$ is dualizable. 

\medskip

We have:
$$\IndCoh(\CY_1)\otimes \IndCoh(\CY_2)=
\IndCoh(\CY_1)\otimes  \left(\underset{S_2\in ((\on{DGSch}^{\on{aff}}_{\on{aft}})_{/\CY_2})^{\on{op}}}
{\underset{\longleftarrow}{lim}}\, \IndCoh(S_2)\right).$$
By \secref{sss:properties of duality}(iii), the latter expression maps isomorphically to
$$\underset{S_2\in ((\on{DGSch}^{\on{aff}}_{\on{aft}})_{/\CY_2})^{\on{op}}}{\underset{\longleftarrow}{lim}}\, 
\left(\IndCoh(\CY_1)\otimes \IndCoh(S_2)\right).$$

We rewrite $\IndCoh(\CY_1)$ by definition as 
$$\underset{S_1\in ((\on{DGSch}^{\on{aff}}_{\on{aft}})_{/\CY_1})^{\on{op}}}{\underset{\longleftarrow}{lim}}\, \IndCoh(S_1),$$
so
\begin{multline*}
\underset{S_2\in ((\on{DGSch}^{\on{aff}}_{\on{aft}})_{/\CY_2})^{\on{op}}}{\underset{\longleftarrow}{lim}}\, \left(\IndCoh(\CY_1)\otimes \IndCoh(S_2)\right)
\simeq \\
\simeq \underset{S_2\in ((\on{DGSch}^{\on{aff}}_{\on{aft}})_{/\CY_2})^{\on{op}}}{\underset{\longleftarrow}{lim}}\, \left(
\left(\underset{S_1\in ((\on{DGSch}^{\on{aff}}_{\on{aft}})_{/\CY_1})^{\on{op}}}{\underset{\longleftarrow}{lim}}\, 
\IndCoh(S_1)\right)\otimes \IndCoh(S_2)\right).
\end{multline*}
Since $\IndCoh(S_2)$ is dualizable, by \secref{sss:properties of duality}(iii),
the latter expression can be rewritten as 
\begin{equation} \label{e:product 3}
\underset{S_2\in ((\on{DGSch}^{\on{aff}}_{\on{aft}})_{/\CY_2})^{\on{op}}}{\underset{\longleftarrow}{lim}}\, \left(
\underset{S_1\in ((\on{DGSch}^{\on{aff}}_{\on{aft}})_{/\CY_1})^{\on{op}}}{\underset{\longleftarrow}{lim}}\, 
\left(\IndCoh(S_1)\otimes \IndCoh(S_2)\right)\right).
\end{equation}

\medskip

Now, as was mentioned in \secref{sss:IndCoh schemes ten}, for quasi-compact schemes $S_1$ and $S_2$, the natural functor
$$\IndCoh(S_1)\otimes \IndCoh(S_2)\to \IndCoh(S_1\times S_2)$$
is an equivalence. 

\medskip

Hence, we obtain that the expression in \eqref{e:product 3} maps isomorphically to
$$\underset{S_2\in ((\on{DGSch}^{\on{aff}}_{\on{aft}})_{/\CY_2})^{\on{op}}}{\underset{\longleftarrow}{lim}}\, \left(
\underset{S_1\in ((\on{DGSch}^{\on{aff}}_{\on{aft}})_{/\CY_1})^{\on{op}}}
{\underset{\longleftarrow}{lim}}\, \left(\IndCoh(S_1\times S_2)\right)\right),$$
which itself is isomorphic to
$$\underset{(S_1,S_2)\in ((\on{DGSch}^{\on{aff}}_{\on{aft}})_{/\CY_1})^{\on{op}}\times ((\on{DGSch}^{\on{aff}}_{\on{aft}})_{/\CY_2})^{\on{op}}}
{\underset{\longleftarrow}{lim}}\, \left(\IndCoh(S_1\times S_2)\right).$$

To summarize, we obtain an equivalence
\begin{multline} \label{e:product 4}
\IndCoh(\CY_1)\otimes \IndCoh(\CY_2)\to \\
\to \underset{(S_1,S_2)\in ((\on{DGSch}^{\on{aff}}_{\on{aft}})_{/\CY_1})^{\on{op}}\times ((\on{DGSch}^{\on{aff}}_{\on{aft}})_{/\CY_2})^{\on{op}}}
{\underset{\longleftarrow}{lim}}\, \left(\IndCoh(S_1\times S_2)\right).
\end{multline}

\medskip

Finally, it is easy to see that the natural functor
$$(\on{DGSch}^{\on{aff}}_{\on{aft}})_{/\CY_1}\times (\on{DGSch}^{\on{aff}}_{\on{aft}})_{/\CY_2}\to
(\on{DGSch}^{\on{aff}}_{\on{aft}})_{/\CY_1\times \CY_2}$$ 
is cofinal. Hence, the functor
\begin{multline*} 
\IndCoh(\CY_1\times \CY_2)=
\underset{S\in ((\on{DGSch}^{\on{aff}}_{\on{aft}})_{/\CY_1\times \CY_2})^{\on{op}}}{\underset{\longleftarrow}{lim}}\, \IndCoh(S)
\to \\
\to \underset{(S_1,S_2)\in ((\on{DGSch}^{\on{aff}}_{\on{aft}})_{/\CY_1})^{\on{op}}\times ((\on{DGSch}^{\on{aff}}_{\on{aft}})_{/\CY_2})^{\on{op}}}
{\underset{\longleftarrow}{lim}}\, \left(\IndCoh(S_1\times S_2)\right)
\end{multline*}
is an equivalence, and the composition
\begin{multline*}
\IndCoh(\CY_1)\otimes \IndCoh(\CY_2)  \to \IndCoh(\CY_1\times \CY_2)=
 \underset{S\in ((\on{DGSch}^{\on{aff}}_{\on{aft}})_{/\CY_1\times \CY_2})^{\on{op}}}{\underset{\longleftarrow}{lim}}\, \IndCoh(S) \to \\
\to \underset{(S_1,S_2)\in ((\on{DGSch}^{\on{aff}}_{\on{aft}})_{/\CY_1})^{\on{op}}\times ((\on{DGSch}^{\on{aff}}_{\on{aft}})_{/\CY_2})^{\on{op}}}
{\underset{\longleftarrow}{lim}}\, \left(\IndCoh(S_1\times S_2)\right)
\end{multline*}
is the map \eqref{e:product 4}. 

\medskip

This proves that the map
$\IndCoh(\CY_1)\otimes \IndCoh(\CY_2)  \to \IndCoh(\CY_1\times \CY_2)$ is an equivalence.

\end{proof}

\ssec{Applications to $\QCoh(\CY)$}  \label{ss:appl to QCoh}

We will now use \corref{c:IndCohdualizable} to prove the following:

\begin{thm} \label{t:QCoh dualizable}
Let $\CY$ be a QCA algebraic stack, which is eventually coconnective 
(see Definition~\ref{d:ev_coconnective}), and locally almost of finite
(as are all algebraic stacks in this section). Then the category 
$\QCoh(\CY)$ is dualizable.
\end{thm}

\begin{rem}
We do not know whether, under the assumptions of the theorem, the category
$\QCoh(\CY)$ is compactly generated.
\end{rem}

\begin{proof}

Recall (see \cite[Sect. 11.7.3]{IndCoh}) that for any eventually coconnective algebraic stack $\CY$, the functor 
$\Psi_\CY:\IndCoh(\CY)\to \QCoh(\CY)$ admits a left adjoint, which is fully faithful (and automatically
continuous by virtue of being a left adjoint).

\medskip

In particular, we obtain that in this case, $\QCoh(\CY)$ is a \emph{retract} of $\IndCoh(\CY)$ in the category
$\StinftyCat_{\on{cont}}$. 

\medskip

The assertion of the theorem follows from the following observation: let $\bO$ be a monoidal
category, which admits inner Hom's, i.e., for $M_1,M_2\in \bO$, there exists an object
$$\underline\Hom_\bO(M_1,M_2)\in \bO,$$ such that we have
$$\on{Maps}_\bO(N,\underline\Hom_\bO(M_1,M_2))\simeq \on{Maps}_\bO(N\otimes M_1,M_2),$$
functorially in $N$.

\begin{lem}
Under the above circumstances, a retract of a (left) dualizable object is (left) dualizable.
\end{lem}

\begin{proof}
It is easy to see that an object $M$ is (left) dualizable if and only if for any $N$, 
the natural map
$$N\otimes \underline\Hom_\bO(M,1)\to \on{Maps}(N,M)$$
is an isomorphism. However, the latter condition survives taking retracts.
\end{proof}

We apply this lemma to $\bO=\StinftyCat_{\on{cont}}$. This category has inner Hom's, which are explicitly given by
$$\underline\Hom_{\StinftyCat_{\on{cont}}}(\bC_1,\bC_2)=\on{Funct}_{\on{cont}}(\bC_1,\bC_2),$$
where the right-hand side has a natural structure of DG category.

\end{proof}

\begin{cor}  \label{c:QCoh on product}
Let $\CY$ satisfy the assumptions of \thmref{t:QCoh dualizable}. Then
for any prestack $\CY'$, the natural functor
$$\QCoh(\CY)\otimes \QCoh(\CY')\to \QCoh(\CY\times \CY')$$
is an equivalence.
\end{cor}

\begin{proof}
This follows from \thmref{t:QCoh dualizable} by \cite[Proposition 1.4.4]{QCoh}, which repeats verbatim 
the proof of \corref{c:indcoh on product}. 
\end{proof}

\begin{rem}
The assertion of \corref{c:QCoh on product}, together with the proof, is valid for \emph{all} prestacks $\CY'$, i.e., 
not necessarily those locally almost of finite type.
\end{rem}

\sssec{}

Let us recall the notion of rigid monoidal DG category from \cite{DG}, Sect. 6.1. This notion can be formulated as follows:
a monoidal category $\bO$ is rigid if:

\begin{itemize}

\item The object $1\in \bO$ is compact.

\item The functor 
\begin{equation} \label{e:comult}
\bO\to \bO\otimes \bO, 
\end{equation}
right adjoint to $\bO\otimes \bO\overset{\otimes}\longrightarrow \bO$, is continuous,
and is compatible with left and right actions of $\bO$.

\end{itemize}

If this happens, the functors
$$\bO\otimes \bO\overset{\otimes}\longrightarrow \bO\overset{\CMaps_\bO(1,-)}\longrightarrow \Vect$$
and 
$$\Vect\to \bO\to \bO\otimes \bO,$$
(where the functor $\Vect\to \bO$ is given by $1\in \bO$, and the functor $\bO\to \bO\otimes \bO$ is
\eqref{e:comult}) define a duality datum between $\bO$ and itself. 

\sssec{}

We have:

\begin{cor} \label{c:QCoh rigid}
Let $\CY$ be as in \thmref{t:QCoh dualizable}. Then the monoidal category $\QCoh(\CY)$ is rigid.
\end{cor}

\begin{proof}
This is \cite[Proposition 2.3.2]{QCoh}: the assertion is true for any prestack (not necessarily of finite type)
with the following three properties: (1) the category $\QCoh(\CY)$ is dualizable, (2) the object
$\CO_\CY\in \QCoh(\CY)$ is compact, and (3) the diagonal morphism $\CY\to \CY\times \CY$ is
schematic, quasi-separated and quasi-compact.
\end{proof}

In particular, we obtain a canonical identification 
$$\bD_\CY^{\on{naive}}:\QCoh(\CY)^\vee\simeq \QCoh(\CY),$$ 
where the duality datum is described as follows:

\medskip

The functor $\epsilon_{\QCoh(\CY)}$ is given by
$$\QCoh(\CY)\otimes \QCoh(\CY)\overset{\boxtimes}\longrightarrow \QCoh(\CY\times \CY)\overset{\Delta^*}\longrightarrow 
\QCoh(\CY)\overset{\Gamma(\CY,-)}\longrightarrow \Vect,$$
and the functor $\mu_{\QCoh(\CY)}$ is given by
$$\Vect\overset{\CO_\CY}\longrightarrow \QCoh(\CY)\overset{\Delta_*}\longrightarrow 
\QCoh(\CY\times \CY)\simeq \QCoh(\CY)\otimes \QCoh(\CY).$$

\ssec{Serre duality on $\IndCoh(\CY)$} \label{ss:Serreduality}

\sssec{}

Recall (see \cite[Sect. 9.2.1]{IndCoh}) that for a quasi-compact DG scheme $Z$, there exists a canonical involutive
equivalence:
$$\bD^{\on{Serre}}_Z:\IndCoh(Z)^\vee\simeq \IndCoh(Z).$$

\medskip

In terms of \secref{sss:properties of duality}(ii'), the above equivalence corresponds to the identification
$$(\IndCoh(Z)^c)^{\on{op}}=\Coh(Z)^{\on{op}}\overset{\BD^{\on{Serre}}_{Z}}\longrightarrow \Coh(Z)=\IndCoh(Z)^c,$$
where the middle arrow is the \emph{Serre duality} functor. Explicitly, for $\CF\in \Coh(Z)$,
$$\BD^{\on{Serre}}_{Z}(\CF)=\underline\Hom_{\QCoh(Z)}(\CF,\omega_Z),$$
which is a priori an object of $\QCoh(Z)$, but in fact can be easily shown to belong to $\Coh(Z)$.

\begin{rem}
In the above formula, $\underline\Hom_{\QCoh(Z)}(-,-)$ denotes the inner Hom of
\cite{DG}, Sect. 5.1, defined whenever a monoidal category (in our case $\QCoh(Z)$)
is acting on a module category (in our case $\IndCoh(Z)$).
\end{rem}

\medskip

Our current goal is to show that the same goes through, when instead of a quasi-compact DG scheme $Z$
we have a QCA algebraic stack $\CY$.

\sssec{}

First, let $\CY$ be any algebraic stack. Recall the (non-cocomplete) category $\Coh(\CY)$, see \secref{ss:coherent}.
We obtain that there exists a canonical equivalence:
\begin{equation} \label{e:Serre coh on stack}
\BD^{\on{Serre}}_\CY: \Coh(\CY)^{\on{op}}\iso \Coh(\CY),
\end{equation}
characterized by the property that for every affine (or quasi-compact) quasi-compact DG scheme $S$ equipped with a \emph{smooth}
map $g:S\to \CY$, we have an identification
$$g^{\IndCoh,*}\circ \BD^{\on{Serre}}_\CY\simeq \BD^{\on{Serre}}_S\circ (g^!)^{\on{op}},$$
as functors $\Coh(\CY)^{\on{op}}\to \Coh(Z)$. Moreover, $\BD^{\on{Serre}}_\CY$ is naturally involutive. 

\begin{prop}  \label{p:Serre duality}
For $\CF_1\in \Coh(\CY)^{\on{op}}$ and $\CF_2\in \IndCoh(\CY)$ we have a canonical isomorphism
\begin{equation} \label{Hom via D}
\CMaps(\BD^{\on{Serre}}_\CY(\CF_1),\CF_2)\simeq \Gamma^{\IndCoh}\left(\CY,\CF_1\sotimes \CF_2\right).
\end{equation}
\end{prop}

\begin{rem}  \label{r:pairing for IndCoh schemes}
The assertion of the proposition when $\CY$ is a quasi-compact DG scheme $Z$ follows
from the definition of the evaluation map
$$\IndCoh(Z)\otimes \IndCoh(Z)\to \Vect,$$
see \cite[Sect. 9.2.2]{IndCoh}.
\end{rem}

\begin{proof}

The left-hand side in \eqref{Hom via D} identifies with 
\begin{equation}  \label{e:LHSrewritten}
\underset{(S,g)\in (\on{DGSch}_{/\CY,\on{smooth}})^{\on{op}})}{\underset{\longleftarrow}{lim}}\, 
\CMaps_{\Coh(Z)}\left(g^!(\BD^{\on{Serre}}_\CY(\CF_1)),g^!(\CF_2)\right).
\end{equation}

\medskip

One can rewrite $\CMaps_{\Coh(S)}\left(g^!(\BD^{\on{Serre}}_\CY(\CF_1)),g^!(\CF_2)\right)$ as
$$\CMaps_{\Coh(S)}\left(\BD^{\on{Serre}}_Z(g^{\IndCoh,*}(\CF_1)),g^!(\CF_2)\right)\simeq
\Gamma^{\IndCoh}\left(S,g^{\IndCoh,*}(\CF_1)\sotimes g^!(\CF_2)\right),$$
where the last isomorphism takes place because of Remark \ref{r:pairing for IndCoh schemes}.

\medskip

Note that for $\CF\in \IndCoh(\CY)$, by \eqref{e:Gamma IndCoh} have:
$$\Gamma^{\IndCoh}(\CY,\CF)\simeq 
\underset{(S,g)\in (\on{DGSch}_{/\CY,\on{smooth}})^{\on{op}})}{\underset{\longleftarrow}{lim}}\, 
\Gamma^{\IndCoh}\left(S,g^{\IndCoh,*}(\CF)\right).$$

\medskip

Therefore, the right-hand side in \eqref{Hom via D} is canonically isomorphic to 
$$\underset{(S,g)\in (\on{DGSch}_{/\CY,\on{smooth}})^{\on{op}})}{\underset{\longleftarrow}{lim}}\, 
\Gamma^{\IndCoh}\left(S,g^{\IndCoh,*}(\CF_1\sotimes \CF_2)\right).$$

\medskip

Therefore, in order to construct the isomorphism in \eqref{Hom via D}, it remains
to construct a compatible family of isomorphisms of functors
\begin{equation} \label{e:! and * diag}
\Delta_S^!\circ (g^{\IndCoh,*}\boxtimes g^!)\simeq g^{\IndCoh,*}\circ \Delta_\CY^!.
\end{equation}

\medskip

The latter isomorphism of functors is valid for any $k$-representable, eventually coconnective 
morphism between prestacks $\pi:\CY_1\to \CY_2$: it follows by applying \eqref{e:! and * stacks} to the Cartesian diagram
$$
\CD
\CY_1  @>>>  \CY_2\times \CY_1  \\
@V{\pi}VV   @VV{\on{id}_{\CY_2}\times \pi}V  \\
\CY_2 @>{\Delta_{\CY_2}}>> \CY_2\times \CY_2.
\endCD
$$

\end{proof}

\sssec{}

Assume now that $\CY$ is a QCA algebraic stack. Then by \thmref{IndCoh},
$$\IndCoh(\CY)\simeq \Ind(\Coh(\CY)).$$
So, by \secref{sss:properties of duality}(ii'), from \eqref{e:Serre coh on stack} we deduce:

\begin{cor}  \label{c:self-duality of IndCoh}
For a QCA algebraic stack $\CY$ there is a natural involutive identification:
\begin{equation}  \label{e:self-duality of IndCoh}
\bD_\CY^{\on{Serre}}:\IndCoh(\CY)^\vee\simeq \IndCoh(\CY).
\end{equation}
\end{cor}

\sssec{}

Will shall now describe explicitly the duality data $\epsilon_{\IndCoh(\CY)}$ and $\mu_{\IndCoh(\CY)}$
that corresponds to the equivalence \eqref{e:self-duality of IndCoh}. We claim:

\begin{prop}  \label{p:pairing for IndCoh} Let $\CY$ be a QCA algebraic stack. Then the duality \eqref{e:self-duality of IndCoh}
has as evaluation $\epsilon_{\IndCoh(\CY)}$ the functor 
\begin{equation} \label{e:counit for IndCoh}
\IndCoh(\CY)\otimes \IndCoh(\CY)\to \IndCoh(\CY\times \CY)\overset{\Delta^!_\CY}\longrightarrow
\IndCoh(\CY)\overset{\Gamma^{\IndCoh}(\CY,-)}\longrightarrow \Vect,
\end{equation}
and as a co-evaluation $\mu_{\IndCoh(\CY)}$ the functor
\begin{equation} \label{e:unit for IndCoh}
\Vect\overset{\omega_\CY\otimes -}\longrightarrow \IndCoh(\CY)\overset{(\Delta_\CY)^{\IndCoh}_*}\longrightarrow
\IndCoh(\CY\times \CY)\simeq \IndCoh(\CY)\otimes \IndCoh(\CY).
\end{equation}
\end{prop}

\begin{proof}

Let $\CF_1,\CF_2$ be two objects of $\Coh(\CY)$. In order to identity $\epsilon_{\IndCoh(\CY)}$ with the functor
\eqref{e:counit for IndCoh}, we need to establish a functorial isomorphism
$$\CMaps_{\Coh(\CY)}(\BD^{\on{Serre}}_\CY(\CF_1),\CF_2)\simeq \Gamma^{\IndCoh}(\CY,\CF_1\sotimes \CF_2).$$
However, this is the content of \propref{p:Serre duality}.

\medskip

In order to prove that $\mu_{\IndCoh(\CY)}$ is given by \eqref{e:unit for IndCoh}, it is sufficient to show 
that the composition 
%\begin{multline*}
$$\IndCoh(\CY)\overset{\on{Id}_{\IndCoh(\CY)}\otimes \text{\eqref{e:unit for IndCoh}}}
\longrightarrow \IndCoh(\CY)\otimes \IndCoh(\CY)\otimes \IndCoh(\CY)
\overset{\text{\eqref{e:counit for IndCoh}}\otimes \on{Id}_{\IndCoh(\CY)}}\longrightarrow \IndCoh(\CY)$$
%\end{multline*}
is isomorphic to the identity functor. 

\medskip

Consider the diagram
$$
\CD
\CY  @>{\Delta_\CY}>> \CY\times \CY  @>{\on{id}\times p_\CY}>>  \CY \\
@V{\Delta_\CY}VV   @VV{\on{id}\times \Delta_\CY}V \\
\CY\times \CY  @>{\Delta_\CY\times \on{id}}>> \CY\times \CY\times \CY \\
@V{p_\CY\times \on{id}}VV  \\
\CY.
\endCD
$$
We need to show that the functor
\begin{equation} \label{e:comp for mu}
(\on{Id}_{\IndCoh(\CY)}\otimes (p_\CY)^{\IndCoh}_*)
\circ (\on{id}\times \Delta_\CY)^!\circ (\Delta_\CY\times \on{id})_*^{\IndCoh}\circ (p^!_\CY\otimes \on{Id}_{\IndCoh(\CY)})
\end{equation}
is isomorphic to the identity functor. 

\medskip

However, in the above diagram the inner square is Cartesian and the arrows in it are schematic and quasi-compact. Therefore,
by the base change isomorphism, we have
$$(\Delta_\CY)_*^{\IndCoh}\circ (\Delta_\CY)^! \simeq
(\on{id}\times \Delta_\CY)^!\circ (\Delta_\CY\times \on{id})_*^{\IndCoh}:\IndCoh(\CY\times \CY)\to 
\IndCoh(\CY\times \CY).$$

\medskip

Therefore, the functor in \eqref{e:comp for mu} is isomorphic to
\begin{multline*}
(\on{Id}_{\IndCoh(\CY)}\otimes (p_\CY)^{\IndCoh}_*)\circ (\Delta_\CY)_*^{\IndCoh}\circ (\Delta_\CY)^! \circ
(p^!_\CY\otimes \on{Id}_{\IndCoh(\CY)})\simeq  \\
\simeq (\on{id}\times p_\CY)_*\circ (\Delta_\CY)_*^{\IndCoh}\circ (\Delta_\CY)^! \circ 
(p_\CY\times \on{id})^!\simeq \\
\simeq \left((\on{id}\times p_\CY)\circ \Delta_\CY\right)^{\IndCoh}_*\circ
\left(\Delta_\CY\circ (p_\CY\times \on{id})\right)^!\simeq (\on{id})^{\IndCoh}_*\circ \on{id}^!\simeq
\on{Id}.
\end{multline*}

\end{proof}

\sssec{}

Let $\pi:\CY_1\to \CY_2$ be a morphism of QCA algebraic stacks. We have the functors
$$\pi_*^{\IndCoh}:\IndCoh(\CY_1)\to \IndCoh(\CY_2) \text{ and } \pi^!:\IndCoh(\CY_2)\to \IndCoh(\CY_1).$$

We claim that these functors are related as follows. Recall the notion of dual functor, see \secref{sss:dual functors}.

\begin{prop} \label{p:duality of * and !}
Under the identifications $\bD^{\on{Serre}}_{\CY_i}:\IndCoh(\CY_i)^\vee\simeq \IndCoh(\CY_i)$, we have:
$$(\pi_*^{\IndCoh})^\vee\simeq \pi^!.$$
\end{prop}

\begin{proof}

We need to show that the object in 
$$\IndCoh(\CY_1)^\vee\otimes \IndCoh(\CY_2)\simeq \IndCoh(\CY_1)\otimes \IndCoh(\CY_2)\simeq \IndCoh(\CY_1\times \CY_2)$$
that corresponds to $\pi_*^{\IndCoh}$ is isomorphic to the object that corresponds to $\pi^!$. The former is given by
$$(\on{id}_{\CY_1}\times \pi)^{\IndCoh}_*\circ (\Delta_{\CY_1})^{\IndCoh}_*(\omega_{\CY_1}),$$
and the latter by
$$(\pi\times\on{id}_{\CY_2})^!\circ (\Delta_{\CY_2})_*^{\IndCoh}(\omega_{\CY_2}).$$

The needed isomorphism follows by base change (see \secref{sss:functoriality IndCoh prestacks pushforward}) from the Cartesian diagram
$$
\CD
\CY_1 @>{\on{Graph}(\pi)}>> \CY_1\times \CY_2 \\
@V{\pi}VV   @VV{\pi\times \on{id}_{\CY_2}}V  \\
\CY_2  @>{\Delta_{\CY_2}}>>  \CY_2\times \CY_2,
\endCD
$$
in which the horizontal arrows are schematic and quasi-compact.

\end{proof}

\sssec{Proof of \propref{p:proof of ren vs nonren indcoh}} \label{sss:proof of ren vs nonren indcoh}

As was mentioned in Remark \ref{r:Gamma ind}, we note that the assertion of the proposition when $\CY_2=\on{pt}$ is the isomorphism 
\eqref{e:Gamma IndCoh}.

\medskip

In the general case, by \propref{p:duality of * and !}, it suffices to show that for $\CF_2\in \Coh(\CY_2)$ and $\CF_1\in \IndCoh(\CY_1)$ the natural map
\begin{equation} \label{e:ren vs nonren indcoh one}
\Gamma^{\IndCoh}(\CY_1,\CF_1\sotimes \pi^!(\CF_2))\to 
\Gamma^{\IndCoh}(\CY_2,\pi^{\IndCoh}_{\on{non-ren},*}(\CF_1)\sotimes \CF_2)
\end{equation}
is an isomorphism. We rewrite the right-hand side as
$$\CMaps_{\IndCoh(\CF_2)}(\BD_{\CY_2}^{\on{Serre}}(\CF_2),\pi^{\IndCoh}_{\on{non-ren},*}(\CF_1)),$$
and further as 
$$\underset{(S,g)\in ((\on{DGSch}_{\on{aft}})_{/\CY_1,\on{smooth}})^{\on{op}}}
{\underset{\longleftarrow}{lim}}\, 
\CMaps_{\IndCoh(\CF_2)}\left(\BD_{\CY_2}^{\on{Serre}}(\CF_2),(\pi\circ g)^{\IndCoh}_*(g^{\IndCoh,*}(\CF_1))\right).$$

\medskip

The latter expression can be rewritten as 
\begin{multline*}
\underset{(S,g)\in ((\on{DGSch}_{\on{aft}})_{/\CY_1,\on{smooth}})^{\on{op}}}
{\underset{\longleftarrow}{lim}}\, \Gamma^{\IndCoh}(\CY_2,(\pi\circ g)^{\IndCoh}_*(g^{\IndCoh,*}(\CF_1))\sotimes \CF_2)
\simeq \\
\simeq \underset{(S,g)\in ((\on{DGSch}_{\on{aft}})_{/\CY_1,\on{smooth}})^{\on{op}}}
{\underset{\longleftarrow}{lim}}\, \Gamma^{\IndCoh}(S,g^{\IndCoh,*}(\CF_1)\sotimes (\pi\circ g)^!(\CF_2)).
\end{multline*}

\medskip

Using the fact that
$$g^{\IndCoh,*}(\CF_1)\sotimes (\pi\circ g)^!(\CF_2)\simeq g^{\IndCoh,*}(\CF_1\sotimes \pi^!(\CF_2))$$
(see \eqref{e:! and * diag}), 
we finally obtain that the right-hand side in \eqref{e:ren vs nonren indcoh one} is isomorphic to
$$\underset{(S,g)\in ((\on{DGSch}_{\on{aft}})_{/\CY_1,\on{smooth}})^{\on{op}}}
{\underset{\longleftarrow}{lim}}\, \Gamma^{\IndCoh}\left(S,g^{\IndCoh,*}(\CF_1\sotimes \pi^!(\CF_2))\right)
\simeq \Gamma^{\IndCoh}(\CY_1,\CF_1\sotimes \pi^!(\CF_2)),$$
as required.

\qed

\section{Recollections: D-modules on DG schemes}  \label{s:Dmods on schemes}  

This section is devoted to a review of the theory of D-modules on (DG) schemes. As was mentioned
in the introduction, this material is well-known at the level of triangulated categories. However, no
comprehensive account seems to exist at the DG level.\footnote{That said, the ``local" aspects of the theory
of D-modules (i.e., when we only need to pull back, but not push forward) is a formal consequence of
$\IndCoh$ by the procedure of passage to the de Rham prestack. Details on that can be found in \cite{GR1}.}

\medskip

We remind that according to the conventions of \secref{ss:laft}, \emph{all DG schemes, algebraic stacks and prestacks are assumed locally almost of finite type, unless specified otherwise.} 

%\medskip

%All (DG) schemes in this section will be assumed quasi-compact, unless specified otherwise. 

%\medskip

%As D-modules on a DG scheme/stack/prestack only depend on the underlying classical
%scheme/stack/prestack, we can pass to the world of the classical algebraic geometry. 

\ssec{The basics}

%In this subsection, all DG schemes will be assumed quasi-compact, unless specified otherwise. 

\sssec{}

To any quasi-compact DG scheme\footnote{According to Sect.~\ref{podstilka} below, $\on{D-mod}(Z)$ depends only on the 
underlying classical scheme ${}^{cl\!}Z$. 
The only reason for working in the format of DG schemes is that
we will discuss the relation between $\on{D-mod}(Z)$ and  the category
$\IndCoh(Z)$, which depends on the DG structure.}
$Z$ one assigns the category $\on{D-mod}(Z)$ of right D-modules on $Z$. 

\sssec{}

By definition,
$$\on{D-mod}(Z):=\IndCoh(Z_{\on{dR}}),$$
see \cite[Sect. 2.3.2]{GR1}. (Note that in {\it loc.cit}, the category $\on{D-mod}(Z)$
is denoted $\on{Crys}^r(Z)$.)

\medskip

In the above formula $Z_{\on{dR}}$ is the de Rham prestack of $Z$, i.e.,
$$\on{Maps}(S,Z_{\on{dR}}):=\on{Maps}(({}^{cl\!}S)_{red},Z),$$
where $({}^{cl\!}S)_{red}$ denotes the classical reduced scheme underlying $S$, see \cite[Sect. 1.1.1]{GR1}.

\sssec{}

For any map $f:Z_1\to Z_2$
of quasi-compact DG schemes, there exists a canonically defined continuous functor
$$f^!:\Dmod(Z_2)\to \Dmod(Z_1).$$

\medskip

If $f$ is proper\footnote{A morphism of DG schemes is said to be proper if the underlying morphism of classical schemes
is.}, the functor $f^!$ admits a \emph{left} adjoint, denoted $f_{\on{dR},*}$. 
If $f$ is an open embedding, the functor $f^!$ admits a continuous \emph{right} adjoint, also denoted $f_{\on{dR},*}$. 

%\sssec{Functoriality}  \label{sss:Dmod functoriality}

%A crucial property of the assignment $Z\rightsquigarrow \Dmod(Z)$ is that it has a structure of functor
%$$\Dmod^!_{(\on{DGSch}_{\on{aft}})^{\on{op}}}:(\on{DGSch}_{\on{aft}})^{\on{op}}\to \StinftyCat_{\on{cont}}.$$

\sssec{Descent} \label{Zariski descent}

The assignment $Z\rightsquigarrow \Dmod(Z)$ satisfies fppf descent. 

\medskip

In particular, it satisfies Zariski descent, so the category $\on{D-mod}(Z)$ is glued from the categories 
$\on{D-mod}(U_i)$, where $\{U_i\}$ is a Zariski-open affine cover of $Z$.

\medskip

Therefore, for many purposes it is sufficient to consider the case of affine DG schemes.

\medskip

In addition, gluing can be used to define $\Dmod(Z)$ on a not necessarily quasi-compact DG scheme, as well
as the functor $f^!:\Dmod(Z_2)\to \Dmod(Z_1)$ for a map $f:Z_1\to Z_2$ of not necessarily 
quasi-compact DG schemes.

\medskip

This will be a particular case of the definition of $\Dmod(\CY)$ on a prestack $\CY$,
see \secref{sss:Dmod on prestacks}.

\sssec{Relation between $\on{D-mod}(Z)$ and $\IndCoh(Z)$}  \label{sss:DR-pi^!}

For a DG scheme $Z$ we have a pair of mutually adjoint (continuous) functors
$$\ind_{\on{D-mod}(Z)}:\IndCoh(Z)\rightleftarrows \on{D-mod}(Z):\oblv_{\on{D-mod}(Z)},$$
with $\oblv_{\on{D-mod}(Z)}$ being conservative. The functor $\oblv_{\on{D-mod}(Z)}$ corresponds to pullback 
along the tautological morphism 
$Z\to Z_{\on{dR}}$.

\medskip

For a morphism of DG schemes $f:Z_1\to Z_2$, we have a commutative diagram
$$
\CD
\IndCoh(Z_1)  @<{\oblv_{\on{D-mod}(Z_1)}}<< \on{D-mod}(Z_1)  \\
@A{f^!}AA    @AA{f^!}A   \\
\IndCoh(Z_2)  @<{\oblv_{\on{D-mod}(Z_2)}}<< \on{D-mod}(Z_2). 
\endCD
$$

\medskip

In particular, by taking $Z_1=Z$ and $Z_2=\on{pt}$, we obtain that 
the dualizing complex $\omega_Z$, initially defined
as an object of $\IndCoh(Z)$, naturally upgrades to (i.e., is the
image under $\oblv_{\on{D-mod}(Z)}$ of) a canonically defined object of
$\on{D-mod}(Z)$. By a slight abuse of notation, we denote the latter
by the same character $\omega_Z$.

\medskip

As a consequence of \lemref{l:! cons IndCoh} and the conservativeness 
of the functor $\oblv_{\on{D-mod}(Z)}$, we obtain:

\begin{lem} \label{l:! cons D}
Let a morphism $f:Z_1\to Z_2$ be surjective on $k$-points. Then the functor
$f^!:\on{D-mod}(Z_2)\to \on{D-mod}(Z_1)$ is conservative.
\end{lem}

\sssec{Tensor product}  \label{sss:tens_schemes}
For a pair of DG schemes $Z_1$ and $Z_2$ we have a canonical (continuous) functor
$$\on{D-mod}(Z_1)\otimes \on{D-mod}(Z_2)\to \on{D-mod}(Z_1\times Z_2),$$
which is an equivalence if $Z_1$ and $Z_2$ are quasi-compact.

\begin{rem}
According to Corollary~\ref{c:D on prod} below, quasi-compactness of \emph{one} of the DG schemes is enough.
\end{rem}
%t suffices that one of the schemes is .}
%\emph{if one of the schemes is quasi-compact}.

\medskip

In particular, we have a functor of tensor product 
$$\on{D-mod}(Z)\otimes \on{D-mod}(Z)\to \on{D-mod}(Z)$$
equal to
$$\on{D-mod}(Z)\otimes \on{D-mod}(Z)\to \on{D-mod}(Z\times Z)\overset{\Delta^!_Z}\longrightarrow \on{D-mod}(Z).$$
We denote this functor by
$$\CM_1,\CM_2\mapsto \CM_1\overset{!}\otimes \CM_2.$$
This defines a symmetric monoidal structure on the category $\Dmod(Z)$. The unit in the category
is $\omega_Z$.

\medskip

By \secref{sss:DR-pi^!}, we have:
$$\oblv_{\on{D-mod}(Z)}(\CM_1)\overset{!}\otimes \oblv_{\on{D-mod}(Z)}(\CM_2)\simeq
\oblv_{\on{D-mod}(Z)}(\CM_1\overset{!}\otimes \CM_2),$$ where 
$$\overset{!}\otimes:\IndCoh(Z)\otimes \IndCoh(Z)\to \IndCoh(Z)$$
is as in \secref{sss:IndCoh schemes ten}.

\medskip

By adjunction, for $\CF\in \IndCoh(Z)$ and $\CM\in \Dmod(Z)$, we have a canonical map
\begin{equation} \label{e:tensor with induction}
\ind_{\on{D-mod}(Z)}\left(\CF\overset{!}\otimes \oblv_{\on{D-mod}(Z)}(\CM)\right)\to
\ind_{\on{D-mod}(Z)}(\CF)\overset{!}\otimes \CM.
\end{equation}
It is easy to show (e.g., using Kashiwara's lemma below) that the map \eqref{e:tensor with induction} is an isomorphism.

\sssec{Kashiwara's lemma} 

If $i:Z_1\to Z_2$ is a closed embedding, \footnote{A map of DG schemes is called
a closed embedding if the map of the underlying classical schemes is.}
then the functor $i_{\on{dR},*}$ induces an equivalence
\begin{equation} \label{e:Kashiwara}
\on{D-mod}(Z_1)\to \on{D-mod}(Z_2)_{Z_1},
\end{equation}
where $\on{D-mod}(Z_2)_{Z_1}$ is the full subcategory of $\on{D-mod}(Z_2)$ that consists
of objects that vanish on the complement $Z_2-Z_1$. The inverse equivalence is given by
$i^!|_{\on{D-mod}(Z_2)_{Z_1}}$. 

\medskip

This observation allows to reduce the local aspects of the theory of D-modules on DG schemes
to those on smooth classical schemes. 

\sssec{Topological invariance}  \label{podstilka}

In particular, if a map $i:Z_1\to Z_2$ is such that the induced 
map $$({}^{cl\!}Z_1)_{red}\to ({}^{cl\!}Z_2)_{red}$$ is an isomorphism,
then the functors
\begin{equation} \label{e:nil embedding}
i_{\on{dR},*}:\on{D-mod}(Z_1)\rightleftarrows \on{D-mod}(Z_2):i^!
\end{equation}
are equivalences. 

\medskip

This shows, in particular, that for any $Z$, pullback along the canonical map
$({}^{cl\!}Z)_{red}\to Z$ induces an equivalence
$$\on{D-mod}(Z)\to \on{D-mod}(({}^{cl\!}Z)_{red}).$$

\medskip

So, when discussing the aspects of the theory of D-modules that do not
involve the functors $\ind_{\on{D-mod}(Z)}$ and $\oblv_{\on{D-mod}(Z)}$,
we can (and will) restrict ourselves to classical schemes, and can even assume 
that they are reduced, without losing in generality. 

\sssec{t-structure}  \label{Dt}

The category $\on{D-mod}(Z)$ has a canonical t-structure. It is defined so that $\on{D-mod}(Z)^{>0}$ consists of 
all $\CF\in \on{D-mod}(Z)$ such that $\oblv_{\on{D-mod}(Z)}(\CF )\in\IndCoh(Z)^{>0}$. 

\medskip

For a closed embedding $i:Z_1\to Z_2$, the functor $i_{\on{dR},*}$ is t-exact. In particular,
the equivalence \eqref{e:Kashiwara} is compactible with t-structures, where the t-structure
on $\on{D-mod}(Z_2)_{Z_1}$ is induced by that on $\on{D-mod}(Z_2)$.

\medskip

By definition, the functor $\oblv_{\on{D-mod}(Z)}$ is left t-exact. If $Z$ is smooth, then 
$\oblv_{\on{D-mod}(Z)}$ is t-exact. For any quasi-compact $Z$ it has finite cohomological amplitude:
to prove this, reduce to the case where $Z$ is affine and then embed $Z$ into a smooth classical scheme.

\medskip

For the same reason, the functor $\ind_{\on{D-mod}(Z)}$ is always t-exact.

\begin{lem} \label{D-properties of t} 
The t-structure on $\on{D-mod}(Z)$ is left-complete and is compatible with filtered colimits.
%i.e., the subcategories $\on{D-mod}(Z)^{\leq 0}$ and $\on{D-mod}(Z)^{> 0}$ are preserved under filtered colimits.
\end{lem}

The meaning of these words is explained in Lemma~\ref{properties of t}.

\begin{proof} 
Compatibility with filtered colimits is clear from the definition of $\on{D-mod}(Z)^{>0}$.
To prove left-completeness, it suffices to consider the case where $Z$ is affine.
In this case it follows from the existence of a conservative
t-exact functor $\Phi:\on{D-mod}(Z)\to\Vect$ commuting with limits. To construct such $\Phi$,
choose an embedding $i:Z\hookrightarrow Y$ with $Y$ affine and smooth, then take
$\Phi$ to be the composition of 
$i_{\on{dR},*}:\on{D-mod}(Z)\to \on{D-mod}(Y)$, $\oblv_{\on{D-mod}(Y)}:\on{D-mod}(Y)\to \IndCoh(Y)$, 
$\Psi_\CY:\IndCoh(Y)\simeq \QCoh(Y)$ and $\Gamma :\QCoh(Y)\to\Vect$.
\end{proof} 

\sssec{Relation between $\on{D-mod}(Z)$ and $\QCoh(Z)$}  \label{sss:Dmod and QCoh}

It follows from \lemref{D-properties of t} that the functor
$$\ind_{\Dmod(Z)}:\IndCoh(Z)\to \Dmod(Z)$$ 
canonically factors as
$$\IndCoh(Z)\overset{\Psi_Z}\longrightarrow \QCoh(Z)\to \Dmod(Z).$$
This is a formal consequence of the fact that the functor $\Psi_Z$
identifies $\QCoh(Z)$ with the left completion of $\IndCoh(Z)$ with
respect to its t-structure, while $\Dmod(Z)$ is left-complete and
$\ind_{\Dmod(Z)}$ is right t-exact.

\medskip

We shall denote the resulting functor $\QCoh(Z)\to \Dmod(Z)$ by
$'\ind_{\Dmod(Z)}$. In addition, we have a functor
$'\oblv_{\Dmod(Z)}:\Dmod(Z)\to \QCoh(Z)$ defined as
$$'\oblv_{\Dmod(Z)}:=\Psi_Z\circ \oblv_{\Dmod(Z)}.$$

It follows from Kashiwara's lemma that the functor $'\oblv_{\Dmod(Z)}$ is also
conservative.

\medskip

Assume now that $Z$ is eventually coconnective (we remind that this implies that the functor 
$\Psi_Z$ admits a fully faithful 1 adjoint). Again, it follows formally that in this case the
functors 
$$'\ind_{\Dmod(Z)}:\QCoh(Z)\rightleftarrows \Dmod(Z):{}'\oblv_{\Dmod(Z)}$$
are mutually adjoint.

\begin{rem}
We emphasize, however, that the latter case is \emph{false} of $Z$ is not essentially
coconnective. E.g., in the latter case the functor $'\ind_{\Dmod(Z)}$ does not send
compact objects to compact ones.
\end{rem}

\begin{rem}
The category $\Dmod(Z)$, equipped with the functor $'\oblv_{\Dmod(Z)}$, is the more
familiar realization of D-modules as right D-modules (but which only works in the
eventually coconnective case). 
\end{rem}

\sssec{The ``left" realization}  \label{sss:left realization}

For completeness let us mention that in addition to $\oblv_{\Dmod(Z)}$, there is another canonically defined forgetful
functor
$$\oblv^{\on{left}}_{\Dmod(Z)}:\Dmod(Z)\to \QCoh(Z),$$
responsible for the realization of $\Dmod(Z)$ as ``left D-modules". 

\medskip

For a map $f:Z_1\to Z_2$ of DG schemes, the following diagram naturally commutes:
$$
\CD
\QCoh(Z_1)  @<{\oblv^{\on{left}}_{\on{D-mod}(Z_1)}}<< \on{D-mod}(Z_1)  \\
@A{f^*}AA    @AA{f^!}A   \\
\QCoh(Z_2)  @<{\oblv^{\on{left}}_{\on{D-mod}(Z_2)}}<< \on{D-mod}(Z_2). 
\endCD
$$

The functors
$\oblv^{\on{left}}_{\Dmod(Z)}$ and $\oblv_{\Dmod(Z)}$ are related by the formula
$$\oblv_{\Dmod(Z)}(\CM)\simeq \oblv^{\on{left}}_{\Dmod(Z)}(\CM)\otimes \omega_Z,$$
where $\otimes$ is understood in the sense of the action of $\QCoh(Z)$ on 
$\IndCoh(Z)$, see \secref{sss:IndCoh schemes}.

\medskip

In addition, we have a functor
$$\ind^{\on{left}}_{\Dmod(Z)}:\QCoh(Z)\to \Dmod(Z)$$ defined by the formula
$$\ind^{\on{left}}_{\Dmod(Z)}(\CF):=\ind_{\Dmod(Z)}(\CF\otimes \omega_Z).$$
It again follows formally that when $Z$ is eventually coconnective, the functors
$$(\ind^{\on{left}}_{\Dmod(Z)},\oblv^{\on{left}}_{\Dmod(Z)})$$ form an adjoint pair.

\medskip

For any $Z$ one has
\begin{equation}  \label{e:two_obls-schemes}
'\oblv_{\Dmod(Z)}(\CM)\simeq \oblv^{\on{left}}_{\Dmod(Z)}(\CM)\otimes \Psi_Z(\omega_Z), \quad \CM\in \Dmod(Z)
\end{equation}
and
\begin{equation}   \label{e:two_inds-schemes}
\ind^{\on{left}}_{\Dmod(Z)}(\CF)\simeq {}'\ind_{\Dmod(Z)}(\CF\otimes \Psi_Z(\omega_Z)), \quad \CF\in\QCoh(Z),
\end{equation}
where $\Psi_Z:\IndCoh(Z)\to \QCoh(Z)$ is the functor of \secref{sss:Ind and QCoh}.

\sssec{Coherence and compact generation}

Let $\on{D-mod}_{\on{coh}}(Z)\subset \on{D-mod}(Z)$ denote the full subcategory of bounded complexes 
whose cohomology sheaves are coherent (i.e., locally finitely generated) D-modules.

\medskip

If $Z$ is quasi-compact, we have $\on{D-mod}_{\on{coh}}(Z)=\on{D-mod}(Z)^c$, and this subcategory
generates $\on{D-mod}(Z)$. I.e., 
$$\on{D-mod}(Z)\simeq \on{Ind}(\on{D-mod}_{\on{coh}}(Z)).$$

In fact, this is a formal consequence of the following three facts: (a) that the functor $\oblv_{\on{D-mod}(Z)}$ is conservative;
(b) that $\ind_{\on{D-mod}(Z)}$ sends $\Coh(Z)$ to $\on{D-mod}_{\on{coh}}(Z)$ (which follows from Kashiwara's
lemma), and (c) that for $Z$ quasi-compact $\Coh(Z)$ compactly generates $\IndCoh(Z)$. 

\ssec{The de Rham cohomology functor on DG schemes}  \label{ss:de Rham pushforward schemes}

\sssec{}  \label{sss:properties of Dmod}

Let $f:Z_1\to Z_2$ be a quasi-compact morphism between DG schemes. In this case the classical theory
of D-modules constructs a continuous functor:
$$f_\dr:\on{D-mod}(Z_1)\to \Dmod(Z_2).$$

\medskip

The following are the some of the key features of this functor:

\smallskip

\noindent(i) The assignment $f\rightsquigarrow f_\dr$ is compatible with composition of functors
in the natural sense.

\smallskip

\noindent(ii) For $f$ proper, the functor $f_\dr$ is the left adjoint to $f^!$.

\smallskip

\noindent(iii) For $f$ an open embedding, the functor $f_\dr$ is the right adjoint to $f^!$.

\smallskip

\noindent (iv) For a Cartesian square
$$
\CD
Z'_1   @>{g_1}>>  Z_1 \\
@V{f'}VV   @VV{f}V  \\
Z'_2   @>{g_2}>>  Z_2
\endCD
$$
we have a canonical isomorphism of functors $\on{D-mod}(Z_1)\to \on{D-mod}(Z'_2)$
\begin{equation} \label{e:Dmod base change}
f'_\dr\circ g_1^!\simeq g_2^!\circ f_\dr.
\end{equation}

\medskip

However, even the formulation of these properties in the framework on $\infty$-categories
is not straightforward. For example, it is not so easy to formulate the compatibility between
the isomorphisms (i) and (iv), and also between (ii) or (iii) and (iv). \footnote{Note that when
$f$ is either proper or open, there is a canonical map in one direction in \eqref{e:Dmod base change}
by adjunction. So, in particular, we must have a compatibility condition that says that in either
of these cases, the two maps in \eqref{e:Dmod base change}: one arising by adjunction and the 
other by the data of (iv), must coincide.}

\medskip

At the same time, an $\infty$-category formulation is necessary for the treatment of the category of D-modules
on stacks, as the latter involves taking limits in $\StinftyCat_{\on{cont}}$.

\sssec{}   \label{sss:corr formalism}

We shall adopt the approach taken in \cite{FG}, Sect. 1.4.3, which was initially suggested by
J.~Lurie; it will be developed in detail in \cite{GR2}.

\medskip

Namely, let $(\on{DGSch}_{\on{aft}})_{\on{corr}}$ be the $(\infty,1)$-category whose objects are the same
as those of $\on{DGSch}_{\on{aft}}$, and where the $\infty$-\emph{groupoid} of $1$-morphisms
$\on{Maps}_{(\on{DGSch}_{\on{aft}})_{\on{corr}}}(Z_1,Z_2)$ is that of correspondences
\begin{gather}  
\xy
(-15,0)*+{Z_1}="A";
(15,0)*+{Z_2.}="B";
(0,15)*+{Z_{1,2}}="C";
{\ar@{->}_{f_l} "C";"A"};
{\ar@{->}^{f_r} "C";"B"};
\endxy
\end{gather}
Compositions in this category are defined by forming Cartesian products:
$$Z_{2,3}\circ Z_{1,2}=Z_{1,3}:$$

\begin{gather}  
\xy
(-15,0)*+{Z_1}="A";
(15,0)*+{Z_2}="B";
(0,15)*+{Z_{1,2}}="C";
(45,0)*+{Z_3.}="D";
(30,15)*+{Z_{2,3}}="E";
(15,30)*+{Z_{1,3}}="F";
{\ar@{->} "C";"A"};
{\ar@{->} "C";"B"};
{\ar@{->} "E";"B"};
{\ar@{->} "E";"D"};
{\ar@{->} "F";"C"};
{\ar@{->} "F";"E"};
\endxy
\end{gather}

\sssec{}   \label{sss:restr corr}

The category $(\on{DGSch}_{\on{aft}})_{\on{corr}}$ contains $\on{DGSch}_{\on{aft}}$ and $(\on{DGSch}_{\on{aft}})^{\on{op}}$
as 1-full subcategories where we restrict $1$-morphisms by requiring that $f_l$ (resp., $f_r$) be an 
isomorphism.

\medskip

The theory of D-modules is a functor
$$\Dmod_{(\on{DGSch}_{\on{aft}})_{\on{corr}}}:(\on{DGSch}_{\on{aft}})_{\on{corr}}\to \StinftyCat_{\on{cont}}.$$

\medskip

At the level of objects, this functor assigns to $Z\in \on{DGSch}_{\on{aft}}$ the category
$\Dmod(Z)$. 

\medskip

The restriction of $\Dmod_{(\on{DGSch}_{\on{aft}})_{\on{corr}}}$ to $\on{DGSch}_{\on{aft}}\subset (\on{DGSch}_{\on{aft}})_{\on{corr}}$, denoted
$\Dmod_{\on{DGSch}_{\on{aft}}}$, expresses our ability to take $f_\dr:\Dmod(Z_1)\to \Dmod(Z_2)$, and corresponds to diagrams
of the form
\begin{gather}  
\xy
(-15,0)*+{Z_1}="A";
(15,0)*+{Z_2.}="B";
(0,15)*+{Z_{1}}="C";
{\ar@{->}_{\on{id}} "C";"A"};
{\ar@{->}^{f} "C";"B"};
\endxy
\end{gather}

\medskip

The restriction of $\Dmod_{(\on{DGSch}_{\on{aft}})_{\on{corr}}}$ to $(\on{DGSch}_{\on{aft}})^{\on{op}}\subset (\on{DGSch}_{\on{aft}})_{\on{corr}}$,
which we denote by $\Dmod^!_{\on{DGSch}_{\on{aft}}}$, expresses our ability to take $f^!:\Dmod(Z_2)\to \Dmod(Z_1)$, and corresponds to diagrams
of the form
\begin{gather}  
\xy
(-15,0)*+{Z_1}="A";
(15,0)*+{Z_2.}="B";
(0,15)*+{Z_{2}}="C";
{\ar@{->}_{f} "C";"A"};
{\ar@{->}^{\on{id}} "C";"B"};
\endxy
\end{gather}

The base change isomorphism of \secref{sss:properties of Dmod}(iv) is encoded by the functoriality
of $\Dmod$. 

\medskip

As is explained in \cite{FG}, Sects. 1.4.5 and 1.4.6, the datum of the functor $\Dmod_{(\on{DGSch}_{\on{aft}})_{\on{corr}}}$
also contains the data of adjunction for $(f^!,f_\dr)$
when $f$ is an open embedding, and for $(f_\dr,f^!)$ when $f$ is proper.

\medskip

Unfortunately, there currently is no reference in the literature for the construction of the functor
$\Dmod_{(\on{DGSch}_{\on{aft}})_{\on{corr}}}$ with the above properties. However, a construction of a similar framework 
for $\IndCoh$ instead of $\Dmod$ has been indicated in \cite{IndCoh}, Sects. 5 and 6. 

\sssec{}  \label{sss:corr and IndCoh}

An additional part of data in the functor $\Dmod_{(\on{DGSch}_{\on{aft}})_{\on{corr}}}$ is the following one:

\medskip

The functor $\Dmod^!_{\on{DGSch}_{\on{aft}}}:(\on{DGSch}_{\on{aft}})^{\on{op}}\to \StinftyCat_{\on{cont}}$ comes
equipped with a natural transformation
$$\oblv_{\Dmod}:\Dmod^!_{\on{DGSch}_{\on{aft}}}\to \IndCoh^!_{\on{DGSch}_{\on{aft}}},$$
where 
$$\IndCoh^!_{\on{DGSch}_{\on{aft}}}:(\on{DGSch}_{\on{aft}})^{\on{op}}\to \StinftyCat_{\on{cont}}$$ 
is the functor of \secref{sss:IndCoh schemes}. \footnote{For an individual morphism, this datum is the one
in \secref{sss:DR-pi^!}.}

\medskip

The functor $\Dmod_{\on{DGSch}_{\on{aft}}}:\on{DGSch}_{\on{aft}}\to \StinftyCat_{\on{cont}}$ comes
equipped with a natural transformation
$$\ind_{\Dmod}:\IndCoh_{\on{DGSch}_{\on{aft}}}\to \Dmod_{\on{DGSch}_{\on{aft}}},$$
where 
$$\IndCoh_{\on{DGSch}_{\on{aft}}}:\on{DGSch}_{\on{aft}}\to \StinftyCat_{\on{cont}}$$ 
is the functor of \cite[Sect. 5.6.1]{IndCoh}.  In particular, for a morphism $f:Z_1\to Z_2$ of quasi-compact
schemes, we have a commutative diagram
\begin{equation} \label{e:DR of induced}
\CD
\Dmod(Z_1)  @<{\ind_{\Dmod(Z_1)  }}<<  \IndCoh(Z_1) \\
@V{f_\dr}VV   @VV{f^{\IndCoh}_*}V  \\
\Dmod(Z_1)  @<{\ind_{\Dmod(Z_1)  }}<<  \IndCoh(Z_1).
\endCD
\end{equation}

\begin{rem}
In principle, one would like to formulate the compatibility of the entire datum of the functor
$\Dmod_{(\on{DGSch}_{\on{aft}})_{\on{corr}}}$ with that of the functor
$\IndCoh_{(\on{DGSch}_{\on{aft}})_{\on{corr:all;all}}}$ of \cite[Sect. 5.6.1]{IndCoh}. However, we cannot do this
while staying in the world of $(\infty,1)$-categories, as some of the natural transformations
involved are not isomorphisms. 
\end{rem}

\sssec{Projection formula}

As in \eqref{e:proj formula IndCoh}, from \eqref{e:Dmod base change} one obtains that $f$ satisfies projection
formula for D-modules: for a map $f:Z_1\to Z_2$ of quasi-compact DG schemes, 
and $\CM_i\in \Dmod(Z_i)$, $i=1,2$ we have a 
canonical isomorphism
\begin{equation} \label{e:Dmod proj formula}
\CM_2\otimes f_\dr(\CM_1)\simeq f_\dr(f^!(\CM_2)\otimes \CM_1),
\end{equation}
functorial in $\CM_i$. 

\sssec{De Rham cohomology}     \label{sss:Gamma_DR} 

For $Z\in \on{DGSch}_{\on{aft}}$ we obtain a functor
$$\Gamma_{\on{dR}}(Z,-):=(p_Z)_\dr:\on{D-mod}(Z)\to \Vect,$$
where $p_Z:Z\to \on{pt}$.

\medskip

This functor is co-representable by an object $k_Z\in \on{D-mod}(Z)$, i.e.,
\begin{equation} \label{e:corepresentability}
\Gamma_{\on{dR}}(Z,\CM)=\CMaps(k_Z,\CM).
\end{equation}

\medskip

As $Z$ was assumed quasi-compact, the functor $\Gamma_{\on{dR}}(Z,-)$ is continuous,
so $k_Z\in \Dmod(Z)$ is compact.

\begin{rem}
By \secref{Zariski descent}, the object $k_Z\in \Dmod(Z)$ is defined for any $Z$, not necessarily quasi-compact.
However, in general, it will fail to be compact as an object of $\Dmod(Z)$.
\end{rem}

\sssec{}  \label{sss:pullback constant}

Let $f:Z_1\to Z_2$ be a map between quasi-compact schemes. Since
$$\Gamma_{\on{dR}}(Z_1,-)\simeq \Gamma_{\on{dR}}(Z_2,f_\dr(-)),$$
we obtain that the \emph{partially defined} left adjoint $f^*_{\on{dR}}$
to $f_\dr$ is defined on $k_{Z_2}$, and we have a canonical isomorphism
$$f^*_{\on{dR}}(k_{Z_2})\simeq k_{Z_1}.$$

\ssec{Verdier duality on DG schemes}   \label{ss:CohVer}

%IAn object of $\on{D-mod}(Z)$ is compact if and only if Iit is a bounded complex whose cohomology 
%objects are %Ifinitely generated. The category of compact objects of $\on{D-mod}(Z)$ is denoted by 
%$\on{D-mod}_{\on{coh}}(Z)$.
%
%Its subcategory of compact objects equals $\on{D-mod}_{\on{coh}}(Z)$, by which we mean the subcategory 
%of bounded complexes with finitely generated cohomology objects. 
%
%\medskip

\sssec{}

For a DG scheme $Z$, there is a (unique) involutive anti self-equivalence
$$\BD^{\on{Verdier}}_Z: (\on{D-mod}_{\on{coh}}(Z))^{\on{op}}\iso \on{D-mod}_{\on{coh}}(Z)$$
(called Verdier duality) such that
\begin{equation}   \label{e:Verdier_schemes}
\CMaps_{\on{D-mod}(Z)}(\BD_Z^{\on{Verdier}}(\CM),\CM')\simeq \Gamma_{\on{dR}}(Z,\CM\sotimes \CM'),
\end{equation}
for 
$\CM\in \on{D-mod}_{\on{coh}}(Z) , \,\CM'\in \Dmod(Z)$.

\medskip

Let $\omega_Z$ and $k_Z$ be as in Sect.~\ref{sss:DR-pi^!} and \ref{sss:Gamma_DR}.
Then $\omega_Z,k_Z\in \on{D-mod}_{\on{coh}}(Z)$ and 
$$k_Z\simeq \BD^{\on{Verdier}}_Z(\omega_Z).$$

\sssec{Verdier and Serre duality}

If $\CF\in \Coh(Z)$ then $\ind_{\on{D-mod}(Z)}(\CF)\in \on{D-mod}_{\on{coh}}(Z)$. We now claim:

\begin{lem}
There exists a canonical isomorphism 
\begin{equation}   \label{e:Verdier-Serre_on_schemes}
\BD_Z^{\on{Verdier}}\left(\ind_{\on{D-mod}(Z)}(\CF)\right) \simeq \ind_{\on{D-mod}(Z)}\left(\BD^{\on{Serre}}_Z(\CF)\right).
\end{equation}
\end{lem}

\begin{proof}
Follows by combining isomorphsms \eqref{e:Verdier_schemes}, \eqref{e:DR of induced} and \eqref{e:tensor with induction}, and
\propref{p:Serre duality} (for DG schemes).
\end{proof}

\sssec{Ind-extending Verdier duality}

For $Z$ quasi-compact, ind-extending Verdier duality, by \secref{sss:properties of duality}(ii'), we obtain an identification
\begin{equation} \label{self duality of D}
\bD_Z^{\on{Verdier}}:\on{D-mod}(Z)^\vee\simeq \on{D-mod}(Z),
\end{equation}
where $\on{D-mod}(Z)^\vee$ is the dual DG category (see Sect.~\ref{sss:dualizabilitydef}).
\medskip

By \eqref{e:Verdier_schemes}, the corresponding pairing 
$$\epsilon_{\on{D-mod}(Z)}:\on{D-mod}(Z)\otimes \on{D-mod}(Z)\to \Vect$$
equals the composition
$$\on{D-mod}(Z)\otimes \on{D-mod}(Z)\to \on{D-mod}(Z\times Z)\overset{\Delta_Z^!}\longrightarrow \on{D-mod}(Z)\overset{\Gamma_{\on{dR}}(Z,-)}
\longrightarrow \Vect.$$

\medskip

As in the proof of \propref{p:pairing for IndCoh}, the base change isomorphism implies that the co-evaluation functor
$$\mu_{\on{D-mod}(Z)}:\Vect\to \on{D-mod}(Z)\otimes \on{D-mod}(Z)$$
is given by
$$\Vect\overset{\omega_Z\otimes -}\longrightarrow \on{D-mod}(Z) \overset{\Delta_{\on{dR},*}}\longrightarrow 
\on{D-mod}(Z\times Z)\simeq \on{D-mod}(Z)\otimes \on{D-mod}(Z).$$

\medskip

The latter implies, in turn, that for a map of quasi-compact DG schemes $f:Z_1\to Z_2$, under the identifications 
$$\bD^{\on{Verdier}}_{Z_i}:\Dmod(Z_i)^\vee\simeq \Dmod(Z_i),$$
we have:
\begin{equation}   \label{e:duality_formula}
(f^!)^\vee\simeq f_{\dr}
\end{equation}
(see \secref{sss:dual functors} for the notion of dual functor).

\medskip

Note also that by \cite[Lemma 2.3.3]{DG}, the isomorphism \eqref{e:Verdier-Serre_on_schemes} can
also be formulated as saying that
$$(\oblv_{\Dmod(Z)})^\vee\simeq \ind_{\Dmod(Z)}$$
with respect to the identifications
$$\bD_Z^{\on{Verdier}}:\Dmod(Z)^\vee\simeq \Dmod(Z)  \text{ and } \bD_Z^{\on{Serre}}:(\IndCoh(Z))^\vee\simeq \IndCoh(Z),$$
given by Verdier and Serre dualities, respectively. 

\sssec{Smooth pullbacks}  \label{coh shift}

If $f$ is smooth then the functor $f_{\on{dR},*}$ admits a left adjoint, which we denote by $f^*_{\on{dR}}$.
Being a left adjoint, the functor $f^*_{\on{dR}}$ is continuous. If $f$ is of constant relative dimension $n$, 
we have a canonical isomorphism 
\begin{equation} \label{e:*! shift}
f^*_{\on{dR}}\simeq f^![-2n].
\end{equation}

One has

\begin{equation}   \label{e:added_by_Drinf1}
f^*_{\on{dR}} (\on{D-mod}_{\on{coh}}(Z_2))\subset \on{D-mod}_{\on{coh}}(Z_1), \quad f^! (\on{D-mod}_{\on{coh}}(Z_2))\subset \on{D-mod}_{\on{coh}}(Z_1),
\end{equation}

\begin{equation}    \label{e:added_by_Drinf2}
\BD_{Z_1}^{\on{Verdier}}\left(f_{\on{dR}}^*(\CM)\right)\simeq
f^!\left(\BD_{Z_2}^{\on{Verdier}}(\CM)\right), \quad \CM\in \on{D-mod}_{\on{coh}}(Z_2).
\end{equation}

\begin{rem}
Assume that $Z_1$ and $Z_2$ are quasi-compact (which we can, as the above assertions are Zariski-local). Recall that
in this case $\Dmod(Z_i)^c=\on{D-mod}_{\on{coh}}(Z_i)$. 
We obtain that 
\eqref{e:added_by_Drinf1} follows from the fact that $f^*_{\on{dR}}$ preserves compactness 
(because it has a continuous right adjoint), and \eqref{e:added_by_Drinf2} follows from formula \eqref{e:duality_formula}
combined with \cite[Lemma 2.3.3]{DG}.
\end{rem}

\sssec{}

For $\CM',\CM''\in \Dmod(Z_2)$ by adjunction and the projection formula \eqref{e:Dmod proj formula}
we obtain a map
\begin{equation} \label{e:tensor * and !}
f^*_{\on{dR}}(\CM'\sotimes \CM'')\to f^*_{\on{dR}}(\CM')\sotimes f^!(\CM'').
\end{equation}

However, it easily follows from \eqref{e:*! shift} that \eqref{e:tensor * and !} is an isomorphism.

%\medskip
%{\bf Old version:}
%When both $Z_1$ and $Z_2$ are quasi-compact, it follows from \cite{DG}, Sect. 2.3.2
%that the functor $\pi^!$ preserves compactness, and for $\CF\in \on{D-mod}_{\on{coh}}(Z_2)$ we have:
%$$\BD_{Z_1}^{\on{Verdier}}\left(\pi_{\on{dR}}^*(\CF)\right)\simeq
%\pi^!\left(\BD_{Z_2}^{\on{Verdier}}(\CF)\right).$$

\section{D-modules on stacks} \label{s:Dmods on stacks}

In this section we review the theory of D-modules on algebraic stacks to be used later in the paper.
On the one hand, this theory is well-known, at least at the level of triangulated categories. However,
as we could not find a single source that contains all the relevant facts, we decided to include the present
section for the reader's convenience. 

\ssec{D-modules on prestacks}   \label{ss:Dmods on prestacks}

\sssec{}  \label{sss:Dmod on prestacks}

Let $\CY$ be a prestack. The category $\on{D-mod}(\CY)$ is defined as 
\begin{equation}   \label{e:def_of_D(Y)}
\underset{(S,g)\in ((\on{DGSch}_{\on{aft}})_{/\CY})^{\on{op}}}{\underset{\longleftarrow}{lim}}\, \on{D-mod}(S),
\end{equation}
where we view the assignment
$$(S,g)\rightsquigarrow \Dmod(S)$$
as a functor between $\infty$-categories 
$$((\on{DGSch}_{\on{aft}})_{/\CY})^{\on{op}}\to \StinftyCat_{\on{cont}},$$
obtained by restriction under the forgetful map $(\on{DGSch}_{\on{aft}})_{/\CY}\to \on{DGSch}_{\on{aft}}$
of the functor
$$\Dmod^!_{\on{DGSch}_{\on{aft}}}:\on{DGSch}_{\on{aft}}^{\on{op}}\to \StinftyCat_{\on{cont}}.$$

\medskip

Concretely, an object $\CM$ of $\Dmod(\CY)$ is an assignment for every $g:S\to \CY$ of an object
$g^!(\CM)\in \Dmod(S)$, and a homotopy-coherent system of isomorphisms 
$$f^!(g^!(\CM))\simeq (g\circ f)^!(\CM)\in \Dmod(S')$$ for maps of DG schemes $f:S'\to S$.

\medskip

In the above limit one can replace the category of quasi-compact DG schemes 
by its subcategory of affine DG schemes, or by a larger category of all DG
schemes; this is due to the Zariski descent property of D-modules, see 
\secref{Zariski descent}.

\medskip

In addition, using the fppf descent property for D-modules (see \secref{Zariski descent}),
we can replace the categories $(\on{DGSch}_{\on{aft}})_{/\CY}$ (resp., $(\on{DGSch}^{\on{aff}}_{\on{aft}})_{/\CY}$)
by any of the indexing categories $A$ as in \secref{sss:change index qc}. The proof follows from 
\cite[Corollary 11.2.4]{IndCoh}. 

\sssec{}

According to \cite[Corollary 2.3.9]{GR1}, we can equivalently define $\on{D-mod}(\CY)$
as 
$$\IndCoh(\CY_{\on{dR}}),$$
where $\CY_{\on{dR}}$ is as an \cite[Sect. 1.1.1]{GR1}.

\sssec{}

Tautologically, for a morphism $\pi:\CY_1\to \CY_2$ between prestacks, we 
have a functor
$$\pi^!:\Dmod(\CY_2)\to \Dmod(\CY_1).$$

\medskip

In particular, for any prestack $\CY$, there exists a canonically defined object
$$\omega_\CY\in \Dmod(\CY)$$
equal to $(p_\CY)^!(k)$ for $p_\CY:\CY\to \on{pt}$.

\sssec{}

It follows from \secref{podstilka} that if a morphism of prestacks $\pi:\CY_1\to \CY_2$
induces an isomorphism of the underlying classical prestacks $^{cl}\CY_1\to {}^{cl}\CY_2$,
then the functor
$$\pi^!:\Dmod(\CY_2)\to \Dmod(\CY_1)$$
is an equivalence. 

\medskip

So, for a prestack $\CY$, the category $\Dmod(\CY)$ only depends on the underlying classical
prestack.

\sssec{}

Just as in Sect.~\ref{sss:tens_schemes}, for a pair of prestacks $\CY_1$ and $\CY_2$ one has a canonical (continuous) 
functor $$\on{D-mod}(\CY_1)\otimes \on{D-mod}(\CY_2)\to \on{D-mod}(\CY_1\times \CY_2)$$
and a functor of tensor product 
$$\on{D-mod}(\CY )\otimes \on{D-mod}(\CY )\to \on{D-mod}(\CY)$$
defined as the composition
$$\on{D-mod}(\CY)\otimes \on{D-mod}(\CY)\to \on{D-mod}(\CY\times \CY)\overset{\Delta^!_\CY}\longrightarrow \on{D-mod}(\CY).$$

\sssec{}  \label{sss:oblv stacks}

The natural transformation $\oblv_{\Dmod}$ of \secref{sss:corr and IndCoh} gives rise to a
continuous \emph{conservative} functor 
$$\oblv_{\Dmod(\CY)}:\Dmod(\CY)\to \IndCoh(\CY),$$
which is compatible with morphisms of prestacks under !-pullbacks.

\medskip

As in the case of DG schemes, we can interpret the functor $\oblv_{\Dmod}$ as pullback along the
tautological morphism of prestacks $\CY\to \CY_{\on{dR}}$.

\medskip

However, it is not clear, and most probably not true, that for a general prestack this functor
admits a left adjoint. Neither is it possible for a general prestack $\CY$ to 
consider the functor $'\oblv_{\Dmod(\CY)}$ (because the functor $\Psi_\CY$ is a feature
of DG schemes or algebraic stacks).

\medskip

In addition, the functor $\oblv^{\on{left}}_{\Dmod(-)}$ for schemes mentioned in 
\secref{sss:left realization} gives rise to a functor
$$\oblv^{\on{left}}_{\Dmod(\CY)}:\Dmod(\CY)\to \QCoh(\CY),$$
which is compatible with morphisms of prestacks under !-pullbacks on $\Dmod$
and usual *-pullbacks on $\QCoh$.

\sssec{Quasi-compact schematic morphisms}   \label{sss:representable}

Let $\pi:\CY_1\to \CY_2$ is a schematic and quasi-compact morphism between prestacks.
The functor of $(\dr)$-pushforward for DG schemes gives rise to a continuous functor 
$$\pi_{\on{dR},*}:\on{D-mod}(\CY_1)\to \on{D-mod}(\CY_2).$$
As in the case of the $(\IndCoh,*)$-pushforward, one constructs the functor $\pi_{\on{dR},*}$
as follows:

\medskip

For $(S_2,g_2)\in (\on{DGSch}_{\on{aft}})_{/\CY}$, we set
$$g_2^!(\pi_\dr(-)):=(\pi_S)_\dr\circ g_1^!(-)$$
for the morphisms in the Cartesian diagram
$$
\CD
S_1 @>{g_1}>>  \CY_1 \\
@V{\pi_S}VV    @VV{\pi}V    \\
S_2  @>{g_2}>>  \CY_2.
\endCD
$$

The data of compatibility of the assignment 
$$(S_2,g_2)\rightsquigarrow (\pi_S)_\dr\circ g_1^!(-)$$
under !-pullbacks for maps in $(\on{DGSch}_{\on{aft}})_{/\CY}$
is given by base change isomorphisms \eqref{e:Dmod base change}.

\medskip

Moreover, the formation $\pi_{\on{dR},*}$ is also endowed with base change isomorphisms
with respect to !-pullbacks for Cartesian squares of prestacks
$$
\CD
\CY'_1  @>>>  \CY_1 \\
@V{\pi'}VV   @VV{\pi}V  \\
\CY'_2  @>>>  \CY_2,
\endCD
$$
where the vertical maps are schematic and quasi-compact. 

\medskip

By construction, the projection formula for morphisms between quasi-compact schemes,
i.e., \eqref{e:Dmod proj formula}, implies one for $\pi$. That is, we have a functorial isomorphism
\begin{equation} \label{e:proj formula Dmod sch}
\CM_2\sotimes \pi_\dr(\CM_1)\simeq \pi_\dr(\pi^!(\CM_2)\sotimes \CM_1),\quad \CM_i\in \Dmod(\CY_i).
\end{equation}

\begin{rem}
We emphasize again that the isomorphisms neither in base change nor in projection
formula arise by adjunction from a priori existing maps.
\end{rem}

\sssec{}  \label{sss:smooth pullback stacks}

Let $\CY_i$ be prestacks, and let  $\pi:\CY_1\to \CY_2$ be a morphism which is $k$-representable for
some $k$. As in the case of $\IndCoh$, we will only need the cases of either $\pi$ being schematic, or 
$1$-representable. Assume also that $\pi$ is smooth.

\medskip

In this case we also have a naturally defined functor
$$\pi^*_{\on{dR}}:\Dmod(\CY_2)\to \Dmod(\CY_1).$$

By \eqref{e:*! shift}, if $\pi$ is of relative dimension $n$, we have a canonical
isomorphism 
$$\pi^*_{\on{dR}}\simeq \pi^![-2n].$$

For $\CM',\CM''\in \Dmod(\CY_2)$, from \eqref{e:tensor * and !} we obtain a canonical
isomorphism
\begin{equation} \label{e:* and ! ten}
\pi^*_{\on{dR}}(\CM'\sotimes \CM'')\to \pi^*_{\on{dR}}(\CM')\sotimes \pi^!(\CM''),
\end{equation}

Finally, assuming that $\pi$ is, in addition, schematic and quasi-compact, we obtain that the functors
$(\pi^*_{\on{dR}},\pi_\dr)$ are naturally adjoint. 

\ssec{D-modules on algebraic stacks}   \label{ss:Dmods on algebraic stacks}

From now until the end of this section we shall assume that $\CY$ is an algebraic stack.

\sssec{}    \label{sss:only_smooth}

As was mentioned above, in the formation of the limit \eqref{e:def_of_D(Y)}, one can replace the category 
$(\on{DGSch}_{\on{aft}})_{/\CY}$ by $\on{DGSch}_{/\CY,\on{smooth}}$. 

\medskip

I.e., the functor
\begin{equation}  \label{e:def_of_D(Y)!}
\Dmod(\CY)\simeq \underset{(S,g)\in ((\on{DGSch}_{\on{aft}})_{/\CY})^{\on{op}}}{\underset{\longleftarrow}{lim}}\, \on{D-mod}(S)\to
\underset{(S,g)\in (\on{DGSch}_{/\CY,\on{smooth}})^{\on{op}}}{\underset{\longleftarrow}{lim}}\, \on{D-mod}(S),
\end{equation}
obtained by restriction, is an equivalence. 

\sssec{}

Furthermore, using $(\on{DGSch}_{/\CY,\on{smooth}})^{\on{op}}$ as indexing category, $\Dmod(\CY)$ can be also
realized as the limit 
\begin{equation}   \label{e:def_of_D(Y)*}
\underset{(S,g)\in (\on{DGSch}_{/\CY,\on{smooth}})^{\on{op}}}{\underset{\longleftarrow}{lim}}\, \on{D-mod}(S),
\end{equation}
which is formed with $f^*_{\on{dR}}:\Dmod(S')\to \Dmod(S)$ as transition functors. 
(This follows from \secref{coh shift}.)

\medskip

In addition, choosing a smooth atlas $f:Z\to \CY$, we have:
\begin{equation} \label{e:Dmod via Cech}
\on{D-mod}(\CY)\simeq \on{Tot}\left(\on{D-mod}(Z^\bullet/\CY)\right),
\end{equation}
where the cosimplicial category is formed using either !-pullback or
$(\on{dR},*)$-pullback functors along the simplicial DG scheme $Z^\bullet/\CY$. 
(The assertion for !-pullbacks follows from the smooth descent property of
D-modules, and that for $(\on{dR},*)$-pullbacks from \secref{coh shift}.)

\sssec{}

For an algebraic stack $\CY$, the category $\on{D-mod}(\CY)$ has a (unique) t-structure such that 
$$\on{D-mod}(\CY)^{>0}=\{\CF\in \on{D-mod}(\CY)\;|\;\oblv_{\on{D-mod}(\CY)}(\CF)\in\IndCoh(\CY)^{>0}\}\, .$$
%$\on{D-mod}(\CY)^{>0}$ consists of all $F\in \on{D-mod}(\CY)$ such that $\oblv_{\on{D-mod}(\CY)}(F)\in\IndCoh(\CY)^{>0}$; 
The properties of this t-structure formulated in Sect.~\ref{Dt} for DG schemes imply similar properties for stacks.
In particular, the t-structure is left-complete and compatible with colimits. 

\ssec{The induction functor}

We are going to show that for algebraic stacks, the functor
$$\oblv_{\on{D-mod}(\CY)}:\on{D-mod}(\CY)\to \IndCoh(\CY)$$
admits a left adjoint, denoted $\ind_{\on{D-mod}(\CY)}$, and establish
some properties of this functor.

\medskip

For the main theorems of this paper we will only need the induction functor
in the case when $\CY$ is a classical (i.e., non-derived) algebraic stack.
However, for the sake of completeness, the discussion in this subsection
is applicable do derived stacks as well.

\begin{rem}
A more streamlined treatment will
be given in \cite{GR2}, where the functor of direct image on $\IndCoh$ will
be developed for morphisms such as $\CY\to \CY_{\on{dR}}$, where $\CY$ is an algebraic stack.
The latter functor \emph{is} the sought-for induction functor.
\end{rem}

\sssec{}

Let $S$ be a DG scheme equipped with a smooth map $g:S\to \CY$. Consider the category
$\Dmod(S)_{\on{rel}_{\CY}}$ of relative right D-modules. By definition,
$$\Dmod(S)_{\on{rel}_{\CY}}:=\IndCoh(S_{\on{dR}}\underset{\CY_{\on{dR}}}\times \CY).$$

\medskip

We have the forgetful functors
$$\oblv_{\on{D-mod}(S)_{\on{rel}_{\CY}}}: \Dmod(S)_{\on{rel}_{\CY}}\to \IndCoh(S)$$ and 
$$\oblv_{\on{D-mod}(S)_{\on{rel}\to \on{abs}}}: \Dmod(S)\to  \Dmod(S)_{\on{rel}_{\CY}}$$
defined as pullbacks along the tautological morphisms
$$S\to S_{\on{dR}}\underset{\CY_{\on{dR}}}\times \CY \text{ and } S_{\on{dR}}\underset{\CY_{\on{dR}}}\times \CY\to S_{\on{dR}},$$
respectively. We have:
$$\oblv_{\on{D-mod}(S)_{\on{rel}_{\CY}}}\circ \oblv_{\on{D-mod}(S)_{\on{rel}\to \on{abs}}}\simeq \oblv_{\on{D-mod}(S)}.$$

\medskip

It is easy to see that the functor $\oblv_{\on{D-mod}(S)_{\on{rel}_{\CY}}}$ is conservative. The category $\Dmod(S)_{\on{rel}_{\CY}}$ 
carries a t-structure characterized by the property that
$$\Dmod(S)_{\on{rel}_{\CY}}^{>0}=\{\CF\in \Dmod(S)_{\on{rel}_{\CY}}\,|\, \oblv_{\on{D-mod}(S)_{\on{rel}_{\CY}}} (\CF)\in \IndCoh(S)^{>0}\}.$$

\sssec{}

For a morphism $f:S'\to S$ in $\on{DGSch}_{/\CY,\on{smooth}}$, we have a naturally defined functor
$$f^!:\Dmod(S)_{\on{rel}_{\CY}}\to \Dmod(S')_{\on{rel}_{\CY}},$$ 
which makes the diagram
$$
\CD
\IndCoh(S')  @<{\oblv_{\on{D-mod}(S')_{\on{rel}_{\CY}}}}<<  
\Dmod(S')_{\on{rel}_{\CY}}   @<{\oblv_{\on{D-mod}(S')_{\on{rel}\to \on{abs}}}}<<  \Dmod(S')  \\
@A{f^!}AA   @AA{f^!}A  @AA{f^!}A  \\
\IndCoh(S)  @<{\oblv_{\on{D-mod}(S)_{\on{rel}_{\CY}}}}<<  \Dmod(S)_{\on{rel}_{\CY}}   
@<{\oblv_{\on{D-mod}(S)_{\on{rel}\to \on{abs}}}}<<  \Dmod(S).
\endCD
$$
commute.

\medskip

The assignment $S\rightsquigarrow \Dmod(S)_{\on{rel}_{\CY}}$ has a structure functor of $\infty$-categories:
$$(\Dmod^!_{\on{rel}_{\CY}})_{\on{DGSch}_{/\CY,\on{smooth}}}:
(\on{DGSch}_{/\CY,\on{smooth}})^{\on{op}}\to \StinftyCat_{\on{cont}},$$
which is equipped with natural transformations
$$\IndCoh^!_{\on{DGSch}_{/\CY,\on{smooth}}}\overset{\oblv_{\on{D-mod}_{\on{rel}_{\CY}}}}\longleftarrow
(\Dmod^!_{\on{rel}_{\CY}})_{\on{DGSch}_{/\CY,\on{smooth}}} \overset{\oblv_{\on{D-mod}_{\on{rel}\to \on{abs}}}}\longleftarrow
\Dmod^!_{\on{DGSch}_{/\CY,\on{smooth}}}$$

\sssec{}

For $(S,g)\in \on{DGSch}_{/\CY,\on{smooth}}$, pullback along the morphism
$$S_{\on{dR}}\underset{\CY_{\on{dR}}}\times \CY\to \CY$$
defines a functor 
$$\IndCoh(\CY) \to \Dmod(S)_{\on{rel}_{\CY}},$$
which by a slight abuse of notation we shall denote by $g^!$. I.e.,
$$g^!\simeq \oblv_{\on{D-mod}(S)_{\on{rel}_{\CY}}}\circ g^!:\IndCoh(\CY)\to \IndCoh(S).$$

Furthermore, we have a functor
\begin{equation} \label{e:IndCoh through relative}
\IndCoh(\CY)\to \underset{(S,g)\in (\on{DGSch}_{/\CY,\on{smooth}})^{\on{op}}}{\underset{\longleftarrow}{lim}}\, \Dmod(S)_{\on{rel}_{\CY}}.
\end{equation}

\begin{lem} \label{l:IndCoh through relative}
The functor \eqref{e:IndCoh through relative} is an equivalence.
\end{lem}

\begin{proof}
The inverse functor to \eqref{e:IndCoh through relative}
is given by the composition
\begin{multline*}
\underset{(S,g)\in (\on{DGSch}_{/\CY,\on{smooth}})^{\on{op}}}{\underset{\longleftarrow}{lim}}\, \Dmod(S)_{\on{rel}_{\CY}}
\overset{\oblv_{\on{D-mod}_{\on{rel}_{\CY}}}}\longrightarrow 
\underset{(S,g)\in (\on{DGSch}_{/\CY,\on{smooth}})^{\on{op}}}{\underset{\longleftarrow}{lim}}\, \IndCoh(S)\simeq \\
\simeq \IndCoh(\CY).
\end{multline*}

\end{proof}

\sssec{}   \label{sss:algebroid}

Assume for a moment that $\CY$ is a classical (i.e., non-derived) stack. Then for $(S,g)\in \on{DGSch}_{/\CY,\on{smooth}}$,
the DG scheme $S$ is also classical.

\medskip

In this case the category $\Dmod(S)_{\on{rel}_{\CY}}$
can be described as modules over the \emph{Lie algebroid} $T_{S/\CY}$ of vector fields on $S$ vertical with respect to
the map $g:S\to \CY$. The following is easy:

\begin{lem} \label{l:adj to rel D-mod cl}
The functor $\oblv_{\on{D-mod}(S)_{\on{rel}_{\CY}}}$ admits a left adjoint; we shall denote this left adjoint by 
$\ind_{\on{D-mod}(S)_{\on{rel}_{\CY}}}$. Furthermore, 

\smallskip

\noindent{\em(a)} The composion 
$$\oblv_{\on{D-mod}(S)_{\on{rel}_{\CY}}}\circ \ind_{\on{D-mod}(S)_{\on{rel}_{\CY}}}:\IndCoh(S)\to \IndCoh(S)$$
has a filtration indexed by non-negative integers with the $i$-th successive quotient isomorphic to the functor
$\on{Sym}^i(T_{S/\CY})\otimes -$. 

\smallskip

\noindent{\em(b)} The functors $\ind_{\on{D-mod}(S)_{\on{rel}_{\CY}}}$ and $\oblv_{\on{D-mod}(S)_{\on{rel}_{\CY}}}$
are both t-exact.

\smallskip

\noindent{\em(c)} Every object $\CF\in \Dmod(S)_{\on{rel}_{\CY}}$ admits a resolution (the relative de Rham complex)
by objects of the form
$$\ind_{\on{D-mod}(S)_{\on{rel}_{\CY}}}\left(\on{Sym}^k(T_{S/\CY}[1])\otimes \ind_{\on{D-mod}(S)_{\on{rel}_{\CY}}}(\CF)\right),$$
$k\in [0,\on{dim.rel.}(S/\CY)]$.

\smallskip

\noindent{\em(d)} For $\CF\in \IndCoh(S)$ and $\CM\in \Dmod(S)_{\on{rel}_{\CY}}$, the natural map
$$\ind_{\on{D-mod}(S)_{\on{rel}_{\CY}}}(\CF\sotimes \oblv_{\on{D-mod}(S)_{\on{rel}_{\CY}}}(\CM))\to
\ind_{\on{D-mod}(S)_{\on{rel}_{\CY}}}(\CF)\sotimes \CM$$
is an isomorphism.

\end{lem}

We obtain:

\begin{cor} \label{c:adj to rel abs D-mod cl}
The functor $\oblv_{\on{D-mod}(S)_{\on{rel}\to \on{abs}}}$ admits a left adjoint; we shall denote this left adjoint by 
$\ind_{\on{D-mod}(S)_{\on{rel}\to \on{abs}}}$.
Furthermore, 

\smallskip

\noindent{\em(a)} 
The functor $\ind_{\on{D-mod}(S)_{\on{rel}\to \on{abs}}}$ is right t-exact and has cohomological amplitude bounded by
$\on{dim.rel.}(S/\CY)$.

\smallskip

\noindent{\em(b)} For $\CF\in \Dmod(S)_{\on{rel}_{\CY}}$ and $\CM\in \Dmod(S)$, the natural map
$$\ind_{\on{D-mod}(S)_{\on{rel}\to \on{abs}}}(\CF\sotimes \oblv_{\on{D-mod}(S)_{\on{rel}\to \on{abs}}}(\CM))\to
\ind_{\on{D-mod}(S)_{\on{rel}\to \on{abs}}}(\CF)\sotimes \CM$$
is an isomorphism.

\end{cor}

\begin{proof}
By \lemref{l:adj to rel D-mod cl}, it is enough to consider (and prove the existence of) the functor
$\ind_{\on{D-mod}(S)_{\on{rel}\to \on{abs}}}$ applied to objects of the form
$$\ind_{\on{D-mod}(S)_{\on{rel}_{\CY}}}(\CF),\quad \CF\in \IndCoh(\CY).$$
However,
$$\ind_{\on{D-mod}(S)_{\on{rel}\to \on{abs}}}(\ind_{\on{D-mod}(S)_{\on{rel}_{\CY}}}(\CF))\simeq
\ind_{\on{D-mod}(S)}(\CF)$$
and the assertion follows.
\end{proof}

\begin{rem}
One can show that the cohomological amplitude of $\ind_{\on{D-mod}(S)_{\on{rel}\to \on{abs}}}$
is in fact bounded by 
$\underset{y\in \CY(k)}{\on{max}}(\dim(\on{Aut}(y)))$.
\end{rem}

\begin{rem}
The notion of Lie algebroid is familiar for classical schemes. Its analog in the framework of
derived algebraic geometry will appear in \cite{GR2}. Assuming this theory, the assertions of 
of \secref{sss:algebroid} are equally applicable when $\CY$ is a derived algebraic stack.
\end{rem}

\sssec{}

Let $\CY$ be again a derived algebraic stack. We have:

\begin{prop} \label{p:adj to rel D-mod} \hfill

\smallskip

\noindent{\em(a)}
The functor $\oblv_{\on{D-mod}(S)_{\on{rel}_{\CY}}}$ admits a left adjoint, denoted $\ind_{\on{D-mod}(S)_{\on{rel}_{\CY}}}$.
Both functors $\ind_{\on{D-mod}(S)_{\on{rel}_{\CY}}}$ and $\oblv_{\on{D-mod}(S)_{\on{rel}_{\CY}}}$
are t-exact.

\smallskip

\noindent{\em(a')} For $\CF\in \IndCoh(S)$ and $\CM\in \Dmod(S)_{\on{rel}_{\CY}}$, the natural map
$$\ind_{\on{D-mod}(S)_{\on{rel}_{\CY}}}(\CF\sotimes \oblv_{\on{D-mod}(S)_{\on{rel}_{\CY}}}(\CM))\to
\ind_{\on{D-mod}(S)_{\on{rel}_{\CY}}}(\CF)\sotimes \CM$$
is an isomorphism.

\smallskip

\noindent{\em(b)} The functor $\oblv_{\on{D-mod}(S)_{\on{rel}\to \on{abs}}}$ admits a left adjoint, 
$\ind_{\on{D-mod}(S)_{\on{rel}\to \on{abs}}}$. The functor $\ind_{\on{D-mod}(S)_{\on{rel}\to \on{abs}}}$ 
is right t-exact and is of cohomological amplitude bounded by $\on{dim.rel.}(S/\CY)$.

\smallskip

\noindent{\em(b')} For $\CF\in \Dmod(S)_{\on{rel}_{\CY}}$ and $\CM\in \Dmod(S)$, the natural map
$$\ind_{\on{D-mod}(S)_{\on{rel}\to \on{abs}}}(\CF\sotimes \oblv_{\on{D-mod}(S)_{\on{rel}\to \on{abs}}}(\CM))\to
\ind_{\on{D-mod}(S)_{\on{rel}\to \on{abs}}}(\CF)\sotimes \CM$$
is an isomorphism.

\end{prop}

\begin{proof}
Both assertions easily reduce to the case when $\CY$ is classical, and there they follow
from \lemref{l:adj to rel D-mod cl} and \corref{c:adj to rel abs D-mod cl}, respectively. 
\end{proof}

\begin{rem}
In terms of the formalism that willbe explained in \cite{GR2}, the functors 
$$\ind_{\on{D-mod}(S)_{\on{rel}_{\CY}}} \text{ and } \ind_{\on{D-mod}(S)_{\on{rel}\to \on{abs}}}$$
correspond to the direct image along the morphisms
$$S\to S_{\on{dR}}\underset{\CY_{\on{dR}}}\times \CY \text{ and }
S_{\on{dR}}\underset{\CY_{\on{dR}}}\times \CY \to S_{\on{dR}},$$
respectively.
\end{rem}

\sssec{}

Let $f:S'\to S$ be again a morphism in $\on{DGSch}_{/\CY,\on{smooth}})$. By adjunction, the diagram
\begin{equation} \label{e:rel abs ind}
\CD
\Dmod(S')_{\on{rel}_{\CY}}   @>{\ind_{\on{D-mod}(S')_{\on{rel}\to \on{abs}}}}>>  \Dmod(S')  \\
@A{f^!}AA   @AA{f^!}A  \\
\Dmod(S)_{\on{rel}_{\CY}}   @>{\ind_{\on{D-mod}(S)_{\on{rel}\to \on{abs}}}}>>  \Dmod(S)
\endCD
\end{equation}
commutes up to a natural transformation.

\begin{lem} \label{l:rel abs ind}
The diagram \eqref{e:rel abs ind} commutes (i.e., the natural transformation above is an isomorphism).
\end{lem}

\begin{proof}
The assertion reduces to the case when $\CY$ is classical, and there it follows from \secref{sss:algebroid}.
\end{proof}

Thus, we obtain that the functors
$(S,g)\mapsto \ind_{\on{D-mod}(S)_{\on{rel}\to \on{abs}}}$ give rise to a natural transformation
$$(\Dmod^!_{\on{rel}_{\CY}})_{\on{DGSch}_{/\CY,\on{smooth}}}
\overset{\ind_{\on{D-mod}_{\on{rel}\to \on{abs}}}}\longrightarrow 
\Dmod^!_{\on{DGSch}_{/\CY,\on{smooth}}}.$$ 

In particular, the assignment
$$(\CF\in \IndCoh(\CY))\mapsto  \ind_{\on{D-mod}(S)_{\on{rel}\to \on{abs}}}(g^!(\CF))$$
defines a functor
$$\IndCoh(\CY)\to \underset{(S,g)\in (\on{DGSch}_{/\CY,\on{smooth}})^{\on{op}}}{\underset{\longleftarrow}{lim}}\, \Dmod(S)=\Dmod(\CY).$$
We denote this functor by $\ind_{\Dmod(\CY)}$. 

\begin{lem}  \label{l:constr of induction}
The functor $\ind_{\Dmod(\CY)}$ is the left adjoint of $\oblv_{\Dmod(\CY)}$. 
\end{lem}

\begin{proof}
This follows from \lemref{l:IndCoh through relative}.
\end{proof}

\sssec{}  \label{sss:Dmod and QCoh on stacks}

Being a left adjoint of a left t-exact functor, the functor $\ind_{\Dmod(\CY)}$ is right t-exact.

\medskip

However, unlike the case of DG schemes, $\ind_{\Dmod(\CY)}$ is no longer t-exact, even when
$\CY$ is a smooth classical stack. Rather, we have the following:

\begin{lem} \label{l:ind finite c.d.}
Assume that $\CY$ is quasi-compact. Then the functor $\ind_{\Dmod(\CY)}$ is of finite
cohomological amplitude.
\end{lem}

\begin{proof}
Follows from \propref{p:adj to rel D-mod}(b).
\end{proof}

\sssec{}

In the sequel we shall use the following property of the functor $\ind_{\Dmod(\CY)}$.

\medskip

First, as in the case of DG schemes, for $\CF\in \IndCoh(\CY)$ and $\CM\in \Dmod(\CY)$, by
adjunction we obtain a map
\begin{equation} \label{e:tensor with induction stacks}
\ind_{\on{D-mod}(\CY)}\left(\CF\overset{!}\otimes \oblv_{\on{D-mod}(\CY)}(\CM)\right)\to
\ind_{\on{D-mod}(\CY)}(\CF)\overset{!}\otimes \CM.
\end{equation}

\begin{lem} \label{l:tensor with induction stacks}
The map \eqref{e:tensor with induction stacks} is an isomorphism.
\end{lem}

\begin{proof}
\lemref{l:IndCoh through relative} reduces the assertion to that of 
\propref{p:adj to rel D-mod}(b'). 
\end{proof}

\sssec{}  \label{sss:left induction stacks}

As in the case of DG schemes, it follows that $\ind_{\Dmod(\CY)}$ canonically factors
through a functor
$$'\ind_{\Dmod(\CY)}:\QCoh(\CY)\to \Dmod(\CY).$$
If $\CY$ is eventually coconnective, it is a left adjoint of the functor
$$'\oblv_{\Dmod(\CY)}:\Dmod(\CY)\to \QCoh(\CY),$$
while the latter is conservative for any algebraic stack.

\medskip

In addition, we have the functors
$$\ind^{\on{left}}_{\Dmod(\CY)}:\QCoh(\CY)\rightleftarrows \Dmod(\CY):\oblv^{\on{left}}_{\Dmod(\CY)}$$
that are mutually adjoint when $\CY$ is eventually coconnective.

\medskip

Similarly to formulas \eqref{e:two_obls-schemes}-\eqref{e:two_inds-schemes}, for any $\CY$, one has
\begin{equation}  \label{e:two_obls-stacks}
'\oblv_{\Dmod(\CY)}(\CM)\simeq \oblv^{\on{left}}_{\Dmod(\CY)}(\CM)\otimes \Psi_\CY(\omega_\CY), \quad \CM\in \Dmod(\CY)
\end{equation}
and 
\begin{equation}   \label{e:two_inds-stacks}
\ind^{\on{left}}_{\Dmod(\CY)}(\CF)\simeq {}'\ind_{\Dmod(\CY)}(\CF\otimes \Psi_\CY(\omega_\CY)), \quad \CF\in\QCoh(\CY),
\end{equation}
where $\Psi_\CY$ is as in \secref{sss:Psi for stacks}.

\ssec{Example: induction for the classifying stack}  \label{ss:ind BG}

In this subsection we shall consider the case of $\CY=BG$, where $G$ is an algebraic group,
and 
$$BG:=\on{pt}/G$$ is its classifying stack. We shall describe the pair of adjoint functors
$$\ind_{\Dmod(BG)}:\IndCoh(BG)\rightleftarrows \Dmod(BG):\oblv_{\Dmod(BG)}$$
more explicitly.

\sssec{}

Take $S=\on{pt}$, and let $\sigma$ denote the tautological map $\on{pt}\to BG$.  We have a commutative diagram
\begin{equation} \label{e:BG diag 1}
\CD
\on{D-mod}(\on{pt})_{\on{rel}_{BG}}   @<{\oblv_{\on{D-mod}(\on{pt})_{\on{rel}\to \on{abs}}}}<<    \Dmod(\on{pt})  \\
@A{\sigma^!}AA     @AA{\sigma^!}A   \\
\IndCoh(BG)  @<{\oblv_{\Dmod(BG)}}<<  \Dmod(BG),
\endCD
\end{equation}
and according to \lemref{l:rel abs ind}, the diagram 
\begin{equation} \label{e:BG diag 2}
\CD
\on{D-mod}(\on{pt})_{\on{rel}_{BG}}   @>{\ind_{\on{D-mod}(\on{pt})_{\on{rel}\to \on{abs}}}}>>    \Dmod(\on{pt})   \\
@A{\sigma^!}AA     @AA{\sigma^!}A   \\
\IndCoh(BG)  @>{\ind_{\Dmod(BG)}}>>  \Dmod(BG),
\endCD
\end{equation}
obtained by taking the left adjoints of the horizontal arrows, is also commutative.

\sssec{}

By \secref{sss:algebroid}, we have
$$\on{D-mod}(\on{pt})_{\on{rel}_{BG}}\simeq \fg\mod,$$
where $\fg:=\on{Lie}(G)$, and $\fg\mod$ denotes the DG category of $\fg$-modules. 

\medskip

The functor 
$$\oblv_{\on{D-mod}(\on{pt})_{\on{rel}\to \on{abs}}}:\Dmod(\on{pt}) \to \on{D-mod}(\on{pt})_{\on{rel}_{BG}}$$
is the functor
$$\on{triv}:\Vect\to \fg\mod$$
that sends a vector space to the $\fg$-module with the trivial action.

\medskip

Its left adjoint
$$\ind_{\on{D-mod}(\on{pt})_{\on{rel}\to \on{abs}}}:\on{D-mod}(\on{pt})_{\on{rel}_{BG}}\to \Dmod(\on{pt})$$
is the functor
$$\on{conv}_\fg:\fg\mod\to \Vect$$
of $\fg$-coinvariants. 

\sssec{}

Note that since $BG$ is smooth, the functor
$$\Psi_{BG}:\IndCoh(BG)\to \QCoh(BG)$$
is an equivalence. We set by definition
$$\Rep(G):=\QCoh(BG).$$

\sssec{}

Let us now assume that $G$ is affine, for the duration of this subsection. 

\medskip

Let $BG^\bullet$ be the \v{C}ech nerve of the map $\sigma:\on{pt}\to BG$. 
The description of $\QCoh(BG)$ as $\on{Tot}(\QCoh(BG^\bullet))$ implies:

\begin{cor}  \label{c:Rep G}
The symmetric monoidal category $\Rep(G)$ identifies with $R_G\on{-comod}$, where
$R_G$ is the regular representation of $G$, viewed as a cocommutative Hopf algebra.
\end{cor}

We claim that $\Rep(G)$ is in fact  ``the usual" DG category corresponding to the derived category
of representations of $G$. Indeed, according to Remark \ref{r:left-completion}, we have a canonical functor
\begin{equation} \label{e:Rep G}
D(\Rep(G)^\heartsuit)\to \Rep(G),
\end{equation}
which identifies $\Rep(G)$ with the left-completion of $D(\Rep(G)^\heartsuit)$.

\medskip

However, we have:

\begin{lem}
The functor \eqref{e:Rep G} is an equivalence.
\end{lem}

\begin{proof}
Follows from the fact that $D(\Rep(G)^\heartsuit)$ is of finite cohomological dimension.
\end{proof}

\sssec{}

Thus, we obtain that the commutative diagrams \eqref{e:BG diag 1} and  \eqref{e:BG diag 2} identify with
$$
\CD
\fg\mod @<{\on{triv}_\fg}<< \Vect \\
@AAA    @AAA  \\
\Rep(G)   @<<<  \Dmod(BG)
\endCD
$$
and
$$
\CD
\fg\mod @>{\on{coinv}_\fg}>> \Vect \\
@AAA    @AAA  \\
\Rep(G)   @>>>  \Dmod(BG),
\endCD
$$
respectively. In both diagrams, the left vertical arrow is the functor
$$\Rep(G) \to \fg\mod,$$
$$V\mapsto \on{Res}^G_\fg\left(V\otimes \det(\fg^\vee)[\dim(G)]\right),$$
where $\on{Res}^G_\fg$ is the usual restriction functor
$$\Rep(G)\to \fg\mod.$$

\sssec{}

In particular, we obtain that the functor 
$$\ind_{\Dmod(BG)}:\IndCoh(BG)\to \Dmod(BG),$$
composed with $\sigma^!$, identifies with the functor $\Rep(G)\to \Vect$ given by
\begin{equation} \label{e:coinv of G rep}
V\mapsto \on{coinv}_\fg\left(\on{Res}^G_\fg(V\otimes \det(\fg^\vee)[\dim(G)])\right).
\end{equation}

\ssec{Additional properties of the induction functor}

The goal of this subsection is to prove \propref{p:DR of induced stack}, which is needed 
for the proof of \propref{p:de Rham and ind}. As the contents of this section will
not be needed elsewhere in the paper, the reader may skip it on the first pass.

\sssec{}

Let $\pi:\wt\CY\to \CY$ be a schematic and quasi-compact
map of algebraic stacks. For $(S,g)\in \on{DGSch}_{/\CY,\on{smooth}}$ set
$$\wt{S}:=\wt\CY\underset{\CY}\times S.$$
Let $\pi_S$ and $\wt{g}$ denote the maps in the diagram
$$
\CD
\wt{S} @>{\pi_S}>>  S  \\
@V{\wt{g}}VV    @VV{g}V   \\
\wt\CY  @>{\pi}>>  \CY.
\endCD
$$

In this case we have a naturally defined functor
$$(\pi_S)^{\IndCoh}_*:\Dmod(\wt{S})_{\on{rel}_{\wt\CY}}\to \Dmod(S)_{\on{rel}_{\CY}}$$
that makes the following diagram commute:
$$
\CD
\Dmod(\wt{S})_{\on{rel}_{\wt\CY}}   @>{(\pi_S)^{\IndCoh}_*}>> \Dmod(S)_{\on{rel}_{\CY}}  \\
@V{\oblv_{\on{D-mod}(\wt{S})_{\on{rel}_{\wt\CY}}}}VV     @VV{\oblv_{\on{D-mod}(S)_{\on{rel}_{\CY}}}}V  \\
\IndCoh(\wt{S})    @>{\pi_S^{\IndCoh,*}}>>  \IndCoh(S).
\endCD
$$

The following isomorphism generalizes the base-change isomorphism for $\IndCoh$:

\begin{lem}  \label{l:base change rel}
There exists a canonical isomorphism 
$$g^!\circ \pi^{\IndCoh}_*\simeq (\pi_S)^{\IndCoh}_* \circ \wt{g}^!$$
as functors $\IndCoh(\wt\CY)\to \Dmod(S)_{\on{rel}_{\CY}}$.
\end{lem}

\sssec{}

In addition, we have the following assertion that generalizes the commutative diagram
\eqref{e:DR of induced}:

\begin{lem} \label{l:DR of induced rel}
The diagram
$$
\CD
\Dmod(\wt{S})    @>{(\pi_S)_\dr}>>  \Dmod(S)  \\
@A{\ind_{\on{D-mod}(\wt{S})_{\on{rel}\to \on{abs}}}}AA     @AA{\ind_{\on{D-mod}(S)_{\on{rel}\to \on{abs}}}}A  \\
\Dmod(\wt{S})_{\on{rel}_{\wt\CY}}   @>{(\pi_S)^{\IndCoh}_*}>> \Dmod(S)_{\on{rel}_{\CY}}
\endCD
$$
canonically commutes. 
\end{lem}

\begin{rem}
In terms of the formalism that will be explained in \cite{GR2}, the commutativity of the diagram
in \lemref{l:DR of induced rel} follows by taking direct images on $\IndCoh$ along the morphisms
in the following diagram
$$
\CD
\wt{S}_{\on{dR}}\underset{\wt\CY_{\on{dR}}}\times \wt\CY   @>>>  S_{\on{dR}}\underset{\CY_{\on{dR}}}\times \CY   \\
@VVV  @VVV     \\
\wt{S}_{\on{dR}}  @>>>   S_{\on{dR}}.
\endCD
$$
\end{rem}

\sssec{}

We now claim:

\begin{prop} \label{p:DR of induced stack}
For the map $\pi:\wt\CY\to \CY$ as above, the following diagram of functors canonically commutes:
$$
\CD
\Dmod(\wt\CY)  @>{\pi_\dr}>>  \Dmod(\CY) \\
@A{\ind_{\Dmod(\wt\CY)}}AA    @AA{\ind_{\Dmod(\CY)}}A   \\
\IndCoh(\wt\CY)    @>{\pi^{\IndCoh}_*}>>  \IndCoh(\CY).
\endCD
$$
\end{prop}

\begin{proof}

By \lemref{l:IndCoh through relative}, we need to show that for every $(S,g)\in \on{DGSch}_{/\CY,\on{smooth}}$ 
the functors 
\begin{equation} \label{e:DR of induced stack}
g^!\circ \ind_{\Dmod(\CY)}\circ \pi^{\IndCoh}_* \text{ and } g^!\circ \pi_\dr\circ \ind_{\Dmod(\wt\CY)},\quad
\IndCoh(\wt\CY)\to \Dmod(S)
\end{equation}
are canonically isomorphic.

\medskip

We rewrite the left-hand side in \eqref{e:DR of induced stack} as
\begin{multline*}
\ind_{\on{D-mod}(S)_{\on{rel}\to \on{abs}}}\circ g^!\circ \pi^{\IndCoh}_*  \overset{\text{\lemref{l:base change rel}}}\simeq \\
\simeq \ind_{\on{D-mod}(S)_{\on{rel}\to \on{abs}}} \circ
(\pi_S)^{\IndCoh}_* \circ \wt{g}^! \overset{\text{\lemref{l:DR of induced rel}}}\simeq \\
\simeq
(\pi_S)_\dr\circ \ind_{\on{D-mod}(\wt{S})_{\on{rel}\to \on{abs}}}\circ \wt{g}^! \simeq \\
\simeq (\pi_S)_\dr\circ \wt{g}^!\circ \ind_{\Dmod(\wt\CY)}\simeq 
g^!\circ \pi_\dr\circ \ind_{\Dmod(\wt\CY)},
\end{multline*}
as required.
\end{proof}

\begin{rem}
In terms of the formalism of \cite{GR2}, the assertion of \propref{p:DR of induced stack}
follows by taking direct images along the commutative diagram
$$
\CD
\wt\CY  @>>>  \CY  \\
@VVV   @VVV   \\
\wt\CY_{\on{dR}}  @>>>  \CY_{\on{dR}}.
\endCD
$$
\end{rem}

\section{De Rham cohomology on an algebraic stack}  \label{s:deRham stack}

\ssec{Definition of De Rham cohomology}   

\sssec{}  \label{sss:DeRham_stacks-def}

The presentation of $\Dmod(\CY)$ as in \eqref{e:def_of_D(Y)*} and \secref{sss:pullback constant}
imply that there exists a canonically defined object
$$k_\CY\in \on{D-mod}(\CY),$$ such that
for every smooth morphism $g:S\to \CY$ with $S$ being a DG scheme one has $$g^*_{\on{dR}}(k_{\CY})= k_Z.$$

\medskip

We define the \emph{not necessarily continuous} functor $\Gamma_{\on{dR}}(\CY,-):\on{D-mod}(\CY)\to \Vect$
as 
$$\Gamma_{\on{dR}}(\CY,\CM):=\CMaps_{\Dmod(\CY)}(k_\CY,-).$$

\sssec{}

By \eqref{e:def_of_D(Y)*}, the functor $\Gamma_{\on{dR}}(\CY,-)$ can be calculated as follows: for
$\CM\in \on{D-mod}(\CY)$, we have:
\begin{equation}   \label{e:DeRham}
\Gamma_{\on{dR}}(\CY,\CM)\simeq \underset{(S,g)\in (\on{DGSch}_{/\CY,\on{smooth}})^{\on{op}}}
{\underset{\longleftarrow}{lim}}\, \Gamma_{\on{dR}}(S,g^*_{\on{dR}}(\CM)).
\end{equation}

\medskip

More economically, for a given smooth atlas $f:Z\to \CY$, by \eqref{e:Dmod via Cech} we have:
$$\Gamma_{\on{dR}}(\CY,\CM)\simeq \on{Tot}\left(\Gamma_{\on{dR}}(Z^\bullet/\CY,\CM|_{Z^\bullet/\CY})\right),$$
where $\CM|_{Z^\bullet/\CY}$ again denotes the $(\on{dR},*)$-pullback.

\sssec{Warning} \label{sss:warning}
Even if $\CY$ is quasi-compact and moreover, even if $\CY$ is QCA, the
functor $\Gamma_{\on{dR}}(\CY,-)$ is \emph{not necessarily continuous} 
(see the Examples in \secref{sss:BG_m} and \secref{ss:BG} below), which means that the object $k_{\CY}\in D(\CY)$ is 
\emph{not necessarily compact.} 

\medskip

See also Corollary~\ref{c:more precise} and Definition~\ref{d:safe}
below for a characterization of those quasi-compact stacks $\CY$ for which the functor 
$\Gamma_{\on{dR}}(\CY,-)$ is continuous. 

\sssec{Example}   \label{sss:BG_m}

Let $\CY:=B\BG_m$. Let us show that the functor $\Gamma_{\on{dR}}(\CY,-)$ is not continuous.
Let $A$ denote the graded algebra formed by 
$\Ext^i (k_{\CY},k_{\CY})=H^i(\Gamma_{\on{dR}}(\CY,k_{\CY}))$.

\medskip

It is easy to see that $A=k[u]$, where $\deg u=2$. The diagram
\begin{equation} \label{e:kkkkk}
k_{\CY}{\buildrel{u}\over{\longrightarrow}} k_{\CY}[2]{\buildrel{u}\over{\longrightarrow}}
k_{\CY}[4]{\buildrel{u}\over{\longrightarrow}} \ldots
\end{equation}
has a zero colimit: the pullback functor under $\on{pt}\to B\BG_m$ is conservative and continuous,
and the pullback of \eqref{e:kkkkk} to $\on{pt}$ consists of zero maps.

\medskip

However, when we apply the functor $\Gamma_{\on{dR}}(\CY,-)$ to
\eqref{e:kkkkk} we obtain the diagram
\[
A{\buildrel{u}\over{\longrightarrow}} A[2]{\buildrel{u}\over{\longrightarrow}}
A[4]{\buildrel{u}\over{\longrightarrow}} \ldots 
\]
whose colimit is nonzero.

\sssec{}

The following key calculation will be performed in \secref{ss:proof of de Rham and ind}: 

\begin{prop} \label{p:de Rham and ind}
For $\CF\in \IndCoh(\CY)$ there exists a canonical
isomorphism
$$\Gamma_{\on{dR}}(\CY,\ind_{\Dmod(\CY)}(\CF))\simeq \Gamma^{\IndCoh}(\CY,\CF).$$
\end{prop}

\ssec{Example: classifying stacks}  \label{ss:BG}

In this subsection we shall analyze the example of $\CY=BG$, where $G$ is a connected algebraic group. 
In particuar, we will show that if $G$ in non-unipotent, then $\Gamma_{\on{dR}}(BG,-)$ is not continuous. 

\sssec{}

Assume first that $G$ is affine. Then ``non-unipotent'' means that $G$ contains a copy of $\BG_m$.   
The morphism $$\pi:B\BG_m\to BG$$ is schematic and quasi-compact, so the 
functor $\pi_\dr$ is continuous. We have
$$\Gamma_{\on{dR}}(B\BG_m,-)\simeq \Gamma_{\on{dR}}(BG,-)\circ \pi_\dr.$$

In particular, the Example in \secref{sss:BG_m} implies that the functor $\Gamma_{\on{dR}}(BG,-)$
is non-continuous. 

\sssec{}

For a general connected $G$, let us describe the category $\Dmod(BG)$ explicitly. Recall (see \secref{ss:ind BG})
that $\sigma$ denotes the morphism 
$$\on{pt}\to BG.$$
The functor $\sigma^!$ admits a left adjoint, denoted $\sigma_!$. 

\medskip

Since $\sigma^!$ is conservative,
and both functors are continuous, the Barr-Beck-Lurie theorem (see e.g. \cite[Sect. 3.1.2]{DG}),
implies that the category $\Dmod(BG)$ identifies with that of modules for the monad $\sigma^!\circ \sigma_!$
acting on $\Dmod(\on{pt})=\Vect$. 

\medskip

The above monad identifies with the associative algebra in $\Vect$
$$B:=\left(\CMaps_{\Dmod(BG)}(\sigma_!(k),\sigma_!(k))\right)^{\on{op}}.$$
Hence, we obtain an equivalence of categories
\begin{equation} \label{e:equiv der category}
\Dmod(BG)\simeq B\mod,
\end{equation}
where $B\in B\mod$ corresponds to the object $\sigma_!(k)\in \Dmod(BG)$,
which is a compact generator of this category.

\sssec{}

By Verdier duality
\begin{equation} \label{e:homology of group}
B\simeq \left(\Gamma_{\on{dR}}(G,k_G)\right)^\vee,
\end{equation}
where the algebra structure on the right-hand side is given by the product operation $G\times G\to G$.
It is well-known that, unless $G$ is unipotent, $B$ is isomorphic to the exterior algebra on generators in 
degrees $-(2m_i-1)$, $m_i\in\BZ^{> 0}$, where $i$ runs through some finite set.  

\medskip

The presentation of $B$ given by \eqref{e:homology of group} shows that the structure of associative
algebra on $B$ canonically upgrades to that of co-commutative Hopf algebra. In particular, $B$ is
augmented. 

\medskip

The augmentation module $$k\in B\mod$$ corresponds to the object $k_{BG}\in \Dmod(BG)$. 
In terms of \eqref{e:homology of group}, the augmentation
corresponds to the map $p_G:G\to \on{pt}$. 

\sssec{}

We obtain that the algebra 
\[A:=\CMaps_{\Dmod(BG)}(k_{BG},k_{BG})\]
is canonically isomorphic to the \emph{Koszul dual} of $B$, i.e.,
\[A\simeq \CMaps_{B\mod}(k,k).\]
Explicitly, $A$ is a polynomial algebra on generators in degrees $2m_1,...,2m_r$ for $m_i$
as above. \footnote{Over $\BC$, the latter observation reproduces a well-known fact about
the cohomology of the classifying space.}

\medskip

In particular, this shows that $k$ is not a compact object in $B\mod$ (otherwise $A$ would
have been finite-dimensional).

\medskip 

The functor $\Gamma_{\on{dR}}(BG,-)$ is given, in terms of \eqref{e:equiv der category},
by $$M\mapsto \CMaps_B(k,M),$$
so it is not continuous.

\sssec{}

Let us assume once again that $G$ is affine, and compare the above description of
the category $\Dmod(BG)$ with \secref{ss:ind BG}.

\medskip

We obtain a pair of commutative diagrams
\begin{equation} \label{e:oblv for BG}
\CD
\fg\mod  @<{\on{triv}_\fg}<<   \Vect \\
@AAA  @AAA  \\
\Rep(G)   @<{\oblv_{\Dmod(BG)}}<<  B\mod,
\endCD
\end{equation}
and  
\begin{equation} \label{e:ind for BG}
\CD
\fg\mod  @>{\on{coinv}_\fg}>>   \Vect \\
@AAA  @AAA  \\
\Rep(G)   @>{\ind_{\Dmod(BG)}}>>  B\mod.
\endCD
\end{equation}

Note that there is a natural forgetful functor 
\begin{equation} \label{e:functions and homology functor}
B\mod\to \Rep(G)
\end{equation}
corresponding to the homomorphism of Hopf algebras
\begin{equation} \label{e:functions and homology}
B^\vee=\Gamma_{\on{dR}}(G,k_G)\simeq \CMaps_{\Dmod(G)}(\omega_G,\omega_G)\to 
\CMaps_{\IndCoh(G)}(\omega_G,\omega_G)\simeq R_G.
\end{equation}

It is easy to see that the functor
$$\oblv_{\Dmod(BG)}:B\mod\to \Rep(G)$$ 
in \eqref{e:oblv for BG} equals the composition of the functor
\eqref{e:functions and homology functor}, followed by the functor
$$V\mapsto V\otimes \det(\fg)[-\dim(G)]:\Rep(G)\to \Rep(G).$$

\medskip

In particular, we obtain that the functor 
$$V\mapsto \on{coinv}_\fg(\on{Res}^G_\fg(V)):\Rep(G)\to \Vect$$
canonically factors as
$$\Rep(G)\to B\mod \overset{\oblv_B}\longrightarrow \Vect,$$
where the functor $\Rep(G)\to B\mod$ is the left adjoint to the functor in 
\eqref{e:functions and homology functor}.

\begin{rem}

Let $BG_{\on{dR}}^\bullet$ denote the simplicial object of $\on{PreStk}$ obtained
by applying the functor $\CY\mapsto \CY_{\on{dR}}$ to $BG^\bullet$. Equivalently, 
$BG_{\on{dR}}^\bullet$ is the \v{C}ech nerve of the map $\on{pt}\to (BG)_{\on{dR}}$.

\medskip

We have:
$$\on{Tot}(\QCoh(BG_{\on{dR}}^\bullet))\simeq \on{Tot}(\IndCoh(BG_{\on{dR}}^\bullet))\simeq 
\on{Tot}(\Dmod(BG^\bullet))\simeq \Dmod(BG),$$
where the latter isomorphism is given by \eqref{e:Dmod via Cech}.

\medskip 

Set by definition 
$$\Rep(G_{\on{dR}}):=\on{Tot}(\QCoh(BG_{\on{dR}}^\bullet)).$$
We can informally interpret the resulting adjunction
$\Rep(G)\rightleftarrows \Rep(G_{\on{dR}})$
as coming from the short exact sequence 
$$1\to \fg\to G\to G_{\on{dR}}\to 1.$$
The latter will be made precise in \cite{GR2} by considering the formal completion 
$G^\wedge$ at $1\in G$, and showing that $G_{\on{dR}}\simeq G/G^\wedge$
and that proving that
$$\fg\mod\simeq \Rep(G^\wedge):=\on{Tot}(\QCoh((G^\wedge)^\bullet)).$$

\end{rem}

\ssec{Coherence and compactness on algebraic stacks}   \label{ss:coherence&compactness}

\sssec{}

Let $$\on{D-mod}_{\on{coh}}(\CY)\subset \on{D-mod}(\CY)$$ be the full 
subcategory consisting of objects $\CM\in \on{D-mod}(\CY)$ such that
$g^!(\CM)\in \on{D-mod}_{\on{coh}}(S)$ for any smooth map $g:S\to \CY$, where $S$
is a DG scheme. (Of course, $\on{D-mod}_{\on{coh}}(\CY)$ is not cocomplete.)

\medskip

It is easy to see that coherence condition is equivalent to requiring that
$g^*_{\on{dR}}(\CM)\in \Dmod(S)$ belong to $\on{D-mod}_{\on{coh}}(S)$
for any smooth map $g:S\to \CY$, where $S$ is a DG scheme. 
(Indeed, for smooth maps, $g^*_{\on{dR}}$ and $g^!$ differ by a cohomological
shift on each connected component of $S$.)

\medskip

It is also clear, that in either definition it suffices to consider those $S$ that
are quasi-compact, or even affine.

\medskip

Finally, it is enough to require either of the above conditions for just one smooth
atlas $f:Z\to \CY$.

\sssec{}

The object $k_\CY\in \on{D-mod}(\CY)$ is always in $\on{D-mod}_{\on{coh}}(\CY)$. On the other hand,
even if $\CY$ is quasi-compact it may happen that $k_\CY$
is not compact (see Sect.\ref{sss:warning}). 

\medskip

{\it So it is not true that $\on{D-mod}(\CY)^c$ equals $\on{D-mod}_{\on{coh}}(\CY)$ for any quasi-compact stack.}
\footnote{According to Corollary~\ref{c:safe_stacks} below, $\on{D-mod}(\CY)^c=\on{D-mod}_{\on{coh}}(\CY)$ for 
those quasi-compact stacks that are \emph{safe} in the sense of  Definition~\ref{d:safe}.}
However, we have:

\begin{lem}  \label{l:compact is coherent}
For any algebraic stack $\CY$ one has the inclusion
%\begin{lem}
%%It is clear that $$
\begin{equation}  \label{e:compacts are coherent}
\on{D-mod}(\CY)^c\subset \on{D-mod}_{\on{coh}}(\CY)\, .
\end{equation}
\end{lem}

\begin{proof}
The proof repeats verbatim that of \propref{coh is compact}(a):

\medskip

We need to show that if $\CM\in \on{D-mod}(\CY)^c$ then for any smooth map $g:S\to \CY$ with $S$ being a 
quasi-compact (or even affine) DG scheme, one has $g^*_{\on{dR}}(\CM)\in \on{D-mod}_{\on{coh}}(S)=\on{D-mod}(S)^c$.

\medskip

However, this is clear since $g^*_{\on{dR}}$
admits a right adjoint that commutes with colimits, namely $g_{\on{dR},*}$ (see 
Sect.~\ref{sss:representable}).
\end{proof}

\sssec{Verdier duality on algebraic stacks} \label{sss:Verdier_on_stacks}  

Let us observe that there exists a canonical involutive anti self-equivalence
\begin{equation}    \label{e:Verdier_on_stacks}
\BD_\CY^{\on{Verdier}}:(\on{D-mod}_{\on{coh}}(\CY))^{\on{op}}\to \on{D-mod}_{\on{coh}}(\CY)
\end{equation}
(called Verdier duality) such that for any 
%\medskip
%Indeed, $\BD_\CY^{\on{Verdier}}$ is uniquely fixed by the requirement that for a 
smooth map $g:S\to \CY$ from a scheme, we have:
$$g^!\circ \BD^{\on{Verdier}}_\CY\simeq \BD_S^{\on{Verdier}}\circ g^*_{\on{dR}}.$$
In other words, to define \eqref{e:Verdier_on_stacks} we use two different realizations
of $\on{D-mod}_{\on{coh}}(\CY)$ as  a limit: the one of \eqref{e:def_of_D(Y)} for the first copy of $\on{D-mod}_{\on{coh}}(\CY)$,
and the one of \eqref{e:def_of_D(Y)*} for the second one.

\begin{lem}  \label{l:Verdier_stacks}
For any $\CM\in \on{D-mod}_{\on{coh}}(\CY)$ and $\CM'\in \on{D-mod}(\CY)$ one has a canonical isomorphism
$$\CMaps_{\on{D-mod}(\CY)}(\BD_\CY^{\on{Verdier}}(\CM),\CM')\simeq\Gamma_{\on{dR}}(\CY,\CM\sotimes \CM').$$
\end{lem}

\begin{proof}

The two sides are calculated as limits over $(S,g)\in (\on{DGSch}_{/\CY,\on{smooth}})^{\on{op}}$
of
$$\CMaps_{\on{D-mod}(S)}\left(g^!(\BD_\CY^{\on{Verdier}}(\CM)),g^!(\CM')\right) \text{ and }
\Gamma_{\on{dR}}\left(S,g_{\on{dR}}^*(\CM\sotimes \CM')\right),$$
respectively. By \eqref{e:Verdier_schemes}, we have
\begin{multline*}
\CMaps_{\on{D-mod}(S)}\left(g^!(\BD_\CY^{\on{Verdier}}(\CM)),g^!(\CM')\right)\simeq
\CMaps_{\on{D-mod}(S)}\left(\BD^{\on{Verdier}}_S(g^*_{\on{dR}}(\CM)),g^!(\CM')\right)\simeq \\
\simeq \Gamma_{\on{dR}}\left(S,g^*_{\on{dR}}(\CM)\sotimes g^!(\CM')\right),
\end{multline*}
so the required isomorphism follows from \eqref{e:* and ! ten}. 
\end{proof}

Combining \lemref{l:Verdier_stacks}, \propref{p:Serre duality}, \lemref{l:tensor with induction stacks} and \propref{p:de Rham and ind},
we obtain: 

\begin{cor}  \label{c:Verdier-Serre_on_stacks}
If $\CF\in \Coh(\CY )$ then $\ind_{\on{D-mod}(\CY)}(\CF)\in \on{D-mod}_{\on{coh}}(\CY )$, and we have:
\begin{equation}   \label{e:Verdier-Serre_on_stacks}
\BD_{\CY}^{\on{Verdier}}\left(\ind_{\on{D-mod}({\CY})}(\CF)\right) \simeq \ind_{\on{D-mod}({\CY})}\left(\BD^{\on{Serre}}_{\CY}(\CF)\right).
\end{equation}
\end{cor}

\ssec{$(\on{dR},*)$-pushforwards for stacks}    \label{ss:de Rham dir image stacks}

\sssec{}  \label{sss:general_push-f}

If $\pi:\CY_1\to \CY_2$ is a map between algebraic stacks. We define the functor
$$\pi_{\on{dR},*}:\on{D-mod}(\CY_1)\to \on{D-mod}(\CY_2)$$
by
\begin{equation}   \label{e:piDR*}
\pi_{\on{dR},*}(\CM):=\underset{(S,g)\in ((\on{DGSch}_{\on{aft}})_{/\CY_1,\on{smooth}})^{\on{op}}}
{\underset{\longleftarrow}{lim}}\, (\pi\circ g)_{\on{dR},*}(g^*_{\on{dR}}(\CM)),
\end{equation}
where $(\pi\circ g)_{\on{dR},*}$ is understood in the sense 
of Sect.~\ref{sss:representable}.

\begin{rem}
Unfortunately, we do not know how to characterize the functor $\pi_{\on{dR},*}$ intrinsically.
Unless $\pi$ is smooth (or, more generally, locally acyclic in an appropriate sense), the left adjoint to $\pi_{\on{dR},*}$ 
will not be defined as a functor $\Dmod(\CY_2)\to \Dmod(\CY_1)$, but rather on the corresponding
pro-categories.  

\medskip

See, however, \corref{c:dR expl}, which gives an explicit formula for maps into $\pi_{\on{dR},*}(-)$
out of a coherent object of $\Dmod(\CY_2)$.  
\end{rem}

\sssec{Warning}        \label{sss:warning1}
The functor $\pi_{\on{dR},*}$ has features similar to those
of the functor $\pi_*$ discussed in \secref{s:dir im gen}. For a general morphism $\pi$, 
{\it it is not continuous} (see Sect.~\ref{sss:warning}); it does not satisfy base change 
(even for open embeddings) or the projection formula (see 
Sects. \ref{ss:base change Dmod} and \ref{ss:proj formula Dmod} 
below for the explanation of what this means). 

\medskip

That said, the restriction of $\pi_{\on{dR},*}$ to $\Dmod(\CY_1)^+$ behaves reasonably,
as is guaranteed by \propref{p:coconnective part Dmod}. 

\medskip

However, on all of $\Dmod(\CY_1)$, the functor $\pi_{\on{dR},*}$
may surprise one's intuition; see \secref{sss:trans contr} for 
a particularly treacherous example.  

\ssec{Properties of the $(\on{dR},*)$-pushforward}

This subsection is devoted to 
proving that $\pi_\dr$ has \emph{some} reasonable properties. As the following discussion 
is purely technical (and will amount to showing that certain limits can be commuted
with certain colimits), the reader can skip it on the first pass, and return to it when
necessary.

\sssec{}  \label{sss:change index}

One can calculate more economically $\pi_{\on{dR},*}$ as follows. 

\medskip

Let $A$ be a category equipped with a functor 
$$a\mapsto (S_a,g_a):A\to (\on{DGSch}_{\on{aft}})_{/\CY_1,\on{smooth}}$$
with the property that the functor given by $(\on{dR},*)$-pullback
$$\Dmod(\CY_1)\to \underset{a\in A^{\on{op}}}{\underset{\longleftarrow}{lim}}\, \Dmod(S_a),$$
is an equivalence, cf. \secref{sss:change index qc}. 

\medskip

In \secref{sss:proof of changing index} we will prove:

\begin{lem} \label{l:changing index}
For $\CM_1\in \Dmod(\CY_1)$, the map
$$\pi_\dr(\CM_1)\to \underset{a\in A^{\on{op}}}{\underset{\longleftarrow}{lim}}\, (\pi\circ g_\alpha)_\dr\circ (g_\alpha)_{\on{dR}}^*(\CM_1)$$
is an isomorphism.
\end{lem}

\sssec{}

Assume for a moment that $\pi$ is schematic and quasi-compact. In this case we obtain two functors, both denoted 
$\pi_\dr$. One such functor, which we shall temporarily denote by $\pi^{(a)}_\dr$ was introduced in \secref{sss:representable} and
was specific to schematic quasi-compact maps. Another functor, which we shall temporarily denote by $\pi^{(b)}_\dr$
is the one from \eqref{e:piDR*}.

\medskip

It is easy to see that there is a natural transformation
\begin{equation}  \label{e:two versions of dr}
\pi^{(a)}_\dr\to \pi^{(b)}_\dr.
\end{equation}

We claim:

\begin{prop} \label{p:dr for sch}
The natural transformation \eqref{e:two versions of dr} is an isomorphism.
\end{prop}

Due to this proposition, we obtain that the notation $\pi_\dr$ is unambiguous. 

\begin{proof}

For $\CM_1\in \Dmod(\CY_1)$ we calculate $\pi^{(b)}_\dr(\CM_1)$ by \lemref{l:changing index} via
the category $A=(\on{DGSch}_{\on{aft}})_{/\CY_2,\on{smooth}}$, see \lemref{l:replace index schematic}. 

\medskip

For
$$(S_2,g_2)\in (\on{DGSch}_{\on{aft}})_{/\CY_2,\on{smooth}}$$
consider the Cartesian diagram
$$
\CD
S_1  @>{g_1}>> \CY_1 \\
@V{\pi_S}VV    @VV{\pi}V   \\
S_2 @>{g_2}>>  \CY_2,
\endCD
$$
and we have:
$$(\pi\circ g_1)_\dr\circ (g_1)^*_{\on{dR}}(\CM_1)\simeq
(g_2\circ \pi_S)_\dr \circ (g_1)^*_{\on{dR}}(\CM_1)\simeq 
(g_2)_\dr\circ (\pi_S)_\dr\circ (g_1)^*_{\on{dR}}(\CM_1).$$

However, it is easy to see that the natural transformation 
$$(g_2)^*_{\on{dR}}\circ \pi^{(a)}_\dr\to (\pi_S)_\dr\circ (g_1)^*_{\on{dR}}$$
arising by adjunction is an isomorphism.

\medskip

Hence,
$$(\pi\circ g_1)_\dr\circ (g_1)^*_{\on{dR}}(\CM_1)\simeq 
(g_2)_\dr\circ (g_2)^*_{\on{dR}}\circ \pi^{(a)}_\dr(\CM_1).$$

Passing to the limit over $(S_2,g_2)$ we obtain the desired isomorphism.

\end{proof}

\sssec{Transitivity}

We note the following property of the functor $\pi_{\on{dR},*}$:

\begin{lem} \label{l:de Rham trans}
There exists a canonical isomorphism of (non-contunuous) functors 
$$\Dmod(\CY_1)\to \Vect:\Gamma_{\on{dR}}(\CY_1,-)\simeq \Gamma_{\on{dR}}(\CY_2,\pi_\dr(-)).$$
\end{lem}

\begin{proof}

Follows from the fact that the partially defined left adjoint $\pi^*_{\on{dR}}$ of $\pi_\dr$
is defined on $k_{\CY_2}$ and 
$$\pi^*_{\on{dR}}(k_{\CY_2})\simeq k_{\CY_1}.$$

\end{proof}

Let now $\phi:\CY_2\to \CY_3$ be another morphism between algebraic stacks. It is easy to
see that there exists a natural transformation
\begin{equation} \label{e:de Rham trans}
\phi_\dr\circ \pi_\dr\to (\phi\circ \pi)_\dr.
\end{equation}

The natural transformation is not always an isomorphism, see \secref{sss:trans contr} for
a counterexample. In what follows we shall need the following statement, proved
in \secref{sss:proof of prop de Rham trans}: 

\begin{prop} \label{p:de Rham trans}
Suppose that $\pi$ is schematic and quasi-compact. Then
the natural transformation \eqref{e:de Rham trans} is an isomorphism.
\end{prop}

We refer the reader to \secref{sss:more trans} where several more situations
are given, in which \eqref{e:de Rham trans} is an isomorphism. 

\sssec{Pushforward of induced D-modules}

Let $\pi:\CY_1\to \CY_2$ be as above. It is easy to see that there exists a canonical
natural transformation between functors $\IndCoh(\CY_1)\to \Dmod(\CY_2)$, namely,
\begin{equation} \label{e:pushforward induced}
\ind_{\Dmod(\CY_2)}\circ \pi^{\IndCoh}_{\on{non-ren},*}\to 
\pi_\dr\circ \ind_{\Dmod(\CY_1)}.
\end{equation}

The following assertion will be proved in \secref{sss:proof pushforward induced}:

\begin{prop} \label{p:pushforward induced}
Suppose that $\CY_1$ and $\CY_2$ are QCA. Then the natural transformation 
\eqref{e:pushforward induced} is an isomorphism.
\end{prop}

\begin{rem}
We do not know whether \eqref{e:pushforward induced} is an isomorphism for an arbitrary morphism
of stacks. Note, however, that when $\CY_2=\on{pt}$, this is true by \propref{p:de Rham and ind}.
\end{rem}

\sssec{}

Note that if $\pi$ is smooth, the functor $\pi_\dr$ commutes with limits, since it admits a left
adjoint. In general, this will not be so. However, we have the following useful property:

\begin{lem} \label{l:de Rham almost cocont}
For $\CM\in \Dmod(\CY_1)$, the natural map
$$\pi_\dr(\CM)\to \underset{n}{\underset{\longleftarrow}{lim}}\, \pi_\dr(\tau^{\geq -n}(\CM))$$
is an isomorphism.
\end{lem}

\begin{proof}

We have:
\begin{multline*}
\underset{n}{\underset{\longleftarrow}{lim}}\, \pi_\dr(\tau^{\geq -n}(\CM))
\simeq \underset{n}{\underset{\longleftarrow}{lim}}\,
\underset{(S,g)\in ((\on{DGSch}_{\on{aft}})_{/\CY_1,\on{smooth}})^{\on{op}}}
{\underset{\longleftarrow}{lim}}\, (\pi\circ g)_{\on{dR},*}(g^*_{\on{dR}}(\tau^{\geq -n}(\CM)))\simeq \\
\simeq 
\underset{(S,g)\in ((\on{DGSch}_{\on{aft}})_{/\CY_1,\on{smooth}})^{\on{op}}}
{\underset{\longleftarrow}{lim}}\,  \underset{n}{\underset{\longleftarrow}{lim}}\,
(\pi\circ g)_{\on{dR},*}(g^*_{\on{dR}}(\tau^{\geq -n}(\CM))).
\end{multline*}

We claim that for each $(S,g)$, the map
$$(\pi\circ g)_{\on{dR},*}(g^*_{\on{dR}}(\CM))\to  \underset{n}{\underset{\longleftarrow}{lim}}\,
(\pi\circ g)_{\on{dR},*}(g^*_{\on{dR}}(\tau^{\geq -n}(\CM)))$$
is an isomorphism.

\medskip

First, since $g^*_{\on{dR}}$ is of bounded cohomological amplitude, we rewrite 
$$\underset{n}{\underset{\longleftarrow}{lim}}\,
(\pi\circ g)_{\on{dR},*}(g^*_{\on{dR}}(\tau^{\geq -n}(\CM)))\simeq
  \underset{n}{\underset{\longleftarrow}{lim}}\,
(\pi\circ g)_{\on{dR},*}(\tau^{\geq -n}(g^*_{\on{dR}}(\CM))).$$

\medskip

Thus, we have reduced the assertion of the lemma to the case when $\CY_1=S$
is a quasi-compact DG scheme. In this case, the functor $\pi_\dr$ has itself
a bounded cohomological amplitude, so
$$\underset{n}{\underset{\longleftarrow}{lim}}\, \pi_\dr(\tau^{\geq -n}(\CM))\simeq
\underset{n}{\underset{\longleftarrow}{lim}}\, \tau^{\geq -n}(\pi_\dr(\CM)).$$

Now, the desired assertion follows from the left-completeness of $\Dmod(\CY_2)$
in its t-structure.

\end{proof}

\ssec{Base change for the $(\on{dR},*)$-pushforward}  \label{ss:base change Dmod}

\sssec{}

Let $\pi:\CY_1\to \CY_2$ be as above, and let $\phi_2:\CY'_2\to \CY_2$ be another morphism of
algebraic stacks. 

\medskip

Consider the Cartesian diagram
$$
\CD
\CY'_1  @>{\phi_1}>>  \CY_1 \\
@V{\pi'}VV  @VV{\pi}V  \\
\CY'_2 @>{\phi_2}>>  \CY_2.
\endCD
$$

For $\CM_1\in \Dmod(\CY_1)$ there exists a canonically defined map
\begin{equation} \label{e:base change morphism Dmod}
\phi_2^!\circ \pi_\dr(\CM_1)\to \pi'_\dr\circ \phi_1^!(\CM_1).
\end{equation}

\begin{defn} \label{defn:base change Dmod} \hfill

\smallskip

\noindent{\em(a)} The triple $(\phi_2,\CM_1,\pi)$ satisfies
base change if the map \eqref{e:base change morphism Dmod} is an isomorphism.

\smallskip

\noindent{\em(b)} The pair $(\CM_1,\pi)$ satisfies
base change if \eqref{e:base change morphism Dmod} is an isomorphism for any $\phi_2$.

\smallskip

\noindent{\em(c)} The morphism $\pi$ satisfies base change if 
\eqref{e:base change morphism Dmod} is an isomorphism for any $\phi_2$ and $\CF_1$.

\end{defn}

\sssec{}

We have the following analog of \propref{p:bootstrap base change}:

\begin{prop} \label{p:bootstrap base change Dmod}
Given $\pi:\CY_1\to \CY_2$, for $\CM_1\in \Dmod(\CY_1)$ the following conditions are equivalent:

\smallskip

\noindent {\em(i)} $(\CM_1,\pi)$ satisfies base change.

\smallskip

\noindent {\em(ii)} $(\phi_2,\CF_1,\pi)$ satisfies base change whenever $\CY_2=S_2\in \on{DGSch}_{\on{aft}}^{\on{aff}}$.

\smallskip

\noindent {\em(iii)} For any $S'_2\overset{f_2}\to S_2\overset{g_2}\to \CY_2$ with $S_2,S'_2\in \on{DGSch}_{\on{aft}}^{\on{aff}}$,
the triple $(f_2,\CF_{S,1},\pi_S)$ satisfies base change, where
$$g_1:S_2\underset{\CY_2}\times \CY_1\to \CY_1,\,\, 
\CM_{S,1}:=g_1^!(\CF_1) \text{ and } \pi_S:S_2\underset{\CY_2}\times \CY_1\to S_2.$$ 

\end{prop}

\begin{proof}

Follows in the same way as \propref{p:bootstrap base change} using the following observation.

\medskip

Let $i\mapsto \bC^i, i\in I$ be a family of cocomplete DG categories,
let $F^{i,j}:\bC^i\to \bC^j$ denote the corresponding
family of functors. Let $\bC:=\underset{i}{\underset{\longleftarrow}{lim}}\, \bC^i$ be their limit.

\medskip

For another category of indices $A$, let $\bc_a, a\in A$ be an $A$-family of objects in $\bC$, i.e., a compatible
family of objects $\bc^i_a\in \bC^i,a\in A$. Denote
$$\bc^i:=\underset{a}{\underset{\longleftarrow}{lim}}\, \bc^i_a\in \bC^i$$
(recall that cocomplete DG categories are closed under limits, see \secref{sss:dg functors}).

\begin{lem}   \label{l:general_observation}
Suppose that the maps $F_{i,j}(\bc^i)\to \bc^j$ are isomorphisms. 
Then 
$$\bc:=\underset{a}{\underset{\longleftarrow}{lim}}\, \bc_a\in \bC$$
corresponds to the system $i\mapsto \bc^i$. 
\end{lem}

We apply this lemma as follows: the category of indices $I$ is $((\on{DGSch})^{\on{aff}}_{/\CY_2})^{\on{op}}$ and
for an object $i=(Z,f)\in I$, set 
$$\bC^i=\Dmod(Z),$$ 
so that $\bC=\Dmod(\CY_2)$. 

\medskip

We take the category of indices $A$ to be $((\on{DGSch}^{\on{aff}}_{\on{aft}})_{\CY_1,\on{smooth}})^{\on{op}}$. 
For each $a=(S,g)\in A$ we set 
$$\bc_a:=(\pi\circ g)_\dr(g^*_{\on{dR}}(\CM_1)),$$
so that $\bc=\pi_\dr(\CM)$ and for $i=(Z,f)$
$$\bc^i=(\pi_Z)_{\on{dR},*}(\wt{f}^!(\CM_1)),$$
where
$$
\CD
Z\underset{\CY_2}\times \CY_1  @>{\wt{f}}>>  \CY_1 \\
@V{\pi_Z}VV   @VV{\pi}V   \\
Z @>{f}>>  \CY_2.
\endCD
$$

\end{proof}

It is clear that schematic quasi-compact morphisms satisfy base change. 

\begin{rem}
In \thmref{t:rel_Dmods} we shall show
that any morphism which is \emph{safe} satisfies base change. 
Furthermore, in \corref{c:safe satisfies base change} we will
show that for a morphism between QCA stacks, a pair $(\CM_1,\pi)$ satisfies base change 
if $\CM_1$ is \emph{safe}. 
\end{rem}

\sssec{}

As in \corref{c:coconnective part} we have:

\begin{prop}   \label{p:coconnective part Dmod}  Let $\pi:\CY_1\to \CY_2$ be a quasi-compact morphism between algebraic 
stacks. Then: 

\smallskip

\noindent{\em(a)} For any $\CM_1\in \Dmod(\CY_1)^+$, the pair $(\CM_1,\pi)$ satisfies base change.
\footnote{The proof uses the fact that for a morphism of affine DG schemes, the functor of $!$-pullback
of D-modules has a finite cohomological amplitude.}

\smallskip

\noindent{\em(b)} If $\CY_2$ is quasi-compact, there exists $m\in\BZ$ such that for any $n\in \BZ$
the functor $\pi_{\on{dR},*}$, when restricted to $\on{D-mod}(\CY_1)^{\geq n}$, maps
to  $\on{D-mod}(\CY_2)^{\geq n-m}$, and as such commutes with filtered colimits.
\end{prop}

\ssec{Projection formula for the $(\on{dR},*)$-pushforward}  \label{ss:proj formula Dmod}

\sssec{}  

In the situation of \secref{sss:general_push-f}, let $\CM_1\in \Dmod(\CY_1)$ and $\CM_2\in \Dmod(\CY_2)$  
be two objects. We claim that there is always a morphism in one direction
\begin{equation} \label{e:map for proj formula Dmod}
\CM_2\sotimes \pi_\dr(\CM_1)\to \pi_\dr(\pi^!(\CM_2)\sotimes \CM_1).
\end{equation}

Indeed, specifying such morphism amounts to a compatible family of maps
$$\CM_2\sotimes \pi_\dr(\CM_1)\to (\pi\circ g)_{\on{dR},*}\left(g^*_{\on{dR}}\left(\pi^!(\CM_2)\sotimes \CM_1\right)\right)$$
for $(S,g)\in (\on{DGSch}_{\on{aft}})_{/\CY_1,\on{smooth}}$. 

\medskip

The required map arises from the map
\begin{multline*}
\CM_2\sotimes \pi_\dr(\CM_1)\to \CM_2\sotimes (\pi\circ g)_\dr(g^*_{\on{dR}}(\CM_1))\simeq \\
\simeq (\pi\circ g)_{\on{dR},*}\left((\pi\circ g)^!(\CM_2)\sotimes g^*_{\on{dR}}(\CM_1)\right)
\simeq(\pi\circ g)_{\on{dR},*}\left(g^! (\pi^!(\CM_2))\sotimes g^*_{\on{dR}}(\CM_1)\right)\simeq \\
\simeq (\pi\circ g)_{\on{dR},*} \left(g^*_{\on{dR}}\left(\pi^!(\CM_2)\sotimes \CM_1\right)\right)
\end{multline*}
where the second arrow is furnished by \secref{sss:representable}, as the morphism $\pi\circ g$ is schematic 
and quasi-compact, and where the last arrow uses the isomorphism \eqref{e:* and ! ten}. 

\sssec{}

We give the following definitions:

\begin{defn} \label{defn:proj formula Dmod} \hfill

\smallskip

\noindent{\em(a)} The triple $(\CM_1,\CM_2,\pi)$ satisfies the projection formula
if the map \eqref{e:map for proj formula Dmod} is an isomorphism.

\smallskip

\noindent{\em(b)} The pair $(\CM_2,\pi)$ satisfies
the projection formula if \eqref{e:map for proj formula Dmod} is an isomorphism for any $\CM_1$.

\smallskip

\noindent{\em(c)} The pair $(\CM_1,\pi)$ satisfies
the projection formula if \eqref{e:map for proj formula Dmod} is an isomorphism for any $\CM_2$.

\smallskip

\noindent{\em(d)} The map $\pi$ satisfies
the projection formula if \eqref{e:map for proj formula Dmod} is an isomorphism for any $\CM_1$ and $\CM_2$.

\end{defn}

\medskip

We also give the following definition:

\begin{defn}  \label{defn:strong proj formula Dmod}  The morphism $\pi$ 
strongly satisfies the projection formula if 
it satisfies base change and for every
$S_2\in (\on{DGSch}^{\on{aff}}_{\on{aft}})_{/\CY_2}$, the morphism 
$$\pi_S:S_2\underset{\CY_2}\times \CY_1\to S_2$$ 
satisfies the projection formula. 
\end{defn}

It is easy to see that if $\pi$ strongly satisfies the projection formula, then it satisfies the projection formula.

\sssec{Examples}  \label{sss:ex proj formula} \hfill

\smallskip

\noindent(i) It is easy to see that if $\pi$ is schematic and quasi-compact, then $\pi$ strongly satisfies the projection formula. 

\medskip

\noindent(ii) In \thmref{t:rel_Dmods} we shall strengthen this to the assertion that any $\pi$ which is safe also strongly 
satisfies the projection formula. 

\medskip

\noindent(iii) In \corref{c:partial proj formula safe}, we will show that if $\pi$ is a morphism between QCA stacks,
and $\CM_1\in \Dmod(\CY_1)$ is safe, then $(\CM_1,\pi)$ satisfies the projection formula.

\medskip

\noindent(iv) Suppose that $ \pi$ is quasi-compact. Then for any $\CM_i\in \Dmod(\CY_i)^+$, the triple
$(\CM_1,\CM_2,\pi)$ satisfies the projection formula. This follows in the same way as in \corref{c:coconnective part}(c), using the fact that
for a quasi-compact algebraic stack $\CY$ the functor $\sotimes$ on $\Dmod(\CY)$ has a bounded cohomological amplitude.

\sssec{A counter-example}  \label{sss:cntr}

It is easy to produce an example of how the projection formula fails when
$\CM_2$ is not compact. E.g., take $\CY_1=B\BG_m$, $\CY_2=\on{pt}$,
$\CM_1=\underset{n\geq 0}\oplus\, k_{B\BG_m}[2n]$
and $\CM_2=V$, where $V$ is any infinite-dimensional vector space.
Computing the two sides of the projection formula via \lemref{l:de Rham almost cocont}, 
it is easy to check that the projection formula fails in
this case.\footnote{Doing this exercise makes it easier to read similar
but more lengthy computations below.}

\medskip

We shall now give an example of how the projection formula fails when $\CM_2$ is compact. 

\medskip

Take $\CY_1=\BA^1\times B\BG_m$ and $\CY_2=\BA^1$, with the morphism $\pi$ being the
projection on the first factor:
$$
\CD
\BA^1\times B\BG_m  @>{p_{\BA^1}\times \on{id}_{B\BG_m}}>>   B\BG_m  \\
@V{\pi}VV    @VV{p_{B\BG_m}}V   \\
\BA^1  @>{p_{\BA^1}}>>  \on{pt}.
\endCD
$$

\medskip

We take $\CM_2:=\ind^{\on{left}}_{\Dmod(\BA^1)}(\CO_{\BA^1})$ and
$$\CM_1:=(p_{\BA^1}\times \on{id}_{B\BG_m})^!\left(\underset{n\geq 0}\oplus\, k_{B\BG_m}[2n]\right).$$
We shall consider $\Dmod(\BA^1)$ in the "left" realization, and in particular, with the t-structure,
for which the functor $\oblv^{\on{left}}_{\Dmod(\BA^1)}$ is t-exact.

\medskip

We calculate both $\pi_\dr(\CM_1)$ and $\pi_\dr(\CM_1\sotimes \pi^!(\CM_2))$ using 
\lemref{l:de Rham almost cocont}. We have:
$$\pi_\dr(\CM_1)\simeq \underset{m}{\underset{\longleftarrow}{lim}}\, \, \,
\pi_\dr\circ (p_{\BA^1}\times \on{id}_{B\BG_m})^!\left(\underset{m\geq n\geq 0}\oplus\, k_{B\BG_m}[2n]\right)$$
and
\begin{multline*} 
\pi_\dr(\CM_1\sotimes \pi^!(\CM_2))\simeq \\
\simeq \underset{m}{\underset{\longleftarrow}{lim}}\, \,\, 
\pi_\dr\left((p_{\BA^1}\times \on{id}_{B\BG_m})^!(\underset{m\geq n\geq 0}\oplus\, k_{B\BG_m}[2n])\sotimes
\pi^!(\ind^{\on{left}}_{\Dmod(\BA^1)}(\CO_{\BA^1}))\right).
\end{multline*}

\medskip

By \propref{p:coconnective part Dmod}(a), for every $m$ we have
\begin{multline*}  %\label{e:before ten}
\pi_\dr\circ (p_{\BA^1}\times \on{id}_{B\BG_m})^!\left(\underset{m\geq n\geq 0}\oplus\, k_{B\BG_m}[2n]\right)
\simeq \\
\simeq (p_{\BA^1})^!\left(\Gamma_{\on{dR}}(B\BG_m,(\underset{m\geq n\geq 0}\oplus\, k_{B\BG_m}[2n]))\right)\simeq \\
\simeq (p_{\BA^1})^!\left(\underset{m\geq n\geq 0,l\geq 0}\oplus\, k[2(n-l)]\right)\simeq
\underset{m\geq n\geq 0,l\geq 0}\oplus\, \CO_{\BA^1}[2(n-l)].
\end{multline*}

In particular, the $0$th cohomology of $\pi_\dr(\CM_1)$ identifies with
$$\underset{m\geq 0}\Pi\, \CO_{\BA^1},$$
the countable product of copies of $\CO_{\BA^1}\in \Dmod^{\on{left}}(\BA^1)$. 

\medskip

Since $\ind^{\on{left}}_{\Dmod(\BA^1)}(\CO_{\BA^1})$ is flat as an $\CO_{\BA^1}$-module, we obtain that the $0$th cohomology
of $\pi_\dr(\CM_1)\sotimes \CM_2$ identifies with
$$\left(\underset{m\geq 0}\Pi\, \CO_{\BA^1}\right)\otimes \ind^{\on{left}}_{\Dmod(\BA^1)}(\CO_{\BA^1})$$
(we note that in the ``left" realization, the tensor product $\sotimes$ corresponds to the usual
tensor product $\otimes$ at the level of the underlying $\CO$-modules). 

\medskip

Note that the forgetful functor
$$\Gamma(\BA^1,-)\circ \oblv^{\on{left}}_{\Dmod(\BA^1)}:\Dmod(\BA^1)\to \Vect$$
commutes with limits, since it admits a left adjoint. Hence, we obtain that the $0$th cohomology
of 
$$\Gamma\left(\BA^1,\oblv^{\on{left}}_{\Dmod(\BA^1)}(\pi_\dr(\CM_1)\sotimes \CM_2)\right)$$
identifies with
$$\left(\underset{m\geq 0}\Pi\, k[t]\right)\underset{k[t]}\otimes k[t,\partial_t]\simeq \left(\underset{m\geq 0}\Pi\, k[t]\right)\otimes V,$$
where $V$ is a vector space such that $k[t,\partial_t]\simeq k[t]\otimes V$ as a $k[t]$-module. The key point is that
$V$ is infinite-dimensional. 

\medskip
 
By \secref{sss:ex proj formula}(iv),
\begin{multline*}
\pi_\dr\left((p_{\BA^1}\times \on{id}_{B\BG_m})^!(\underset{m\leq n\geq 0}\oplus\, k_{B\BG_m}[2n])\sotimes
\pi^!(\CM_2)\right)\simeq \\
\simeq (p_{\BA^1})^!(\underset{m\geq n\geq 0,l\geq 0}\oplus\, k[2(n-l)])\sotimes \ind_{\Dmod(\BA^1)}(\CO_{\BA^1})
\simeq \\
\simeq  \left(\underset{m\geq n\geq 0,l\geq 0}\oplus\, \CO_{\BA^1}[2(n-l)]\right)\otimes \ind_{\Dmod(\BA^1)}(\CO_{\BA^1}).
\end{multline*}

Hence, the $0$th cohomology of  
$$\Gamma\left(\BA^1,\oblv^{\on{left}}_{\Dmod(\BA^1)}(\pi_\dr(\CM_1\sotimes \pi^!(\CM_2)))\right)$$
identifies with
$$\underset{m\geq 0}\Pi\, k[t,\partial_t]\simeq \underset{m\geq 0}\Pi\, (k[t]\otimes V).$$

Finally, the canonical map
$$\left(\underset{m\geq 0}\Pi\, k[t]\right)\otimes V\to \underset{m\geq 0}\Pi\, (k[t]\otimes V)$$
is \emph{not} an isomorphism because $V$ is infinite-dimensional.

\ssec{Proofs of properties of the $(\on{dR},*)$-pushforward}

\sssec{}

First, we are now going to prove the following assertion, which has multiple consequences:

\begin{lem} \label{l:partial projection formula}
For a map of algebraic stacks $\pi:\CY_1\to \CY_2$, and $\CM_2\in \Dmod_{\on{coh}}(\CY_2)$
and any $\CM_1\in \Dmod(\CY_1)$, the map
$$\Gamma_{\on{dR}}(\CY_2,\CM_2\sotimes \pi_\dr(\CM_1))\to 
\Gamma_{\on{dR}}(\CY_2,\pi_\dr(\pi^!(\CM_2)\sotimes \CM_1)),$$
induced by \eqref{e:map for proj formula Dmod}, is an isomorphism.
\end{lem}

\begin{rem}
Note, however, that in the situation of the \lemref{l:partial projection formula},
the map \eqref{e:map for proj formula Dmod} itself does not have to
be an isomorphism, see example in \secref{sss:cntr} above.
\end{rem}

\begin{proof}[Proof of \lemref{l:partial projection formula}]

Note that the assumption that $\CM_2\in \Dmod_{\on{coh}}(\CY_2)$ implies that the functor 
$$\CM'_2\mapsto \Gamma_{\on{dR}}(\CY_2,\CM'_2\sotimes\CM_2):\Dmod(\CY_2)\to \Vect$$
commutes with limits. Indeed, this is because the above functor identifies with 
$$\CMaps_{\Dmod(\CY_2)}(\BD_{\CY_2}^{\on{Verdier}}(\CM_2),-),$$
by \lemref{l:Verdier_stacks}. 

\medskip

Applying the definition of $\pi_\dr$, we obtain that 
$\Gamma_{\on{dR}}\left(\CY_2,\pi_\dr(\CM_1)\sotimes \CM_2\right)$ identifies
with the limit over $(S,g)\in \on{DGSch}_{/\CY_1,\on{smooth}}$ of
$$\Gamma_{\on{dR}}\left(\CY_2,(\pi\circ g)_\dr(g^*_{\on{dR}}(\CM_1))\sotimes\CM_2\right).$$

\medskip

Since the morphism $\pi\circ g$ is schematic and quasi-compact, the projection formula \eqref{e:proj formula Dmod sch}
is applicable, and the latter expression can be rewritten as
$$\Gamma_{\on{dR}}\left(\CY_2,(\pi\circ g)_\dr\left(g^*_{\on{dR}}(\CM_1)\sotimes (\pi\circ g)^!(\CM_2)\right)\right),$$
and by \lemref{l:de Rham trans} further as
$$\Gamma_{\on{dR}}\left(S,g^*_{\on{dR}}(\CM_1)\sotimes (\pi\circ g)^!(\CM_2)\right)
\simeq \Gamma_{\on{dR}}\left(S,g^*_{\on{dR}}(\CM_1)\sotimes g^!(\pi^!(\CM_2))\right).$$
Applying the isomorphism of \eqref{e:* and ! ten}, we rewrite the latter as
$$\Gamma_{\on{dR}}\left(S,g^*_{\on{dR}}\left(\CM_1\sotimes \pi^!(\CM_2)\right)\right).$$

Now, the resulting limit over $(S,g)$ is isomorphic to 
$\Gamma_{\on{dR}}(\CY_1,\CM_1\sotimes \pi^!(\CM_2))$ by definition.

\end{proof}

Note that \lemref{l:partial projection formula} gives the following, somewhat more explicit 
characterization of the $(\on{dR},*)$-pushforward functor:

\begin{cor} \label{c:dR expl}
For $\pi:\CY_1\to \CY_2$, $\CM_1\in \Dmod(\CY_1)$ and $\CM_2\in \Dmod_{\on{coh}}(\CY_2)$,
we have a canonical isomorphism
$$\CMaps(\CM_2,\pi_\dr(\CM_1))\simeq \Gamma_{\on{dR}}(\CY_1,\pi^!(\BD^{\on{Verdier}}_{\CY_2}(\CM_2))\sotimes \CM_1).$$
\end{cor}

\begin{proof}

Using \lemref{l:Verdier_stacks}, we rewrite the left-hand side as
$$\Gamma_{\on{dR}}(\CY_2,\BD^{\on{Verdier}}_{\CY_2}(\CM_2)\sotimes \pi_\dr(\CM_1)),$$
and further, using \lemref{l:partial projection formula}, as
$$\Gamma_{\on{dR}}(\CY_2,\pi_\dr(\pi^!(\BD^{\on{Verdier}}_{\CY_2}(\CM_2))\sotimes \CM_1)),$$
and finally as
$$\Gamma_{\on{dR}}(\CY_1,\pi^!(\BD^{\on{Verdier}}_{\CY_2}(\CM_2))\sotimes \CM_1)$$
using \lemref{l:de Rham trans}.

\end{proof}

\sssec{Proof of \propref{p:de Rham trans}}  \label{sss:proof of prop de Rham trans}

It is easy to see that for an algebraic stack $\CY$, the category $\Dmod(\CY)$ is generated 
by its subcategory $\Dmod_{\on{coh}}(\CY)$. Hence, using \lemref{l:Verdier_stacks}, we obtain that 
it is enough to show that for $\CM_1\in \Dmod(\CY_1)$
and $\CM_3\in \Dmod_{\on{coh}}(\CY_3)$, the map
$$\Gamma_{\on{dR}}\left(\CY_3,\CM_3\sotimes (\phi_\dr\circ \pi_\dr(\CM_1))\right)\to
\Gamma_{\on{dR}}\left(\CY_3,\CM_3\sotimes (\phi\circ \pi)_\dr(\CM_1)\right)$$
is an isomorphism. 

\medskip

Applying Lemmas \ref{l:partial projection formula} and \ref{l:de Rham trans}, we rewrite
\begin{multline} \label{e:Rham trans one}
\Gamma_{\on{dR}}\left(\CY_3,\CM_3\sotimes (\phi_\dr\circ \pi_\dr(\CM_1))\right)
\simeq 
\Gamma_{\on{dR}}\left(\CY_3,\phi_\dr(\phi^!(\CM_3)\sotimes \pi_\dr(\CM_1))\right)\simeq \\
\simeq \Gamma_{\on{dR}}\left(\CY_2,\phi^!(\CM_3)\sotimes \pi_\dr(\CM_1)\right)
\end{multline}
and 
\begin{multline} \label{e:Rham trans two}
\Gamma_{\on{dR}}\left(\CY_3,\CM_3\sotimes (\phi\circ \pi)_\dr(\CM_1)\right)\simeq
\Gamma_{\on{dR}}\left(\CY_3,(\phi\circ \pi)_\dr((\phi\circ \pi)^!(\CM_3)\sotimes \CM_1)\right)\simeq \\
\simeq \Gamma_{\on{dR}}\left(\CY_1,(\phi\circ \pi)^!(\CM_3)\sotimes \CM_1\right).
\end{multline}

Since $\pi$ is schematic and quasi-compact, the projection formula is applicable, and we obtain
\begin{equation} \label{e:proj formula applied}
\Gamma_{\on{dR}}\left(\CY_2,\phi^!(\CM_3)\sotimes \pi_\dr(\CM_1)\right)\simeq
\Gamma_{\on{dR}}\left(\CY_2,\pi_\dr(\pi^!\circ \phi^!(\CM_3)\sotimes \CM_1)\right).
\end{equation}

Hence, using \lemref{l:de Rham trans}, we obtain that the expression in \eqref{e:Rham trans one} is 
also isomorphic to
$$\Gamma_{\on{dR}}\left(\CY_1,(\phi\circ \pi)^!(\CM_3)\sotimes \CM_1\right),$$
as required.

\qed

\sssec{}  \label{sss:more trans}

Note that the only non-tautological point of the proof of \propref{p:de Rham trans} is the isomorphism
\eqref{e:proj formula applied}.

\medskip

Hence, more generally, we obtain that the map \eqref{e:de Rham trans} is an isomorphism in 
the following situations:

\medskip

\noindent(i) When $\pi$ satisfies the projection formula.

\medskip

\noindent(ii) When $\phi^!$ sends $\Dmod_{\on{coh}}(\CY_3)$ to $\Dmod_{\on{coh}}(\CY_2)$
(this is due to \lemref{l:partial projection formula}). This happens, e.g., when $\phi$ is smooth.

\medskip

\noindent(iii) When $\CM_1\in \Dmod(\CY_1)^+$ (this is due to \secref{sss:ex proj formula}(iv)).

\sssec{}   \label{sss:trans contr}

The natural transformation \eqref{e:de Rham trans} fails to be an isomorphism in the following example.

\medskip

We take $\CY_1=B\BG_m$, $\CY_2=\on{pt}$ and $\CY_3=\BA^1$, where $\phi$ is the inclusion of 
$0$ into $\BA^1$. We take $\CM_1\in \Dmod(B\BG_m)$ equal to $\underset{n\geq 0}\oplus\, k_{B\BG_m}[2n]$. We claim that
$$\CM_3:=(\phi\circ \pi)_\dr(\CM_1)\in \Dmod(\BA^1)$$
is \emph{not} supported at $0$. Indeed, using the fact that the functors 
$$\Gamma(\BA^1,-):\QCoh(\BA^1)\to \Vect \text{ and } \oblv^{\on{left}}_{\Dmod(\BA^1)}:\Dmod(\BA^1)\to \QCoh(\BA^1)$$
commute with limits, we calculate $\Gamma(\BA^1,\oblv^{\on{left}}_{\Dmod(\BA^1)}(\CM_3))$ via
\lemref{l:de Rham almost cocont}. 

\medskip

Note that 
$$(\phi\circ \pi)_\dr(k_{B\BG_m})\simeq \underset{m\geq 0}\oplus\, \delta[-2m],$$
where $\delta$ is the $\delta$-function at $0\in \BA^1$. We obtain that 
$H^0\left(\Gamma(\BA^1,\oblv^{\on{left}}_{\Dmod(\BA^1)}(\CM_3))\right)$ is the \emph{product} of $\BN$-many copies
of $\Gamma(\BA^1,\oblv^{\on{left}}_{\Dmod(\BA^1)}(\delta))$.
In particular, the generator $t\in \Gamma(\BA^1,\CO_{\BA_1})$ acts on it
non-nilpotently. 

\sssec{Proof of \lemref{l:changing index}}   \label{sss:proof of changing index} 

As in the proof of \propref{p:de Rham trans}, it suffices to show that for 
$\CM_2\in \Dmod_{\on{coh}}(\CY_2)$, the map
$$\Gamma_{\on{dR}}(\CY_2,\CM_2\sotimes \pi_\dr(\CM_1))\to
\Gamma_{\on{dR}}\left(\CY_2,\CM_2\sotimes \left(
\underset{a\in A^{\on{op}}}{\underset{\longleftarrow}{lim}}\, (\pi\circ g_\alpha)_\dr\circ (g_\alpha)_{\on{dR}}^*(\CM_1)\right)\right)$$
is an isomorphism.

\medskip

As in the proof of \lemref{l:partial projection formula}, we have:
\begin{multline*}
\Gamma_{\on{dR}}\left(\CY_2,\CM_2\sotimes \left(
\underset{a\in A^{\on{op}}}{\underset{\longleftarrow}{lim}}\, (\pi\circ g_\alpha)_\dr\circ (g_\alpha)_{\on{dR}}^*(\CM_1)\right)\right)
\simeq \\
\simeq \underset{a\in A^{\on{op}}}{\underset{\longleftarrow}{lim}}\, 
\Gamma_{\on{dR}}\left(S_\alpha,(g_\alpha)^*_{\on{dR}}(\pi^!(\CM_2)\sotimes \CM_1)\right)\simeq \\
\simeq \underset{a\in A^{\on{op}}}{\underset{\longleftarrow}{lim}}\, 
\CMaps_{\Dmod(S_\alpha)}\left(k_{S_\alpha},(g_\alpha)^*_{\on{dR}}(\pi^!(\CM_2)\sotimes \CM_1)\right)
\end{multline*}

However, by the assumption on $A$, the latter expression is isomorphic to
$$\CMaps_{\Dmod(\CY_1)}(k_{\CY_1},\pi^!(\CM_2)\sotimes \CM_1)\simeq \Gamma_{\on{dR}}(\CY_1,\pi^!(\CM_2)\sotimes \CM_1),$$
which is isomorphic to
$$\Gamma_{\on{dR}}(\CY_2,\CM_2\sotimes \pi_\dr(\CM_1)),$$
by Lemmas \ref{l:partial projection formula} and \ref{l:de Rham trans}

\qed

\ssec{Proof of \propref{p:de Rham and ind}}  \label{ss:proof of de Rham and ind}

\sssec{}

First, we are going to construct a map in one direction:
\begin{equation} \label{e:de Rham and ind}
\Gamma_{\on{dR}}(\CY,\ind_{\Dmod(\CY)}(\CF))\to \Gamma^{\IndCoh}(\CY,\CF).
\end{equation}

By definition, the left-hand side and the right-hand side are the limits over
$$(S,g)\in (\on{DGSch}_{/\CY,\on{smooth}})^{\on{op}}$$ of
$$\Gamma_{\on{dR}}(S,g^*_{\on{dR}}\circ \ind_{\Dmod(\CY)}(\CF)) \text{ and }
\Gamma^{\IndCoh}(S,g^{\IndCoh,*}(\CF)),$$
respectively. 

\medskip

We rewrite
$$\Gamma^{\IndCoh}(S,g^{\IndCoh,*}(\CF))\simeq \Gamma_{\on{dR}}(S,\ind_{\Dmod(S)}\circ g^{\IndCoh,*}(\CF)).$$

We claim that there is a canonical map
$$g^*_{\on{dR}}\circ \ind_{\Dmod(\CY)}(\CF)\to \ind_{\Dmod(S)}\circ g^{\IndCoh,*}(\CF)$$
that functorially depends on $(S,g)$. The map in question arises by the $(g^*_{\on{dR}},g_\dr)$ adjunction from the 
isomorphism of \propref{p:DR of induced stack}.

\medskip

Thus, we obtain a compatible system of maps
\begin{equation} \label{e:de Rham and ind terms}
\Gamma_{\on{dR}}(S,g^*_{\on{dR}}\circ \ind_{\Dmod(\CY)}(\CF)) \to
\Gamma^{\IndCoh}(S,g^{\IndCoh,*}(\CF)),
\end{equation}
giving rise to the desired map \eqref{e:de Rham and ind}.

\medskip

Note, however, that the individual maps in \eqref{e:de Rham and ind terms} are \emph{not} isomorhisms.

\sssec{}  \label{sss:com diag}

The following property of the map \eqref{e:de Rham and ind} follows from the construction. Let
$\pi:\wt\CY\to \CY$ be a schematic and quasi-compact map. 

\medskip

Then for $\wt\CF\in \IndCoh(\wt\CY)$
the following diagram commutes:
$$
\CD
\Gamma_{\on{dR}}(\wt\CY,\ind_{\Dmod(\wt\CY)}(\wt\CF))   @>{\text{\eqref{e:de Rham and ind}}}>>   \Gamma^{\IndCoh}(\wt\CY,\wt\CF)   \\
@V{\text{\lemref{l:de Rham trans}}}V{\sim}V  \\
\Gamma_{\on{dR}}(\CY,\pi_\dr\circ \ind_{\Dmod(\wt\CY)}(\wt\CF))  & & @VV{\sim}V  \\
@V{\text{\propref{p:DR of induced stack}}}V{\sim}V     \\
\Gamma_{\on{dR}}(\CY,\ind_{\Dmod(\CY)}\circ \pi^{\IndCoh}_*(\wt\CF)) @>{\text{\eqref{e:de Rham and ind}}}>>  
\Gamma^{\IndCoh}(\CY,\pi^{\IndCoh}_*(\wt\CF)).
\endCD
$$

\sssec{Two reduction steps}

We note that for $\CF\in \IndCoh(\CY)$, the maps
$$\Gamma_{\on{dR}}(\CY,\ind_{\Dmod(\CY)}(\CF))\to
\underset{n}{\underset{\longleftarrow}{lim}}\, \Gamma_{\on{dR}}(\CY,\ind_{\Dmod(\CY)}(\tau^{\geq -n}(\CF)))$$
and
$$\Gamma^{\IndCoh}(\CY,\CF)\to 
\underset{n}{\underset{\longleftarrow}{lim}}\, \Gamma^{\IndCoh}(\CY,\tau^{\geq -n}(\CF))$$
are both isomorphisms. 

\medskip

Another way to phrase this is that both functors are right Kan extensions of their restrictions to
$\IndCoh(\CY)^+$. 

\medskip

Hence, in order to show that \eqref{e:de Rham and ind} is an isomorphism, it is enough to
do so for $\CF\in \IndCoh(\CY)^+$.  Passing to a Zariski cover, we may assume that $\CY$
is quasi-compact.

\sssec{}

Choose a smooth cover $g:Z\to \CY$, where $Z\in \dgSch_{\on{aft}}$, and consider its 
\v{C}ech nerve $Z^\bullet/\CY$. Let $g^i$ denote the corresponding map $Z^i/\CY\to \CY$.

\medskip

Consider the resulting map
$$\CF\to \on{Tot}\left((g^i)^{\IndCoh}_*\circ (g^i)^{\IndCoh,*}(\CF)\right).$$

It is an isomorphism in $\IndCoh(\CY)$, because the terms are uniformly bounded below,
and the corresponding map
$$\Psi_\CY(\CF)\to \Psi_\CY\left(\on{Tot}\left((g^i)^{\IndCoh}_*\circ (g^i)^{\IndCoh,*}(\CF)\right)\right)\simeq 
\on{Tot}\left((g^i)_*\circ (g^i)^*(\Psi_\CY(\CF))\right)$$
is an isomorphism in $\QCoh(\CY)$. 

\medskip

Since $\Gamma^{\IndCoh}(\CY-)\simeq \Gamma(\CY,-)\circ \Psi_\CY$, the map
\begin{multline*}
\Gamma^{\IndCoh}(\CY,\CF)\to 
\on{Tot}\left(\Gamma^{\IndCoh}(\CY,(g^i)^{\IndCoh}_*\circ (g^i)^{\IndCoh,*}(\CF))\right)\simeq \\
\simeq \on{Tot}\left(\Gamma^{\IndCoh}(Z^i/\CY,(g^i)^{\IndCoh,*}(\CF))\right)
\end{multline*}
is also an isomorphism.

\sssec{}

We claim that the map 
\begin{multline*}
\ind_{\Dmod(\CY)}(\CF)\to \on{Tot}\left(\ind_{\Dmod(\CY)}\circ (g^i)^{\IndCoh}_*\circ (g^i)^{\IndCoh,*}(\CF)\right)
\overset{\text{\propref{p:DR of induced stack}}}\simeq \\
\simeq \on{Tot}\left((g^i)_\dr\circ \ind_{\Dmod(Z^i/\CY)}\circ (g^i)^{\IndCoh,*}(\CF)\right)
\end{multline*}
is an isomorphism. 

\medskip

This is obtained as in \corref{c:coconnective part}(c) using the fact that 
the functor $\ind_{\Dmod(\CY)}$ is of bounded cohomological amplitude.

\medskip

Note also that the functor $\Gamma_{\on{dR}}(\CY,-)$ commutes with limits. Indeed, the functor in
question is given by
$\CMaps_{\Dmod(\CY)}(k_\CY,-)$. 

\medskip

Hence, we obtain that the natural map
\begin{multline*}
\Gamma_{\on{dR}}(\CY,\ind_{\Dmod(\CY)}(\CF))\to \on{Tot}\left(
\Gamma_{\on{dR}}\left(\CY,(g^i)_\dr\circ \ind_{\Dmod(Z^i/\CY)}\circ (g^i)^{\IndCoh,*}(\CF)\right)\right)\simeq\\
\overset{\text{\lemref{l:de Rham trans}}}\simeq 
\on{Tot}\left(\Gamma_{\on{dR}}\left(Z^i/\CY,\ind_{\Dmod(Z^i/\CY)}\circ (g^i)^{\IndCoh,*}(\CF)\right)\right)
\end{multline*}
is an isomorphism.

\sssec{}

Now, it follows from \secref{sss:com diag} that the map in \eqref{e:de Rham and ind} fits into the commutative diagram
$$
\CD
\Gamma_{\on{dR}}(\CY,\ind_{\Dmod(\CY)}(\CF))  @>>>  
\on{Tot}\left(\Gamma_{\on{dR}}\left(Z^i/\CY,\ind_{\Dmod(Z^i/\CY)}\circ (g^i)^{\IndCoh,*}(\CF)\right)\right)  \\
@VVV   @VVV   \\
\Gamma^{\IndCoh}(\CY,\CF)  @>>>  
\on{Tot}\left(\Gamma^{\IndCoh}(Z^i/\CY,(g^i)^{\IndCoh,*}(\CF))\right),
\endCD
$$
where the right vertical arrow is a co-simplicial isomorphism coming from
$$\Gamma_{\on{dR}}\left(Z^i/\CY,\ind_{\Dmod(Z^i/\CY)}(-)\right)\simeq
\Gamma^{\IndCoh}(Z^i/\CY,-).$$

This implies that the left vertical arrow is an isomorphism, as required.

\qed

\section{Compact generation of $\on{D-mod}(\CY)$}  \label{s:2Dmods} 

In this section we will finally prove the result that caused out to write this paper: that for a QCA
algebraic stack $\CY$, the category $\Dmod(\CY)$ is compactly generated. After all the preparations
we have made, the proof will be extremely short. In Sect. \ref{ss:Somecorollaries} we shall establish
some additional favorable properties of the category $\Dmod(\CY)$. 

\medskip

Throughout this section, we will assume that unless specified otherwise, 
all our (pre)stacks are QCA algebraic stacks in the sense of 
Definition~\ref{d:QCA} (in particular, they are quasi-compact).

\ssec{Proof of compact generation}
\begin{thm} \label{t:Dmods}
The category $\on{D-mod}(\CY)$ is compactly generated. More precisely, objects of $\on{D-mod}(\CY)$ of the form
\begin{equation}   \label{e:compacts_in_D(Y)}
\ind_{\on{D-mod}(\CY)} (\CF), \;\quad \CF\in\Coh(\CY)
\end{equation}
are compact and generate $\on{D-mod}(\CY)$.
\end{thm}

\begin{proof}
%We claim that $\on{D-mod}(\CY)$ is compactly generated by the essential image of
%$\Coh(\CY)$ under $\ind_{\on{D-mod}(\CY)}$. 
%
%\medskip
%First, we claim that $\ind_{\on{D-mod}(\CY)}$ sends $\Coh(\CY)$ to $\on{D-mod}(\CY)^c$. Indeed, since 
(i) By \propref{coh is compact}, the objects of $\Coh(\CY)$ are compact in 
$\IndCoh(\CY)$.\footnote{Recall that the proof of this fact is based on formula~\eqref{e:local_Hom} 
and \thmref{main}.} Since $\ind_{\on{D-mod}(\CY)}$ is the left adjoint of a functor that commutes with colimits,
it sends compact objects to compact ones. So objects of the form \eqref{e:compacts_in_D(Y)} are compact.
%$\ind_{\on{D-mod}(\CY)}$ sends $\Coh(\CY)$ to $\on{D-mod}(\CY)^c$.

\medskip

(ii) By Proposition~ \ref{p:coh_generates}, $\Coh(\CY)$  generates $\IndCoh(\CY)$. So
it remains to show that the essential image of $\ind_{\on{D-mod}(\CY)}$ generates
$\on{D-mod}(\CY)$. This follows from the fact that the functor $\oblv_{\on{D-mod}(\CY)}$
is conservative.
\end{proof}

\begin{rem}
Note that, unlike the case of DG schemes, the subcategory $$\Dmod(\CY)^c\subset \Dmod(\CY)$$ is
\emph{not} preserved by the truncation functors. We note that this is also the case for the category
$\QCoh(-)$ on non-regular schemes. By contrast, $\IndCoh(-)^c$ on schemes and QCA algebraic 
stacks is compatible with the t-structure.
\end{rem}

\ssec{Variant of the proof of Theorem~\ref{t:Dmods}}   \label{ss:variant}

For the reader who prefers to avoid the (potentially unfamiliar) category $\IndCoh(\CY)$,
below we give an alternative argument, which does not use $\IndCoh(\CY)$ explicitly. 
Since the assertion is about categorical properties of $\Dmod(\CY)$,
we may assume that $\CY$ is a classical stack (rather than a DG stack). \footnote{The same proof
is applicable when $\CY$ is an eventually coconnective QCA stack.}

\sssec{}

Recall the pair of adjoint functors
$${}'\ind_{\Dmod(\CY)}:\QCoh(\CY)\rightleftarrows \Dmod(\CY):{}'\oblv_{\Dmod(\CY)},$$
see \secref{sss:left induction stacks}, and recall that 
${}'\oblv_{\Dmod(\CY)}$ is conservative. 

\medskip

Hence, by \corref{c:Coh_generates_QCoh}, in order to prove \thmref{t:Dmods}, it is sufficient to show that
the functor ${}'\ind_{\Dmod(\CY)}$ sends $\Coh(\CY)\subset \QCoh(\CY)$ to $\Dmod(\CY)^c$.

\medskip

I.e., we need to show that for $\CF\in\Coh(\CY)$, the functor $\Dmod(\CY)\to\Vect$ defined by
$$\CM\mapsto \CMaps_{\Dmod(\CY)}({}'\ind_{\on{D-mod}(\CY)}(\CF ),\CM),\quad\quad \CM\in \on{D-mod}(\CY)$$
is continuous. 

\medskip

The idea of the proof is the same as that of \propref{coh is compact}: namely, we represent the above functor as a composition of a continuous functor
\begin{equation} \label{e:inHom Dmod}
\Dmod(\CY)\to \QCoh(\CY), \quad \CM\mapsto \underline\Hom_{\QCoh(\CY)}(\CF,{}'\oblv_{\Dmod(\CY)}(\CM))
\end{equation}
and the functor $\Gamma(\CY,-): \QCoh(\CY)\to\Vect$, which is also
continuous by \thmref{main}.

 %The idea of the proof is the same as that of \propref{coh is compact}: namely, we represent the above functor 
  %as a composition
 %$$\Gamma\left(\CY,\underline\Hom_{\QCoh(\CY)}(\CF,{}'\oblv_{\Dmod(\CY)}(\CM))\right),$$
 %where 
 %\begin{equation} \label{e:inHom Dmod}
 %\CM\mapsto \underline\Hom_{\QCoh(\CY)}(\CF,{}'\oblv_{\Dmod(\CY)}(\CM)):\Dmod(\CY)\to \QCoh(\CY)
 %\end{equation}
 %is continuous, while the functor $\Gamma(\CY,-)$ is continuous by \thmref{main}.

\medskip

The functor \eqref{e:inHom Dmod} is the functor of ``internal Hom" from a coherent sheaf to a D-module.
The content of the proof is to show that the latter is well-defined and has the expected properties. 

\sssec{}

We rewrite the expression $\CMaps_{\Dmod(\CY)}({}'\ind_{\on{D-mod}(\CY)}(\CF ),\CM)$ as 
$$\CMaps_{\QCoh(\CY)}(\CF,{}'\oblv_{\Dmod(\CY)}(\CM)).$$

\medskip

We introduce the object $$\underline\Hom'_{\QCoh(\CY)}(\CF,{}'\oblv_{\Dmod(\CY)}(\CM))\in \QCoh(\CY)$$
as follows.

\medskip

Let $\on{Sch}^{\on{aff}}_{/\CY,\on{smooth}}$ denote the full subcategory of the category of affine schemes
over $\CY$, where we restrict objects $(S\in \on{Sch}^{\on{aff}}_{\on{aft}},g:S\to \CY)$ to those for which $g$ is smooth.
We restrict $1$-morphisms to those $f:S'\to S$ for which $f$ is smooth.

\medskip

For $$(S,g)\in (\on{Sch}^{\on{aff}}_{/\CY,\on{smooth}})^{\on{op}},$$
we set 
\begin{multline}   \label{e:inthom}
\Gamma\left(S,g^*(\underline\Hom'_{\QCoh(\CY)}(\CF,{}'\oblv_{\Dmod(\CY)}(\CM)))\right):=\\
=\CMaps_{\QCoh(S)}\left(g^!(\CF),{}'\oblv_{\Dmod(S)}(g^!(\CM))\right).
\end{multline}
Here $g^!$ is well-defined as a functor $\QCoh(\CY)\to \QCoh(S)$ since $g$ is smooth.

\sssec{}

We claim:

\begin{lem} \label{l:inhom well defined}
For $f:S'\to S$ in $\on{Sch}^{\on{aff}}_{/\CY,\on{smooth}}$, the natural map
$$f^*\left(g^*(\underline\Hom'_{\QCoh(\CY)}(\CF,{}'\oblv_{\Dmod(\CY)}(\CM)))\right)\to
(g\circ f)^*\left(\underline\Hom'_{\QCoh(\CY)}(\CF,{}'\oblv_{\Dmod(\CY)}(\CM))\right)$$
is an isomorphism. 
\end{lem}

The proof will be given in \secref{sss:proof of inhom well defined}. The above lemma ensures that the assignment
$$(S,g)\mapsto g^*\left(\underline\Hom'_{\QCoh(\CY)}(\CF,{}'\oblv_{\Dmod(\CY)}(\CM))\right)$$
indeed defines an object of $\QCoh(\CY)$. We denote it by $\underline\Hom'_{\QCoh(\CY)}(\CF,{}'\oblv_{\Dmod(\CY)}(\CM))$. 

\medskip

Since $g^!$ is isomorphic to $g^*$ up to a twist by a line bundle, we have:
\begin{multline*}
\CMaps_{\QCoh(\CY)}(\CF,{}'\oblv_{\Dmod(\CY)}(\CM))\simeq \\
\simeq\underset{(S,g)\in \on{Sch}^{\on{aff}}_{/\CY,\on{smooth}}}{\underset{\longleftarrow}{lim}}\,
\CMaps_{\QCoh(S)}\left(g^*(\CF),g^*({}'\oblv_{\Dmod(\CY)}(\CM))\right)\simeq \\
\simeq 
\underset{(S,g)\in \on{Sch}^{\on{aff}}_{/\CY,\on{smooth}}}{\underset{\longleftarrow}{lim}}\,
\CMaps_{\QCoh(S)}\left(g^!(\CF),g^!({}'\oblv_{\Dmod(\CY)}(\CM))\right)= \\
=\underset{(S,g)\in \on{Sch}^{\on{aff}}_{/\CY,\on{smooth}}}{\underset{\longleftarrow}{lim}}\,
\Gamma\left(S,g^*(\underline\Hom'_{\QCoh(\CY)}(\CF,{}'\oblv_{\Dmod(\CY)}(\CM)))\right)= \\
=\Gamma\left(\CY,\underline\Hom'_{\QCoh(\CY)}(\CF,{}'\oblv_{\Dmod(\CY)}(\CM))\right).
\end{multline*}
Applying \thmref{main}(i), we obtain that it suffices to show that  the functor
$$\CM\mapsto \underline\Hom'_{\QCoh(\CY)}(\CF,{}'\oblv_{\Dmod(\CY)}(\CM))$$
commutes with colimits in $\CM$. 

\begin{rem}
A similar manipulation shows that for $\CF_1\in \QCoh(\CY)$,
\begin{multline*}
\CMaps_{\QCoh(\CY)}(\CF_1\otimes \CF,{}'\oblv_{\Dmod(\CY)}(\CM))\simeq \\
\simeq \CMaps_{\QCoh(\CY)}\left(\CF_1,\underline\Hom'_{\QCoh(\CY)}(\CF,{}'\oblv_{\Dmod(\CY)}(\CM))\right);
\end{multline*}
in other words, $\underline\Hom'_{\QCoh(\CY)}(\CF,{}'\oblv_{\Dmod(\CY)}(\CM))$ \emph{is} the internal Hom
object
$$\underline\Hom_{\QCoh(\CY)}(\CF,{}'\oblv_{\Dmod(\CY)}(\CM)).$$
\end{rem}

\sssec{}
Now let us prove continuity of the functor \eqref{e:inHom Dmod}.
Since for $(S,g)\in \on{Sch}^{\on{aff}}_{/\CY,\on{smooth}}$, 
the functor $g^*$ is continuous, it suffices
to show that for every $(S,g)$ as above, the functor
$$\CM\mapsto g^*(\underline\Hom'_{\QCoh(\CY)}(\CF,{}'\oblv_{\Dmod(\CY)}(\CM)))$$
is continuous. 

\medskip

We rewrite 
\begin{multline} \label{e:on S}
\Gamma\left(S,g^*(\underline\Hom'_{\QCoh(\CY)}(\CF,{}'\oblv_{\Dmod(\CY)}(\CM)))\right)\simeq \\
\simeq \CMaps_{\Dmod(S)}\left({}'\ind_{\on{D-mod}(S)}(g^!(\CF)),g^!(\CM)\right).
\end{multline}
Now, $g^!(\CF)\in \Coh(S)$, and since $S$ is a scheme, the functor ${}'\ind_{\on{D-mod}(S)}$
is known to send $\Coh(S)$ to $\Dmod(S)^c$. This implies that the right-hand side in \eqref{e:on S}
commutes with colimits in $\CM$. 

\qed

\sssec{Proof of \lemref{l:inhom well defined}}   \label{sss:proof of inhom well defined}

This will be parallel to the proof of \lemref{l:inner transition map}. 

\medskip

Let $f:S'\to S$ be a smooth map 
between affine schemes. Let $\CF$ be an object of $\Coh(S)$,
and $\CM$ an object of $\Dmod(S)$. We claim that the natural map
\begin{multline} \label{e:inner Hom D}
H^0\left(f^*\left(\underline\Hom_{\QCoh(S)}(\CF,{}'\oblv_{\Dmod(S)}(\CM))\right)\right)\to \\
\to H^0\left(\underline\Hom_{\QCoh(S')}\left(f^!(\CF),{}'\oblv_{\Dmod(S)}(f^!(\CM))\right)\right).
\end{multline}
is an isomorphism. 

\medskip

Note that the assumption that $f$ is smooth and the fact that the categories $\Dmod(S)$ and $\Dmod(S')$ are
of finite cohomological dimension, imply that both sides in
\eqref{e:inner Hom D} will remain unchanged if we replace $\CM$ by $\tau^{\geq -n}(\CM)$ for $n\gg 0$.

\medskip

Note also that \eqref{e:inner Hom D} is evidently an isomorphism if $\CF\in \QCoh(S)^c=\QCoh(S)^{\on{perf}}$.
Now replace $\CF$ by $\CF_1$, where $\CF_1\in \QCoh(S)^c$ is equipped with a map to $\CF$, such that 
$$\on{Cone}(\CF_1\to \CF)\in \QCoh(S)^{\leq -n}$$
with $n\gg 0$.

\qed

\ssec{Some corollaries of Theorem~\ref{t:Dmods}}   \label{ss:Somecorollaries}

\sssec{}

First we claim:

\begin{cor}   \label{c:compactDmods}
$\on{D-mod}(\CY)^c$ is Karoubi-generated by objects of the form 
$\ind_{\on{D-mod}(\CY)} (\CF)$, $\CF\in\Coh(\CY)$. %\eqref{e:compacts_in_D(Y)}.
\end{cor}

Recall that for a cocomplete DG category $\bC$ and its not necessarily cocomplete DG subcategories
$\bC'_0\subset \bC'$, ones says that a subcategory $\bC'_0$ to Karoubi-generates $\bC'$ if 
the latter is the smallest among DG subcategories of $\bC$ that contain $\bC'_0$ and are closed
under direct summands. This is a condition on corresponding homotopy categories (i.e., it is
insensitive to the $\infty$-category structure). 

\begin{proof}
This follows from \secref{sss:ind-compl}.
\end{proof}

\sssec{}

As yet another corollary of \thmref{t:Dmods}, we obtain:

\begin{cor}  \label{c:D on prod}
Let $\CY$ be a QCA stack and $\CY'$ \emph{any} prestack. Then the natural functor
$$\on{D-mod}(\CY)\otimes \on{D-mod}(\CY')\to \on{D-mod}(\CY\times \CY')$$
is an equivalence.
\end{cor}

\begin{proof}

The proof repeats verbatim that of \corref{c:indcoh on product}. It applies to \emph{any} prestack
$\CY$, for which the category $\Dmod(\CY)$ is dualizable.

\end{proof}

\ssec{Verdier duality on a QCA stack}

\sssec{}

In Sect.~\ref{sss:Verdier_on_stacks} we defined an involutive anti self-equivalence
$$\BD_\CY^{\on{Verdier}}:(\on{D-mod}_{\on{coh}}(\CY))^{\on{op}}\to \on{D-mod}_{\on{coh}}(\CY).$$
\begin{cor}  \label{duality on compact}
This functor induces an involutive anti self-equivalence 
\begin{equation}   \label{e:Verdier-compact_on_stacks}
\BD_\CY^{\on{Verdier}}:(\on{D-mod}(\CY)^c)^{\on{op}}\iso \on{D-mod}(\CY)^c
\end{equation}
\end{cor}

\begin{proof}
The nontrivial statement to prove is that $\BD^{\on{Verdier}}_\CY$ preserves $\on{D-mod}(\CY)^c$.
By Corollary~\ref{c:compactDmods}, it suffices to show that $\BD^{\on{Verdier}}_\CY$ preserves
$\ind_{\on{D-mod}(\CY)}(\Coh(\CY))$. The latter follows from \corref{c:Verdier-Serre_on_stacks}.
%However, on $\Coh(\CY)$ we have:
%$$\BD^{\on{Verdier}}_\CY\circ \ind_{\on{D-mod}(\CY)}\simeq \ind_{\on{D-mod}(\CY)}\circ \BD^{\on{Serre}}_\CY.$$
\end{proof}

\begin{cor}  %\label{self-duality of IndCoh}
The equivalence \eqref{e:Verdier-compact_on_stacks}
%$\BD_\CY^{\on{Verdier}}:(\on{D-mod}(\CY)^c)^{\on{op}}\iso \on{D-mod}(\CY)^c$ 
%provided by Corollary~\ref{duality on compact}
uniquely extends to an equivalence
%$$\on{D-mod}(\CY)^\vee\iso \on{D-mod}(\CY).$$
\begin{equation}  \label{e:self-duality of Dmod}
\bD^{\on{Verdier}}_\CY:\on{D-mod}(\CY)^\vee\iso \on{D-mod}(\CY).
\end{equation}
\end{cor}

\begin{proof}
By Theorem~\ref{t:Dmods}, $\on{D-mod}(\CY)=\Ind(\on{D-mod}(\CY)^c)$. By 
Sect.~\ref{sss:properties of duality}(ii'), this implies
that $\on{D-mod}(\CY)^\vee=\Ind((\on{D-mod}(\CY)^c)^{\on{op}})$, so
$$\on{D-mod}(\CY)^\vee=\Ind((\on{D-mod}(\CY)^c)^{\on{op}})\simeq \Ind(\on{D-mod}(\CY)^c)\simeq \on{D-mod}(\CY).$$
\end{proof}

\sssec{}   \label{sss:pairing on QCA}

According to \secref{sss:dualizabilitydef}, the self-duality given by \eqref{e:self-duality of Dmod} corresponds to a pair
of functors
\begin{equation} \label{e:pairing Dmod}
\epsilon_{\Dmod(\CY)}:\Dmod(\CY)\otimes \Dmod(\CY)\to \Vect
\end{equation}
and
\begin{equation} \label{e:unit Dmod}
\mu_{\Dmod(\CY)}:\Vect\to \Dmod(\CY)\otimes \Dmod(\CY).
\end{equation}

We shall also use the notation $\langle-,-\rangle_{\Dmod(\CY)}$ to denote the functor
$$\Dmod(\CY)\times \Dmod(\CY)\to \Dmod(\CY)\otimes \Dmod(\CY)\overset{\epsilon_{\Dmod(\CY)}}\longrightarrow \Vect.$$

From \lemref{l:Verdier_stacks}, we obtain:

\begin{lem} \label{l:pairing Dmod comp}
For $\CM\in \Dmod(\CY)^c$ and $\CM'\in \Dmod(\CY)$ we have
$$\langle\CM,\CM'\rangle_{\Dmod(\CY)}=\Gamma_{\on{dR}}(\CY,\CM\sotimes \CM').$$
\end{lem}

In \corref{c:pairing Dmod} we shall describe the functor $\epsilon_{\Dmod(\CY)}$ on the entire category
$$\Dmod(\CY)\otimes \Dmod(\CY)\simeq \Dmod(\CY\times \CY)$$ explicitly. 
Furthermore, in \secref{sss:proof unit Dmod} we will prove:

\begin{prop} \label{p:unit Dmod}
The object 
$$\mu_{\Dmod(\CY)}(k)\in \Dmod(\CY)\otimes \Dmod(\CY)\simeq \Dmod(\CY\times \CY)$$
identifies canonically with $(\Delta_\CY)_\dr(\omega_\CY)$.
\end{prop}

\sssec{}

Let now $\pi:\CY_1\to \CY_2$ be a schematic quasi-compact morphism between QCA stacks. Recall the notion of the dual
functor, see \secref{sss:dual functors}. We claim:

\begin{prop}  \label{p:duality on morphisms schematic}
The functors 
$$\pi_\dr:\Dmod(\CY_1)\to \Dmod(\CY_2) \text{ and } \pi^!:\Dmod(\CY_2)\to \Dmod(\CY_1)$$
are related by $(\pi_\dr)^\vee\simeq \pi^!$ in terms of the self-dualities $\bD^{\on{Verdier}}_{\CY_i}:\Dmod(\CY_i)^\vee\simeq \Dmod(\CY_i)$.
\end{prop}

\begin{proof}
It suffices to construct a functorial isomorphism for $\CM_i\in \Dmod(\CY_i)^c$:
$$\langle \CM_2,\pi_\dr(\CM_1)\rangle_{\Dmod(\CY_2)}\simeq 
\langle \pi^!(\CM_2),\CM_1\rangle_{\Dmod(\CY_1)}.$$

By \lemref{l:pairing Dmod comp} we rewrite the left-hand side as 
$$\Gamma_{\on{dR}}(\CY_2,\CM_2\sotimes \pi_\dr(\CM_1)),$$
which by the projection formula \eqref{e:proj formula Dmod sch}
identifies with
$$\Gamma_{\on{dR}}(\CY_2,\pi_\dr(\pi^!(\CM_2)\sotimes \CM_1)).$$

However, by \lemref{l:de Rham trans}, the latter identifies with 
$$\Gamma_{\on{dR}}(\CY_1,\pi^!(\CM_2)\sotimes \CM_1),$$
which in turn identifies with
$\langle \pi^!(\CM_2),\CM_1\rangle_{\Dmod(\CY_1)}$
again by \lemref{l:pairing Dmod comp}.

\end{proof}

\section{Renormalized de Rham cohomology and safety}  \label{s:renormalized}

As we saw in \secref{sss:warning}, for a QCA algebraic stack $\CY$, the functor $\Gamma_{\on{dR}}(\CY,-)$
is not necessarily continuous. In this section we shall introduce a new functor, denoted $\Gamma_\rd(\CY,-)$
that we will refer to as ``renormalized de Rham cohomology". This functor will be continuous, and we will
have a natural transformation
$$\Gamma_\rd(\CY,-)\to \Gamma_{\on{dR}}(\CY,-).$$
We shall also introduce a class of objects on $\Dmod(\CY)$, called \emph{safe}, for which the above natural transformation
is an equivalence.

\medskip

In this section all algebraic stacks will be assumed QCA, unless specified otherwise. 

\ssec{Renormalized de Rham cohomology}   \label{ss:ren de Rham}

Recall the notion of the dual functor from \secref{sss:dual functors}. 

\begin{defn} \label{d:ren dr}
For a QCA algebraic stack $\CY$ we define the \emph{continuous} functor 
$$\Gamma_\rd(\CY,-):\Dmod(\CY)\to \Vect$$
to be the dual of 
$$\pi_\CY^!:\Vect\to \Dmod(\CY),\quad k\mapsto \omega_\CY$$
under the identifications 
$$\bD^{\on{Verdier}}_\CY:\Dmod(\CY)^\vee\simeq \Dmod(\CY) \text{ and }\Vect^\vee\simeq \Vect.$$
\end{defn}

Note that if $\CY$ is a scheme $Z$, by \eqref{e:duality_formula}, we have $\Gamma_{\rd}(Z,-)\simeq \Gamma_{\on{dR}}(Z,-)$.

\begin{rem}
Presumably, the functor analogous to $\Gamma_{\rd}(Z,-)$ can be defined in other
sheaf-theoretic situations, e.g., for the derived category of sheaves with constructible cohomologies
for stacks over the field of complex numbers.
\end{rem}

\sssec{}

Here is a more explicit description of the functor $\Gamma_{\rd}(Z,-)$. 

\begin{lem} \label{l:renormalized de Rham as ind-extension}
The functor $\Gamma_\rd(\CY,-)$ (see \secref{sss:pairing on QCA} for the notation)
is canonically isomorphic to
the ind-extension of the functor
$$\Gamma_{\on{dR}}(\CY,-)|_{\Dmod(\CY)^c}:\Dmod(\CY)^c\to \Vect.$$
\end{lem}

\begin{proof}
We only have to show that the pairing $\langle-,-\rangle_{\Dmod(\CY)}$
corresponding to the self-duality of $\Dmod(\CY)$ satisfies
$$\langle \CM,p_\CY^!(k)\rangle_{\Dmod(\CY)}\simeq \Gamma_{\on{dR}}(\CY,\CM)$$
for $\CM\in \Dmod(\CY)^c$. However, this is immediate from \lemref{l:pairing Dmod comp}.
\end{proof}

\begin{cor} \label{c:nat trans to ren}
There exists a canonically defined natural transformation
\begin{equation} \label{e:nat trans to ren}
\Gamma_\rd(\CY,-)\to \Gamma_{\on{dR}}(\CY,-),
\end{equation}
which is an isomorphism when restricted to compact objects.
\end{cor}

In general, the failure of the natural transformation \eqref{e:nat trans to ren} to be
an isomorphism is a measure to which the functor $\Gamma_{\on{dR}}(\CY,-)$
fails to be continuous.

\begin{example}   \label{ex:ren for BG}

As an illustration, let us compute the functor $\Gamma_{\rd}(\CY,-)$ for $\CY=BG$, see \secref{ss:BG}.
Let $B$ be as in \eqref{e:homology of group}. We saw in {\it loc.cot.} that the functor $\Gamma_{\on{dR}}(BG,-)$ 
is given by $\CMaps_{B\mod}(k,-)$.

\medskip

We claim now that the functor $\Gamma_{\rd}(BG,-)$ is given by
$$M\mapsto k\underset{B}\otimes M[-2\dim(G)+\delta],$$
where $\delta$ is the degree of the highest cohomology group
of $\Gamma_{\on{dR}}(G,k_G)$.

\medskip

Explicitly,
$$
\begin{cases}
&\delta=0, \text{ if $G$ is unipotent}; \\
&\delta=\dim(G), \text{ if $G$ is reductive}; \\
&\delta=2\dim(G), \text{ if $G$ is an abelian variety}.
\end{cases}
$$

\medskip

Recall that $\sigma$ denotes the map $\on{pt}\to BG$, and recall that 
$\sigma_!(k)$ is a compact generator of $\Dmod(BG)$. Hence, it suffices to show that 
$$\Gamma_{\on{dR}}(BG,\sigma_!(k))\simeq k[-2\dim(G)+\delta],$$
as modules over $B\simeq \CMaps_{\Dmod}(\sigma_!(k),\sigma_!(k))$. 

\medskip

Note that $$\sigma_!(k)\simeq \sigma_\dr(k)[-2\dim(G)+\delta],$$
so the required assertion follows from the isomorphism
$$\Gamma_{\on{dR}}(BG,\sigma_\dr(k))\simeq \Gamma_{\on{dR}}(\on{pt},k)=k.$$

\end{example}

\begin{example}  \label{ex:ren dr of ind}

We claim that the functor 
$$\Gamma_\rd(\CY,-)\circ \ind_{\Dmod(\CY)}$$ identifies canonically with
$\Gamma^{\IndCoh}(\CY,-)$. 

\medskip

Both functors are continuous, so it is enough to construct the isomorphism
on the subcategory $\Coh(\CY)\subset \IndCoh(\CY)$. In the latter case the assertion follows from
\lemref{l:renormalized de Rham as ind-extension} and \propref{p:de Rham and ind}.

\medskip

Moreover, we obtain that the natural transformation \eqref{e:nat trans to ren} induces an isomorphism
$$\Gamma_\rd(\CY,-)\circ \ind_{\Dmod(\CY)}\to \Gamma_{\on{dR}}(\CY,-)\circ \ind_{\Dmod(\CY)}.$$
As we shall see shortly, the latter isomorphism is a general phenomenon that holds for all
\emph{safe} objects of $\Dmod(\CY)$.

\end{example}

\ssec{Safe objects of $\Dmod(\CY)$}

\begin{defn}
An object $\CM\in \on{D-mod}(\CY)$ is said to be \emph{safe} if the functor
$$\CM'\mapsto \Gamma_{\on{dR}}(\CY,\CM\sotimes \CM'):\on{D-mod}(\CY)\to \Vect$$
is continuous. 
\end{defn}
%\sssec{}

Its is clear that safe objects of $\on{D-mod}(\CY)$ form a (non-cocomplete) DG subcategory
(i.e., the condition of being safe survives taking cones). 

\medskip

It is also clear that the subcategory of safe objects in $\on{D-mod}(\CY)$ is a tensor ideal
with respect to $\sotimes$. Indeed, if $\CM$ is safe, then so are all $\CM\sotimes \CM'$.

\sssec{}

The notion of safety is what allows us to distinguish compact objects among the larger
subcategory $\CM\in \on{D-mod}_{\on{coh}}(\CY)$:

\begin{prop}  \label{compactness via safety}
%An object $\CF\in \on{D-mod}_{\on{coh}}(\CY)$ is compact if and only if it is safe.
Then the following properties of an object $\CM\in \on{D-mod}_{\on{coh}}(\CY)$ are equivalent:
\begin{enumerate}
\item[(a)] $\CM$ is compact;

\item[(b)]  $\CM$ is safe;

\item[(c)]  $\BD^{\on{Verdier}}_\CY (\CM)$ is safe.
\end{enumerate}
\end{prop}

\begin{proof}
By Lemma~\ref{l:Verdier_stacks}, (a) is equivalent to (c). So (b) is equivalent to the
compactness of $\BD^{\on{Verdier}}_\CY (\CF )$. The latter is equivalent to (a) by
\corref{duality on compact} (it is here that we use that $\CY$ is QCA).
%Lemma~\ref{l:Verdier_stacks} immediately implies that an object $\CF\in \on{D-mod}_{\on{coh}}(\CY)$ 
%is compact if and only if $\BD^{\on{Verdier}}_\CY (\CF )$ is safe.
%It remains to use  \corref{duality on compact}.
\end{proof}

Note, however, safe objects do not have to be coherent or cohomologically bounded: 

\begin{example}  \label{ex:induced is safe}

We claim that all objects of the form $\ind_{\Dmod(\CY)}(\CF)$, $\CF\in \IndCoh(\CY)$, are safe. Indeed, by
\lemref{l:tensor with induction stacks} and \propref{p:de Rham and ind}, for $\CM\in \Dmod(\CY)$ 
$$\Gamma_{\on{dR}}(\CY,\ind_{\Dmod(\CY)}(\CF)\sotimes \CM)\simeq
\Gamma^{\IndCoh}(\CY,\CF\sotimes \oblv_{\Dmod(\CY)}(\CM)),$$
and the latter functor is continuous.

\end{example}

\sssec{}

The following will be useful in the sequel:

\begin{lem}  \label{l:safety preserved schematic pullback}
Let $\pi:\CY_1\to \CY_2$ be schemtaic. If $\CM_2\in \Dmod(\CY_2)$ is safe, then so is
$\pi^!(\CY_2)\in \Dmod(\CY_1)$.
\end{lem}

\begin{proof}
We need to show that the functor
$$\CM_1\mapsto \Gamma_{\on{dR}}(\CY_1,\pi^!(\CM_2)\sotimes \CM_1)$$
commutes with colimits. By \lemref{l:de Rham trans}, the latter expression
can be rewritten as
$$\Gamma_{\on{dR}}(\CY_2,\pi_\dr(\pi^!(\CY_2)\sotimes \CM_1)).$$

Now, since $\pi$ is schematic and quasi-compact, the projection formula and \eqref{e:proj formula Dmod sch} is applicable,
and we can rewrite the latter expression as
$$\Gamma_{\on{dR}}(\CY_2,\CM_2\sotimes \pi_\dr(\CM_1)).$$

Now, the required assertion follows from the fact that the functor $\pi_\dr$
commutes with colimits.

\end{proof}

\begin{rem}
In \lemref{l:safety under functors} we will extend the assertion of the above lemma to the case when $\pi$
is not necessarily schematic, but merely \emph{safe}. However, the lemma obviously
fails for general morphisms: consider, e.g., $B\BG_m\to \on{pt}$.
\end{rem}

\sssec{De Rham cohomology of safe objects}

The following proposition is crucial for the sequel:

\begin{prop}  \label{p:pairing w safe}
Let $\CM_1\in \Dmod(\CY)$ be safe. Then for any $\CM_2\in \Dmod(\CY)$, the natural transformation
\eqref{e:nat trans to ren} induces an isomorphism
$$\Gamma_\rd(\CY,\CM_1\sotimes \CM_2)\to \Gamma_{\on{dR}}(\CY,\CM_1\sotimes \CM_2).$$
\end{prop}

\begin{proof}

By \lemref{l:renormalized de Rham as ind-extension}, we have:
\begin{equation} \label{e:dr of ten 1}
\tau^{\leq 0}\left(\Gamma_\rd(\CY,\CM_1\sotimes \CM_2)\right)\simeq 
\underset{\CM\in \Dmod(\CY)^c{}_{/\CM_1\sotimes \CM_2}}{\underset{\longrightarrow}{colim}}\, \tau^{\leq 0}\left(\Gamma_{\on{dR}}(\CY,\CM)\right).
\end{equation}

Using the fact that 
$$\tau^{\leq 0}\left(\Gamma_{\on{dR}}(\CY,\CM)\right)\simeq \on{Maps}_{\Dmod(\CY)}(k_\CY,\CM),$$
we can rewrite \eqref{e:dr of ten 1} as the co-end of the functors
$$\CM\mapsto \on{Maps}_{\Dmod(\CY)}(\CM,\CM_1\sotimes \CM_2) \text{ and } \CM\mapsto \on{Maps}_{\Dmod(\CY)}(k_\CY,\CM)$$
out of $\Dmod(\CY)^c$. Using the Verdier duality anti-equivalence of $\Dmod(\CY)^c$, we rewrite the above co-end as
the co-end of the functors
$$\CM'\mapsto \tau^{\leq 0}\left(\Gamma_{\on{dR}}(\CY,\CM_1\sotimes \CM_2\sotimes \CM')\right) \text{ and }
\CM'\mapsto \on{Maps}_{\Dmod(\CY)}(\CM',\omega_\CY),$$
as functors out of $(\Dmod(\CY)^c)^{\on{op}}$.

\medskip

The latter co-end can be rewritten as 

\begin{equation} \label{e:dr of ten 2}
\underset{\CM'\in \Dmod(\CY)^c_{/\omega_\CY}}{\underset{\longrightarrow}{colim}}\, \tau^{\leq 0}(\Gamma_{\on{dR}}
(\CY,\CM_1\sotimes \CM_2\sotimes \CM')).
\end{equation}

However, tautologically,
$$\underset{\CM'\in \Dmod(\CY)^c_{/\omega_\CY}}{\underset{\longrightarrow}{colim}}\, \CM'\simeq \omega_\CY,$$
and hence
$$\underset{\CM'\in \Dmod(\CY)^c_{/\omega_\CY}}{\underset{\longrightarrow}{colim}}\, \CM_2\sotimes \CM'\simeq \CM_2.$$

Hence, the assumption that $\Gamma_{\on{dR}}(\CY,\CM_1\sotimes -)$ commutes with colimits implies that
the expression in \eqref{e:dr of ten 2} maps isomorphically to $\tau^{\leq 0}\left(\Gamma_{\on{dR}}(\CY,\CM_1\sotimes \CM_2)\right)$,
as required.

\end{proof}

As a particular case, we obtain:

\begin{cor}  \label{c:dr and rd of safe}
If $\CM\in \Dmod(\CY)$ is safe, the natural transformation
\eqref{e:nat trans to ren} induces an isomorphism
$$\Gamma_\rd(\CY,\CM)\to \Gamma_{\on{dR}}(\CY,\CM).$$
\end{cor}

In addition:

\begin{cor}  \label{c:crit for safety}
An object $\CM\in \Dmod(\CY)$ is safe if and only if the natural transformation
\eqref{e:nat trans to ren} induces an isomorphism
$$\Gamma_\rd(\CY,\CM\sotimes \CM')\to \Gamma_{\on{dR}}(\CY,\CM\sotimes \CM')$$
for any $\CM'\in \Dmod(\CY)$.
\end{cor}

Combining \propref{p:pairing w safe} with \propref{compactness via safety}, we obtain:

\begin{cor} \label{c:pairing on compact}
If one of the objects $\CM'$ or $\CM''$ is compact, then the map
$$\Gamma_\rd(\CY,\CM'\sotimes \CM'')\to \Gamma_{\on{dR}}(\CY,\CM'\sotimes \CM'')$$
is an isomorphism.
\end{cor}

\sssec{}

The notion of safe object allows to give a more explicit description of the pairing 
$\langle-,-\rangle_{\Dmod(\CY)}$:

\begin{lem}  \label{l:descr of pairing}
For $\CM',\CM''\in \Dmod(\CY)$, the natural map 
$$\Gamma_\rd(\CY,\CM'\sotimes \CM'')\to \langle \CM',\CM''\rangle_{\Dmod(\CY)}$$
is an isomorphism.
\end{lem}

\begin{proof}
By definition, both functors in the corollary are continuous, so it is enough to verify the assertion for
$\CM'$ and $\CM''$ compact. By definition, the right-hand side is
$\Gamma_{\on{dR}}(\CY,\CM'\sotimes \CM'')$. So, the assertion follows from \corref{c:pairing on compact}.
\end{proof}

\begin{cor} \label{c:pairing Dmod}
The functor
$$\epsilon_{\Dmod(\CY)}:\Dmod(\CY)\otimes \Dmod(\CY)\to \Vect$$
identifes canonically with
$$\Dmod(\CY)\otimes \Dmod(\CY)\simeq \Dmod(\CY\times \CY)\overset{\Delta_\CY^!}\longrightarrow \Dmod(\CY)\overset{\Gamma_{\rd}(\CY,-)}
\longrightarrow \Vect.$$
\end{cor}

\sssec{Proof of \propref{p:unit Dmod}}  \label{sss:proof unit Dmod}

The functor
$$\mu_{\Dmod(\CY)}:\Vect\to \Dmod(\CY)\otimes \Dmod(\CY)\simeq \Dmod(\CY\times \CY)$$
is the dual of the functor $\epsilon_{\Dmod(\CY)}$ under the identifications
$$\Vect^\vee\simeq \Vect  \text{ and }  
\bD_{\CY\times \CY}^{\on{Verdier}}:\Dmod(\CY\times \CY)^\vee\simeq \Dmod(\CY\times \CY).$$

Hence, the required assertion follows from \corref{c:pairing Dmod} and \propref{p:duality on morphisms schematic}.

\qed

\ssec{The relative situation}  \label{ss:ren direct image}

\sssec{}

Let $\pi:\CY_1\to \CY_2$ be a map between QCA algebraic stacks, and consider
the functor $\pi^!:\Dmod(\CY_2)\to \Dmod(\CY_1)$. 

\begin{defn} \label{d:ren dir im}
We define the \emph{continuous} functor 
$$\pi_{\blacktriangle}:\Dmod(\CY_1)\to \Dmod(\CY_2)$$
to be the dual of $\pi^!$ under the identifications $\Dmod(\CY_i)^\vee\simeq \Dmod(\CY_i)$.
\end{defn}

We shall refer to the functor $\pi_{\blacktriangle}$ as the ``renormalized direct image". 

\medskip

Note that \propref{p:duality on morphisms schematic} implies that if $\pi$ is schematic, then $\pi_{\blacktriangle}\simeq \pi_\dr$. 

\sssec{}

It follows from the construction that the assignment $\pi\rightsquigarrow \pi_{\blacktriangle}$ is compatible
with compositions, i.e., for 
$$\CY_1\overset{\pi}\longrightarrow \CY_2\overset{\phi}\longrightarrow \CY_3$$
there exists a canonical isomorphism
$$\phi_{\blacktriangle}\circ \pi_{\blacktriangle}\simeq (\phi\circ \pi)_{\blacktriangle}.$$
Indeed, this isomorphism follows by duality from $\pi^!\circ \phi^!\simeq (\phi\circ \pi)^!$.

\sssec{}

We claim that the functor $\pi_{\blacktriangle}$ satisfies the projection formula \emph{by definition}: 

\begin{lem} \label{l:proj formula for renormalized}
For $\CM_1\in \Dmod(\CY_1)$ and $\CM_2\in \Dmod(\CY_2)$ we have a canonical
isomorphism
$$\pi_{\blacktriangle}(\CM_1)\sotimes \CM_2\simeq \pi_{\blacktriangle}(\CM_1\sotimes \pi^!(\CM_2)),$$
functorial in $\CM_i\in \Dmod(\CY_i)$.
\end{lem}

\begin{proof}

It suffices to construct a functorial isomorphism
$$\langle \pi_{\blacktriangle}(\CM_1)\sotimes \CM_2,\CM'_2\rangle_{\Dmod(\CY_2)}\simeq
\langle \pi_{\blacktriangle}(\CM_1\sotimes \pi^!(\CM_2)),\CM'_2\rangle_{\Dmod(\CY_2)}$$
functorial in $\CM_2,\CM'_2\in \Dmod(\CY_2)$, $\CM_1\in \Dmod(\CY_1)$.

\medskip

By \lemref{l:descr of pairing},
\begin{multline*}
\langle \pi_{\blacktriangle}(\CM_1)\sotimes \CM_2,\CM'_2\rangle_{\Dmod(\CY_2)}
\simeq \Gamma_\rd(\CY_2,\pi_{\blacktriangle}(\CM_1)\sotimes \CM_2\sotimes \CM'_2)\simeq \\
\simeq \langle \pi_{\blacktriangle}(\CM_1),\CM_2\sotimes\CM'_2)\rangle_{\Dmod(\CY_2)}.
\end{multline*}

By the definition of $\pi_{\blacktriangle}$, the latter identifies with
$$\langle \CM_1,\pi^!(\CM_2\sotimes\CM'_2)\rangle_{\Dmod(\CY_1)}.$$

Again, by \lemref{l:descr of pairing}, the latter expression can be rewritten as
$$\Gamma_\rd(\CY_1,\CM_1\sotimes \pi^!(\CM_2\sotimes \CM'_2))\simeq
\langle \CM_1\sotimes \pi^!(\CM_2),\pi^!(\CM'_2)\rangle_{\Dmod(\CY_1)},$$
and again by the definition of $\pi_{\blacktriangle}$, further as
$$\langle \pi_{\blacktriangle}(\CM_1\sotimes \pi^!(\CM_2)),\CM'_2\rangle_{\Dmod(\CY_2)},$$
as required.

\end{proof}

\sssec{Calculating $\pi_{\blacktriangle}$}

It turns out that safe objects are adjusted to calculating the functor $\pi_{\blacktriangle}$:

\begin{prop} \label{p:ren dir image of safe}
There is a canonical natural transformation 
\begin{equation} \label{e:nat trans to ren rel}
\pi_{\blacktriangle}\to \pi_\dr,
\end{equation}
which is an isomorphism when evaluated on safe objects. 
\end{prop}

\begin{proof}

We need to show that for $\CM_1\in \Dmod(\CY_1)$ and 
$\CM_2\in \Dmod(\CY_2)^c$ there exists a canonical map
\begin{equation} \label{e:calc dir image one}
\langle \CM_1,\pi^!(\CM_2)\rangle_{\Dmod(\CY_1)}\to \langle \pi_\dr(\CM_1),\CM_2\rangle_{\Dmod(\CY_2)},
\end{equation}
which is an isomorphism if $\CM_1$ is safe.

\medskip

By \lemref{l:descr of pairing} the left-hand side in \eqref{e:calc dir image one} identifies with
$$\Gamma_{\rd}(\CY_1,\CM_1\sotimes \pi^!(\CM_2)).$$
The latter expression maps to 
$$\Gamma_{\on{dR}}(\CY_1,\CM_1\sotimes \pi^!(\CM_2)),$$
and by \propref{p:pairing w safe}, this map is an isomorphism 
if $\CM_1$ is safe.

\medskip

By \lemref{l:descr of pairing} and \corref{c:pairing on compact}, the right-hand side
in \eqref{e:calc dir image one} identifies with
$$\Gamma_{\on{dR}}(\CY_2,\pi_\dr(\CM_1)\sotimes\CM_2).$$

Thus, we obtain the following diagram of maps
\begin{multline*}
\langle \CM_1,\pi^!(\CM_2)\rangle_{\Dmod(\CY_1)}\simeq
\Gamma_{\rd}(\CY_1,\CM_1\sotimes \pi^!(\CM_2))\to \\
\to \Gamma_{\on{dR}}(\CY_1,\CM_1\sotimes \pi^!(\CM_2))\simeq  
\Gamma_{\on{dR}}(\CY_2,\pi_\dr(\CM_1\sotimes \pi^!(\CM_2)))\leftarrow \\
\leftarrow \Gamma_{\on{dR}}(\CY_2,\pi_\dr(\CM_1)\sotimes\CM_2)\simeq 
\langle \pi_\dr(\CM_1),\CM_2\rangle_{\Dmod(\CY_2)},
\end{multline*}
where the left-pointing arrow comes from \eqref{e:map for proj formula Dmod}.
The assertion of the proposition follows now from \lemref{l:partial projection formula}.

\end{proof}

\begin{cor} \label{c:renormalized as ind-extension}
The functor $\pi_{\blacktriangle}$ is canonically isomorphic to the
ind-extension of the functor
$$\pi_\dr|_{\Dmod(\CY_1)^c}:\Dmod(\CY_1)^c\to \Dmod(\CY_2).$$
\end{cor}

\sssec{}

As another corollary of \propref{p:ren dir image of safe}, we obtain:

\begin{cor}  \label{c:partial proj formula safe}
For $\CM_1\in \Dmod(\CY_1)$ and $\CM_2\in \Dmod(\CY_2)$, the map 
$$\pi_\dr(\CM_1)\sotimes\CM_2\to \pi_\dr(\CM_1\sotimes\pi^!(\CM_2))$$
of \eqref{e:map for proj formula Dmod} is an isomorphism provided that $\CM_1$
is safe.
\end{cor}

\begin{proof}
We have a commutative diagram
$$
\CD
\pi_{\blacktriangle}(\CM_1)\sotimes\CM_2  @>>>  \pi_{\blacktriangle}(\CM_1\sotimes\pi^!(\CM_2))  \\
@VVV   @VVV  \\
\pi_\dr(\CM_1)\sotimes\CM_2  @>>>  \pi_\dr(\CM_1\sotimes\pi^!(\CM_2)),
\endCD
$$
in which the upper horizontal arrow is an isomorphism by \lemref{l:proj formula for renormalized}.
Now, the vertical arrows are isomorphisms by \propref{p:ren dir image of safe}.
\end{proof}

\sssec{Base change for the renormalized direct image}

Consider a Cartesian diagram of QCA algebraic stacks:
$$
\CD
\CY'_1  @>{\phi_1}>> \CY_1 \\
@V{\pi'}VV   @VV{\pi}V  \\
\CY'_2  @>{\phi_2}>> \CY_2
\endCD
$$

We claim that there exists a canonical natural transformation
\begin{equation} \label{e:base change for ren Dmod}
\pi'_{\blacktriangle}\circ \phi_1^!\to \phi_2^!\circ \pi_{\blacktriangle}.
\end{equation}

Indeed, both functors being continuous, it is enough to construct
the morphism in question on $\Dmod(\CY_1)^c$.  By  \eqref{e:base change morphism Dmod}
for any $\CM\in \Dmod(\CY_1)$ we have a map 
$$\phi_2^!\circ \pi_\dr(\CM)\to \pi'_\dr\circ \phi_1^!(\CM).$$
Moreover, by \propref{p:coconnective part Dmod}, this map is 
an isomorphism whenever $\CM\in \Dmod(\CY_1)^+$. 

\medskip

Thus, for $\CM\in \Dmod(\CY_1)^c$ we have
$$\phi_2^!\circ \pi_{\blacktriangle}(\CM)\simeq \phi_2^!\circ \pi_\dr(\CM)\simeq 
\pi'_\dr\circ \phi_1^!(\CM),$$
and the latter receives a map from $\pi'_{\blacktriangle}\circ \phi_1^!(\CM)$.

\medskip

We now claim:

\begin{prop} \label{p:base change for ren Dmod}
The map \eqref{e:base change for ren Dmod} is an isomorphism.
\end{prop}

\begin{proof}

By transitivity and the definition of $\Dmod(\CY'_2)$, it is enough to show that the map
in question is an isomorphism when $\CY'_2$ is a DG scheme. In particular, in this
case the morphism $\phi_2$, and hence $\phi_1$, is schematic. However, in this
case for $\CM\in \Dmod(\CY_1)^c$, the object
$$\phi_1^!(\CM)\in \Dmod(\CY'_1)$$
is safe, by \lemref{l:safety preserved schematic pullback}. Therefore, the map
$$\pi'_{\blacktriangle}\circ \phi_1^!(\CM)\to \pi'_\dr\circ \phi_1^!(\CM),$$
used in the construction of \eqref{e:base change for ren Dmod}, 
is an isomorphism, by \propref{p:ren dir image of safe}.

\end{proof}

\begin{rem}
Using \eqref{e:base change for ren Dmod} one can extend the functor $\pi_{\blacktriangle}$
to arbitrary QCA morphisms $\pi:\CY_1\to \CY_2$ between prestacks in a way compatible with base change.
\end{rem}

In the course of the proof of \propref{p:base change for ren Dmod} we have also established:

\begin{cor}  \label{c:safe satisfies base change}
If $\CM_1\in \Dmod(\CY_1)$ is safe, then the map $\phi_2^!\circ \pi_\dr(\CM_1)\to \pi'_\dr\circ \phi_1^!(\CM_1)$
of \eqref{e:base change morphism Dmod} is an isomorphism.
\end{cor}

\sssec{Renormalized direct image of induced D-modules}

Generalizing Example \ref{ex:ren dr of ind}, we claim: 

\begin{prop}  \label{p:ren dir image and induction}
There exists a canonical isomorphism of functors
$$\pi_{\blacktriangle}\circ \ind_{\Dmod(\CY_1)}\simeq \ind_{\Dmod(\CY_2)}\circ \pi_*^{\IndCoh}.$$
\end{prop}

\begin{proof}
Since both functors are continuous, it suffices to show that for $\CF_1\in \IndCoh(\CY_1)^c$ and $\CM_2\in \Dmod(\CY_2)^c$, there 
exists a canonical isomorphism
$$\langle \pi_{\blacktriangle}(\ind_{\Dmod(\CY_1)}(\CF_1)),\CM_2\rangle_{\Dmod(\CY_2)}\simeq 
\langle \ind_{\Dmod(\CY_2)}(\pi_*^{\IndCoh}(\CF_1)),\CM_2\rangle_{\Dmod(\CY_2)}.$$

By the definition of $\pi_{\blacktriangle}$, \lemref{l:descr of pairing} and \corref{c:pairing on compact}, 
the left-hand side identifies with
$$\Gamma_{\on{dR}}(\CY_1,\ind_{\Dmod(\CY_1)}(\CF_1)\sotimes \pi^!(\CM_2)),$$
while the right-hand side identifies with
$$\Gamma_{\on{dR}}(\CY_2,\ind_{\Dmod(\CY_2)}(\pi_*^{\IndCoh}(\CF_1))\sotimes\CM_2).$$
Using \lemref{l:tensor with induction stacks}, the two expressions can be rewritten as 
$$\Gamma_{\on{dR}}\left(\CY_1,\ind_{\Dmod(\CY_1)}(\CF_1\sotimes \pi^!(\oblv_{\Dmod(\CY_2)}(\CM_2)))\right)$$
and
$$\Gamma_{\on{dR}}\left(\CY_2,
\ind_{\Dmod(\CY_2)}(\pi^{\IndCoh}_*(\CF_1)\sotimes \oblv_{\Dmod(\CY_1)}(\CM_2))\right),$$
respectively, and further, using \propref{p:de Rham and ind}
as
$$\Gamma^{\IndCoh}\left(\CY_1,\CF_1\sotimes \pi^!(\oblv_{\Dmod(\CY_2)}(\CM_2))\right)$$
and 
$$\Gamma^{\IndCoh}\left(\CY_2,\pi^{\IndCoh_*}(\CF_1)\sotimes \oblv_{\Dmod(\CY_1)}(\CM_2)\right),$$
respectively.

\medskip

Now, the required isomorphism follows from \propref{p:duality of * and !}.

\end{proof}

\sssec{}    \label{sss:proof pushforward induced}

We are now ready to prove \propref{p:pushforward induced} stated in \secref{ss:de Rham dir image stacks}. Indeed, 
it follows by combining Propositions \ref{p:ren dir image and induction}, \ref{p:ren dir image of safe}
and Example \ref{ex:induced is safe}.

\qed

\ssec{Cohomological amplitudes}

\sssec{}

Let us note that by \lemref{p:coconnective part Dmod}, the functor
$\pi_\dr$ is \emph{left t-exact up to a cohomological shift}. We claim that the functor
$\pi_{\blacktriangle}$ exhibits an opposite behavior:

\begin{prop}  \label{p:ren dir image estimate}
There exists an integer $m$ such that $\pi_{\blacktriangle}$  sends
$$\Dmod(\CY_1)^{\leq 0}\to \Dmod(\CY_2)^{\leq m}.$$
\end{prop}

\begin{proof}

It is clear from the $(\ind_{\Dmod},\oblv_{\Dmod})$ adjunction that
for any algebraic stack $\CY$, the category $\Dmod(\CY)^{\leq 0}$
is generated under the operation of taking filtered colimits by
objects of the form $\ind_{\Dmod(\CY)}(\CF)$ for $\CF\in \IndCoh(\CY)^{\leq 0}$.

\medskip

Hence, by \propref{p:ren dir image and induction}, it suffices to show that
there exists an integer $m$, such that $\pi_*^{\IndCoh}$ sends 
$$\IndCoh(\CY_1)^{\leq 0}\to \IndCoh(\CY_2)^{\leq m}.$$

\medskip

Recall that the functor $\Psi_{\CY}$ induces an equivalence $\IndCoh(\CY)^+\to \QCoh(\CY)^+$
for any algebraic stack $\CY$. Therefore, it suffices to show that there exists an integer $m$, 
such that $\pi_*$ sends 
$$\QCoh(\CY_1)^{\leq 0}\to \QCoh(\CY_2)^{\leq m}.$$
However, this follows from \corref{c:relative}(ii).

\end{proof}

\sssec{}

Let us observe that the safety of an object $\CM\in \Dmod(\CY)^b$ makes both functors $$\Gamma_{\on{dR}}(\CY,\CM\sotimes -) \text{ and }
\Gamma_\rd(\CY,\CM\sotimes -)$$ cohomologically bounded. More precisely:

\begin{lem} \label{l:estimates on safety only if} \hfill

\smallskip

\noindent{\em(a)} Let $\CM$ be an safe object of $\Dmod(\CY)^-$. Then 
the functor $$\CM_1\mapsto \Gamma_{\on{dR}}(\CY,\CM\sotimes \CM_1)$$ is right t-exact 
up to a cohomological shift. The estimate on the shift depends only on $\CY$ and the integer
$m$ such that $\CM\in \Dmod(\CY)^{\leq m}$.

\smallskip

\noindent{\em(b)} Let $\CM$ be a safe object of $\Dmod(\CY)^+$. Then 
the functor $$\CM_1\mapsto \Gamma_\rd(\CY,\CM\sotimes \CM_1)$$ is left t-exact up to
a cohomological shift. The estimate on the shift depends only on $\CY$ and the integer
$m$ such that $\CM\in \Dmod(\CY)^{\geq -m}$.

\end{lem}

\begin{proof}

It is easy to see that on \emph{any} quasi-compact algebraic stack, the functor $\sotimes$
is both left and right t-exact up to a cohomological shift. Both assertions follow from the
fact that if $\CM$ is safe, 
$$\Gamma_\rd(\CY,\CM\sotimes \CM_1)\simeq \Gamma_{\on{dR}}(\CY,\CM\sotimes \CM_1)$$
(by \corref{c:crit for safety}), using \lemref{p:coconnective part Dmod} and \propref{p:ren dir image estimate}, respectively.

\end{proof}

We now claim that the above lemma admits a partial converse:

\begin{prop}  \label{p:estimates on safety} \hfill

\smallskip

\noindent{\em(a)} Let $\CM$ be an object of $\Dmod(\CY)^b$. Then it is safe if 
the functor $$\CM_1\mapsto \Gamma_{\on{dR}}(\CY,\CM\sotimes \CM_1)$$ is right t-exact 
up to a cohomological shift. 

\smallskip

\noindent{\em(b)} Let $\CM$ be an object of $\Dmod(\CY)^b$. Then it is safe if 
the functor $$\CM_1\mapsto \Gamma_\rd(\CY,\CM\sotimes \CM_1)$$ is left t-exact up to
a cohomological shift. 

\end{prop}

Combing this with \lemref{l:estimates on safety only if}, we obtain:

\begin{cor}
Let $\CM_i$ be a (possibly infinite) collection of safe objects of $\Dmod(\CY)$ that are
contained in $\Dmod(\CY)^{\geq -m,\leq m}$ for some $m$. Then $\underset{i}\oplus\, \CM_i$
is also safe.
\end{cor}

\begin{proof}[Proof of \propref{p:estimates on safety}]

To prove point (a), it suffices to show that the functor
$$\CM_1\mapsto H^0\left(\Gamma_{\on{dR}}(\CY,\CM\sotimes \CM_1)\right)$$ 
commutes with direct sums. Let $k$ be the integer such that the functor
$$\CM_1\mapsto \Gamma_{\on{dR}}(\CY,\CM\sotimes \CM_1)$$
sends $\Dmod(\CY)^{\leq 0}\to \Vect^{\leq k}$. Let $d$ be an integer such that
$\sotimes$ sends 
$$\Dmod(\CY)^{\leq 0}\times \Dmod(\CY)^{\leq 0}\to \Dmod(\CY)^{\leq d}.$$

\medskip

For a family of objects $\alpha\mapsto \CM_1^\alpha$, consider the following diagram in which the columns
are parts of long exact sequences:

$$
\CD
\underset{\alpha}\oplus\, H^0\left(\Gamma_{\on{dR}}(\CY,\CM\sotimes \tau^{<-k-d}(\CM_1^\alpha))\right)  
@>>> H^0\left(\Gamma_{\on{dR}}(\CY,\CM\sotimes (\underset{\alpha}\oplus\, \tau^{<-k-d}(\CM_1^\alpha)))\right)\\
@VVV   @VVV  \\
\underset{\alpha}\oplus\, H^0\left(\Gamma_{\on{dR}}(\CY,\CM\sotimes \CM_1^\alpha)\right)  @>>>
H^0\left(\Gamma_{\on{dR}}(\CY,\CM\sotimes (\underset{\alpha}\oplus\, \CM_1^\alpha))\right)  \\
@VVV @VVV  \\
\underset{\alpha}\oplus\, H^0\left(\Gamma_{\on{dR}}(\CY,\CM\sotimes \tau^{\geq -k-d}(\CM_1^\alpha))\right)  @>>> 
H^0\left(\Gamma_{\on{dR}}(\CY,\CM\sotimes (\underset{\alpha}\oplus\, \tau^{\geq-k-d}(\CM_1^\alpha)))\right) \\
@VVV   @VVV  \\
\underset{\alpha}\oplus\, H^1\left(\Gamma_{\on{dR}}(\CY,\CM\sotimes \tau^{<-k-d}(\CM_1^\alpha))\right)  
@>>> H^1\left(\Gamma_{\on{dR}}(\CY,\CM\sotimes (\underset{\alpha}\oplus\, \tau^{<-k-d}(\CM_1^\alpha)))\right).
\endCD
$$

The top and the bottom horizontal arrows are maps between zero objects by assumption. Hence, the middle
vertical arrows in both columns are isomorphisms. 
The second from the bottom horizontal arrow is an isomorphism
by \lemref{p:coconnective part Dmod}. Hence, the second from the top 
horizontal arrow is also an isomorphism, as required.

\medskip

Let us now prove point (b). We shall show that under the assumptions on $\CM$, the functor
$\CM_1\mapsto \Gamma_{\on{dR}}(\CY,\CM\sotimes \CM_1)$ is \emph{right} t-exact, up to a cohomological
shift, thereby reducing the assertion to point (a). 

\medskip

Let $n$ be the integer such that 
$$H^i\left(\Gamma_{\rd}(\CY,\CM\sotimes \CM_1)\right)=0$$ 
for all $i>n$ and $\CM_1\in \Dmod(\CY)^{\leq 0}$. Such an integer exists because $\CM$ is bounded and 
the functor $\Gamma_{\rd}(\CY,-)$ is \emph{right t-exact}
up to a cohomological shift. 

\medskip

We will show that the same integer works for $\Gamma_{\on{dR}}(\CY,-)$, i.e.,
$$H^i\left(\Gamma_{\on{dR}}(\CY,\CM\sotimes \CM_1)\right)=0$$ 
for all $i>n$ and $\CM_1\in \Dmod(\CY)^{\leq 0}$. 

\medskip

First, we claim that it is sufficient to show this for
$\CM_1\in \Dmod(\CY)^\heartsuit$. Indeed, it is clear that the assertion for $\CM_1\in \Dmod(\CY)^\heartsuit$
implies the assertion for all $\CM_1\in \Dmod(\CY)^b\cap \Dmod(\CY)^{\leq 0}$. In general,
we use the fact that the functor $\Gamma_{\on{dR}}(\CY,-)$ commutes with limits and the fact that
for $\CM\in \Dmod(\CY)^b$ and any $\CM_1$, the map
$$\CM\sotimes \CM_1\to \underset{m}{\underset{\longleftarrow}{lim}}\, (\CM\sotimes \tau^{\geq -m}(\CM_1))$$
is an isomorphism (which in turn follows from the fact that the t-structure on $\Dmod(\CY)$ is left-complete,
and the functor $\CM\sotimes -$ is of bounded cohomological amplitude).

\medskip

Next, by \lemref{p:coconnective part Dmod}(b), we can assume that $\CM_1\in \Dmod(\CY)^\heartsuit\cap \Dmod_{\on{coh}}(\CY)$.
We will show that
$$H^i\left(\Gamma_{\on{dR}}(\CY,\CM\sotimes \CM_1)\right)=0$$ 
for all $i>n$ and $\CM_1\in \Dmod_{\on{coh}}(\CY)^{\leq 0}$. 

\medskip

We shall use the following lemma, proved in \secref{sss:proof of approx}:

\begin{lem}  \label{l:approx} 
Let $\CY$ be a QCA stack, and $\CN$ an object of 
$\Dmod_{\on{coh}}(\CY)$.

\smallskip

\noindent{\em(a)} For a given integer $k$ there exists $\CN^k\in \Dmod(\CY)^c$ equipped
with a map $\CN^k\to \CN$, whose cone belongs to $\Dmod(\CY)^{\leq -k}$.

\smallskip

\noindent{\em(b)} For a given integer $k$ there exists $\CN^k\in \Dmod(\CY)^c$ equipped
with a map $\CN\to \CN^k$, whose cone belongs to $\Dmod(\CY)^{\geq k}$.

\end{lem}

For $\CM_1\in \Dmod_{\on{coh}}(\CY)^{\leq 0}$, let $\CM_1\to \CM_1^k$ be as in \lemref{l:approx}(b). 
Let $$\CL^k:=\on{Cone}(\CM_1\to \CM_1^k).$$ Consider the diagram
$$
\CD
\Gamma_\rd(\CY,\CM\sotimes \CM_1)   @>>> \Gamma_{\on{dR}}(\CY,\CM\sotimes \CM_1)  \\
@VVV  @VVV  \\
\Gamma_\rd(\CY,\CM\sotimes \CM_1^k)   @>>> \Gamma_{\on{dR}}(\CY,\CM\sotimes \CM_1^k)  \\
@VVV    @VVV \\
\Gamma_\rd(\CY,\CM\sotimes \CL^k)   @>>> \Gamma_{\on{dR}}(\CY,\CM\sotimes \CL^k), 
\endCD
$$
in which the columns are exact triangles. The middle horizontal arrow is an isomorphism
by \corref{c:pairing on compact}. We now claim that for $j=i$ and $i-1$ (or any finite range
of indices) and $k\gg 0$, both 
$$H^j\left(\Gamma_\rd(\CY,\CM\sotimes \CL^k)\right) \text{ and }
H^j\left(\Gamma_{\on{dR}}(\CY,\CM\sotimes \CL^k)\right)$$
are zero. Indeed, for $\Gamma_{\on{dR}}(\CY,\CM\sotimes \CL^k)$ this follows from 
\lemref{p:coconnective part Dmod}. For $\Gamma_\rd(\CY,\CM\sotimes \CL^k)$ this follows
on the assumption on $\CM$. 

\medskip

Hence, $H^i(\Gamma_\rd(\CY,\CM\sotimes \CM_1))\to H^i(\Gamma_{\on{dR}}(\CY,\CM\sotimes \CM_1))$ 
is an isomorphism, and the assertion follows.

\end{proof}

\sssec{Proof of \lemref{l:approx}}   \label{sss:proof of approx} To prove point (a), the usual argument reduces the assertion to the following one: 

\medskip

\noindent For
$\CN\in \Dmod(\CY)^\heartsuit\cap \Dmod_{\on{coh}}(\CY)$, there exists an object
$\CN_0\in \Dmod(\CY)^{\leq 0}\cap \Dmod(\CY)^c$ and a \emph{surjective} map in $\Dmod(\CY)^\heartsuit$:
$$H^0(\CN_0)\to \CN.$$

Write
$$\oblv_{\Dmod(\CY)}(\CN)=\underset{\alpha}\cup\, \CF_\alpha, \quad \CF_\alpha\in \Coh(\CY)^\heartsuit.$$
The objects $\ind_{\Dmod(\CY)}(\CF_\alpha)$ are compact, and for some $\alpha$ the resulting map
$$\ind_{\Dmod(\CY)}(\CF_\alpha)\to \ind_{\Dmod(\CY)}(\oblv_{\Dmod(\CY)}(\CN))\to \CN$$
will induce a surjection on $H^0$.

\medskip

Point (b) is obtained from point (a) by Verdier duality, using the fact that 
$$\BD_\CY^{\on{Verdier}}\left(\on{D-mod}_{\on{coh}}(\CY)^{\leq 0}\right)\subset \on{D-mod}_{\on{coh}}(\CY)^{\geq -d}$$
for some integer $d$ depending only on $\CY$. 

\qed

\ssec{Expressing ($\dr$)-pushforward through the renormalized version}

\sssec{}

Let $\pi:\CY_1\to \CY_2$ be again a morphism between QCA stacks. One can regard the functor
$\pi_{\blacktriangle}$ as a fundamental operation, and wonder whether one can recover the functor
$\pi_\dr$ intrinsically through it. 

\medskip

The latter turns out to be possible, once we take into account the t-structure on $\Dmod(\CY_i)$, and below 
we explain how to it.

\sssec{}

First, according to \lemref{l:de Rham almost cocont}, the functor $\pi_\dr$ can be recovered from its restriction
to $\Dmod(\CY_1)^{\geq -n}$ for every fixed $n$, by taking the limit of its values on the truncations. 

\medskip

Second, according to \propref{p:coconnective part Dmod}(b), the restriction of $\pi_\dr$ to $\Dmod(\CY_1)^{\geq -n}$ 
commutes with filtered colimits, while $\Dmod(\CY_1)^{\geq -n}$ is generated under filtered colimits by the subcategory
$\Dmod(\CY_1)^{\geq -n}\cap \Dmod_{\on{coh}}(\CY_1)$. 

\medskip

Hence, it remains to show how to express 
$\pi_\dr|_{\Dmod_{\on{coh}}(\CY_1)}$ in terms of $\pi_{\blacktriangle}|_{\Dmod_{\on{coh}}(\CY_1)}$.

\sssec{}

Let $\CN$ be an object of $\Dmod_{\on{coh}}(\CY_1)$. For an integer $k\gg 0$, let $\CN\to \CN^k$ 
be as in \lemref{l:approx}(b). Note that since $\CN^k$ is compact, the map
$$\pi_{\blacktriangle}(\CN^k)\to \pi_\dr(\CN^k)$$
is an isomorphism.

\medskip

From \propref{p:coconnective part Dmod}(b) we obtain:

\begin{lem} There exists an integer $m$, depending only on $\pi$, such that
the map $\CN\to \CN^k$ induces an isomorphism
$$\tau^{\leq k-m}(\pi_\dr(\CN))\to \tau^{\leq k-m}(\pi_\dr(\CN^k)).$$
\end{lem}

\sssec{}

The above procedure can be summarized as follows: 

\begin{prop} \hfill

\smallskip

\noindent{\em(a)} The functor $\pi_\dr$ maps isomorphically to the right Kan extension of its retsriction to
$\Dmod(\CY_1)^+$. 

\smallskip

\noindent{\em(b)} For every $n$, the restriction of $\pi_\dr$ to $\Dmod(\CY_1)^{\geq -n}$ receives an
isomorphism from the left Kan extension of its further restriction to 
$\Dmod(\CY_1)^{\geq -n}\cap \Dmod_{\on{coh}}(\CY_1)$.

\smallskip

\noindent{\em(c)} The restriction of $\pi_\dr$ to $\Dmod_{\on{coh}}(\CY_1)$ maps isomorphically
to the right Kan extension of its further restriction to $\Dmod(\CY_1)^c$.

\end{prop}

Recall that the restrictions of $\pi_\dr$ and $\pi_{\blacktriangle}$ to
$\Dmod(\CY_1)^c$ are canonically equivalent. So the above proposition
indeed expresses $\pi_\dr$ in terms of $\pi_{\blacktriangle}$.

\section{Geometric criteria for safety}  \label{s:3Dmods}

\subsection{Overview of the results} 
The results of this section have to do with a more explicit description of 
the subcategory of safe objects in $\Dmod(\CY)$. By \propref{compactness via safety}, this decsription
will characterize the subcategory $\Dmod(\CY)^c$ inside $\on{D-mod}_{\on{coh}}(\CY)$.

\medskip

We will introduce a notion of \emph{safe} algebraic stack (see Definition~\ref{d:safe}). We will show that a quasi-compact algebraic
stack $\CY$ is safe if and only if all objects of $\Dmod(\CY)$ are safe. In particular, for a quasi-compat $\CY$, the equality 
$\on{D-mod}(\CY)^c= \on{D-mod}_{\on{coh}}(\CY)$ holds if and only if $\CY$ is safe.

\medskip

For an arbitrary QCA stack $\CY$ we shall formulate an explicit safety criterion for objects of $\on{D-mod}(\CY )$ (Theorem~\ref{criter for safety}). 
Note that safety for objects can be checked strata-wise (see Corollary~\ref{stratification}).

\medskip

This section is organized as follows. In Sect.~\ref{ss:safe_stacks} we formulate the results and give some easy proofs. 
The more difficult Theorems~\ref{t:rel_Dmods} and \ref{criter for safety} are proved in Sect.~\ref{ss:rel_Dmods}-\ref{ss:proof_key}.

\medskip

As we shall be only interested in the categorical aspects of $\Dmod(\CY)$, with no restriction of generality
we can assume that all schemes and algebraic stacks discussed in this section are classical. 

\medskip

\noindent{\bf Change of conventions:} For the duration of this section ``prestack" will mean 
``classical prestack", and ``algebraic stack" will mean ``classical algebraic stack".

\ssec{Formulations}   \label{ss:safe_stacks}

\sssec{Safe algebraic stacks and morphisms}

\begin{defn}  \label{d:safe}  \hfill

\smallskip

\noindent{\em(a)}
An algebraic stack $\CY$ is {\em locally safe\,} if for every geometric point $y$ of $\CY$ 
the neutral connected component of its automorphism group,
$\on{Aut} (y)$, is unipotent.
%has trivial \emph{connected reductive
%component}, i.e., the quotient of $\on{Aut} (y)$ by its unipotent radical is finite.
\smallskip

\noindent{\em(b)} A morphism of algebraic stacks is {\em locally safe\,} if all its geometric fibers are. 

\smallskip

\noindent{\em(c)} 
An algebraic stack (resp. a morphism of algebraic stacks) is \emph{safe} if it is quasi-compact and
locally safe.
\end{defn}

\begin{rem}   \label{r:safe&*}
A safe algebraic stack is clearly QCA in the sense of Definition~\ref{d:QCA}.
\end{rem}

\begin{thm}  \label{t:rel_Dmods}
Let $\pi:\CY\to \CY'$ be a quasi-compact morphism of algebraic stacks. Then the functor $\pi_{\on{dR},*}$ is continuous if
and only if $\pi$ is safe. In the latter case $\pi_{\on{dR},*}$ strongly satisfies the projection formula.
\footnote{See \secref{ss:proj formula Dmod} for the explanation of what this means.}
\end{thm}

This theorem is proved in Sect.~\ref{ss:rel_Dmods} below.

\begin{cor}
If $\pi$ is safe, the canonical map
$$\pi_{\blacktriangle}\to \pi_{\on{dR},*}$$ is an isomorphism.
\end{cor}

\begin{proof}
Both functors are continuous,
and the map in question is an isomorphism on compact objects by \propref{p:ren dir image of safe}.
\end{proof}

\begin{cor}  \label{c:more precise}
Let $\CY$ be a quasi-compact stack. 
Then the functor $\Gamma_{\on{dR}}(\CY,-)$ is continuous if and only if $\CY$ is safe. 
\end{cor}

\begin{cor}   \label{c:safe_stacks}
The following properties of a quasi-compact algebraic stack $\CY$ are equivalent:
\begin{enumerate}
\item[(i)] $\on{D-mod}(\CY)^c= \on{D-mod}_{\on{coh}}(\CY)$;
\item[(ii)] $k_{\CY}\in \on{D-mod}(\CY)^c\,$;
\item[(iii)] The functor $\Gamma_{\on{dR}}(\CY,-)$ is continuous;
\item[(iv)] All objects of $\on{D-mod}(\CY)$ are safe.
\item[(v)] $\CY$ is safe.

\end{enumerate}
\end{cor}

\begin{proof}
By \corref{c:more precise}, (iii)$\Leftrightarrow$(v). Since $\CMaps(k_{\CY},-)=\Gamma_{\on{dR}}(\CY,-)$ 
we have (ii)$\Leftrightarrow$(iii). Clearly (i)$\Rightarrow$(ii). The equivalence (iii)$\Leftrightarrow$(iv)
is tautological. It remains to prove that 
(iii)$\Rightarrow$(i). 

\medskip

The problem is to show that any $\CM\in \on{D-mod}_{\on{coh}}(\CY)$ is compact, i.e.,
the functor $\CMaps(\CM,-)$ is continuous. This follows from (iii) and the formula
$$\CMaps_{\Dmod(\CY)}(\CM,\CM')\simeq \Gamma_{\on{dR}}(\CY,\BD_\CY^{\on{Verdier}}(\CM)\sotimes \CM'),
\quad\quad \CM\in \on{D-mod}_{\on{coh}}(\CY), \,\CM'\in \on{D-mod}(\CY),$$
which is the content of Lemma~\ref{l:Verdier_stacks}.
\end{proof}

\sssec{Characterization of safe objects of $\on{D-mod}(\CY)$} \label{ss:safe_objects}

Let now $\CY$ be a QCA algebraic stack (in particular, it is quasi-compact).

\begin{thm}    \label{criter for safety} %\hfill
Let $\CY$ be a QCA algebraic stack and $\CM\in \on{D-mod}(\CY)^b$. Then
the following conditions are equivalent:

\begin{enumerate}
\item $\CM$ is safe;

\smallskip

\item For any schematic quasi-compact morphism $\pi:\CY'\to \CY$
and any morphism $\varphi:\CY'\to S$ with $S$ being a quasi-compact scheme, the object
$\varphi_{\on{dR},*}\left(\pi^!(\CM)\right))\in \on{D-mod}(S)$
belongs to $\on{D-mod}(S)^b$; 

\smallskip

\item For any schematic quasi-compact morphism $\pi:\CY'\to \CY$
and any morphism $\varphi:\CY'\to S$ with $S$ being a quasi-compact scheme, the object
$\varphi_{\blacktriangle}\left(\pi^!(\CF)\right))\in \on{D-mod}(S)$
belongs to $\on{D-mod}(S)^b$; 

\smallskip

\item  For any schematic quasi-compact morphism $\pi:\CY'\to \CY$
and any morphism $\varphi:\CY'\to S$ with $S$ being a quasi-compact scheme, 
the canonical morphism 
$$\varphi_{\blacktriangle}\left(\pi^!(\CM)\right))\to\varphi_{\on{dR},*}\left(\pi^!(\CM)\right))$$
is an isomorphism;

\smallskip

\item[(2$'$)] - $(4'):$ same as in {\em (2)-(4)}, but $\pi$ is required to be a finite \'etale
map onto a locally closed substack of $\CY$. 

\smallskip

\item  
$\CM$ belongs to the smallest (non-cocomplete) DG subcategory 
$\CT (\CY )\subset \on{D-mod}(\CY)$ containing all objects of the form $f_{\on{dR},*}(\CN)$, 
where $f:S\to\CY$ is a morphism with $S$ being a quasi-compact scheme and $\CN\in \on{D-mod}(S)^b$.
\end{enumerate}
\smallskip
%\noindent{\em(B)}
\end{thm}

\begin{rem}
Note, however, that the subcategory of safe objects in $\Dmod(\CY)$ is \emph{not} preserved by
the truncation functors.
\end{rem}

\ssec{Proof of Theorem~\ref{t:rel_Dmods}}  \label{ss:rel_Dmods}

\sssec{If $\pi_{\on{dR},*}$ is continuous then $\pi$ is safe.}   \label{sss:easy_rel_Dmods}

Up to passing to a field extension, we have to show that for any point $\xi:\on{pt}\to\CY'$
and any $k$-point $y$ of the fiber $\CY_{\xi}$, the group $G:=\on{Aut}(y)$ cannot contain
%$\BG_m$. Let $H\subset G$ be any algebraic subgroup. 
a connected non-unipotent\footnote{If $G$ were assumed affine, then ``non-unipotent" could be replaced by ``isomorphic to $\BG_m$". 
Accordingly, at the end of Sect.~\ref{sss:easy_rel_Dmods} it would suffice to refer to the 
example of \secref{sss:BG_m} instead of the example of ~\secref{ss:BG}. } algebraic subgroup $H\subset G$.
We have a commutative diagram
$$
\CD 
BH @>{f}>>  \CY \\
@V{p}VV   @VV{\pi}V  \\
\on{pt}  @>{\xi}>>  \CY'
\endCD
$$
in which $f$ is the composition $BH\to BG\hookrightarrow \CY_{\xi}\to\CY'$.
By assumption,
$\pi_{\on{dR},*}$ is continuous. By Sect.~\ref{sss:representable}, $f_{\on{dR},*}$ is also continuous
since $f$ is schematic and quasi-compact.
So the composition $\pi_{\on{dR},*}\circ f_{\on{dR},*}=\xi_{\on{dR},*}\circ p_{\on{dR},*}$ is continuous. But $\xi_{\on{dR},*}$ is
continuous (by Sect.~\ref{sss:representable}) and conservative (e.g., compute $\xi^!\circ \xi_{\on{dR},*}$ by base change). Therefore $p_{\on{dR},*}$ is continuous. 
%So $H\not\simeq\BG_m$ by Example~\ref{ex:BG_m}. 
This contradicts the Example of \secref{ss:BG}.

\qed

To prove the other statements from Theorem~\ref{t:rel_Dmods}, we need to introduce some
definitions.

\sssec{Unipotent group-schemes}

Let $\CX$ be a prestack. A group-scheme over $\CX$ is a group-like object 
$\CG\in \on{PreStk}_{/\CX}$, such that the structure morphism 
$\CG\to \CX$ is schematic.

\medskip

We shall say that $\CG$ is unipotent if its pullback to any scheme gives a unipotent group-scheme
over that scheme (a group-scheme is said to be unipotent if its geometric fibers are unipotent). 

\medskip

If $\CG$ is smooth
and unipotent, then the exponential map defines an isomorphism between $\CG$ and the
vector group-scheme of the corresponding sheaf of Lie algebras, as objects of $\on{PreStk}_{/\CX}$.
This fact is stated in \cite[Sect. XV.3 (iii)]{Ra} without a proof, although the proof is not difficult. 
\footnote{For our purposes, it will suffice to know that this fact when $\CX$ is a scheme, generically
on $\CX$, in which case it is obvious.}

\begin{lem}  \label{l:unip pullback}
If $\CG$ is a smooth unipotent group-scheme over $\CX$, then the $!$-pullback functor 
$\Dmod(\CX)\to \Dmod(\CG)$ is fully faithful.
\end{lem}

\begin{proof}

By the definition of $\Dmod$ on prestacks, it is sufficient to prove this fact when $\CX=S$ is an affine DG scheme. 
Further Zariski localization reduces us to the
fact that the pullback functor 
$$\Dmod(S)\to \Dmod(S\times \BA^n)$$
is fully faithful.\footnote{The reader who is not willing to use the isomorphism given by the exponential
map on all of $S$, can prove the lemma by subdividing $S$ into strata.}

\end{proof}

\sssec{Unipotent gerbes}

\begin{defn}       \label{d:unipotent_gerbe}
We say that a morphism of prestacks $\CZ\to\CX$ is a \emph{unipotent gerbe} if there exists
an fppf cover $\CX'\to\CX$ such that $\CZ':=\CZ\underset{\CX}\times \CX'$ is isomorphic
to the classifying stack of a smooth unipotent group-scheme over $\CX'$.
\end{defn}

\begin{lem} \label{l:unipotent_gerbes} 
Let $\pi:\CZ\to \CX$ be a unipotent gerbe. Then the functor
$$\pi^!:\on{D-mod}(\CX)\to \on{D-mod}(\CZ)$$
is an equivalence.
\end{lem}

\begin{proof}
The statement is local in the fppf topology on $\CX$, 
so we can assume that $\CZ=B\CG$ for some smooth unipotent group-scheme $\CG$ over $\CX$. Then 
$$\on{D-mod}(\CZ)\simeq \on{Tot}\left(\on{D-mod}(\CZ^\bullet/\CX)\right),$$
where $\CZ^\bullet/\CX$ is the \v{C}ech nerve of $\CZ\to \CX$. 

\medskip

%Since $\CG$ is unipotent the functors $\on{D-mod}(Z^n)\to \on{D-mod}(Z^{n+1})$ corresponding to the face maps 
%$Z^{n+1}\to Z^n$ are fully faithful. The functors $\on{D-mod}(Z^0)\to \on{D-mod}(Z^1)$ corresponding to the two face maps 
%$Z^1\to Z^0$ are isomorphic (because the maps are equal). So the functor $\on{Tot}\left(\on{D-mod}(Z^\bullet)\right)\to 
%\on{D-mod}(Z^0)$ is an equivalence.
Each of the $n+1$ face maps $\CZ^n/\CX\to \CZ^0/\CX$ identifies with the natural projection $$p_n:\CG^{\times n}\to \CX,$$
where $\CG^{\times n}=\CG\underset{\CX}\times...\underset{\CX}\times \CG$.

\medskip

Since $\CG$ is unipotent, by \lemref{l:unip pullback}
the functor
$p_n^!:\on{D-mod}(\CX)\to \on{D-mod}(\CG^{\times n})$ is fully faithful. I.e., $p_n^!$ 
identifies $\on{D-mod}(\CX)$ with a full subcategory
$\CC^n\subset \on{D-mod}(\CG^{\times n})$. Therefore,
$$\on{D-mod}(\CZ)\simeq \on{Tot}\left(\on{D-mod}(\CZ^\bullet/\CX)\right) \simeq \on{Tot}\left(\CC^\bullet\right)\simeq \on{D-mod}(\CX).$$
\end{proof}

Assume that in the situation of \lemref{l:unipotent_gerbes}, $\CZ$ and $\CX$ were algebraic stacks. In this case
the functor $\pi_\dr:\on{D-mod}(\CZ)\to \on{D-mod}(\CX)$ is defined.

\begin{cor}  \label{c:unipotent_gerbes} 
Suppose that $\CZ$ and $\CX$ are algebraic stacks, and $\pi$ is equidimensional.
Under these circumstances, the functor $\pi_\dr$ is the inverse of 
$\pi^!$, up to a cohomological shift.
\end{cor}

\begin{proof}
Since $\pi$ is smooth, the functor $\pi^*_{\on{dR}}$ is defined and is the left adjoint of $\pi_\dr$.
The assertion now follows from the fact that $\pi^*_{\on{dR}}$ is isomorphic to $\pi^!$, up to
a cohomological shift, see \secref{sss:smooth pullback stacks}.
\end{proof}

\sssec{Nice open substacks}
To proceed with the proof of Theorem~\ref{t:rel_Dmods} and also for Theorem \ref{criter for safety}, we need the following variant of 
Lemma~\ref{LM}.

\begin{lem}  \label{l:varLM}
Let $\CY\ne\emptyset$ be a reduced classical algebraic stack over a field of characteristic 0 such that the automorphism group of 
any geometric point of $\CY$ is affine. Then there exists a diagram
\begin{equation}  \label{e:varLM}
\begin{array}{ccccc}
&&\CZ&\to &X\times BG \\
&&\downarrow&& \\
\CY&\supset&\oCY&&
\end{array}
\end{equation}
in which
\begin{itemize}
\item %\noindent 
$\oCY\subset\CY$ is a non-empty open substack;
\item
the morphism $\pi:\CZ\to\oCY$ is schematic, finite, surjective, and \'etale;
\item
$X$ is a scheme;
\item
$G$ is a connected reductive algebraic group over $k$;
\item
the morphism $\psi:\CZ\to X\times BG$ is a unipotent gerbe (in the sense of Definition~\ref{d:unipotent_gerbe}).
\end{itemize} 
\end{lem}

\begin{rem}   \label{r:Gtriv}
If $\CY$ is safe then $G$ clearly has to be trivial.
\end{rem}

\begin{proof}
Let $\oCY$ be an open among the locally closed substacks given by Lemma~\ref{LM}. Let
$\oCY\to X'$, $X\to X'$ and $\CG$ be the corresponding data supplied by that lemma.

\medskip

Since we are in characteristic $0$, the group-scheme $\CG$ is smooth over $X$ by Cartier's theorem. 
After shrinking $X'$ and $X$ we can assume that $\CG$ is affine over $X$.
After further shrinking, we can assume that the group-scheme $\CG$ admits a factorization
\begin{equation} %\label{factor group-scheme}
1\to \CG_{un}\to \CG\to \CG_{red}\to 1,
\end{equation}
where $\CG_{un}$ and $\CG_{red}$ are smooth group-schemes with
$\CG_{un}$ being unipotent and $\CG_{red}$ being reductive and locally constant.
After replacing $X$ by a suitable \'etale covering, $\CG_{red}$ becomes constant, i.e., isomorphic to
$X\times G$ for some reductive algebraic group over $k$. 

\medskip

Now set 
$$\CZ:=\oCY\underset{X'}\times X=B\CG.$$ 
We have a morphism $\CZ=B\CG \to B\CG_{red}=X\times BG$. Thus we get a diagram \eqref{e:varLM},
which has the required properties except that $G$ is not necessarily connected. Finally, replace $G$ by its
neutral connected component $G^{\circ}$ and replace $\CZ$ by $\CZ\underset{BG}\times BG^{\circ}$.
\end{proof}

\sssec{Proof of \thmref{t:rel_Dmods}}  \label{sss:qcompact} 

By definition, we may assume that $\CY'=S$ is an affine DG scheme. 
In this case $\CY$ is quasi-compact (because $\pi$ is). So by Noetherian induction, we can assume that
the theorem holds for the restriction of $\pi$ to any closed substack $\imath:\CX\hookrightarrow\CY$, $\CX\ne\CY$.
Take $\CX:=(\CY-\oCY)$, where $\oCY$ is as in Lemma~\ref{l:varLM}. Then the exact triangle
\[\imath_\dr(\imath^!(\CM))\to \CM\to \jmath_\dr(\jmath^!(\CM)), \quad\quad  \CM\in \on{D-mod}(\CY ), \; \jmath:\oCY\hookrightarrow\CY\]
shows that it suffices to prove the theorem for $\pi|_{\oCY}\,$. \footnote{Note that this step relies in Propositions 
\ref{p:de Rham trans} and \ref{p:dr for sch}.}

\medskip

The morphism $p:\CZ\to\oCY$ is schematic, finite, surjective, and etale, so the functor $p_\dr\circ p^!$, which is
isomorphic to $p_\dr\circ p^*_{\on{dR}}$, 
contains $\on{Id}_{\on{D-mod}(\CY )}$ as a direct summand. Therefore it suffices to prove the theorem for the composition
\begin{equation}   \label{e:the_composition}
\CZ\to\oCY\hookrightarrow\CY{\buildrel{\pi}\over{\longrightarrow}}S. 
\end{equation}
Using Remark~\ref{r:Gtriv} and the assumption that $S$ is a scheme, we can decompose 
the morphism \eqref{e:the_composition} as 
$$\CZ{\buildrel{f}\over{\longrightarrow}}X{\buildrel{g}\over{\longrightarrow}}S,$$
where $f$ is the canonical map $\CZ\to X$.

\medskip

It remains to show that each of the
functors $f_\dr$ and $g_\dr$ has the properties stated in the theorem. This is clear for
$g$ as it is a morphism between quasi-compact schemes (see \secref{ss:de Rham pushforward schemes}).
For $f$, this follows from \lemref{l:unipotent_gerbes}. 

\qed

\ssec{Proof of Theorem~\ref{criter for safety}}   \label{ss:proof of safety_crit}

\sssec{Stability of safety}   \label{sss:safetystability}

%Before attacking \propref{criter for safety}, let us establish some elementary properties 
%of the safety condition.

%\medskip

 %Projection formula for (quasi-compact) schematic morphisms 
% (see Sect.~\ref{sss:representable}) implies:

\begin{lem} \label{l:safety under functors} Let $\pi:\CY_1\to \CY$ be a morphism of QCA stacks.
%(quasi-compact) schematic morphism.

%\smallskip
\begin{enumerate}
%\noindent{\em(a)}
\item[(a)] If $\CM_1\in \on{D-mod}(\CY_1)$ is safe then so is $\pi_{\on{dR},*}(\CM_1)\in \on{D-mod}(\CY)$. 

%\smallskip
%
%\noindent{\em(b)}
\item[(b)] If $\pi$ is safe and $\CM\in \on{D-mod}(\CY)$ is safe then so is $\pi^!(\CM)\in \on{D-mod}(\CY_1)$.
\end{enumerate}
%
%\smallskip
%
%\noindent{\em(3)} If $\pi$ is a locally closed embedding, then the condition
%in {\em(1)} is ``if and only if". 
%
%\noindent{\em(1$'$)} If $\pi$ is a locally closed embedding and $\pi_{\on{dR},*}(\CF_1)\in \on{D-mod}(\CY)$
%is safe then so is $\CF_1\in \on{D-mod}(\CY_1)$.
%
%\smallskip
%
%%\noindent{\em(4)} If $\pi$ is a closed embedding, and $\CF$ is supported
%%on $\CY_1$, then the condition in {\em(2)} is ``if and only if". 
%
%\noindent{\em(2$'$)} Suppose that $\pi$ is a closed embedding and $\CF$ is supported
%on $\CY_1$. If $\pi^!(\CF)\in \on{D-mod}(\CY_1)$ is safe then so is $\CF\in \on{D-mod}(\CY)$.
\end{lem}

\begin{proof} \hfill

\smallskip

\noindent (a) We need to show that the functor 
$$\CN\mapsto\Gamma_{\on{dR}} (\CY,\pi_{\on{dR},*}(\CM_1)\sotimes\CN),\quad \CN\in \on{D-mod}(\CY)$$
is continuous. However, by \corref{c:partial proj formula safe}, the right-hand side is isomorphic to
$$\Gamma_{\on{dR}} (\CY,\pi_{\on{dR},*}(\CM_1\sotimes\pi^!(\CN)),$$
i.e., $\Gamma_{\on{dR}} (\CY_1 ,\CM_1\sotimes\pi^!(\CN))$, and the latter is continuous since $\CM_1$ is
safe.

\medskip

\noindent (b) The functor 
$$\CN_1\mapsto\Gamma_{\on{dR}} (\CY_1,\pi^!(\CM)\sotimes\CN_1),\quad \CN_1\in \on{D-mod}(\CY_1)$$
is continuous because the projection formula
$$\Gamma_{\on{dR}} (\CY_1,\pi^!(\CM)\sotimes\CN_1)\simeq
\Gamma_{\on{dR}} (\CY ,\CM\sotimes\,\pi_{\on{dR},*}(\CN_1)),$$
is valid by \lemref{l:proj formula for renormalized}, since $\pi_\dr\simeq \pi_{\blacktriangle}$.
%
%\smallskip
%
%(1$'$) Write $\CF_1$ as $\pi^!(\pi_*(\CF_1))$ and apply statement (2).
%
%\smallskip
%
%(2$'$) Write $\CF$ as $\pi_*(\pi^!(\CF))$ and apply statement (1).
\end{proof}

\begin{cor}  \label{stratification}
Let $\imath_j:\CY_j\hookrightarrow \CY$, $j=1,\ldots,n$, be locally closed substacks 
such that $\CY=\underset{j}\cup\, \CY_j$. Then an object $\CM\in \on{D-mod}(\CY )$ is safe if
and only if $\imath_j^!(\CM)$ is safe for each $j$.
\end{cor}

\begin{proof}
It suffices to consider the case where $n=2$, $\CY_1$ is a closed substack, and 
$\CY_2=(\CY-\CY_1)$. The ``only if" statement holds by \lemref{l:safety under functors}(b).
To prove the ``if" statement, consider the exact triangle 
$(\imath_1)_\dr(\imath_1^! (\CM))\to\CM\to (\imath_2)_\dr(\imath_2^! (\CM))$.
By \lemref{l:safety under functors}, 
$$(\imath_1)_\dr(\imath_1^! (\CM)) \text{ and } (\imath_2)_\dr(\imath_2^!)(\CM)$$ are both safe, so $\CM$ is safe.
\end{proof}

%\begin{lem} \label{l:2unipotent_gerbes} 
%In the situation of Lemma~\ref{l:unipotent_gerbes} the equivalences $\pi^!$
%and $\pi_{\on{dR},*}$ preserve the class of safe objects.
%\end{lem}
%
%\begin{proof}
%This follows from Lemma~\ref{l:unipotent_gerbes} and the fact that $\pi^!$ is a tensor functor.
%\end{proof}

%\sssec{Safety and stratifications}
%
%Let $$\CY=\underset{k}\cup\, \CY_k$$
%be a decomposition of $\CY$ into locally closed substacks. Let
%$\imath_k:\CY_k\hookrightarrow \CY$ denote the corresponding
%locally closed embeddings. 
%
%medskip
%
%%From \lemref{safety under functors} we obtain:
%
%\begin{cor}  %\label{stratification}
%An object $\CF\in \on{D-mod}(\CY)$ is safe if and only if all $\imath_k^!(\CF_k)\in \on{D-mod}(\CY_k)$
%are safe.
%end{cor}
%
%\begin{proof}
%The ``only if" statement follows from \lemref{safety under functors}(2).
%The ``if" statement follows from \lemref{safety under functors}(2$'$) and the fact that safe objects
%form a triangulated subcategory.
%\end{proof}

\sssec{The mapping telescope argument}   \label{sss:telescope}
\begin{lem}  \label{l:telescope}
Let $\CT (\CY )\subset \on{D-mod}(\CY )$ be as in condition (5) of Theorem~\ref{criter for safety}.
Then $\CT (\CY )$ is closed under direct summands.
\end{lem}

\begin{proof}
The subcategory $\CT (\CY )$ has the following property: if $\CM\in \CT (\CY )$ then the infinite direct sum
$$\CM\oplus \CM\oplus \CM\oplus\ldots$$ also belongs to $\CT (\CY )$. Indeed, it 
suffices to check this if $\CM=f_{\on{dR},*}(\CN)$, where $f:S\to\CY$ is a morphism with $S$ being a quasi-compact scheme 
and $\CN\in \on{D-mod}(S)^b\,$. 

\medskip

Now suppose that $\CM\in \CT (\CY )$ and $\CN'\in \on{D-mod}(\CY )$ is a direct summand of $\CM$. Let $p:\CM\to\CM$ be 
the corresponding projector. The usual formula
\[
\CM'=\colim (\CM{\buildrel{p}\over{\longrightarrow}}\CM{\buildrel{p}\over{\longrightarrow}}\CM\to\ldots )=
\Cone (\CM\oplus \CM\oplus \CM\oplus\ldots\to\CM\oplus \CM\oplus \CM\oplus\ldots )
\]
shows that $\CM'\in \CT (\CY )$. 
\end{proof}

\sssec{The key proposition} We shall deduce \thmref{criter for safety} from the following proposition.  

\medskip

Let $X$ be a quasi-compact scheme and $G$ a connected algebraic group.
Consider the algebraic stack $X\times BG$. Let $\varphi :X\times BG\to X$ and $\sigma:X\to X\times BG$ be the
natural morphsims. 

\medskip

Let 
$\CT_X (X\times BG)\subset \on{D-mod}(X\times BG)$ denote the smallest (non-cocomplete) DG subcategory 
containing all objects of the form 
$\sigma_{\on{dR},*}(\CN)$, $\CN\in \on{D-mod}(X)^b\,$. 

\begin{prop}     \label{p:key}
For an object $\CM\in \on{D-mod}(X\times BG)^b$ the following conditions are equivalent:
\begin{enumerate}
\item[(i)] $\CM\in \CT_X (X\times BG)$;
\item[(ii)] $\varphi_{\on{dR},*}(\CM)\in \on{D-mod}(X)^b$;
\item[(iii)] $\varphi_{\blacktriangle}(\CM)\in \on{D-mod}(X)^b$. 
\end{enumerate}
\end{prop}
 
\sssec{Proof of Theorem~\ref{criter for safety} modulo Proposition~\ref{p:key}} \hfill  \label{sss:modulo}

\smallskip

It is clear that (2)$\Rightarrow$(2$'$), (3)$\Rightarrow$(3$'$), (4)$\Rightarrow$(4$'$).

\medskip

The direct image functor preserves boundedness from below (see \lemref{p:coconnective part Dmod}),
while the renormalized direct image functor preserves boundedness from above (see \propref{p:ren dir image estimate}). 
So condition (4) implies (2) and (3), while condition (4$'$) implies (2$'$) and (3$'$).

\medskip

By Lemma~\ref{l:safety under functors}(a), condition (5) implies (1). Condition (1) implies condition (4) by
Lemma~\ref{l:safety under functors}(b) combined with \corref{c:dr and rd of safe}.

\medskip

Thus it remains to prove that (2$'$)$\Rightarrow$(5) and (3$'$)$\Rightarrow$(5). 

\medskip

Let $\CM\in \Dmod(\CY)$
satisfy either (2$'$) or (3$'$). By Noetherian induction and Lemma~\ref{l:safety under functors},
it suffices to show that there exists a non-empty
open substack $\oCY$ of $\CY$, such that the restriction $\CM|_{\oCY}$ satisfies condition (5). 

\medskip

We take $\oCY$ to be as in \lemref{l:varLM}.  Consider the following diagram, in which the square is Cartesian:
\begin{equation} \label{e:LM3}
\CD
\CZ'  @>{\psi'}>>  X \\
@V{\sigma'}VV    @VV{\sigma}V   \\
\CZ  @>{\psi}>> X\times BG \\
@V{\pi}VV   \\
\oCY.
\endCD
\end{equation}
The map $\psi':\CZ'\to X$ is a unipotent gerbe. By further shrinking $X$, we can assume that it
admits a section; denote this section by $g$. 

\medskip

Note that $\CM$ is a direct summand of $\pi_\dr(\pi^!(\CM))$. Hence, by \lemref{l:telescope}, it suffices to show the
following:

\begin{prop}
The object $\pi^!(\CM)\in \Dmod(\CZ)$ belongs to the smallest 
(non-cocomplete) DG subcategory of $\Dmod(\CZ)$ containing all objects of the form
$f_\dr(\CN)$, where $f=\sigma'\circ g$ and $\CN\in \Dmod(X)^b$. 
\end{prop}

\begin{proof}

We apply conditions (2$'$) or (3$'$) with $\CY'=\CZ$, and $\phi$ being the composition
$$\CZ\overset{\psi}\to X\times BG\overset{\varphi}\to X.$$
Consider the object $\psi_\dr(\pi^!(\CM))$. It is bounded because $\pi^!(\CM)$ is bounded,
and $\psi_\dr$ is an equivalence, which is t-exact up to a cohomological shift (see \lemref{l:unipotent_gerbes}). 
Moreover,
$$\psi_{\blacktriangle}(\pi^!(\CM))\simeq \psi_\dr(\pi^!(\CM))$$
because $\psi$, being a unipotent gerbe, is safe.

\medskip

Consider now the objects 
$$\varphi_\dr\left(\psi_\dr(\pi^!(\CM))\right)\simeq \phi_\dr(\pi^!(\CM)) \text{ and }
\varphi_{\blacktriangle}\left(\psi_{\blacktriangle}(\pi^!(\CM))\right)\simeq \phi_{\blacktriangle}(\pi^!(\CM))$$
(note that the first isomorphism uses \secref{sss:more trans} in any of the (i), (ii) or (iii) versions).

\medskip

Condition (2$'$) (resp., (3$'$)) implies that the former (resp., latter) object is in
$\Dmod(X)^b$. Hence, by the implications (ii)$\Rightarrow$(i)   (resp., (iii)$\Rightarrow$(i)) in
\ref{p:key}, 
we obtain that $\psi_\dr(\pi^!(\CM))$ belongs to the subcategory $\CT_X (X\times BG)$.

\medskip

Consider the Cartesian square in \eqref{e:LM3}. Since $\psi_\dr$ is an equivalence
(\lemref{l:unipotent_gerbes}), we obtain that the object 
$\pi^!(\CM)$ belongs to the smallest (non-cocomplete) DG subcategory of $\Dmod(\CZ)$ containing all objects of the form 
$\sigma'_\dr(\CN')$, $\CN'\in \on{D-mod}(\CZ')^b\,$.  

\medskip

Recall that $g$ denotes a section of the map $\psi'$. By \lemref{l:unipotent_gerbes}, $\psi'_\dr$
is an equivalence, and $g_\dr$ is its left inverse. Hence, $g_\dr$ is an equivalence as well. 
So, every object $\CN'\in \on{D-mod}(\CZ')^b$ is of the form
$g_\dr(\CN)$ for $\CN\in \Dmod(X)^b$, which implies the required assertion.

\end{proof}

This finishes the proof Theorem \ref{criter for safety} modulo Proposition~\ref{p:key}.

%Clearly, condition (2$'$) is a particular case of condition (2). In Sect.~\ref{sss:1implies2} below 
%we will show that the implication (1) $\Rightarrow$ (2) is a rather formal consequence of the definition.
%To prove that (2$'$) $\Rightarrow$ (1) we will have to appeal in Sect.~\ref{impl 1 3} to a stratification
%of $\CY$ given by \lemref{LM}.

\ssec{Proof of Proposition~\ref{p:key}} \label{ss:proof_key}
%The implications (ii)$\Leftarrow$(i)$\Rightarrow$(iii) are clear.
We already know from Sect.~\ref{sss:modulo} that (ii)$\Leftarrow$(i)$\Rightarrow$(iii).

\sssec{}  \label{sss:(ii)implies(i)}

As a preparation for the proof of the implication (ii)$\Rightarrow$(i), we observe:

\begin{lem}  \label{l:using_connectedness}
If $\CM\in \on{D-mod}(X\times BG)^{\ge r}$ then $$\Cone \, 
\left(\CM\to \sigma_{\on{dR},*}(\sigma_{\on{dR}}^*(\CM))\right)\in \on{D-mod}(X\times BG)^{\ge r+1}\,.$$
\end{lem}

\begin{proof}
Use that the fibers of $\sigma$ are connected (because $G$ is assumed to be connected).
\end{proof}

\begin{cor}  \label{c:right_resolution}
Let $\CM\in \on{D-mod}(X\times BG )^b$. Then for every $m\in\BZ$ there exists an exact triangle
\begin{equation}   \label{e:triangle}
\CM\to\CE\to\CM'\to\CM[1]
\end{equation}
with $\CE\in\CT_X (X\times BG ),\; \CM'\in \on{D-mod}(X\times BG )^{\ge m}\cap 
\on{D-mod}(X\times BG )^b\,$.
\end{cor}

\sssec{}

Let $\cd (X)$ denote the cohomological dimension of $\on{D-mod}(X)$.
Since $X$ is a quasi-compact scheme  $\cd (X)<\infty$ 
(and in fact, $\cd (X)\le 2\cdot\dim X$). By definition, 
%Denote it by $\cd (X)$. By definition, 
%$\Ext^i(M,N)=0$ if $i>\cd (X)$, $M\in \on{D-mod}(X)^{\ge 0}$, $N\in \on{D-mod}(X)^{\le 0}$. So
\begin{equation}   \label{e:cd}
\Ext^j(\CN,\CL)=0 \;\mbox{ if }\; \CN\in \on{D-mod}(X)^{\ge m},\, \CL\in \on{D-mod}(X)^{\le n},\, j>n-m+\cd (X) \, . 
\end{equation}

\begin{lem}   \label{l:Extvanishing}
Let $\CM$ be an object of $\Dmod(X\times BG)$ such that $\varphi_\dr(\CM)\in \on{D-mod}(X)^{\le n}$, and let $\CM'$ be a bounded object in
$\on{D-mod}(X\times BG )^{\ge m}$. Then 
$$\Ext^i(\CM',\CM)=0 \text{ for }i>n-m+\cd (X)+d,$$ 
where $d:=\dim G$.
\end{lem}

\begin{proof}
We can assume that $\CM'$ lives in a single degree $\ge m$. Then $\CM'=\varphi^*_{\on{dR}}(\CN)[d]$
for some $\CN\in \on{D-mod}(X)^{\ge m}$. Applying \eqref{e:cd} to $\CL=\varphi_\dr(\CM)$ we see that
the group
\[
\Ext^i(\CM',\CM)=\Ext^{i-d}(\varphi^*_{\on{dR}}(\CN),\CM)\simeq\Ext^{i-d}(\CN,\varphi_{\on{dR},*}(\CM))
\]
is zero if $i-d>n-m+\cd (X)$.
\end{proof}

\sssec{}

We are now ready to prove the implication (ii)$\Rightarrow$(i) in Proposition~\ref{p:key}.

\medskip

Suppose that $\varphi_{\on{dR},*}(\CM)\in \on{D-mod}(X)^{\le n}$. Apply Corollary~\ref{c:right_resolution}
for $m=n+\cd (X)+d$. In the corresponding exact triangle \eqref{e:triangle} the morphism 
$\CM'\to\CM[1]$ is homotopic to 0 by Lemma~\ref{l:Extvanishing}. So 
$$\CM\oplus\CM'\simeq\CE\in\CT_X (X\times BG ).$$ 
Now the next lemma implies that
$\CM\in\CT_X (X\times BG )$. 

\begin{lem}  \label{l:2telescope}
The subcategory $\CT_X (X\times BG )\subset \on{D-mod}(X\times BG )$ is closed under direct summands.
\end{lem}

\begin{proof}
The same argument as in the proof of Lemma~\ref{l:telescope}.
\end{proof}

\sssec{Proof of the implication (iii)$\Rightarrow$(i)}  \label{sss:(iii)implies(i)}

\begin{lem}    \label{l:compact_supports}
For any connected algebraic group the functors $$\sigma^!:\on{D-mod}(X\times BG)\to \on{D-mod}(X) 
\text{ and }\varphi^!:\on{D-mod}(X)\to \on{D-mod}(X\times BG)$$ have 
left adjoints $\sigma_!:\on{D-mod}(X)\to \on{D-mod}(X\times BG)$ and 
$\varphi_!:\on{D-mod}(X\times BG )\to \on{D-mod}(X)$. Moreover, 
\[
\varphi_!\simeq\varphi_{\blacktriangle}[2(\dim(G)-\delta)], \quad \sigma_!\simeq \sigma_{\on{dR},*}[\delta-2\dim(G)]\, ,
\]
where $\delta$ is the degree of the highest cohomology group
of $\Gamma_{\on{dR}}(G,k_G)$.
\end{lem}

\begin{proof}

By \corref{c:D on prod}, 
$$\Dmod(X\times BG)\simeq \Dmod(X)\otimes \Dmod(BG),$$
and all functors involved in the lemma are continuous. Hence, they each decompose as
$$\on{Id}_{\Dmod(X)}\otimes \text{ Corresponding functor for $BG$}.$$ 
So, it is sufficient to consider the case when $X=\on{pt}$. The assertion in the latter case
essentially follows from Example \ref{ex:ren for BG}:

\medskip

The fact that $\sigma_!\simeq \sigma_{\on{dR},*}[\delta-2\dim(G)]$ is evident: it suffices
to compute both sides on $k\in \Vect=\Dmod(\on{pt})$. To show that
$$\Gamma_{\on{dR},!}(BG,-):=\varphi_!$$ exists and satisfies
$$\Gamma_{\on{dR},!}(BG,-)\simeq \Gamma_\rd(BG,-)[2(\dim(G)-\delta)],$$
it suffices to show that $\Gamma_{\on{dR},!}(BG,-)$ is defined on the compact generator $\sigma_!(k)$
of $\Dmod(BG)$, and 
$$\Gamma_{\on{dR},!}(BG,\sigma_!(k))\simeq \Gamma_\rd(BG,\sigma_!(k))[2(\dim(G)-\delta)],$$
as modules over $\CMaps_{\Dmod}(\sigma_!(k),\sigma_!(k))$.

\medskip

However, $\Gamma_{\on{dR},!}(BG,\sigma_!(k))\simeq k$, and required isomorphism
was established in Example \ref{ex:ren for BG}:
$$\Gamma_\rd(BG,\sigma_!(k))\simeq \Gamma_\rd(BG,\sigma_\dr(k))[-2\dim(G)+\delta]\simeq
k[-2\dim(G)+\delta].$$

\end{proof}

Lemma~\ref{l:compact_supports} allows to prove the implication (iii)$\Rightarrow$(i) from 
Proposition~\ref{p:key} by mimicking the arguments from Sect.~\ref{sss:(ii)implies(i)}. 
For example, the role of Lemma~\ref{l:using_connectedness} is played by the following

\begin{lem}  \label{l:2using_connectedness}
If $\CM\in \on{D-mod}(X\times BG )^{\le r}$ then 
$$\on{Cone} \, \left(\sigma_!(\sigma^! (\CM))\to\CM\right)[-1]\in \on{D-mod}(X\times BG)^{\le r-1}\,.$$ \qed
\end{lem}

%Corollary~\ref{c:right_resolution}, and 
%Lemma~\ref{l:Extvanishing} is played by the following statements.

%\begin{lem}  \label{l:2using_connectedness}
%If $\CM\in \on{D-mod}(X\times BG )^{\le r}$ then $\Cone \, ( s_! s^! (\CM )\to\CM )[1]\in \on{D-mod}(X\times BG )^{\le r-1}\,$. \qed
%\end{lem}

%\begin{cor}  \label{c:2right_resolution}
%Let $\CM\in \on{D-mod}(X\times BG )^b$. Then for every $n\in\BZ$ there exists an exact triangle
%\begin{equation}   \label{e:2triangle}
%\CM'\to\CE\to\CM\to\CM'[1], \quad\quad \CE\in\CT_X (X\times BG ),\; \CM'\in \on{D-mod}(X\times BG )^{\le n}\cap 
%\on{D-mod}(X\times BG )^b\, .
%\end{equation}
%\end{cor}

%\begin{lem}   \label{l:2Extvanishing}
%Suppose that $\varphi_!(\CM)\in \on{D-mod}(X)^{\ge m}$ and $\CM'$ is a bounded object in
%$\on{D-mod}(X\times BG)^{\le n}$. Then $\Ext^i (\CM,\CM')=0$ for $i>n-m+\cd (X)-d$, where $d:=\dim G$.
%\end{lem}

%\begin{proof}
%We can assume that $\CM'$ lives in a single degree $\le n$. Then $\CM'=\varphi^! (N)[d]$
%for some $N\in \on{D-mod}(X)^{\le n}$. Applying \eqref{e:cd} for $M:=\varphi_!(\CM )$ we see that
%the group
%\[
%\Ext^i(\CM,\CM')=\Ext^{i+d}(\CM ,\varphi^! (N))=\Ext^{i+d}(\varphi_!(\CM),N)
%\]
%is zero if $i+d>n-m+\cd (X)$.
%\end{proof}

\ssec{Proper morphisms of stacks}

\sssec{}

Recall the definition of a \emph{proper} (but not necessarily schematic) morphism between algebraic stacks;
see \cite[Definition 7.11]{LM}. 

\medskip

As a simple application of the theory developed above, in this subsection we will prove the following:

\begin{prop} \label{p:proper map}
Let $\pi:\CY_1\to \CY_2$ be a proper map between algebraic stacks. Then the functor
$\pi_\dr:\Dmod(\CY_1)\to \Dmod(\CY_2)$ sends $\Dmod_{\on{coh}}(\CY_1)\to \Dmod_{\on{coh}}(\CY_2)$. 
\end{prop}

The rest of this subsection is devoted to the proof of the proposition.

\sssec{Step 1}

First, we recall that the definition of properness includes separatedness. This implies that the groups of
automorhisms of points of the geometric fibers of $\pi$ are finite. In particular, $\pi$ is safe. 

\medskip

By \thmref{t:rel_Dmods},
$\pi_\dr$ satisfies base change. This allows to assume that $\CY_2$ is an affine DG scheme. In this case
$\CY_1$ is a safe QCA stack, and by \corref{c:safe_stacks}
$$\Dmod_{\on{coh}}(\CY_1)=\Dmod(\CY_1)^c.$$

Hence, it is enough to show that $\pi_\dr$ sends $\Dmod(\CY_1)^c$ to
$\Dmod(\CY_2)^c$.

\medskip

The category $\Dmod(\CY_1)^c$ is Karoubi-generated by the essential image of $\Coh(\CY_1)$ under the functor
$\ind_{\Dmod(\CY_1)}$. So, it is sufficient to show that the composition $\pi_\dr\circ \ind_{\Dmod(\CY_1)}$ sends
$\Coh(\CY_1)$ to $\Dmod(\CY_2)^c$. 

\medskip

However, by \propref{p:pushforward induced},
$$\pi_\dr\circ \ind_{\Dmod(\CY_1)}\simeq \ind_{\Dmod(\CY_2)}\circ \pi^{\IndCoh}_*.$$
Hence, it is enough to show that the functor
$\pi^{\IndCoh}_*$ sends $\Coh(\CY_1)$ to $\Coh(\CY_2)$. 

\sssec{Step 2}

Consider the functor
$$\pi:\QCoh(\CY_1)\to \QCoh(\CY_2).$$ 

We have a commutative diagram of functors
$$
\CD
\QCoh(\CY_1)^+   @<{\Psi_{\CY_1}}<< \IndCoh(\CY_1)^+  \\
@V{\pi_*}VV   @VV{\pi^{\IndCoh}_*}V   \\
 \QCoh(\CY_2)^+   @<{\Psi_{\CY_2}}<< \IndCoh(\CY_2)^+,
\endCD
$$
where the horizontal arrows are equivalences. 

\medskip

Hence, it suffices to show that $\pi_*$ sends $\Coh(\CY_1)\subset \QCoh(\CY_1)^+$ to
$\Coh(\CY_2)\subset \QCoh(\CY_2)^+$.

\medskip 

By \corref{c:relative}, $\pi_*$ sends $\QCoh(\CY_1)^b$ to $\QCoh(\CY_2)^b$. Hence, it remains to show
that $\pi_*$ sends objects from $\QCoh(\CY_1)^\heartsuit\cap \Coh(\CY_1)$ to objects in $\QCoh(\CY_2)$
with coherent cohomologies. 

\medskip

However, the latter is the content of \cite[Theorem 1]{F} (see also
\cite[Theorem 15.6(iv)]{LM}, combined with \cite[Theorem 1.2]{Ol}).

\section{More general algebraic stacks}  \label{s:gen alg stacks}

\ssec{Algebraic spaces and LM-algebraic stacks}

\sssec{}  \label{sss:alg spaces}

We define the notion of algebraic space as in \cite{Stacks}, Sect. 4.1.1. We shall always impose the condition
that our algebraic spaces be quasi-separated (i.e., the diagonal morphism $\CX\to \CX\times \CX$ is quasi-compact).
\footnote{Note that the diagonal morphism of an algebraic space is always separated. In fact, for any presheaf
of \emph{sets} $\CX$, the diagonal of the diagonal is an isomorphism.}

\medskip

Thus, our definition is equivalent (the DG version) of that of \cite{LM} (this relies on the DG version of Artin's 
theorem about the existence of an \'etale atlas, see Corollary 8.1.1 of \cite{LM}).

\medskip

An algebraic space is an algebraic stack in the sense of the definition of \secref{sss:algebraic stacks}.
Vice versa, an algebraic stack $\CX$ is an algebraic space if and only if the following equivalent conditions hold:

\begin{itemize}

\item The underlying classical stack $^{cl}\CX$ is a sheaf of sets (rather than groupoids).

\item The diagonal map $\CX\to \CX\times \CX$ induces a \emph{monomorphism} at the
level of underlying classical prestacks.

\end{itemize}

\sssec{}

Let us recall that a morphism between prestacks $\pi:\CY_1\to \CY_2$ is called
representable, if its base change by any affine DG scheme yields an algebraic space.

\sssec{LM-algebraic stacks}    \label{sss:LM-algebraic}

We shall now enlarge the class of algebraic stacks as follows. We say that it is \emph{LM-algebraic} if 

\begin{itemize}

\item The diagonal morphism $\CY\to \CY\times \CY$ is representable, %and as such is
quasi-separated, and quasi-compact.

\item There exists a DG scheme $Z$ and a map $f:Z\to \CY$ (automatically representable,
by the previous condition) such that $f$ is smooth and surjective.

\end{itemize}

%\begin{rem}
%As in Remark \ref{r:qs}, it is likely that one can replace the condition 
%that the diagonal morphism $\CY\to \CY\times \CY$ be separated by the weaker
%but more natural one of it being quasi-separated. The only place where this
%condition is used is \lemref{LM}.
%\end{rem}

\sssec{The extended QCA condition}

The property of being QCA makes sense for LM-algebraic stacks. We shall call these objects 
QCA LM-algebraic stacks.

\medskip

We can now enlarge the class of QCA morphisms between prestacks accordingly. We shall say that
a morphism is LM-QCA if its base change by an affine DG scheme yields QCA LM-algebraic stack. 

\ssec{Extending the results}

\sssec{}

The basic observation that we make is that a quasi-compact algebraic space is automatically QCA.  In
particular, we obtain that quasi-compact representable morphisms are QCA. 

\medskip

Note also (for the purposes of considering D-modules) that a quasi-compact algebraic space is
safe in the sense of Definition \ref{d:safe}. In particular, a quasi-compact representable morphism
is safe.

\sssec{}

Let us now recall where we used the assumption on algebraic stacks that the diagonal morphism
$$\CY\to \CY\times \CY$$ 
should be schematic. 

\medskip

In all three contexts ($\QCoh$, $\IndCoh$ and $\Dmod$) we needed the following property. Let $S$
be an affine (or, more generally, quasi-separated and quasi-compact) DG scheme equipped with a 
smooth map $g:S\to \CY$. We considered the naturally defined functors
\begin{align*} 
&g^*:\QCoh(\CY)\to \QCoh(S), \quad g^{\IndCoh,*}:\IndCoh(\CY)\to \IndCoh(S) \text{  and  }\\
&g^*_{\on{dR}}:\Dmod(\CY)\to \Dmod(S).
\end{align*}
We needed these functors to admit \emph{continuous} right adjoints
\begin{align*} 
&g_*:\QCoh(S)\to \QCoh(\CY), \quad g^{\IndCoh}_*:\IndCoh(S)\to \IndCoh(\CY) \text{  and  }\\
&g_\dr:\Dmod(S)\to \Dmod(\CY),
\end{align*}
respectively. 

\medskip

Now, this was indeed the case, because the map $g$ is itself schematic, quasi-separated and quasi-compact. 

\sssec{}

Now, we claim that the same is true for LM-algebraic stacks. Indeed, if $\CY$ is 
an LM-algebraic stack and $S$ is a DG scheme, then any morphism $g:S\to \CY$
is representable, quasi-separated and quasi-compact.

\medskip

In particular, if $S$ is an affine (or, more generally, quasi-separated and quasi-compact) DG scheme,
the morphism $g$ is QCA (and safe). 

\medskip

We obtain that Corollary \ref{c:relative} implies the corresponding fact for $g_*$. 

\medskip

Corollary \ref{c:* adj for indcoh}, applied after a base change by all maps $f:Z\to \CY$ where 
$Z\in \on{DGSch}^{\on{aff}}_{\on{aft}}$, implies the required property of $g^{\IndCoh}_*$. 

\medskip

Finally, \thmref{t:rel_Dmods}, again applied after a base change by all maps $f:Z\to \CY$ where 
$Z\in \on{DGSch}^{\on{aff}}_{\on{aft}}$, implies the required property of $g_\dr$. 

\sssec{}

Another ingredient that went into the proofs of the main results was \propref{p:stratification}.
However, it is easy to see that its proof works for LM-algebraic stacks with no modification.

\medskip

The rest of the ingredients in the proofs are without change.

\sssec{}

In application to the category $\QCoh(-)$, we have the following generalization of \thmref{main}:

\begin{thm} \label{t:main, generalized}  \hfill

\smallskip

\noindent{\em(a)} Suppose that an LM-algebraic stack $\CY$ is QCA. Then the functor
$\Gamma:\QCoh(\CY)\to \Vect$ is continuous. Moreover, there exists an integer $n_\CY$
such that $H^i(\Gamma(\CY,\CF))=0$ for all $i>n_\CY$ for $\CF\in \QCoh(\CY)^{\leq 0}$.

\smallskip

\noindent{\em(b)} Let $\pi:\CY_1\to \CY_2$ be a LM-QCA morphism between prestacks.
%such that its base change by an affine scheme yields a LM-algebraic stack satisfying ($\ast$). 
Then the functor $\pi_*:\QCoh(\CY_1)\to \QCoh(\CY_2)$ is continuous.
\end{thm}

\sssec{}

In application to $\IndCoh$, we have:

\begin{thm} \label{IndCoh, generalized} 
Suppose that an LM-algebraic stack $\CY$ is QCA. Then 
the category $\IndCoh(\CY)$ is compactly generated, and its subcategory
of compact objects identifies with $\Coh(\CY)$.
\end{thm}

In particular, the statements of Corollary \ref{c:indcoh on product} and 
Theorem~\ref{t:QCoh dualizable} hold for LM-algebraic stacks as well. 

\sssec{}

In application to D-modules, we have:

\begin{thm} \label{t:Dmods, generalized} \hfill

\smallskip

\noindent{\em(a)}
If  an LM-algebraic stack $\CY$ is QCA then the category $\on{D-mod}(\CY)$ is compactly
generated. An object of $\on{D-mod}_{\on{coh}}(\CY)$ is compact if and only if it is safe. 

\smallskip

\noindent{\em(b)} Let $\pi:\CY_1\to \CY_2$ be a quasi-compact morphism between LM-algebraic stacks.
Then the functor $\pi_{\on{dR},*}$ is continuous if and only if $\pi$ is safe. 

\end{thm}

Note that in \thmref{criter for safety}(2)-(4) we can replace the words ``schematic" by "representable".

\end{document}